\documentclass[12pt,reqno]{amsart}

\usepackage{amsthm, amsmath, amsfonts, amssymb, graphicx, tikz-cd, makecell, comment}
\usepackage[shortlabels]{enumitem}
\usetikzlibrary{backgrounds}
\usepackage[hidelinks]{hyperref}
\usepackage[capitalize]{cleveref}
\usepackage{bbm}
\usepackage{needspace}

\newtheorem{theorem}{Theorem}[section]
\newtheorem{lemma}[theorem]{Lemma}
\newtheorem{proposition}[theorem]{Proposition}
\newtheorem{corollary}[theorem]{Corollary}

\theoremstyle{definition}
\newtheorem{definition}[theorem]{Definition}

\newtheorem{remark}[theorem]{Remark}
\newtheorem{claim}[theorem]{Claim}
\newtheorem{notation}[theorem]{Notation}

\crefname{theorem}{Theorem}{Theorems}
\crefname{lemma}{Lemma}{Lemmas}
\crefname{proposition}{Proposition}{Propositions}
\crefname{corollary}{Corollary}{Corollaries}
\crefname{conjecture}{Conjecture}{Conjectures}
\crefname{definition}{Definition}{Definitions}
\crefname{example}{Example}{Examples}
\crefname{question}{Question}{Questions}
\crefname{remark}{Remark}{Remarks}
\crefname{claim}{Claim}{Claims}
\crefname{notation}{Notation}{Notations}
\crefname{equation}{Equation}{Equations}
\crefname{figure}{Figure}{Figures}

\AddToHook{env/theorem/begin}{\crefalias{theorem}{theorem}}
\AddToHook{env/lemma/begin}{\crefalias{theorem}{lemma}}
\AddToHook{env/proposition/begin}{\crefalias{theorem}{proposition}}
\AddToHook{env/corollary/begin}{\crefalias{theorem}{corollary}}
\AddToHook{env/conjecture/begin}{\crefalias{theorem}{conjecture}}
\AddToHook{env/definition/begin}{\crefalias{theorem}{definition}}
\AddToHook{env/example/begin}{\crefalias{theorem}{example}}
\AddToHook{env/question/begin}{\crefalias{theorem}{question}}
\AddToHook{env/remark/begin}{\crefalias{theorem}{remark}}
\AddToHook{env/claim/begin}{\crefalias{theorem}{claim}}
\AddToHook{env/notation/begin}{\crefalias{theorem}{notation}}

\newenvironment{proofclaim}{\paragraph{\emph{Proof of the Claim}.}}{\hfill\\} 
\newenvironment{proofsketch}{\paragraph{\emph{Sketch of proof}.}}{\hfill$\qed$\\}

\newcounter{mycount}

\newcommand{\myref}[1]{\hyperref[#1]{#1}}

\newcommand{\Pries}{{\sf{Pries}}}
\newcommand{\Spec}{{\sf{Spec}}}

\newcommand{\JMS}{{\sf{JM_{SLat}}}}
\newcommand{\JML}{{\sf{JM}}}

\newcommand{\CGS}{{\sf{CG_\SLat}}}
\newcommand{\CG}{{\sf{CG}}}

\newcommand{\BDGMS}{{\sf BDGM_{SLat}}}
\newcommand{\BDGM}{{\sf BDGM}}

\newcommand{\HsS}{{\sf{Hs_{SLat}}}}
\newcommand{\HsL}{{\sf{Hs}}}

\newcommand{\HD}{{\sf{DH}}}

\newcommand{\Hg}{{\sf{Hg}}}
\newcommand{\Urq}{{\sf{Urq}}}
\newcommand{\GvG}{{\sf{GvG}}}
\newcommand{\Plos}{Plo\v{s}\v{c}ica}
\newcommand{\Plo}{{\sf{Plo}}}

\newcommand{\AlgLatS}{{\sf{AlgLat_{\sigma}}}}
\newcommand{\AlgLatL}{{\sf{AlgLat_{\lambda}}}}

\newcommand{\AL}{{\sf{AL}}}

\newcommand{\CohLatS}{{\sf{CohLat_{\sigma}}}}
\newcommand{\CohLatL}{{\sf{CohLat_{\lambda}}}}

\newcommand{\SLat}{{\sf SLat}}
\newcommand{\Lat}{{\sf Lat}}

\newcommand{\FSpace}{F-space}

\newcommand{\LC}{{\sf LC}\,}
\newcommand{\RC}{{\sf RC}\,}
\newcommand{\LCUp}{(\LC\U)_p}
\newcommand{\RCUp}{(\RC\U)_p}
\newcommand{\LCUO}{(\LC\U)_0}
\newcommand{\RCUO}{(\RC\U)_0}

\newcommand{\OF}{{\sf OF}}
\newcommand{\KOF}{{\sf KOF}}
\newcommand{\CLF}{{\sf CLF}}
\newcommand{\ClopUp}{{\sf ClopUp}}

\newcommand{\ClosedUp}{{\sf ClosedUp}}
\newcommand{\OpenUp}{{\sf OpenUp}}

\newcommand{\fsat}{{\sf fsat}}
\newcommand{\FSat}{{\sf FSat}}

\newcommand{\Filt}{{\sf Filt}}
\newcommand{\Idl}{{\sf Idl}}
\newcommand{\Up}{{\sf Up}}

\newcommand{\cl}{\Delta}

\newcommand{\bd}{\blacklozenge}
\newcommand{\bs}{\blacksquare}

\newcommand{\up}{{\uparrow}}
\newcommand{\down}{{\downarrow}}
\newcommand\twoheaduparrow{%
\mathrel{\mathchoice{\raise2pt\hbox{\ooalign{\hss$\uparrow$\hss\cr\lower2pt\hbox{$\uparrow$}}}} {\raise2pt\hbox{\ooalign{\hss$\uparrow$\hss\cr\lower2pt\hbox{$\uparrow$}}}}  {\raise1.5pt\hbox{\ooalign{\hss$\scriptstyle\uparrow$\hss\cr\lower1.5pt\hbox{$\scriptstyle\uparrow$}}}}{\raise1.1pt\hbox{\ooalign{\hss$\scriptscriptstyle\uparrow$\hss\cr\lower1.1pt\hbox{$\scriptscriptstyle\uparrow$}}}}}}

\newcommand{\C}{\mathcal{C}}
\newcommand{\D}{\mathcal{D}}
\newcommand{\G}{\mathcal{G}}
\renewcommand{\H}{\mathcal{H}}
\newcommand{\K}{\mathcal{K}}
\renewcommand{\L}{L\,}

\newcommand{\U}{\mathcal{U}}
\newcommand{\Pl}{\mathcal{P}}
\newcommand{\Z}{\mathcal{Z}}
\newcommand{\HHs}{\mathbb{H}\mathbbm{s}}

\newcommand{\EE}{\mathbb{E}}
\newcommand{\FF}{\mathbb{F}}
\newcommand{\GG}{\mathbb{G}}
\newcommand{\HH}{\mathbb{H}}
\newcommand{\LL}{\mathbb{L}}
\newcommand{\UU}{\mathbb{U}}
\newcommand{\DD}{\mathbb{D}}
\newcommand{\PP}{\mathbb{P}}

\newcommand\nr[1]{\not\mathrel{#1}}
\newcommand\rel[1]{\mathrel{#1}}

\newcommand{\rad}{.09}
\newcommand{\spac}{10}
\newcommand{\scale}{.6}

\setlist[enumerate,1]{label={\upshape(\arabic*)},ref=\arabic*} 

\makeatletter 
\edef\plabelformat{(\string#2\string#1\string#3)}
\edef\plabelrangeformat{(\string#3\string#1,\string#2\string#6)}
\newcommand{\plabel}[1]{\label{#1}
\immediate\write\@auxout{\noexpand\crefformat{#1}{\noexpand\cref{#1}\plabelformat}
\noexpand\crefmultiformat{#1}{\noexpand\cref{#1}\plabelformat}{,\plabelformat}{,\plabelformat}{,\plabelformat}
\noexpand\crefrangeformat{#1}{\noexpand\cref{#1}\plabelrangeformat}}}
\makeatother

\newcommand{\abbrevref}[1]{%
  \begingroup
    \crefname{definition}{Def.}{Defs.}%
    \crefname{remark}{Rem.}{Rems.}%
    \cref{#1}%
  \endgroup
}

\makeatletter
\@namedef{subjclassname@2020}{\textup{2020} Mathematics Subject Classification}
\makeatother

\begin{document}

\title[Duality Theory for Bounded Lattices]{Duality Theory for Bounded Lattices: \\ A Comparative Study}

\author{G.~Bezhanishvili}
\address{New Mexico State University}
\email{guram@nmsu.edu}

\author{L.~Carai}
\address{University of Milan}
\email{luca.carai.uni@gmail.com}

\author{P.~J.~Morandi}
\address{New Mexico State University}
\email{pmorandi@nmsu.edu}

\thanks{The third author was supported in part by NSF DMS-2231414}

\begin{abstract}
There are numerous generalizations of the celebrated Priestley duality for bounded distributive lattices 
to the non-distributive setting. The resulting dualities rely on an earlier foundational work of such authors as Nachbin, Birkhoff-Frink, Bruns-Lakser, Hofmann-Mislove-Stralka, and others.  
We undertake a detailed comparative study of the existing dualities for arbitrary bounded (non-distributive) lattices, including supplying the dual description of bounded lattice homomorphisms where it was lacking. This is achieved by working with relations instead of functions. As a result, we arrive at a landscape of categories that provide various generalizations of the category of Priestley spaces. We provide explicit descriptions of the functors yielding equivalences of these categories, together with explicit dual equivalences with the category of bounded lattices and bounded lattice homomorphisms. 
\end{abstract}

\maketitle
\tableofcontents

\section{Introduction} \label{sec: intro}

There is a well-developed duality theory for distributive lattices. It was first established by Stone \cite{Sto37} as a generalization of his celebrated duality for boolean algebras \cite{Sto36}. The resulting spaces are known as spectral spaces, which also arise as prime spectra of commutative rings \cite{Hoc69} (see also \cite{DST19}). An alternative duality was established by Priestley \cite{Pri70,Pri72} in the language of ordered Stone spaces, which became known as Priestley spaces. Cornish \cite{Cor75} showed that the two approaches are two sides of the same coin by establishing an isomorphism between the categories $\Spec$ of spectral spaces and $\Pries$ of Priestley spaces (see \cref{sec: preliminaries} for details).

The above approaches heavily depend on the Prime Filter Lemma (that distinct elements of a distributive lattice can be separated by a prime filter), which is no longer available for arbitrary lattices. Over the years, various approaches have been developed to generalize Stone and Priestley dualities to arbitrary (non-distributive) lattices. The first such approach was undertaken by Urquhart \cite{Urq78}, followed by Hartung \cite{Har92}, \Plos\ \cite{Plo95}, Dunn and Hartonas \cite{HD97}, Hartonas \cite{Har97}, Gehrke and van Gool \cite{GvG14}, and Jipsen and Moshier \cite{MJ14a}.  More recent approaches include  \cite{Har18,CG20,BDGM24,Har24}.

While some of these authors pointed out connections with earlier works, a comprehensive comparison of all these dualities is still lacking, and it is the goal of this article to fill in this gap. In addition, we provide an axiomatization of the category arising from the Dunn-Hartonas approach, and describe the morphisms dual to arbitrary lattice homomorphisms that was missing in the approaches of Urquhart, Hartung, \Plos, and Gehrke-van Gool. This is done by working with relations instead of functions between the appropriate spaces, a technique that originates in modal logic and has been used to describe duals of morphisms of distributive and implicative semilattices (see \cite{Cel03,Cel03b,BJ11,BJ13}). In the non-distributive case it was used by Celani and Gonz\'{a}lez \cite{CG20}.

An important underlying theme in these considerations is Nachbin's well-known correspondence between semilattices and algebraic lattices \cite{Nac49} (see also \cite{BF48}), and the resulting Pontryagin-style duality for semilattices \cite{HMS74}. This can be achieved by either working with join-semilattices and the ideal functor or meet-semilattices and the filter functor. Since most of the above approaches follow the latter, so will we. 

Every algebraic lattice is a continuous lattice, and hence carries the Scott and Lawson topologies (see, e.g., \cite{GHKLMS03}), which are central to some of the dualities mentioned above. These topologies can be defined in the more general setting of posets, giving rise to a more general approach to various Stone-like dualities. We refer to \cite{GHKLMS03} as well as to \cite{Ern91,Ern93,Ern04} and the references therein. For our purposes, working with the Scott topology yields the Jipsen-Moshier duality for semilattices, and working with the Lawson topology the BDGM duality \cite{BDGM24}. If the semilattice under consideration is a lattice, then the corresponding algebraic lattice is arithmetic (binary meets of compact elements are compact). While it is not necessary, it is convenient (and common) to restrict one's attention to bounded lattices. This results in further restriction that the top element of an arithmetic lattice is compact (equivalently, finite meets of compact elements are compact). Borrowing the terminology from pointfree topology \cite{Joh82}, we call such complete lattices coherent. Then Jipsen-Moshier duality for bounded lattices is established by working with the Scott topology  (see \cref{sec: JM}), while the BDGM duality by working with the Lawson topology  on coherent lattices (see \cref{sec: BDGM}). The equivalence of the two approaches is then a simple matter of appropriate restriction of the Cornish isomorphism between $\Spec$ and $\Pries$ (see \cref{sec: BDGM}). 

It is interesting to note that the spaces that Hartonas works with in \cite{Har97} are nothing more but BDGM-spaces in disguise. Thus, we derive Hartonas duality from the BDGM duality (see \cref{sec: Hs}).

While in non-distributive lattices we don't have sufficient supply of prime filters, we do have sufficient supply of the filters that are meet-irreducible elements of the lattice of filters. This allows to represent arbitrary (semi)lattices as subbasic closed sets of the topology on the set of meet-irreducible filters, an approach undertaken in \cite{CG20} (see also \cite{Har24}). This approach is outlined in \cref{sec: CG}, where we discuss how to go back and forth between the Jipsen-Moshier and Celani-Gonz\'{a}lez spaces.

Dunn-Hartonas work not only with the filters $\Filt(A)$, but also with the ideals $\Idl(A)$ of a lattice $A$, and the relation $R \subseteq \Filt(A) \times \Idl(A)$ given by $F \rel{R} I$ iff $F \cap I \ne \varnothing$. They call such triples canonical L-frames and develop their duality in the language of these triples. We provide an axiomatization of such objects, and show that the resulting category is equivalent to the category of Hartonas spaces. This yields an equivalence between the Jipsen-Moshier, Bezhanishvili et al, Celani-Gonz\'{a}lez, Hartonas, and Dunn-Hartonas approaches (see \cref{sec: DH}).

The remaining approaches are more faithful to the Priestley approach to distributive lattices in that these authors work with not all filters and ideals, but smaller collections that are big enough to separate distinct elements of the lattice. As a consequence, axiomatizations of the corresponding spaces are more involved. We show that the Gehrke-van Gool spaces can be obtained from Dunn-Hartonas spaces $(X, R, Y)$ by selecting appropriate closed subspaces of $X$ and $Y$ and restricting the relation $R$ to those (see \cref{sec: DH and GvG}). This is done by working with distributive joins and meets, an idea that goes back to MacNeille \cite{Mac37} (see also \cite{BL70,HK71}), which allows us to work with an appropriate weakening of the notion of prime element. Distributive joins and meets are also known as exact \cite{Bal84} (see also \cite{BPP14,BPW16}). Hartung spaces can then be obtained by further restricting Gehrke-van Gool spaces to subsets $X_0$ of $X$ and $Y_0$ of $Y$ that are maximal in certain sense, and taking $R_0 = R \cap (X_0 \times Y_0)$ (see \cref{sec: GvG and Hg}). We point out that $X_0$ and $Y_0$ are no longer closed subsets of $X$ and $Y$. Because of this, instead of the subspace topology of the original topology, we work with the subspace topology of the open downset topology, which is not  as well behaved (see \cite[pp.~447-448]{GvG14}).

Urquhart spaces can be obtained from Hartung spaces $(X_0,R_0,Y_0)$ by working with the subset $Z$ of  $X_0 \times Y_0$ consisting of pairs $(x, y)$ such that $x$ is $y$-maximal and $y$ is $x$-maximal, which is equipped with two natural orders $\le_1$ and $\le_2$ (see \cref{sec: Hg and Urq}). 
\Plos~spaces can be thought of as an alternative presentation of Urquhart spaces, where the two quasi-orders $\le_1$ and $\le_2$ are replaced by a single reflexive relation $R$ (see \cref{sec: Urq and Plo}).
Finally, we show that each Urquhart space $(Z,\le_1,\le_2)$ gives rise to the Dunn-Hartonas space $(X,R,Y)$, thus completing the circle of correspondences (see \cref{sec: Urq and DH}).

We point out that the approaches of Urquhart, Hartung, \Plos, and Gehrke-van Gool did not have morphisms fully developed. For example, Urquhart and Hartung developed duals of onto bounded lattice homomorphisms, while Gehrke-van Gool did the same for the more general class of admissible homomorphisms (see \cite[Def.~3.19]{GvG14}). These duals are described by appropriate pairs of functions. This approach does not generalize to arbitrary bounded lattice homomorphisms, which could be explained by the fact that restrictions of a 
Dunn-Hartonas morphism $(f, g) \colon (X,R,Y) \to (X',R',Y')$ to various subsets of $X, X'$ and $Y, Y'$ that the other authors work with may no longer be well-defined functions. 
We overcome this by working with pairs of relations between various such subsets, instead of pairs of functions, 
thus yielding a series of functors from the appropriate categories which we show are equivalences (see \cref{sec: HD GvG Hg Urq}). As a result, we arrive at the diagram of equivalences in \cref{figure: intro}, which captures in a nutshell how these various dualities for lattices relate to each other. In the diagram, the double arrows represent isomorphism and the single arrows equivalence of categories. In \cref{table: intro} we describe the main categories at play.

Letting $\Lat$ be the category of bounded lattices and bounded lattice homomorphisms, we show that composing each of the above equivalences with the corresponding dual equivalence with $\Lat$ yields a diagram that commutes up to natural isomorphism (see \cref{sec: conclusion}). We conclude the paper by demonstrating how each of these dualities restricts to the distributive case. As expected, the dualities of Celani-Gonz\'{a}lez, Gehrke-van Gool, Hartung, Urquhart, and \Plos~all collapse to Priestley duality. We also indicate how each of the above dualities works on a non-distributive lattice, thus showcasing the similarities and differences between all these different approaches (see \cref{sec: similarities and differences}). 

\vspace{1mm}
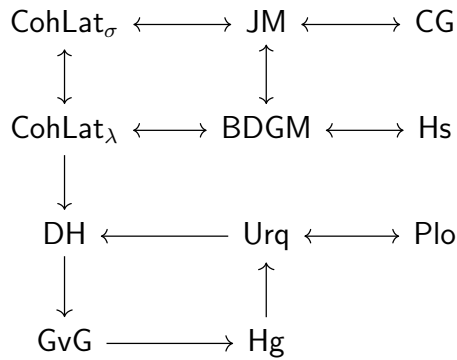
\begin{figure}[ht] 
\[
\begin{tikzcd}
\CohLatS \arrow[d, <->] \arrow[r, <->]& \JML \arrow[d, <->] \arrow[r] & \CG \\
\CohLatL \arrow[r, <->] \arrow[d, ->] & \BDGM \arrow[r, <->] & \HsL\\
 \HD \arrow[d] & \Urq \arrow[l, ->] & \Plo \arrow[l, <->]\\
\GvG \arrow[r] & \Hg \arrow[u] &
\end{tikzcd}
\]
\caption{Diagram of the equivalences}
\label{figure: intro}
\end{figure}
\vspace{1mm}

\begin{table}[ht]
\centering
\begin{tabular}{|lll|} \hline
\textbf{Category} & \textbf{Objects}  & \textbf{Location}\\ \hline
$\CohLatS$ & Coherent lattices with Scott topology & \abbrevref{CohLatS}\\
$\CohLatL$ & Coherent lattices with Lawson topology & \abbrevref{CohLatL}\\
$\JML$ & Jipsen-Moshier spaces for lattices   &  \abbrevref{def: JML-space}\\
$\BDGM$ & BDGM-spaces & \abbrevref{def: BDGM space} \\
$\HsL$ & Hartonas spaces for lattices & \abbrevref{def: HsL space}\\
$\CG$ & Celani-Gonz\'{a}lez spaces & \abbrevref{def: CG} \\
$\HD$ & Dunn-Hartonas spaces & \abbrevref{def: DH morphisms}\\
$\GvG$ & Gehrke-van Gool spaces& \abbrevref{def: GvG}\\
$\Hg$ & Hartung spaces& \abbrevref{def: Hg}\\
$\Urq$ & Urquhart spaces & \abbrevref{def: Urq}\\ 
$\Plo$ & \Plos\ spaces & \abbrevref{def: Plo space}\\ \hline
\end{tabular}
\vspace{3mm}
\caption{Table of the categories}
\label{table: intro}
\end{table}


\section{Preliminaries} \label{sec: preliminaries}

In this preliminary section we recall the Cornish isomorphism \cite{Cor75} between $\Spec$ and $\Pries$, the Nachbin theorem \cite{Nac49} connecting semilattices with algebraic lattices, and the resulting Pontryagin-style duality between semilattices and algebraic lattices \cite{HMS74}. We also specialize this to an equivalence between bounded lattices and what we term coherent lattices (the term that we borrow from pointfree topology \cite[p.~65]{Joh82}).

For a topological space $X$, let $\mathcal O(X)$ be the frame of opens of $X$. We recall \cite[Sec.~II.3.4]{Joh82} that $X$ is {\em coherent} if the compact opens form a bounded sublattice of $\mathcal O(X)$ that is a basis for $X$. Observe that each coherent space is compact. A nonempty closed subset $C$ of $X$ is \emph{irreducible} if it cannot be written as the union of two proper closed subsets. The space $X$ is \emph{sober} if each closed irreducible subset is the closure of a unique point. Observe that each sober space is $T_0$. If $X$ is coherent and sober, then $X$ is called a \emph{spectral space} \cite[p.~43]{Hoc69} (see also \cite[Def.~1.1.5] {DST19} or \cite[p.~65]{Joh82}; in the latter these spaces are called coherent). A map $f \colon X \to Y$ between spectral spaces is a \emph{spectral map} provided the pullback of a compact open is compact open. Note that each spectral map is continuous. Let $\Spec$ be the category of spectral spaces and spectral maps. 

A compact ordered space $X$ is a \emph{Priestley space} if it satisfies the Priestley separation axiom: If $x \not\le y$, then there is a clopen upset $U$ containing $x$ and missing $y$. Let $\Pries$ be the category of Priestley spaces and continuous order-preserving maps.

\begin{theorem} [Cornish's isomorphism \cite{Cor75}] \label{thm: Cornish}
$\Spec$ is isomorphic to  $\Pries$. 
\end{theorem}

\begin{remark}
We recall the functors establishing this isomorphism. If $(X, \tau)$ is a spectral space, then $(X, \pi, \le)$ is a Priestley space, where $\pi$ is the patch topology and $\le$ is the specialization order of $\tau$. Conversely, if $(X, \pi, \le)$ is a Priestley space, then $(X, \tau)$ is a spectral space, where $\tau$ is the topology of open upsets. In addition, a map between spectral spaces is spectral iff it is continuous and order-preserving with respect to the corresponding Priestley spaces.
\end{remark}

We next recall Nachbin's well-known result \cite{Nac49} that algebraic frames are exactly the ideal lattices of join-semilattices. By a \emph{join-semilattice} we mean a poset in which all finite joins exist. In particular, each join-semilattice has a least element 0. 

An element $k$ of a complete lattice $L$ is \emph{compact} if $k\le\bigvee S$ implies $k\le\bigvee T$ for some finite $T\subseteq S$. Let $K(L)$ be the set of compact elements of $L$. Then it is easy to see that $K(L)$ is a join-semilattice, where the order on $K(L)$ is the restriction of the order on $L$. The following definitions are well known (see, e.g., \cite[Def.~1-4.2]{GHKLMS03}, \cite[p.~252]{Joh82}).

\begin{definition}
We call a complete lattice $L$
\begin{enumerate}
\item a \emph{compact lattice} if $1 \in K(L)$;
\item an \emph{algebraic lattice} if $K(L)$ join-generates $L$; 
\item an \emph{arithmetic lattice} if $L$ is algebraic and $K(L)$ is closed under binary meets;
\item a \emph{coherent lattice} if $L$ is a compact arithmetic lattice.
\end{enumerate} 
\end{definition}

\begin{remark}
It is immediate from the definition that an algebraic lattice $L$ is coherent iff $K(L)$ is a bounded sublattice of $L$.
\end{remark}

For a join-semilattice $A$, let $\Idl(A)$ be the lattice of ideals of $A$.

\begin{theorem} [Nachbin's theorem \cite{Nac49}] \plabel{thm: Nachbin}
Let $L$ be a complete lattice.
\begin{enumerate}
\item $L$ is algebraic iff $L\cong\Idl(A)$ for a join-semilattice $A$. 
\item $L$ is arithmetic iff $L\cong\Idl(A)$ for a lattice $A$.
\item \label[thm: Nachbin]{Nachbin ideal} $L$ is coherent iff $L\cong\Idl(A)$ for a bounded lattice $A$. 
\end{enumerate}
\end{theorem}

Nachbin's theorem has an obvious reformulation for the filter lattices of meet-semilattices (by which we mean posets in which all finite meets exist). Following the standard practice in duality theory for lattices, we will have a blanket assumption that all lattices are bounded. We will thus concentrate on the following two results, where $\Filt(A)$ is the lattice of filters of a meet-semilattice $A$:
\begin{itemize}
\item $L$ is an algebraic lattice iff $L\cong\Filt(A)$ for a meet-semilattice $A$; 
\item $L$ is a coherent lattice iff $L\cong\Filt(A)$ for a bounded lattice $A$.
\end{itemize}

There are several natural topologies on $\Filt(A)$. To recall them, we introduce the following standard notation: 

\begin{notation}
Let $X$ be a poset.
\begin{enumerate}
\item For $S \subseteq X$, we write $\up S$ for the upset and $\down S$ for the downset generated by $S$. Then $S$ is an upset if $S=\up S$ and a downset if $S=\down S$. 
\item For $S, T  \subseteq X$ we write $-S$ for set-theoretic complement and $S - T$ for set difference. 
\end{enumerate}
\end{notation}

\begin{definition} \cite[Def.~II-1.3]{GHKLMS03}
Let $L$ be an algebraic lattice. An upset $U$ of $L$ is \emph{Scott open} if $\bigvee S \in U$ implies $S \cap U \ne \varnothing$ for each directed $S \subseteq L$. 
\end{definition}

The Scott open subsets of $L$ form a topology, known as the \emph{Scott topology} and denoted by $\sigma(L)$. In what follows we will freely use the following well-known facts about the Scott topology (see \cite[Sec.~II-1, II-2]{GHKLMS03}).

\begin{theorem} \plabel{thm: Scott facts}
Let $L$ be an algebraic lattice.
\begin{enumerate}
\item \label[thm: Scott facts]{specialization order} The specialization order of the Scott topology on $L$ is the existing order of $L$.
\item \label[thm: Scott facts]{compact elements} Let $a \in L$. Then $\up a$ is Scott open iff $a \in K(L)$.
\item \label[thm: Scott facts]{basis} $\{ \up k : k \in K(L) \}$ is a basis for $\sigma(L)$.
\item \label[thm: Scott facts]{compact sober} The Scott topology on $L$ is a spectral topology.
\item \label[thm: Scott facts]{Scott continuous} Let $f \colon L \to M$ be a map between algebraic lattices. Then $f$ is Scott-continuous iff $f$ preserves directed joins. 
\end{enumerate}
\end{theorem}

We next recall the Lawson topology (see \cite[Sec.~III-1]{GHKLMS03}).

\begin{definition}
Let $L$ be an algebraic lattice.
\begin{enumerate}
\item  The \emph{lower topology} $\omega(L)$ on $L$ is generated by the subbasis $\{ - \up x : x \in L\}$.  
\item  The \emph{Lawson topology} $\lambda(L)$ is the join of the Scott and lower topologies. 
\end{enumerate}
\end{definition}

We will freely use the following well-known facts about the Lawson topology (see \cite[Sec.~III-1, V-5]{GHKLMS03}). A \emph{meet-semilattice homomorphism} between two meet-semilattices is a map  $f \colon L \to M$ that preserves finite meets (and hence $f(1) = 1$).

\begin{theorem} \plabel{thm: Lawson facts}
Let $L$ be an algebraic lattice.
\begin{enumerate}
\item \label[thm: Lawson facts]{Lawson is patch}  The Lawson topology on $L$ is the patch topology of the Scott topology. 
\item \label[thm: Lawson facts]{Lawson upset is Scott}  An upset of $L$ is Lawson open iff it is Scott open.
\item \label[thm: Lawson facts]{Lawson continuous} Let $f \colon L \to M$ be a meet-semilattice homomorphism between algebraic lattices. Then $f$ is Lawson-continuous iff $f$ is Scott-continuous and preserves  arbitrary meets.
\end{enumerate}
\end{theorem}

\needspace{3\baselineskip}
\begin{remark} \plabel{rem: Cornish and AlgLat}
\hfill
\begin{enumerate}
\item \label[rem: Cornish and AlgLat]{objects} Putting together Theorems~\ref{thm: Cornish}, \ref{thm: Scott facts}\eqref{compact sober}, and \ref{thm: Lawson facts}\eqref{Lawson is patch}, for each algebraic lattice $L$, we have that $(L, \sigma(L))$ is a spectral space and $(L, \lambda(L), \le)$ is the corresponding Priestley space. 
\item \label[rem: Cornish and AlgLat]{morphisms} In \cite{GHKLMS03}, meet-semilattices are not assumed to have a top, and hence meet-semilattice homomorphisms need not preserve the top. Under our definition of meet-semilattice homomorphisms, \cite[Thm.~III-1.8]{GHKLMS03} takes the stronger form of \cref{Lawson continuous} (namely, preserving nonempty meets becomes equivalent to preserving arbitrary meets). 
\end{enumerate}
\end{remark}

We next define the categories of algebraic lattices with the Scott topology and algebraic lattices with the Lawson topology to which the Cornish isomorphism restricts.

\begin{definition} \plabel{def: KAlgLat}
\hfill
\begin{enumerate}
\item \label[def: KAlgLat]{KAlgLatS} Let $\AlgLatS$ be the category of algebraic lattices with the Scott topology and Scott-continuous maps that preserve arbitrary meets.
\item \label[def: KAlgLat]{KAlgLatL} Let $\AlgLatL$ be the category of algebraic lattices with the Lawson topology and meet-semilattice homomorphisms that are Lawson-continuous.
\end{enumerate}
\end{definition}

\begin{remark} \label{rem: why spectral}
Clearly isomorphisms in both $\AlgLatS$ and $\AlgLatL$ are order-isomorphisms.
Let $f \colon L \to M$ be an $\AlgLatS$-morphism. We show that $f$ is spectral. Let $U \subseteq M$ be compact open. It is enough to show that $f^{-1}(U)$ is compact. By \cref{basis}, it suffices to assume that $U = \up k$ for some $k \in K(M)$. Set $x = \bigwedge f^{-1}(U)$, so $f^{-1}(U) \subseteq \up x$. Because $f$ preserves arbitrary meets, $f(x) = \bigwedge \{ f(z) : z \in f^{-1}(U)\}$. This shows that $k \le f(x)$. Therefore, $x \in f^{-1}(U)$, and hence $f^{-1}(U) = \up x$. Thus, $f^{-1}(U)$ is compact since each principal upset is compact. Consequently, $\AlgLatS$ is a (non-full) subcategory of $\Spec$. Similarly, $\AlgLatL$ is a (non-full) subcategory of $\Pries$.
\end{remark}

As an immediate consequence of the above we obtain:

\begin{theorem} \label{thm: iso between AlgLat_S and AlgLat_L}
The Cornish isomorphism between $\Spec$ and $\Pries$ restricts to an isomorphism between $\AlgLatS$ and $\AlgLatL$. 
\end{theorem}

\begin{remark}\label{rem: AL}
Since the Scott and Lawson topologies are intrinsic to algebraic lattices, each of $\AlgLatS$ and $\AlgLatL$ is also isomorphic to the category $\sf{AL}$, considered in \cite[Def.~IV-1.13]{GHKLMS03}, whose objects are algebraic lattices and whose morphisms are maps preserving arbitrary meets and directed joins.
\end{remark}

Now suppose that $L, M$ are coherent lattices and $f \colon L \to M$ is an $\AlgLatS$-morphism. Since $f$ preserves arbitrary meets, it has a left adjoint $\ell \colon M \to L$. By \cite[Cor.~IV-1.12]{GHKLMS03}, $\ell[K(M)] \subseteq K(L)$. In addition, we have:

\begin{lemma} \label{lem: ell preserves meets}
Let $f \colon L \to M$ be an $\AlgLatS$-morphism with left adjoint $\ell$. The following are equivalent.
\begin{enumerate}
\item $\ell$ preserves finite meets.
\item $\ell$ preserves finite meets of compact elements.
\item $f^{-1}$ preserves finite joins of principal filters.
\item $f^{-1}$ preserves finite joins of principal filters of compact elements.
\end{enumerate}
\end{lemma}

\begin{proof}
(1)$\Longleftrightarrow$(2) 
One direction is obvious. For the other, let $S\subseteq L$ be finite. Because $\ell$ is order-preserving, $\ell(\bigwedge S) \le \bigwedge \{ \ell(x) : x \in S\}$. To see the equality, let $x \in S$ and $k \in K(M)$ with $k \le \ell(x)$. Since $\ell$ preserves arbitrary joins, $\ell(x) = \bigvee \{ \ell(m) : m \in \down x \cap K(L) \}$. Because $k$ is compact and this join is directed, there is $m \in K(L)$ with $m \le x$ and $k \le \ell(m)$. Now, if $k \in K(M)$ with $k \le \bigwedge \{ \ell(x) : x \in S\}$, then for each $x \in S$ there is $m_x \in K(L)$ with $m_x \le x$ and $k \le \ell(m_x)$. Since $\ell$ preserves finite meets of compact elements, 
\[
k \le \ell\left(\bigwedge \{ m_x : x \in S\}\right) \le \ell\left(\bigwedge S \right).
\]
This yields the reverse inequality, and hence $\ell$ preserves all finite meets.

(1)$\Longleftrightarrow$(3) Since $f^{-1}(\up y) = \up \ell(y)$ for each $y \in M$, for $S \subseteq M$ finite, we have
\[
f^{-1}\left(\bigvee \{ \up y : y \in S\}\right) = f^{-1}\left( \big\uparrow \bigwedge S\right) = \big\uparrow \ell\left(\bigwedge S\right)
\]
and
\[
\bigvee \{ f^{-1}(\up y) : y \in S \} = \bigvee \{ \up \ell(y) : y \in S\} = \big\uparrow \bigwedge \{ \ell(y) : y \in S \}.
\]
From this we see that $f^{-1}$ preserves finite joins of principal filters iff $\ell$ preserves finite meets. Therefore, (1) and (3) are equivalent.

(2)$\Longleftrightarrow$(4) This follows from the same argument as for (1)$\Longleftrightarrow$(3).
\end{proof}

\begin{definition} \plabel{def: CohLat}
\hfill
\begin{enumerate}
\item \label[def: CohLat]{CohLatS} Let $\CohLatS$ be the category of coherent lattices and $\AlgLatS$-morphisms whose left adjoint preserves finite meets.
\item \label[def: CohLat]{CohLatL} Let $\CohLatL$ be the category of coherent lattices and $\AlgLatL$-morphisms whose left adjoint preserves finite meets.
\end{enumerate}
\end{definition}

Clearly $\CohLatS$ is a (non-full) subcategory of $\AlgLatS$ and $\CohLatL$ is a (non-full) subcategory of $\AlgLatL$. As an immediate consequence we obtain: 

\begin{theorem} \label{thm: iso between CohLat_S and CohLat_L}
The isomorphism of Theorem~\emph{\ref{thm: iso between AlgLat_S and AlgLat_L}} restricts to an isomorphism between $\CohLatS$ and $\CohLatL$.
\end{theorem}

We recall from the introduction that $\Lat$ is the category of bounded lattices and bounded lattice homomorphisms. Let also  $\SLat$ be the category of meet-semilattices and meet-semilattice homomorphisms.

\begin{theorem} \plabel{thm: CohLat = Lat}
\hfill
\begin{enumerate}
\item \label[thm: CohLat = Lat]{thm: Nachbin 1} The categories $\AlgLatS$ and $\AlgLatL$ are each dually equivalent to $\SLat$.
\item \label[thm: CohLat = Lat]{thm: Nachbin 2} The categories $\CohLatS$ and $\CohLatL$ are each dually equivalent to $\Lat$.
\end{enumerate}
\end{theorem}

\begin{proofsketch}
\eqref{thm: Nachbin 1} In \cite[Thm.~I.3.9]{HMS74} (see also \cite[Thm.~IV-1.16]{GHKLMS03}), Nachbin's theorem was generalized to a duality between $\SLat$ and the category of Stone semilattices. Since the latter is (isomorphic to) $\AlgLatL$, the result follows by applying \cref{thm: iso between AlgLat_S and AlgLat_L}.

\eqref{thm: Nachbin 2} The dual equivalence of $\AlgLatL$ and $\SLat$ restricts to a dual equivalence of $\CohLatL$ and $\Lat$ (see, e.g., \cite[Prop.~3.19]{HMS74}). Now apply \cref{thm: iso between CohLat_S and CohLat_L}.
\end{proofsketch}

\begin{remark}\ \label{rem: functors giving the duality}
\begin{enumerate}
\item \cref{thm: Nachbin 1} together with \cref{rem: AL} yields that $\SLat$ is also dually equivalent to $\AL$ (see \cite[Cor.~5.9(3)]{Ern91}, \cite[Thm.~IV.1.16]{GHKLMS03}).

\item The functors establishing the dual equivalences of \cref{thm: CohLat = Lat} are described as follows. One functor sends a semilattice or lattice $A$ to $\Filt(A)$ and a morphism $\alpha \colon A \to B$ to $\alpha^{-1} \colon \Filt(B) \to \Filt(A)$. The other sends an algebraic or coherent lattice $L$ to $(K(L), \ge)$  and a morphism $f \colon L \to M$ to the restriction of its left adjoint.
\end{enumerate}
\end{remark}


\section{The Jipsen-Moshier approach} \label{sec: JM}

In \cite{MJ14a}, Jipsen and Moshier defined two categories of topological spaces and showed that they are dually equivalent to $\SLat$ and $\Lat$. In this section we show that these two categories of spaces are isomorphic to $\AlgLatS$ and $\CohLatS$, respectively.

For a topological space $X$, we write $\le$ for the specialization order on $X$. A nonempty subset $F$ of $X$ is a \emph{filter} if for each $x, y \in F$, there is $z \in F$ with $z \le x, y$. Observe that when $X$ is a meet-semilattice, this notion is equivalent to the standard notion of filter. Let $\OF(X)$ be the open filters and $\KOF(X)$ the compact open filters of $X$. A subset $A$ of $X$ is $F$-\emph{saturated} if it is an intersection of open filters. For $A \subseteq X$, set
\[
\fsat(A) = \bigcap \{ U \in \OF(X) : A \subseteq U\}.
\]
Then $\fsat$ is a closure operator on $X$, and so its fixpoints form the complete lattice $\FSat(X)$, where meet is given by intersection and join by $\bigvee \mathcal S = \fsat\left(\bigcup \mathcal S\right)$ (see, e.g., \cite[Prop. 7.2(ii)]{DP02}).

\begin{definition} \label{def: MJ-space}
\hfill
\begin{enumerate}
\item \label[def: MJ-space]{MJ objects} Let $X$ be a topological space. We call $X$ a {\em Jipsen-Moshier space}, or JM-space for short, if $X$ is sober and $\KOF(X)$ forms a basis that is closed under finite intersections.
\item \label[def: MJ-space]{MJ morphisms}Let $\JMS$ be the category of Jipsen-Moshier spaces and maps $f \colon X \to Y$ satisfying $U \in \KOF(Y)$ implies $f^{-1}(U) \in \KOF(X)$.
\end{enumerate}
\end{definition}

\begin{remark} \label{rem: JM maps are continuous}
JM-spaces are called HMS-spaces in \cite{MJ14a}, to emphasize their reliance on Hofmann-Mislove-Stralka duality for semilattices \cite{HMS74}. 
The definition of a JM-space is a strengthening of the definition of a spectral space in that the coherence condition on compact open subsets holds for the smaller collection $\KOF(X)$, and similarly for morphisms. Observe that $\KOF(X)$ is a meet-subsemilattice of $\FSat(X)$ and every $\JMS$-morphism is automatically continuous.
\end{remark}

Because we will be viewing algebraic lattices as topological spaces, we will typically denote them by $X$, $Y$, $Z$. 
We recall the following results from \cite[Sec.~2, 3]{MJ14a}. 
\begin{theorem} \plabel{thm: JM results}
Let $X$ be a topological space and $\le$ its specialization order.
\begin{enumerate}
\item \label[thm: JM results]{compact is principal} A filter of $X$ with respect to $\le$ is compact iff it is principal. 
\item \label[thm: JM results]{JMS is AlgLatS} $X \in \JMS$ iff $(X, \le)$ is an algebraic lattice and the topology of $X$ is the Scott topology.
\end{enumerate}
\end{theorem}

\begin{remark}
We briefly comment on the proof of \cref{JMS is AlgLatS}. Using \cref{thm: Scott facts}, it is easy to see that each algebraic lattice with the Scott topology is a JM-space. The key then is to prove that a JM-space is an algebraic lattice and the topology is the Scott topology. This can be seen as follows. 
Since every JM-space is sober, directed joins exist. This together with the fact that compact open filters are principal implies that arbitrary meets exist, and hence $X$ is a complete lattice. In addition, $X$ is algebraic since compact open filters are precisely the principal upsets of compact elements, which also yields that the topology is the Scott topology.
\end{remark}

Putting \cref{compact elements,compact is principal} together yields the following result, which will be used throughout.

\begin{lemma} \label{lem: KOF = up k's}
Let $X$ be an algebraic lattice. Viewing $X$ as a topological space with the Scott topology, $\KOF(X) = \{ \up k : k \in K(X)\}$.
\end{lemma}
 
We next turn our attention to morphisms. The next result can be derived from \cite[Sec.~5]{MJ14a}. To keep the paper self-contained, we give a short proof.

\begin{lemma} \label{lem: morphisms of JM}
Let $f \colon X \to Y$ be a map between algebraic lattices. Then $f$ is a $\JMS$-morphism iff $f$ is an $\AlgLatS$-morphism.
\end{lemma}

\begin{proof}
Suppose that $f$ is a $\AlgLatS$-morphism. Let $U \in \KOF(Y)$. By \cref{rem: why spectral}, $f$ is spectral, and hence $f^{-1}(U)$ is compact open. In addition, since $f$ preserves arbitrary meets, $f^{-1}(U)$ is a filter. Thus, $f^{-1}(U) \in \KOF(X)$, and so $f$ is a $\JMS$-morphism.

Conversely, let $f$ be a $\JMS$-morphism. Then $f$ is Scott-continuous by \cref{rem: JM maps are continuous,JMS is AlgLatS}. If $U$ is compact open, then $U$ is a finite union from $\KOF(Y)$. Therefore, $f^{-1}(U)$ is a finite union from $\KOF(X)$. Consequently, $f$ is spectral. 
To see that $f$ preserves arbitrary meets, first observe that $f$ is order preserving since the orders on $X$ and $Y$ are the specialization orders for the Scott topology by \cref{specialization order}. Now let $S \subseteq X$ and set $x = \bigwedge S$. Then $f(x) \le \bigwedge f[S]$. For the reverse inequality, let $l \in K(Y)$ with $l \le \bigwedge f[S]$. Then $l \le f(z)$ for each $z \in S$. By \cref{lem: KOF = up k's}, there is $k \in K(X)$ with $f^{-1}(\up l) = \up k$. Therefore, if $z \in S$, then $k \le z$. This yields $k \le x$, and so $l \le f(k) \le f(x)$, again since $f$ is order preserving. Because $\bigwedge f[S]$ is the join of all such $l$, we get $\bigwedge f[S] \le f(x)$, and hence $f(x) = \bigwedge f[S]$. Thus, $f$ preserves arbitrary meets, and hence is an $\AlgLatS$-morphism.
\end{proof}
 
\begin{theorem} \label{thm: AlgLat = JM}
$\AlgLatS$ is isomorphic to $\JMS$.
\end{theorem}

\begin{proof}
By \cref{JMS is AlgLatS}, $X\in \AlgLatS$ iff $(X, \sigma(X)) \in \JMS$. Moreover, by \cref{lem: morphisms of JM},  ${f\in{\sf Hom}_\AlgLatS(X, Y)}$ iff $f\in{\sf Hom}_\JMS(X, Y)$. Thus, the assignment $X \mapsto (X, \sigma(X))$ and $f \mapsto f$ defines the desired isomorphism of categories. 
\end{proof}

The next result  is an immediate consequence of the above and \cref{thm: Nachbin 1}.

\begin{corollary} \cite[Thm.~5.2]{MJ14a}\label{cor: JMS = SLat}
$\JMS$ is dually equivalent to $\SLat$.
\end{corollary}

\begin{remark} \label{rem: functors giving JMS = SLat}
The functors yielding the duality of \cref{cor: JMS = SLat} are essentially those described in \cref{rem: functors giving the duality}. That is, $X \in \JMS$ is sent to $\KOF(X)$, which is isomorphic to $(K(X), \ge)$, and $A \in \SLat$ to $\Filt(A)$. A $\JMS$-morphism $f \colon X \to X'$ is sent to the restriction of $f^{-1}$ to $\KOF(X')$, and a $\SLat$-morphism $f \colon A \to B$ to $f^{-1} \colon \Filt(B) \to \Filt(A)$.
\end{remark}

We now consider a (non-full) subcategory of $\JMS$ that is dually equivalent to $\Lat$.

\begin{definition} \label{def: JML-space}
Let $\JML$ be the category of Jipsen-Moshier spaces such that $\KOF(X)$ forms a bounded sublattice of $\FSat(X)$ and $\JMS$-morphisms $f$ such that $f^{-1}$ preserves finite joins of compact open filters. 
\end{definition}

Putting \cref{JMS is AlgLatS} together with \cite[Thm.~3.2]{MJ14a} yields:

\begin{lemma} \label{lem: arithmetic lattice case for JM}
Let $X \in \JMS$. Then $X \in \JML$ iff $X$ is a coherent lattice and the topology of $X$ is the Scott topology.
\end{lemma}

\begin{remark}
The key idea of the proof is as follows. By \cref{lem: KOF = up k's}, $\KOF(X)$ is isomorphic to $(K(X), \ge)$. Thus, $\KOF(X)$ is a bounded sublattice of $\FSat(X)$ iff $K(X)$ is a bounded sublattice of $X$. 
\end{remark}

In addition, putting \cref{lem: ell preserves meets,lem: KOF = up k's,lem: morphisms of JM} together yields:

\begin{lemma} \label{lem: morphisms of JML}
Let $f \colon X \to Y$ be an $\AlgLatS$-morphism between coherent lattices topologized with the Scott topology. Then $f$ is a $\CohLatS$-morphism iff $f$ is a $\JML$-morphism.
\end{lemma}

We thus obtain:

\begin{theorem} \label{thm: CohLatL = JML}
$\CohLatS$ is isomorphic to $\JML$. 
\end{theorem}

\begin{proof}
By \cref{lem: arithmetic lattice case for JM}, $X \in \CohLatS$ iff $(X, \sigma(X))  \in \JML$. By \cref{lem: morphisms of JML}, we have that $f\in{\sf Hom}_\CohLatS(X, Y)$ iff $f\in{\sf Hom}_\JML(X, Y)$. Therefore, the isomorphism between $\AlgLatS$ and $\JMS$ restricts to an isomorphism between $\CohLatS$ and $\JML$. 
\end{proof}

The next result  is an immediate consequence of the above and \cref{thm: Nachbin 2}.

\begin{corollary} \cite[Thm.~5.2]{MJ14a}\label{cor: JML = Lat}
$\JML$ is dually equivalent to $\Lat$.
\end{corollary}

\begin{remark} \label{rem: functor JM to Lat}
The functors yielding the duality of \cref{cor: JML = Lat} are the restrictions of those given in \cref{rem: functors giving JMS = SLat}.
\end{remark}


\section{The Bezhanishvili-Dmitrieva-de Groot-Moraschini approach} \label{sec: BDGM}

In this section we consider the spaces introduced in \cite{BDGM24}, which we term BDGM-spaces. The same way JM-spaces are a strengthening of spectral spaces, BDGM-spaces are a strengthening of Priestley spaces, and the Cornish isomorphism between $\Spec$ and $\Pries$ restricts to an isomorphism between the categories of JM-spaces and BDGM-spaces. As we pointed out in the previous section, these categories come in two flavors, depending on whether we want to obtain a dual equivalence for $\SLat$ or $\Lat$. We thus arrive at two categories $\BDGMS$ and $\BDGM$. We show that $\BDGMS$ is isomorphic to $\AlgLatL$, from which we derive that $\BDGMS$ is isomorphic to $\JMS$. We then show that the above isomorphism restricts to an isomorphism between $\BDGM$ and $\CohLatL$, from which we derive that $\BDGM$ is isomorphic to $\JML$.

\begin{definition}\ \label{def: BDGMS space}
\begin{enumerate}
\item We call an ordered topological space $X$ a {\em BDGM-space} if it is a compact space, a meet-semilattice, and the following strengthening of the Priestley separation axiom holds: 
\[
\mbox{If } x \not\le y, \mbox{ then there is a clopen filter } F \mbox{ such that } x\in F \mbox{ and } y\notin F.
\]
\item A {\em BDGM-morphism} between BDGM-spaces is a continuous meet-semilattice morphism. 
\item Let $\BDGMS$ be the category of BDGM-spaces and BDGM-morphisms.
\end{enumerate}
\end{definition} 

\begin{remark} \label{rem: properties of BDGM spaces}
It is clear from the above definition that each BDGM-space is a Priestley space, 
and so $\le$ is a closed order on $X$. In particular, $\up x, \down x$ are closed for each $x\in X$. Moreover, an argument similar to that for Priestley spaces shows that 
\[
\{F - G : F, G \in \CLF(X) \}
\]
is a subbasis for the topology of $X$, where $\CLF(X)$ denotes the collection of clopen filters of~$X$.
\end{remark}

It is trivial to see that each algebraic lattice with the Lawson topology is a BDGM-space. The following lemma shows the converse.

\begin{lemma} \plabel{lem: BDGM basic properties}
Let $X$ be a BDGM-space.
\begin{enumerate}
\item \label[lem: BDGM basic properties]{BDGM 2} $X$ is a complete lattice.
\item \label[lem: BDGM basic properties]{BDGM 1} $\CLF(X) = \{ \up k : k \in K(X) \}$.
\item \label[lem: BDGM basic properties]{BDGM 3} $X$ is an algebraic lattice and the topology is the Lawson topology.
\end{enumerate}
\end{lemma}

\begin{proof}
(\ref{BDGM 2}) 
Since $1 \in X$, it is enough to show that $\bigwedge S$ exists in $X$ for each nonempty $S \subseteq X$. Let $T$ be the set of all finite meets from $S$. Because $X$ is compact and each $\down t$ is closed, $\bigcap \{ \down t : t \in T\}$ is nonempty. This intersection is the set $S^l$ of lower bounds of $S$. If $x_1, \dots, x_n \in S^l$ and $t_1, \dots, t_m \in T$, then $x_i \le t_j$ for each $i,j$, so
\[
t_1 \wedge \dots \wedge t_m \in \up x_1 \cap \dots \cap \up x_n \cap \down t_1 \cap \dots \cap \down t_m.
\]
Therefore, $\{ \up x : x \in S^l\} \cup \{ \down t : t \in T\}$ has the finite intersection property, and hence compactness implies that there is
\[
s \in \bigcap \{ \up x : x \in S^l\} \cap \bigcap \{ \down t : t \in T\}.
\]
Thus, $s = \bigwedge S$. 

(\ref{BDGM 1}) Let $F \in \CLF(X)$. Then $\{F\} \cup \{ \down x : x \in F\}$ is a collection of closed sets 
with the finite intersection property. Therefore, there is $ k \in F \cap \bigcap \{ \down x : x \in F\}$, which forces $F = \up k$.  To see that $k \in K(X)$, suppose that $S \subseteq X$ with $k \le \bigvee S =: s$. Then $\up s \subseteq \up k$, so $\bigcap \{ \up x : x \in S \} = \up s \subseteq \up k$. Since each $\up x$ is closed and $\up k = F$ is clopen, compactness implies that there are $x_1, \dots, x_n \in S$ with $\up x_1 \cap \dots \cap \up x_n \subseteq \up k$, and so $k \le x_1 \vee \dots \vee x_n$. Thus, $k$ is compact.

Conversely, let $k \in K(X)$. The separation axiom for BDGM-spaces shows that the filter $\up k$ is an intersection of clopen filters. Therefore, by the previous paragraph, there is $S \subseteq K(X)$ with $\up k = \bigcap \{ \up l : l \in S\}$. This implies that $k = \bigvee S$. Since $k$ is compact, $k$ is a finite join from $S$, and hence $\up k$ is a finite intersection of clopen filters. Thus, $\up k$ is clopen.

(\ref{BDGM 3}) By (\ref{BDGM 2}), for $X$ to be an algebraic lattice, it is enough to show that each $x\in X$ is a join from $K(X)$. By the separation axiom for BDGM-spaces and (\ref{BDGM 1}), $\up x = \bigcap \{ \up k : k \in K(X), \ k \le x \}$. Thus, $x=\bigvee (\down x\cap K(X))$. 
Since $\{ \up k - \up l : k, l \in K(X) \}$ is a subbasis for the topology on $X$, 
the topology must be the Lawson topology.
\end{proof}

\begin{proposition} \plabel{prop: BDGMS = AlgLatL}
\hfill
\begin{enumerate}
\item \label[prop: BDGMS = AlgLatL]{BDGMS = AlgLatL on objects} Let $X$ be an ordered Stone space. Then $X\in\BDGMS$ iff $X \in \AlgLatL$.
\item \label[prop: BDGMS = AlgLatL]{BDGMS = AlgLatL on morphisms} A map $f \colon X \to Y$ is a $\BDGMS$-morphism iff $f$ is an $\AlgLatL$-morphism.
\end{enumerate}
\end{proposition}

\begin{proof}
(\ref{BDGMS = AlgLatL on objects}) As we pointed out above, if $X \in \AlgLatL$, then $X \in \BDGMS$. The converse follows from \cref{BDGM 3}.

(\ref{BDGMS = AlgLatL on morphisms}) By (\ref{BDGMS = AlgLatL on objects}), $X, Y \in \BDGMS$ iff $X, Y \in \AlgLatL$. Furthermore, the topology on a BDGM-space is the Lawson topology by \cref{BDGM 3}. Thus, $f$ is a $\BDGMS$-morphism iff $f$ is an $\AlgLatL$-morphism.
\end{proof}

As an immediate consequence, we obtain:

\begin{theorem} \label{thm: BDGMS = AlgLatL}
$\BDGMS$ is isomorphic to $\AlgLatL$.
\end{theorem}

As a corollary, we obtain:

\begin{corollary} \plabel{cor: BDGMS = SLat}
\hfill
\begin{enumerate}
\item \label[cor: BDGMS = SLat]{JMS = BDGMS} \cite[Rem.~2.15]{BDGM24} The Cornish isomorphism restricts to an isomorphism between $\JMS$ and $\BDGMS$.
\item \label[cor: BDGMS = SLat]{BDGMS = SLat}\cite[Thm.~2.11]{BDGM24} $\BDGMS$ is dually equivalent to $\SLat$.
\end{enumerate}
\end{corollary}

\begin{proof}
For (\ref{JMS = BDGMS}) apply \cref{thm: iso between AlgLat_S and AlgLat_L,thm: AlgLat = JM,thm: BDGMS = AlgLatL}, and for (\ref{BDGMS = SLat}) apply Theorems~\ref{thm: CohLat = Lat}(\ref{thm: Nachbin 1})  and \ref{thm: BDGMS = AlgLatL}.
\end{proof}

\begin{remark} \label{rem: BDGMS duality with SLat}
The duality between $\BDGMS$ and $\SLat$ is obtained by sending $X$ to $\CLF(X)$ and $f \colon X \to Y$ to $f^{-1} \colon \CLF(Y) \to \CLF(X)$. For the other direction, $A$ is sent to $\Filt(A)$ and $\alpha \colon A \to B$ to $\alpha^{-1} \colon \Filt(B) \to \Filt(A)$.
\end{remark}

We now consider a (non-full) subcategory of $\BDGMS$ that is dually equivalent to $\Lat$.

\begin{definition} \label{def: BDGM space}
Let $\BDGM$ be the category of BDGM-spaces satisfying 
\begin{itemize}
\item finite joins of clopen filters are clopen,
\end{itemize}
and $\BDGM$-morphisms $f \colon X \to Y$ satisfying
\begin{itemize}
\item $f(x) = 1$ implies $x = 1$;
\item $y' \wedge z' \le f(x)$ implies that there are $y, z$ with $y' \le f(y)$, $z' \le f(z)$, and $y \wedge z \le x$. 
\end{itemize}
\end{definition} 

\begin{remark}
The above spaces and morphisms were introduced in \cite{BDGM24} under the name of L-spaces and L-morphisms. 
\end{remark}   

For the next lemma, observe that if $f$ is a $\BDGMS$-morphism, it has a left adjoint by \cref{BDGMS = AlgLatL on morphisms}.

\begin{lemma} \label{lem: when is f a BDGM morphism}
Let $f \colon X \to Y$ be a $\BDGMS$-morphism and $\ell$ its left adjoint. The following are equivalent.
\begin{enumerate}
\item $f$ is a $\BDGM$-morphism.
\item $f^{-1}$ preserves finite joins of principal filters.
\item $f^{-1}$ preserves finite joins of principal filters of compact elements.
\item $\ell$ preserves finite meets.
\item $\ell$ preserves finite meets of compact elements.
\end{enumerate}
\end{lemma}

\begin{proof}
By \cref{lem: ell preserves meets}, conditions (2) to (5) are equivalent. Therefore, it is enough to show that (1) and (4) are equivalent.

(1)$\Longrightarrow$(4) Since $1 \le f \ell (1)$, we have $f \ell (1)=1$, which yields $\ell(1)=1$. Therefore, $\ell$ preserves the empty meet. Let $y',z' \in Y$. Then $y' \wedge z' \le f \ell (y' \wedge z')$. Since $f$ is a $\BDGM$-morphism, there are $y, z \in X$ with $y' \le f(y)$, $z' \le f(z)$, and $y \wedge z \le \ell (y' \wedge z')$. Thus, $\ell(y') \le y$ and $\ell(z') \leq z$, which imply that $\ell(y') \wedge \ell(z') \le \ell (y' \wedge z')$. Because $\ell$ is order-preserving, we conclude that it preserves finite meets.

(4)$\Longrightarrow$(1) If $f(x)=1$, then $1 = \ell(1) \le \ell f(x) \le x$. Therefore, $x=1$. Next, suppose that $x \in X$ and $y', z' \in Y$ with $y' \wedge z' \le f(x)$. Then $\ell(y') \wedge \ell(z')=\ell(y' \wedge z') \le x$. We also have that $y' \le f\ell(y')$ and $z' \le f\ell(z')$. Letting $y = \ell(y')$ and $z= \ell (z')$ yields $y' \le f(y)$, $z' \le f(z)$, and $y \wedge z \le x$. Thus, $f$ is a $\BDGM$-morphism.
\end{proof}

\begin{proposition} \plabel{prop: BDGM = CohLatL}
Let $X, Y \in \BDGMS$ and $f \colon X \to Y$ be a $\BDGMS$-morphism. 
\begin{enumerate}
\item \label[prop: BDGM = CohLatL]{BDGM = CohLatL on objects} $X \in \BDGM$ iff $X \in \CohLatL$.
\item \label[prop: BDGM = CohLatL]{BDGM = CohlatL on morphisms} $f$ is a $\BDGM$-morphism iff $f$ is a $\CohLatL$-morphism.
\end{enumerate}
\end{proposition}

\begin{proof}
(\ref{BDGM = CohLatL on objects}) By \cref{BDGMS = AlgLatL on objects}, $X \in \AlgLatL$. Now, $X \in \BDGM$ iff finite joins of clopen filters are clopen. By \cref{BDGM 1}, this occurs iff $K(X)$ is closed under finite meets, which is equivalent to $X \in \CohLatL$. 

(\ref{BDGM = CohlatL on morphisms}) By \cref{BDGMS = AlgLatL on morphisms}, $f$ is an $\AlgLatL$-morphism. Thus, it is sufficient to apply \cref{lem: ell preserves meets,lem: when is f a BDGM morphism}. 
\end{proof}

As an immediate consequence, we obtain:

\begin{theorem} \plabel{thm: BDGM = CohLatL}
\hfill
\begin{enumerate}
\item \label[thm: BDGM = CohLatL]{BDGM = CohLatL} $\BDGM$ is isomorphic to $\CohLatL$.
\item \label[thm: BDGM = CohLatL]{JML = BDGM} \cite[Rem.~2.15]{BDGM24} The Cornish isomorphism restricts to an isomorphism between $\JML$ and $\BDGM$.
\end{enumerate}
\end{theorem}

\begin{proof}
For (1) apply \cref{thm: BDGMS = AlgLatL,prop: BDGM = CohLatL}; and for (2) apply (1), \cref{thm: iso between CohLat_S and CohLat_L,thm: CohLatL = JML}, and \cref{JMS = BDGMS}.
\end{proof}

Putting \cref{{BDGM = CohLatL}} together with
\cref{thm: Nachbin 2} yields: 

\begin{corollary} \label{cor: BDGM = Lat}
\cite[Thm.~2.14]{BDGM24} $\BDGM$ is dually equivalent to $\Lat$.
\end{corollary}

\begin{remark} \label{rem: duality between BDGM and Lat}
The duality above is a restriction of the duality given in \cref{rem: BDGMS duality with SLat}.
\end{remark}

The following diagrams show the relationships between the various categories of \cref{sec: preliminaries,sec: JM,sec: BDGM}. Double arrows with no label represent isomorphisms of categories and dual equivalences when labeled with $d$.

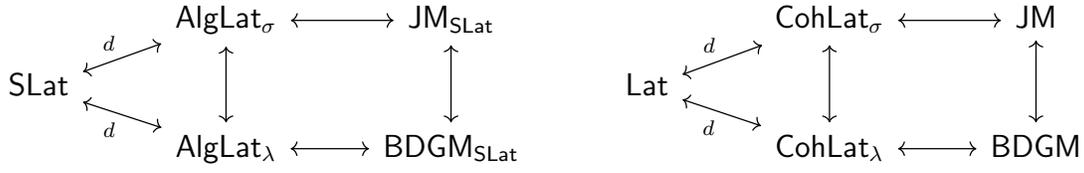
\[
\begin{tikzcd}[column sep = 2.5pc, row sep=0.50pc]
& \AlgLatS \arrow[r, <->] \arrow[dd, <->]& \JMS \arrow[dd, <->] && \CohLatS \arrow[r, <->] \arrow[dd, <->] & \JML \arrow[dd, <->] \\
\SLat \arrow[ur, <->, "d"] \arrow[dr, <->, "d"'] &&& \Lat \arrow[ur, <->,"d"] \arrow[dr, <->,"d"']\\
& \AlgLatL \arrow[r, <->] & \BDGMS && \CohLatL \arrow[r, <->] & \BDGM	
\end{tikzcd}
\]

\section{The Hartonas approach} \label{sec: Hs}

In this section we consider the spaces introduced by Hartonas \cite{Har97}, 
giving rise to two categories that we show are isomorphic to $\BDGMS$ and $\BDGM$, respectively. 

As usual, for a poset $(X, \le)$ and $A \subseteq X$, we let $A^u$ be the set of upper bounds and $A^l$ the set of lower bounds of $A$. 
We then have the closure operator $\Gamma$ on $X$ given by $\Gamma(A) = A^{lu}$.  So $\Gamma X = \{ \Gamma(A) : A \subseteq X\}$ is a complete lattice, where meet is intersection, join is given by $\bigvee A_i = \Gamma(\bigcup_i A_i)$, top is $\Gamma(X) = X$, and bottom is $\Gamma(\varnothing) = X^u$, which is $\{1\}$ if $X$ has a top, and $\varnothing$ otherwise. 
We note that $\Gamma(\{x\}) = \up x$ for each $x \in X$. 

\begin{definition} \label{def: lambda X}
If $X$ is a partially ordered Stone space, we set
\[
\Lambda X =\{ U \in \Gamma X : U \textrm{ is clopen} \}. 
\]
\end{definition}

We recall that a subset $S$ of a complete lattice $X$ is \emph{meet-dense} in $X$ if each element of $X$ is a meet from $S$. 

\begin{definition} \plabel{def: Hartonas space}
\hfill
\begin{enumerate}
\item Let $X$ be a partially ordered Stone space.  We call $X$ an {\em Hartonas space} if
\begin{enumerate}
\item \label[def: Hartonas space]{def: Hartonas space 3} $X$ is a complete lattice;
\item \label[def: Hartonas space]{def: Hartonas space 1} the topology on $X$ is generated by $\Lambda X \cup \{- U : U \in \Lambda X\}$;
\item \label[def: Hartonas space]{def: Hartonas space 4} $\Lambda X$ is meet-dense in $\Gamma X$.
\end{enumerate}
\item \label[def: Hartonas space]{Hs continuity} Let $\HsS$ be the category of Hartonas spaces and continuous maps $f \colon X \to Y$ such that $U \in \Lambda Y$ implies $f^{-1}(U) \in \Lambda X$. 
\end{enumerate}
\end{definition}

\begin{remark} \plabel{rem: properties of Hartonas spaces}
\hfill
\begin{enumerate}
\item \label[rem: properties of Hartonas spaces]{missing axiom} Our definition generalizes that of Hartonas \cite[Def.~4.3]{Har97} to obtain a category equivalent to $\BDGMS$ and dually equivalent to $\SLat$. The axiom \cite[Def.~4.3(2)]{Har97} which is dropped from \cref{def: Hartonas space} will be added in \cref{def: HsL space} to obtain a category equivalent to $\BDGM$ and dually equivalent to $\Lat$.  
\item In \cref{def: Hartonas space}, we start with a partial order $\le$ on a Stone space $X$ and define $\Gamma$ from it. On the other hand, Hartonas starts with $\Gamma$ on $X$ which arises from a binary relation $\prec$, and defines $\le$ by $x \le y$ if $\Gamma(y) \subseteq \Gamma(x)$. The partial order $\le$ and the closure operator $\Gamma$ are definable from one another, but we prefer to start with $\le$ since it is easier to connect the resulting duality with  BDGM-duality.
\item \label[rem: properties of Hartonas spaces]{Gamma elements are filters} The condition that $X$ is a complete lattice is equivalent to Hartonas's condition that $\Gamma X = \{\up x : x \in X\}$. Also, the compactness condition in \cite[Def.~4.2(1)]{Har97} is superfluous since it follows from compactness of $X$. 
\item \label[rem: properties of Hartonas spaces]{up x is closed}
Let $x \in X$. Then $\up x$ is closed since $\up x \in \Gamma X$, so is an intersection from $\Lambda X$ by \cref{def: Hartonas space 4}, and hence is an intersection of clopen sets.  (Note that $\Lambda X$ is closed under finite intersections.)
\item By \crefrange{def: Hartonas space 1}{Hs continuity}, each $\HsS$-morphism is continuous.
\end{enumerate}
\end{remark}

\begin{lemma} \label{lem: KOF = LX}
Let $X$ be an algebraic lattice. 
Viewing $X$ as a topological space with the Lawson topology, $\Lambda X = \{ \up k : k \in K(X)\} = \CLF(X)$.
\end{lemma}

\begin{proof}
By definition, $\Lambda X \subseteq \CLF(X)$.  
By \cref{BDGMS = AlgLatL on objects}, $X$ is a BDGM-space. Therefore, $\CLF(X) = \{ \up k : k \in K(X)\}$ by \cref{BDGM 1}. Finally, $\up k \in \Lambda X$ for each $k\in K(X)$ since it is a clopen principal filter, hence an element of $\Lambda X$ since $\Gamma X = \{\up x : x \in X\}$ (see \cref{Gamma elements are filters}).
\end{proof}

\begin{remark} \label{rem: CLF = KOF} 
Putting \cref{lem: KOF = LX} together with \cref{lem: KOF = up k's} yields that $\KOF(X) = \CLF(X)$ for each algebraic lattice $X$.
\end{remark}

\begin{proposition} \plabel{prop: BDGMS = HsS}
\hfill
\begin{enumerate}
\item \label[prop: BDGMS = HsS]{BDGMS = HsS objects} A partially ordered Stone space is an Hartonas space iff it is a BDGM-space.
\item \label[prop: BDGMS = HsS]{BDGMS = HsS morphisms} A map $f \colon X \to Y$ is a $\BDGMS$-morphism iff $f$ is an $\HsS$-morphism.
\end{enumerate}
\end{proposition}

\begin{proof}
(\ref{BDGMS = HsS objects}) Let $X$ be an Hartonas space. Then $X$ is compact. If $x \not\le y$, then $\up y \not\subseteq \up x$. Since $\Lambda X$ is meet-dense in $\Gamma X$, there is $F \in \Lambda X$ with $\up y \not\subseteq F$ and $\up x \subseteq F$. Therefore, $F$ is a clopen filter containing $x$ but not $y$, and hence $X$ is a BDGM-space.

Conversely, let $X$ be a BDGM-space. By \cref{BDGM 2}, $X$ is a complete lattice, so \cref{def: Hartonas space 3} follows. By Lemmas~\ref{lem: BDGM basic properties}(\ref{BDGM 3}) and \ref{lem: KOF = LX}, $\Lambda X = \CLF(X)$. By \cref{rem: properties of BDGM spaces}, $\{ F - G : F, G \in \CLF(X) \}$ is a subbasis for the topology of $X$, so \cref{def: Hartonas space 1} follows. Finally, for \cref{def: Hartonas space 4}, if $A \in \Gamma X$, then $A = \up x$ for some $x \in X$. Because $\up x$ is the intersection of clopen filters, 
we see that $\Lambda X$ is meet-dense in $\Gamma X$. Thus, $X$ is an Hartonas space.

(\ref{BDGMS = HsS morphisms}) Let $f$ be a $\BDGMS$-morphism. Then $f^{-1}(F) \in \CLF(X)$ for each $F \in \CLF(Y)$. Now apply \cref{lem: KOF = LX} to conclude that $f$ is an $\HsS$-morphism.

Conversely, let $f$ be an $\HsS$-morphism, so $f^{-1}(\Lambda Y) \subseteq \Lambda X$. Since $\Lambda Y$ is meet-dense in $\Gamma Y$, we obtain that $f^{-1}(\Gamma Y) \subseteq \Gamma X$. Therefore, $f^{-1}$ pulls principal filters to principal filters. Thus, there is a map $\ell : Y \to X$ satisfying $f^{-1}(\up y) = \up \ell(y)$. This says exactly that $\ell$ is left adjoint to $f$. Therefore, $f$ preserves all meets, and so is a $\SLat$-morphism. Thus, $f$ is a $\BDGMS$-morphism.
\end{proof}

\cref{prop: BDGMS = HsS} yields:

\begin{theorem} \label{thm: CohLat = Hs}
$\HsS$ is isomorphic to $\BDGMS$.
\end{theorem}

\begin{corollary} \plabel{cor: HsS = AlgLatL}
\hfill
\begin{enumerate}
\item $\HsS$ is isomorphic to $\AlgLatL$ and $\JMS$.
\item \label[cor: HsS = AlgLatL]{cor: HsS = SLat} $\HsS$ is dually equivalent to $\SLat$.
\end{enumerate}
\end{corollary}

\begin{proof}
(1) That $\HsS$ is isomorphic to $\AlgLatL$ follows from \cref{thm: BDGMS = AlgLatL,thm: CohLat = Hs}, and that $\HsS$ is isomorphic to $\JMS$ then follows from \cref{thm: iso between AlgLat_S and AlgLat_L,thm: AlgLat = JM}.

(2) Apply (1) and \cref{thm: Nachbin 1}. 
\end{proof}

\begin{remark} \label{rem: functors for Hs}
The functors that yield the duality of \cref{cor: HsS = SLat} are those of BDGM-duality (see \cref{rem: BDGMS duality with SLat,rem: duality between BDGM and Lat}).
\end{remark}

We now restrict the objects and morphisms of $\HsS$ to yield a category dually equivalent to $\Lat$.

\begin{definition} \label{def: HsL space}
Let $\HsL$ be the category of Hartonas spaces such that $\Lambda X$ is a bounded sublattice of $\Gamma X$, and $\HsS$-morphisms $f \colon X \to Y$ such that $f^{-1} \colon \Lambda Y \to \Lambda X$ preserves finite joins.
\end{definition}

\begin{remark}
The spaces defined in \cref{def: HsL space} are the same as those of \cite[Def.~4.3]{Har97} (see \cref{missing axiom}). Since Hartonas only works with semilattice morphisms between lattices, the extra condition on morphisms does not appear in \cite{Har97}. This extra condition is necessary to establish duality between $\HsL$ and $\Lat$.
\end{remark}

\begin{proposition} \label{prop: BDGM = HsL}
Let $X, Y \in \HsS$.
\begin{enumerate}
\item \label[prop: BDGM = HsL]{BDGM = HsL objects} $X \in \HsL$ iff $X \in \BDGM$.
\item \label[prop: BDGM = HsL]{BDGM = HsL morphisms} Let $f \colon X \to Y$ be an $\HsS$-morphism. Then $f$ is an $\HsL$-morphism iff $f$ is a $\BDGM$-morphism.
\end{enumerate}
\end{proposition}

\begin{proof}
(\ref{BDGM = HsL objects}) Let $X \in \HsS$. By \cref{BDGMS = HsS objects}, $X \in \BDGMS$. Moreover, by \cref{lem: KOF = LX}, $\CLF(X) = \Lambda X$. Therefore, $\CLF(X)$ is closed under finite joins iff $\Lambda X$ is closed under finite joins. Thus, $X \in \HsL$ iff $X \in \BDGM$. 

(\ref{BDGM = HsL morphisms}) Let $f \colon X \to Y$ be an $\HsS$-morphism. By \cref{BDGMS = HsS morphisms}, $f$ is a $\BDGMS$-morphism. Thus, by \cref{lem: when is f a BDGM morphism,lem: KOF = LX}, $f$ is a $\BDGM$-morphism iff $f$ is an $\HsL$-morphism.
\end{proof}

As an immediate consequence of \cref{thm: CohLat = Hs,prop: BDGM = HsL}, we obtain:

\begin{theorem} \label{thm: CohLatL = HsL}
$\HsL$ is isomorphic to $\BDGM$.
\end{theorem}

\begin{corollary} \plabel{cor: JML = HsL}
\hfill
\begin{enumerate}
\item \label[cor: JML = HsL]{thm: JML = HsL}
$\HsL$ is isomorphic to $\JML$.
\item \label[cor: JML = HsL]{cor: HsL = Lat}
$\HsL$ is dually equivalent to $\Lat$.
\end{enumerate}
\end{corollary}

\begin{proof}
(1) Apply Theorems~\ref{thm: CohLatL = HsL} and \ref{thm: BDGM = CohLatL}(\ref{JML = BDGM}).

(2) Apply \cref{thm: CohLatL = HsL,cor: BDGM = Lat}.
\end{proof}

\begin{remark} \plabel{rem: functors for Hs = Lat}
\hfill
\begin{enumerate}
\item \label[rem: functors for Hs = Lat]{Hs = Lat}The functors yielding the duality of \cref{cor: HsL = Lat} are the restrictions of those given in \cref{rem: functors for Hs}.
\item Hartonas worked with meet-semilattice homomorphisms between lattices. Because of this, the lattices he considered only required having top but not necessarily bottom. His duality theorem \cite[Thm.~4.4]{Har97} can be obtained by replacing coherent lattices with arithmetic lattices and $\HsL$-morphisms with $\HsS$-morphisms.
\end{enumerate}
\end{remark}

\section{The Celani-Gonz\'{a}lez approach} \label{sec: CG}

In this section we describe the Celani-Gonz\'{a}lez approach \cite{CG20} (see also \cite{Har24}). Like Jipsen and Moshier, they first describe the category $\CGS$ dual to $\SLat$ and show that $\CGS$ is equivalent to $\JMS$. Next, they restrict their attention to the category $\CG$ and prove that it is dual to $\Lat$. It follows that $\CG$ is equivalent to $\JML$. In this section we show that the functors yielding the equivalence between $\CGS$ and $\JMS$ restrict to give an equivalence between $\CG$ and $\JML$. As a consequence, we conclude our series of equivalences of the upper part of \cref{figure: intro} (see the introduction). 

Celani and Gonz\'{a}lez work with pairs $\C = (X, \K)$ where $X$ is a topological space and $\K$ is a fixed subbasis of $X$. We set 
\[
\L\C = \{ -U : U \in \K\}
\]
and define the corresponding closure operator $\cl_\C$ by 
\[
\cl_\C(Y) = \bigcap \{ A \in \L\C : Y \subseteq A\}.
\]
We call the fixpoints of this closure operator {\em $\Delta_\C$-closed sets}.

Let $Y \subseteq X$ be $\Delta_\C$-closed. A family $\Z \subseteq \L\C$ is a \emph{$Y$-family} if for all $A, B \in \Z$, there are $H \in \L\C$ and $C \in \Z$ such that $Y \subseteq H$ and $A \cap H, B \cap H \subseteq C$.
For a binary relation $S\subseteq X_1 \times X_2$, define $\Box_S \colon \wp(X_2) \to \wp(X_1)$ between the powersets by $\Box_S B = -S^{-1}[-B]$. The notation $\Box_S$ is motivated by its use in modal logic (see, e.g., \cite[p.~64]{CZ97}).

\begin{definition} \cite[Defs.~3.5, 3.14]{CG20} \label{def: CGS}
\begin{enumerate}
\item Let $X$ be a topological space with a fixed subbasis $\K$. The pair $ \C = (X, \K)$ is a \emph{Celani-Gonz\'{a}lez space}, or {\em CG-space} for short, provided
\begin{enumerate}
\item $X$ is a $T_0$-space and $X = \bigcup \K$;
\item each $U \in \K$ is compact open and $\K$ is closed under finite unions (so $\varnothing \in \K$);
\item for each $U, V \in \K$, if $x \in U \cap V$, then there are $W, D \in \K$ with $x \in D - W$ and $D \subseteq (U \cap V) \cup W$;
\item if $Y$ is $\Delta_\C$-closed and $\Z \subseteq \L\C$ is a $Y$-family such that $Y -A \ne \varnothing$ for all $A \in \Z$, then $Y \cap \bigcap \{ -A : A \in \Z \} \ne \varnothing$.
\end{enumerate}
\item Let $\C_1 = (X_1, \K_1)$ and $\C_2 = (X_2, \K_2)$ be CG-spaces. A relation $S \subseteq X_1 \times X_2$ is a {\em CG-morphism} provided
\begin{enumerate}
\item if $B \in \L\C_2$, then $\Box_S B \in \L\C_1$;
\item $S[x]$ is $\Delta_{\C_2}$-closed for each $x \in X_1$.
\end{enumerate}
\item Let $\CGS$ be the category of CG-spaces and CG-morphisms.
\end{enumerate}
\end{definition}

It is proved in \cite[Sec.~3]{CG20} that $\CGS$ is indeed a category, where the composition $S_2 \star S_1$ of two CG-morphisms $S_1, S_2$ is defined by
\[
x (\rel{S_2 \star S_1}) z \iff x \in \Box_{S_1}\Box_{S_2} B \textrm{ implies } z \in B \textrm{ for each } B \in \L\C_3
\] 
and the identity morphism of $(X, \K)$ is the dual  
$\le$ of the specialization order $\sqsubseteq$ of $X$.

\begin{remark} \plabel{rem: basic properties of CGS}
Let $\C = (X, \K)$ be a CG-space.
\begin{enumerate}
\item \label[rem: basic properties of CGS]{SX are upsets} $\K$ is a join-semilattice, where join is union and bottom is $\varnothing$. Consequently, $\L\C$ is a meet-semilattice, where meet is intersection and top is $X$. Since open sets are $\sqsubseteq$-upsets, elements of $\L\C$ are $\le$-upsets.
\item \label[rem: basic properties of CGS]{up x is intersection from SX} If $x \in X$, then $\up x = \bigcap \{ B \in \L\C : x \in B\}$, where $\up$ is calculated with respect to $\le$. To see this, the inclusion $\subseteq$ is clear since each $B \in \L\C$ is a $\le$-upset by (\ref{SX are upsets}). For the reverse inclusion, if $x \not\le y$, then $y \not\sqsubseteq x$, and since $\K$ is a subbasis, there is $U \in \K$ with $y \in U$ and $x \notin U$. Letting $B = -U$, we have $B \in \L\C$, $x \in B$, and $y \notin B$, completing the proof.
\item \label[rem: basic properties of CGS]{CG equivalent condition for Box} The first condition of a CG-morphism is equivalent to: If $V \in \K_2$, then $S^{-1}[V] \in \K_1$.
\end{enumerate}
\end{remark}

\begin{definition} \plabel{def: CG}
\cite[Defs.~3.25, 3.29]{CG20} Let $\CG$ be the category of those $(X, \K) \in \CGS$ that satisfy 
\begin{enumerate}
\item \label[def: CG]{CG X in K} $X \in \K$;
\item \label[def: CG]{CG meet in K} $\bigcup \{ W \in \K : W \subseteq U \cap V \} \in \K$ for all $U, V \in \K$,
\end{enumerate}
and of those CG-morphisms $S \subseteq X_1 \times X_2$ that satisfy 
\begin{enumerate}
\setcounter{enumi}{2}
\item \label[def: CG]{CG serial} $S[x] \ne \varnothing$ for each $x\in X_1$ (that is, $S$ is serial);
\item \label[def: CG]{CG Box join} $\Box_S\cl_{\C_2}(B_1 \cup B_2) \subseteq \cl_{\C_1}(\Box_S B_1 \cup \Box_S B_2)$ for all $B_1, B_2 \in \L\C_2$.
\end{enumerate}
\end{definition}

\begin{remark} \plabel{rem: basic properties of CG}
Let $\C = (X, \K)$ be a CG-space.
\begin{enumerate}
\item \label[rem: basic properties of CG]{CG meet in K alternate} \cref{CG meet in K} says that if $U, V \in \K$, then the meet $U \wedge V$ exists in
$\K$. Consequently, both $\K$ and $\L\C$ are bounded lattices.  The bottom of $\L\C$ is $\varnothing$ by \cref{CG X in K}.
\item \label[rem: basic properties of CG]{CG Box preserves join} The reverse inclusion in \cref{CG Box join} holds automatically, so the condition says that $\Box_S$ preserves binary joins in 
$\L\C$. This is equivalent to 
$S^{-1}$ preserving binary meets in $\K$. 
\end{enumerate}
\end{remark}

Clearly $\CG$ is a (non-full) subcategory of $\CGS$. 
By \cite[Thm.~3.24]{CG20}, $\CGS$ is dually equivalent to $\SLat$, and $\CG$ is dually equivalent to $\Lat$ by \cite[Thm.~3.32]{CG20}. Consequently, $\CG$ is equivalent to $\JML$. To obtain the equivalence of the upper part of \cref{figure: intro}, we give a direct proof that $\CG$ is equivalent to $\JML$, from which the duality between $\CG$ and $\Lat$ follows from \cref{cor: JML = Lat}.

We start by recalling the functor $\mathbb{M} \colon \JMS \to \CGS$ of \cite{CG20}. For $X \in \JMS$ let $X_m$ be the set of meet-irreducible elements of $X$, where we recall that $x \in X - \{ 1 \}$ is \emph{meet-irreducible} if $x = y \wedge z$ 
implies that $x = y$ or $x = z$. 
Set
\[
\K_X = \{  X_m - U : U \in \KOF(X) \}, 
\]
so $\K_X = \{   X_m - \up k: k \in K(X)\}$ by \cref{lem: KOF = up k's}.
By \cite[Thm.~5.20]{CG20}, 
$ \C_X := (X_m, \K_X)$ is a CG-space. 

\begin{remark}
The topology on $X_m$ is the subspace topology of the lower topology on the algebraic lattice $X$. To see this, the lower topology on $X$ has $\{\up x : x \in X\}$ as a subbasis of closed sets. Since $\up x = \bigcap \{ \up k : k \in K(X), k \le x\}$, $-\up x = \bigcup \{ -\up k : k \in K(X), k \le x\}$. 
Therefore, $\{ -\up k : k \in K(X)\}$ is a subbasis of open sets for $X$, and so $\{ X_m - \up k : k \in K(X)\}$ is a subbasis of open sets for $X_m$. Thus, the topology on $X_m$ is the subspace topology of the lower topology on $X$.
\end{remark}

If $f \colon X \to Y$ is a $\JMS$-morphism, define $S_f \subseteq X_m \times Y_m$ by
\[
x \rel{S_f} y \iff f(x) \le y. 
\]
Then $S_f$ is a $\CGS$-morphism, and the assignment $X \mapsto  \C_X$ and $f \mapsto S_f$ defines the functor $\mathbb{M} \colon \JMS \to \CGS$ 
\cite[Thm.~5.30]{CG20}. \cref{thm: functor M} shows that $\mathbb{M}$ restricts to a functor $\JML\to\CG$, the proof of which requires the following two lemmas.

\begin{lemma}  \plabel{lem: 5.11 and 5.13 of CG}
Let $X \in \JMS$, $x\in X$, and $k, l \in K(X)$.
\begin{enumerate}
\item \cite[Lem.~5.11]{CG20} \label[lem: 5.11 and 5.13 of CG]{5.11 of CG} $x = \bigwedge (\up x \cap X_m)$.
\item \cite[Prop.~5.13]{CG20} \label[lem: 5.11 and 5.13 of CG]{5.13 of CG} $ X_m - \up k \subseteq  X_m - \up l \iff k \le l$.
\end{enumerate}
\end{lemma}

\begin{lemma}  \plabel{lem: JML to CG facts}
Let $f \colon X \to Y$ be a $\JML$-morphism and $\ell$ its left adjoint.
\begin{enumerate}
\item \label[lem: JML to CG facts]{Box and up k} If $k \in K(Y)$, then $\Box_{S_f}(\up k \cap Y_m) = \up \ell(k) \cap X_m$.
\item \label[lem: JML to CG facts]{cl and union} If $k, l \in K(Y)$, then $\cl_{ \C_Y}((\up k \cap Y_m) \cup (\up l \cap Y_m)) = \up (k \wedge l) \cap Y_m$.
\end{enumerate}
\end{lemma}

\begin{proof}
(\ref{Box and up k}) For $x \in X_m$, we have
\begin{align*}
x \in \Box_{S_f}(\up k \cap  X_m) &\iff S_f[x] \subseteq \up k \cap X_m \\
&\iff \up f(x) \cap X_m \subseteq \up k \cap X_m \\
&\iff k \le f(x) \\
&\iff x \in f^{-1}(\up k) \cap X_m \\
&\iff x \in \up \ell(k) \cap X_m,
\end{align*}
where the third equivalence follows from \cref{5.11 of CG}.

(\ref{cl and union}) 
By \cref{5.13 of CG}, for each $m \in K(Y)$ we have
\begin{align*}
(\up k \cap Y_m) \cup (\up l \cap  Y_m) \subseteq \up m \cap Y_m &\iff m \le k, l \\
&\iff m \le k \wedge l \\
&\iff \up (k\wedge l) \subseteq \up m.
\end{align*} 
Because $k \wedge l \in K(Y)$ and $\cl_{ \C_Y}(B) = \bigcap \{ \up m \cap  Y_m : B \subseteq \up m\}$ for each $B \subseteq Y_m$ (see \cref{lem: KOF = up k's}), the result follows.
\end{proof}

\begin{proposition} \label{thm: functor M}
\hfill
\begin{enumerate}
\item If $X \in \JML$, then $ \C_X \in \CG$.
\item If $f \colon X \to Y$ is a $\JML$-morphism, then $S_f \subseteq X_m \times Y_m$ is a $\CG$-morphism.
\end{enumerate}
\end{proposition}

\begin{proof}
(1) First, $ \C_X \in \CGS$ by \cite[Thm.~5.20]{CG20}. Since $X \in \JML$, $\up 1 \in \KOF(X)$, so $X_m =  X_m - \up 1 \in  \K_X$. Next, let $U, V \in \K_X$. Then there are $k, l \in K(X)$ with $U = X_m - \up k$ and $V =  X_m - \up l$. Let $W \in \K_X$ with $W \subseteq U \cap V$, and write $W = X_m - \up m$ for some $m \in K(X)$. Then $m \le k, l$ by \cref{5.13 of CG}, so $m \le k \wedge l$. Because $X$ is coherent, $k\wedge l \in K(X)$, so $\up (k\wedge l) \in \KOF(X)$. Thus, $W \subseteq  X_m - \up (k \wedge l)$. This shows that $\bigcup \{ W \in \K_X : W \subseteq U \cap V \} \subseteq   X_m - \up (k\wedge l)$, and the reverse inclusion holds since $ X_m - \up (k \wedge l) \in \K_X$. Consequently, $ \C_X \in \CG$.

(2) By \cite[Prop.~5.24]{CG20}, $S_f$ is a $\CGS$-morphism. Let $x \in X_m$. If $\ell$ is the left adjoint of $f$, then $f(x) = 1$ implies $x = 1$ (because $\ell(1) = 1$).  Thus, since $1 \notin X_m$, we have $f(x) \ne 1$. Therefore, \cref{5.11 of CG} yields $y \in X_m$ with $f(x) \le y$. Consequently, $y \in S_f[x]$, so $S_f[x] \ne \varnothing$. 
Next, let $B_1, B_2 \in \L \C_Y$. Then there are $m_1, m_2 \in K(Y)$ with $B_i = \up m_i \cap Y_m$. By \cref{cl and union}, 
\[
\cl_{ \C_Y}(B_1 \cup B_2) = \cl_{ \C_Y}((\up m_1 \cap Y_m) \cup (\up m_2 \cap Y_m)) = \up (m_1 \wedge m_2) \cap Y_m,
\]
so
\[
\Box_{S_f}\cl_{ \C_Y}(B_1 \cup B_2) = \up \ell(m_1 \wedge m_2) \cap X_m
\]
by \cref{Box and up k}. Also, using \cref{lem: JML to CG facts} again,
\begin{align*}
\cl_{ \C_X}(\Box_{S_f}B_1 \cup \Box_{S_f}B_2) &= \cl_{ \C_X}(\Box_{S_f}(\up m_1 \cap Y_m) \cup \Box_{S_f}(\up m_2 \cap Y_m)) \\
&= \cl_{ \C_X}((\up \ell(m_1) \cap X_m) \cup (\up \ell(m_2) \cap X_m)) \\
&= \up (\ell(m_1) \wedge \ell(m_2)) \cap X_m.
\end{align*}
Because $\ell(m_1 \wedge m_2) = \ell(m_1) \wedge \ell(m_2)$, we conclude
\begin{align*}
\Box_{S_f}\cl_{ \C_Y}(B_1 \cup B_2) &=  \up \ell(m_1 \wedge m_2) \cap X_m\\
&=  (\up \ell(m_1) \wedge \ell(m_2)) \cap X_m\\
&= \cl_{ \C_X}(\Box_{S_f}B_1 \cup \Box_{S_f}B_2).
\end{align*}
Consequently, $S_f$ is a $\CG$-morphism.
\end{proof}

We next recall the functor $\mathbb{H} \colon \CGS \to \JMS$ (defined immediately before \cite[Thm.~5.30]{CG20}). Let $\C = (X, \K)$ be a CG-space. Set
\[
\H(\C) = \left\{ \bigcup \mathcal{S} \in \wp(X) : \mathcal{S} \subseteq \K \right\},
\]
and for $U \in \K$ set
\[
\Psi_U = \{ H \in \H(\C) : U \subseteq H\}. 
\]
By \cite[Thm.~5.9]{CG20}, if we topologize $\H(\C)$ with the subbasis $\{ \Psi_U : U \in \K\}$, we obtain that $\H(\C) \in \JMS$. In particular, it is a complete lattice, where join is union, and meet is given by $\bigwedge \mathcal{S} = \bigcup \{ U \in \K : U \subseteq \bigcap \mathcal{S} \}$.

Let $\C_1 = (X_1, \K_1)$ and $\C_2 = (X_2, \K_2)$ be CG-spaces. For a $\CGS$-morphism $S \subseteq X_1 \times X_2$, define $f_S \colon \H(\C_1) \to \H(\C_2)$ by 
\[
f_S(H) = \bigcup \{ V \in \K_2 :  S^{-1}[V] \subseteq H\}.
\]
Then $f_S$ is a $\JMS$-morphism, and the assignment $\C \mapsto \H(\C)$ and $S \mapsto f_S$ defines the functor $\mathbb{H} \colon \CGS \to \JMS$ 
\cite[Thm.~5.30]{CG20}.
\cref{thm: functor H} shows that $\mathbb{H}$ restricts to a functor $\CG \to \JML$, the proof of which requires the following:

\begin{lemma}  \plabel{lem: 5.7 of CG}
Let $\C = (X, \K)$ be a CG-space. 
\begin{enumerate}
\item \label[lem: 5.7 of CG]{5.7 of CG} $\KOF(\H(\C)) = \{ \Psi_U : U \in \K \}$.
\item \label[lem: 5.7 of CG]{CG K(H)} $K(\H(\C)) = \K$.
\end{enumerate}
\end{lemma}

\begin{proof}
(\ref{5.7 of CG}) This is proved in \cite[Prop.~5.7]{CG20}. 

(\ref{CG K(H)}) Since $\Psi_U = \up U$ for each $U \in \K$, it is enough to observe that $\H(\C) \in \JMS$ and apply (\ref{5.7 of CG}) and \cref{lem: KOF = up k's}. 
\end{proof}

\begin{proposition} \label{thm: functor H}
\hfill
\begin{enumerate}
\item If $\C \in \CG$, then $\H(\C) \in \JML$.
\item If $\C_1 ,\C_2 \in \CG$ and $S \colon \C_1 \to \C_2$ is a $\CG$-morphism, then $f_S \colon \H(\C_1) \to \H(\C_2)$ is a $\JML$-morphism.
\end{enumerate}
\end{proposition}

\begin{proof}
(1) Let  $\C = (X, \K)$. By \cite[Thm.~5.9]{CG20}, $\H(\C) \in \JMS$. To see that $\H(\C) \in \JML$, we need to see that $\KOF(\H(\C))$ is a bounded sublattice of $\FSat(\H(\C))$. 
By \cref{5.7 of CG}, we need to show that if $U, V \in \K$, then $\Psi_U \vee \Psi_V \in \KOF(\H(\C))$, where the join is taken in $\FSat(\H(\C))$. Because $U \wedge V \in \K$, the join is contained in $\Psi_{U \wedge V}$. If $\mathcal{F}$ is an open filter of $\H(\C)$ containing $\Psi_U \cup \Psi_V$, then $U, V \in \mathcal{F}$, so $U \wedge V \in \mathcal{F}$, and hence $\Psi_{U \wedge V} \subseteq \mathcal{F}$. Thus, $\Psi_U \vee \Psi_V = \Psi_{U \wedge V} \in \KOF(\H(\C))$. Since $ \{X\} = \Psi_X$, the result follows.

(2) Let $\C_1 = (X_1, \K_1)$ and $\C_2 =(X_2, \K_2)$. By \cite[Prop.~5.27]{CG20}, $f_S$ is a $\JMS$-morphism, and hence an $\AlgLatS$-morphism. By \cref{lem: morphisms of JML}, to see that $f_S$ is a $\JML$-morphism, it suffices to show that the left adjoint of $f_S$ preserves finite meets of compact elements. It is straightforward to verify that the left adjoint of $f_S$ is given by $S^{-1} \colon \H(\C_2) \to \H(\C_1)$ (which is well defined by \cref{CG equivalent condition for Box}). By \cref{CG Box preserves join}, $S^{-1}$ preserves binary meets from $\K_2$. The result 
follows since $K(\H(\C_2)) = \K_2$ by \cref{CG K(H)}.
\end{proof}

We are ready to prove that $\CG$ is equivalent to $\JML$.

\begin{theorem} \label{thm: CG = JM}
The functors $\mathbb{M}$ and $\mathbb{H}$ restrict to yield an equivalence between $\CG$ and $\JML$.
\end{theorem}

\begin{proof}
By \cite[Thm.~20]{CG20},  $\mathbb{M}$ and $\mathbb{H}$ yield an equivalence between $\JMS$ and $\CGS$. Therefore, in view of \cref{thm: functor M,thm: functor H}, it is enough to observe that each $\JMS$-isomorphism is a $\JML$-isomorphism and that each $\CGS$-isomorphism is a $\CG$-isomorphism. The former is obvious since these are order-isomorphisms. To see the latter, let $S$ be a $\CGS$-isomorphism from $\C_1$ to $\C_2$ with inverse $S'$. It is sufficient to show that $S$ is a $\CG$-morphism. By \cite[Thm.~3.24]{CG20}, $\Box_S \colon \L\C_2 \to \L\C_1$ is  an $\SLat$-isomorphism between lattices. Therefore, $\Box_S$ is a $\Lat$-isomorphism, and so $S$ satisfies \cref{CG Box join}. To see that $S$ also satisfies \cref{CG serial}, let $S[x] =\varnothing$. Then $x \in \Box_S\Box_{S'}B$ for each $B \in \L\C_1$. Thus, 
\[
(\forall x' \in X_1,\forall B \in \L\C_1)\,(x (\rel{S' \star S}) x' \Longrightarrow x' \in B).
\]
Since ${\le} = S' \star S$, we have that 
\[
(\forall x' \in X_1,\forall B \in \L\C_1)\,(x \le x' \Longrightarrow x' \in B),
\]
and hence $\up x = \bigcap \L\C_1$.
On the other hand, since $\bigcup \K = X$, we get that $\bigcap \L\C_1 = \varnothing$.  The obtained contradiction shows that $S[x] \ne \varnothing$. Thus, each $\CGS$-isomorphism is a $\CG$-isomorphism, concluding the proof. 
\end{proof}

As an immediate consequence of the above theorem 
and \cref{cor: JML = Lat}, we obtain:

\begin{corollary}
\cite[Thm.~3.32]{CG20} $\CG$ is dually equivalent to $\Lat$.
\end{corollary}

\begin{remark} \label{rem: CG = Lat}
The functor from $\CG$ to $\Lat$ sends $\C = (X,\K)$ to $\L\C$ and $S \colon \C_1 \to \C_2$ to $\Box_S \colon \L\C_2 \to \L\C_1$. To describe the functor in the other direction, for a lattice $A$, let $X_A$ be the set of meet-irreducible filters of $A$, let  $\mathfrak s \colon A \to \wp (X_A)$ be the Stone map 
\[
\mathfrak s(a) = \{ x \in X_A : a \in x \},
\]
and let $\K_A = \{ -\mathfrak s(a) : a \in A \}$. Then the functor from $\Lat$ to $\CG$ sends $A$ to the pair $\C(A) = (X_A,\K_A)$ and $\alpha \colon A \to B$ to $S_\alpha \colon X_B \to X_A$ given by $y \rel{S_\alpha} x$ if $\alpha^{-1}(y) \subseteq x$.    
\end{remark}

Putting together the results of \cref{sec: preliminaries,sec: JM,sec: BDGM,sec: Hs,sec: CG}, we arrive
at the  isomorphisms and equivalences of the upper part of \cref{figure: intro}:

\[
\begin{tikzcd}[column sep = 5pc]
\CohLatS \arrow[d, <->, "\ref{thm: iso between CohLat_S and CohLat_L}"'] \arrow[r, <->, "\ref{thm: CohLatL = JML}"]& \JML \arrow[d, <->, "\ref{thm: BDGM = CohLatL}(\ref{JML = BDGM})"] \arrow[r, "\ref{thm: CG = JM}"] & \CG \\
\CohLatL \arrow[r, <->, "\ref{thm: BDGM = CohLatL}(\ref{BDGM = CohLatL})"'] & \BDGM \arrow[r, <->, "\ref{thm: CohLatL = HsL}"'] & \HsL
\end{tikzcd}
\]


\section{The Dunn-Hartonas approach} \label{sec: DH}

Prior to Hartonas \cite{Har97}, Dunn and Hartonas \cite{HD97} developed another duality for $\Lat$. For this they introduced the notion of an \FSpace\ (filter space), a canonical L-frame (lattice frame), and proved that the category of canonical L-frames is dually equivalent to $\Lat$. An explicit axiomatization of the category of canonical L-frames was not given in \cite{HD97}. We fill in this gap and show that it is equivalent to $\CohLatL$, thus recovering Dunn-Hartonas duality from Hartonas duality.
\begin{definition} \plabel{def: FSpace} \cite[Def.~2.1]{HD97}
A topological space $X$ is an \emph{\FSpace} if $X$ is a partially ordered Stone space such that
\begin{enumerate}
\item \label[def: FSpace]{def: FSpace 1} $X$ is a complete lattice;
\item \label[def: FSpace]{def: FSpace 2} There is a family $X^*$ of principal clopen upsets such that $X^* \cup \{- V : V \in X^* \}$ is a subbasis of open sets;
\item \label[def: FSpace]{def: FSpace 4} The set $\{ x \in X : \up x \in X^*\}$ is join-dense in $X$;
\item \label[def: FSpace]{def: FSpace 5} $X^*$ is closed under finite intersections.
\end{enumerate}
\end{definition}

We show that F-spaces are precisely algebraic lattices with the Lawson topology. For this we require the following lemma.

\begin{lemma} \plabel{lem: properties of FSpaces}
Let $X$ be an \FSpace.
\begin{enumerate}
\item \label[lem: properties of FSpaces]{lem: properties of FSpaces 1} The set $\{ x \in X : \up x \in X^*\}$ is closed under finite joins.
\item \label[lem: properties of FSpaces]{lem: properties of FSpaces 2} $\up x$ is closed for each $x \in X$.
\item \label[lem: properties of FSpaces]{lem: properties of FSpaces 3} If $\up x$ clopen, then $x \in K(X)$.
\item \label[lem: properties of FSpaces]{lem: properties of FSpaces 4} $X^* = \{ \up x : x \in K(X) \}$.
\end{enumerate}
\end{lemma}

\begin{proof}
\eqref{lem: properties of FSpaces 1} If $S$ is a finite subset of $\{ x \in X : \up x \in X^*\}$, then $\up \bigvee S = \bigcap \{ \up x : x \in S \}$. Therefore, $\bigvee S \in \{ x \in X : \up x \in X^*\}$ by \cref{def: FSpace 5}.

\eqref{lem: properties of FSpaces 2} Let $x \in X$. By \cref{def: FSpace 4}, $x = \bigvee S$ for some $S \subseteq X$ with $\up k \in X^*$ for each $k \in S$. Therefore, $\bigcap \{ \up k : k \in S \} = \up x$. Because each $\up k$ is clopen, $\up x$ is closed.

\eqref{lem: properties of FSpaces 3} Let $\up x$ be clopen and $x \le \bigvee S$ for some $S \subseteq X$. Then $\bigcap \{ \up y : y \in S \} \subseteq \up x$. Since each $\up y$ is closed by (2) and $\up x$ is open, compactness shows that there are $y_1, \dots, y_n \in S$ with $\up y_1 \cap \dots \cap \up y_n \subseteq \up x$. Therefore, $x \le y_1 \vee \dots \vee y_n$, and so $x \in K(X)$.

\eqref{lem: properties of FSpaces 4} If $\up x \in X^*$, then $x \in K(X)$ by (3). For the reverse inclusion, $x \in K(X)$ and \cref{def: FSpace 4} imply that $x$ is a finite join $x = k_1 \vee \dots \vee k_n$ for some $k_i \in X$ with $\up k_i \in X^*$. Thus, $\up x = \up k_1 \cap \dots \cap \up k_n \in X^*$ by \cref{def: FSpace 5}.
\end{proof}

\begin{proposition} \label{prop: FSpace = algebraic lattice}
Let $X$ be a partially ordered Stone space. Then $X$ is an \FSpace\ iff $X$ is an algebraic lattice with the Lawson topology, and $X^* = \CLF(X)$.
\end{proposition}

\begin{proof}
Let $X$ be an \FSpace. Then $X$ is a complete lattice by \cref{def: FSpace 1}. By \cref{def: FSpace 4,lem: properties of FSpaces 4}, $X$ is algebraic. Because $\{ \up k : k \in K(X) \} \cup \{ - \up l : l \in K(X)\}$ is a subbasis for the Lawson topology, the topology of $X$ is the Lawson topology by \cref{def: FSpace 2}. Moreover,  Lemmas~\ref{lem: KOF = LX} and \ref{lem: properties of FSpaces}(\ref{lem: properties of FSpaces 4}) show that $X^* = \CLF(X)$.

Conversely, suppose that $X$ is an algebraic lattice. Then $X$ is a complete lattice, so \cref{def: FSpace 1} holds. We set $X^* = \{ \up k : k \in K(X) \}$. Then \cref{def: FSpace 2} holds since $\{ \up k : k \in K(X)\} \cup \{ -\up l : l \in K(X) \}$ is a subbasis for the Lawson topology. Because $X$ is an algebraic lattice, $K(X)$ is join-dense in $X$, so \cref{def: FSpace 4} holds. Finally, $X^*$ is closed under finite intersections since $K(X)$ is closed under finite joins. Thus, $X $ is an \FSpace.
\end{proof}

By \cite[p.~399]{HD97}, the canonical L-frame associated to a bounded lattice $A$ is the triple $(\Filt(A), R, \Idl(A))$ such that $F \rel{R} I$ iff $F \cap I \ne \varnothing$. In order to axiomatize such triples, it is more convenient to work with the complement $-R$ of the relation $R$ since it turns out to be an interior relation. For Stone spaces $X, Y$ we recall (see, e.g., \cite[Def.~6]{Esa78}, \cite[Def.~5.1]{BGHJ19}) that a
relation $R \subseteq X \times Y$ is \emph{interior} provided that
\begin{enumerate}
\item $R[x]$ is closed for each $x \in X$ and $R[U]$ is clopen for each clopen $U \subseteq X$;
\item $R^{-1}[y]$ is closed for each $y \in Y$ and $R^{-1}[V]$ is clopen for each clopen $V \subseteq Y$.
\end{enumerate}

\begin{remark}
We recall \cite[p.~53]{Hal62} that a relation $R$ between Stone spaces $X$ and $Y$ is continuous if $R[x]$ is closed for each $x \in X$ and $R^{-1}[V]$ is clopen for each clopen $V \subseteq Y$. Interior relations are strengthening of continuous relations in that $R$ is interior iff $R$ and $R^{-1}$ are continuous. 
\end{remark}
The next two definitions are reminiscent to what happens in positive modal logic; see \cite{Gol89,CJ99} or more recent \cite{Gol20,MB23}.

\begin{definition} \plabel{def: DH relation}
Let $X, Y \in \AlgLatL$ and $R \subseteq X \times Y$ be a relation. We call $R$ a \emph{DH-relation} provided 
\begin{enumerate}
\item \label[def: DH relation]{def: DH relation interior} $R$ is an interior relation;
\item \label[def: DH relation]{def: DH relation filter} For each $x \in X$ and $y \in Y$, both $-R[x]$ and $-R^{-1}[y]$ are filters.
\item \label[def: DH relation]{def: DH relation 3}  $R[x] \subseteq R[x'] \Longrightarrow x' \le x$;
\item  \label[def: DH relation]{def: DH relation 4} $R^{-1}[y] \subseteq R^{-1}[y'] \Longrightarrow y' \le y$.
\end{enumerate}
\end{definition}

\begin{remark} \label{rem: R and le for DH}
By \cref{def: DH relation filter}, the reverse implications in \crefrange{def: DH relation 3}{def: DH relation 4} also hold. To see the reverse implication of \cref{def: DH relation 3}, suppose $x' \le x$. If $R[x] \not\subseteq R[x']$, then there is $y \in R[x]$ such that $y \notin R[x']$. Therefore, $x' \in -R^{-1}[y]$. Since $-R^{-1}[y]$ is a filter and $x' \le x$, we have $x \in -R^{-1}[y]$, so $y \notin R[x]$, a contradiction. The argument for the reverse implication of \cref{def: DH relation 4} is similar.
\end{remark}
\begin{definition} \plabel{def: DH morphisms}
\hfill
\begin{enumerate}
\item We call a triple $(X, R, Y)$ a \emph{Dunn-Hartonas space}, or \emph{DH-space} for short, if $X, Y \in \CohLatL$ and $R \subseteq X \times Y$ is a DH-relation.
\item \label[def: DH morphisms]{DH maps}A \emph{DH-morphism} between DH-spaces $(X_1, R_1, Y_1)$ and $(X_2, R_2, Y_2)$ is a pair of $\CohLatL$-morphisms $f \colon X_1 \to X_2$ and $g \colon Y_1 \to Y_2$ which satisfies
\begin{enumerate}
\item \label[def: DH morphisms]{def: DH morphisms 1} If $x \rel{R_1} y$, then $f(x) \rel{R_2} g(y)$;
\item \label[def: DH morphisms]{def: DH morphisms 2} If $x' \rel{R_2} g(y)$, then there is $x \in X_1$ with $x \rel{R_1} y$ and $x' \le f(x)$;
\item \label[def: DH morphisms]{def: DH morphisms 3} If $f(x) \rel{R_2} y'$, then there is $y \in Y_1$ with $x \rel{R_1} y$ and $y' \le g(y)$.
\end{enumerate}
\item Let $\HD$ be the category of DH-spaces and DH-morphisms, where composition is given by $(f_2, g_2) \circ (f_1, g_1) = (f_2 \circ f_1, g_2\circ g_1)$ and $(1_X, 1_Y)$ is the identity morphism on $(X, R, Y)$.
\end{enumerate}
\end{definition}

\begin{remark} \label{rem: morphisms}
Let $(X_1, R_1, Y_1), (X_2, R_2, Y_2) \in \HD$ and $f \colon X_1 \to X_2$ and $g \colon Y_1 \to Y_2$ be a pair of functions. It is straightforward to verify that $(f, g)$ is a $\HD$-isomorphism iff $f, g$ are order-isomorphisms such that $x \rel{R_1} y$ iff $f(x) \rel{R_2} g(y)$. When this occurs, $(f^{-1}, g^{-1})$ is the inverse of $(f, g)$ in $\HD$. This will be used throughout.
\end{remark}

We now establish an equivalence between $\HD$ and $\CohLatL$. It follows from \cref{def: DH morphisms} that there is a forgetful functor $\HD \to \CohLatL$:

\begin{proposition} \label{prop: functor F}
There is a functor $\FF \colon \HD \to \CohLatL$ defined by $\FF(X, R, Y) = X$ and $\FF(f, g) = f$ for each $(X, R, Y) \in \HD$ and each $\HD$-morphism $(f, g)$.
\end{proposition}

To define a functor in the opposite direction, we require some preparation.

\begin{notation} \label{not: Galois connection}
Let $R \subseteq X \times Y$ be a DH-relation. Using the notation of \cite[p.~124]{Bir79}, we denote by $\phi \colon \wp X \to \wp Y$ and $\psi \colon \wp Y \to \wp X$ the induced (antitone) Galois connection maps associated to the complement $-R$ of the relation $R$.
Therefore,
\[
\phi A = -R[A] \textrm{ and }\psi B = -R^{-1}[B]
\]
for each $A \subseteq X$ and $B \subseteq Y$. If $x \in X$ and $y \in Y$, we write $\phi(x)$ for $\phi(\{x\})$ and $\psi(y)$ for $\psi(\{y\})$.
\end{notation}

\begin{lemma} \plabel{lem: properties of DH relations}
Let $R \subseteq X \times Y$ be a DH-relation with $X, Y \in \AlgLatL$.
\begin{enumerate}
\item \label[lem: properties of DH relations]{lem: open filter} If $y \in Y$, then $\psi(y)$ is an open filter of $X$.
\item \label[lem: properties of DH relations]{lem: open filter 2} If $x \in X$, then $\phi(x)$ is an open filter of $Y$.
\item \label[lem: properties of DH relations]{phi psi is 1} $\psi\phi U = U$ for $U \in \CLF(X)$ and $\phi\psi V = V$ for $V \in \CLF(Y)$.
\item \label[lem: properties of DH relations]{lem: dual iso} $\phi$ and $\psi$ restrict to yield dual isomorphisms between $\CLF(X)$ and $\CLF(Y)$.
\end{enumerate}
\end{lemma}

\begin{proof}
\eqref{lem: open filter}, \eqref{lem: open filter 2} By \cref{def: DH relation filter}, $\psi(y)$ and $\phi(x)$ are filters. They are open by \cref{def: DH relation interior}.

\eqref{phi psi is 1} Let $U \in \CLF(X)$. Then $U = \up k$ for some $k \in K(X)$ by \cref{lem: KOF = LX}. We have 
\begin{align*}
\psi\phi U &= -R^{-1}[-R[U]] = \{ x \in X : R[x] \subseteq R[\up k] \} = \{ x \in X : R[x] \subseteq R[k] \} \\
&= \{ x \in X : k \le x \} = U,
\end{align*}
where the third equality holds since $R^{-1}[y]$ is a downset for each $y \in Y$, and the fourth equality by \cref{rem: R and le for DH}. Similarly, if $V \in \CLF(Y)$, then $V = \phi\psi V$. 

\eqref{lem: dual iso} Let $U \in \CLF(X)$. Then $\phi U$ is clopen since $R$ is interior, and
\[
\phi U = -R[U] = -\bigcup \{ R[x] : x \in U \} = \bigcap \{\phi(x) : x \in U\}
\]
is a filter by \eqref{lem: open filter 2}. Thus, $\phi U \in \CLF(Y)$. A similar argument, using (1) rather than (2), shows that $\psi V \in \CLF(X)$ for each $V \in \CLF(Y)$. Since $\phi$ and $\psi$ are order-reversing maps, it suffices to apply \eqref{phi psi is 1}.
\end{proof}

\begin{remark} \label{rem: X and Y coherent}
Let $R \subseteq X \times Y$ be a DH-relation with $X, Y$ compact algebraic lattices. By \cref{lem: dual iso}, $\CLF(X)$ and $\CLF(Y)$ are dually isomorphic meet-semilattices. Therefore, each is a bounded lattice. It then follows from \cref{lem: KOF = LX} that $K(X)$ and $K(Y)$ are bounded lattices, and hence that $X$ and $Y$ are coherent lattices. In particular, the bounds of $\CLF(X)$ are $\up 1$ and $X = \up 0$, and join and meet are given by 
\[
\up k \vee \up l = \up (k \wedge l) \ \textrm{ and }\ \up k \wedge \up l = \up k \cap \up l = \up (k \vee l)
\]
for $k, l \in K(X)$.
\end{remark}

\begin{lemma} \plabel{lem: OF is coherent}
Let $X \in \CohLatL$. Then $\OF(X)$ is isomorphic to $\Idl(\CLF(X))$. Consequently, $\OF(X) \in \CohLatL$ and $K(\OF(X)) = \CLF(X)$.
\end{lemma}

\begin{proof}
Every open filter is a Lawson-open upset, so a Scott-open set by \cref{Lawson upset is Scott}. Therefore, it is a union from $\CLF(X)$ by \cref{basis,lem: KOF = LX}. This union is a directed union by the description of join in $\CLF(X)$ given in \cref{rem: X and Y coherent}. It is then straightforward to see that the desired isomorphism is the map $f \colon \OF(X) \to \Idl(\CLF(X))$ defined by $f(F) = \{ U \in \CLF(X) : U \subseteq F\}$, whose inverse $g \colon \Idl(\CLF(X)) \to \OF(X)$ is given by $g(I) = \bigcup \{ U \in \CLF(X) : U \in I\}$. Thus, $\OF(X) \in \CohLatL$ by \cref{Nachbin ideal}, and $K(\OF(X)) = \CLF(X)$ because the compact elements of $\Idl(\CLF(X))$ are the principal downsets, which are sent to $\CLF(X)$ via $g$.
\end{proof}

Recalling that we are axiomatizing the complement of the relation of a canonical L-frame, we define:

\begin{definition} \label{def: R_X}
If $X \in \CohLatL$, define $R \subseteq X \times \OF(X)$ by $x \rel{R} U$ if $x \notin U$.
\end{definition}

\begin{proposition} \label{lem: triple in HD}
If $X \in \CohLatL$, then $(X, R, \OF(X)) \in \HD$.
\end{proposition}

\begin{proof}
Since $X \in \CohLatL$, we have $\OF(X) \in \CohLatL$ by \cref{lem: OF is coherent}. We show that $R$ is a DH-relation. For \cref{def: DH relation filter}, if $x \in X$, then $-R[x] = \{ U \in \OF(X) : x \in U\}$ is clearly a filter. If $U \in \OF(X)$, then $-R^{-1}[U] = \{ x \in X : x \in U \} = U$, so is a filter. 
For \cref{def: DH relation 3}, suppose that $R[x] \subseteq R[x']$. If $x' \not\le x$, then there is $k \in K(X)$ with $k \le x'$ and $k \not\le x$. If $V = \up k$, then $V \in \OF(X)$ with $V \in R[x]$ and $V \notin R[x']$. This contradiction shows that $x' \le x$. To see \cref{def: DH relation 4}, let $U, U' \in \OF(X)$ with $R^{-1}[U] \subseteq R^{-1}[U']$. If $U' \not\subseteq U$, then there is $x \in X$ with $x \in U' - U$. Therefore, $x \in R^{-1}[U]$ and $x \notin R^{-1}[U']$, a contradiction. Thus, $U' \subseteq U$.

It is left to verify \cref{def: DH relation interior}. Let $x \in X$. To see that $R[x]$ is closed, let $V \in -R[x]$, so $x \in V$. Because $\up x = \bigcap \{ \up k : k \in K(X), k \le x\}$, compactness of the Lawson topology shows that there are $k_1, \dots, k_n \in K(X)$ with each $k_i \le x$ and $\up k_1 \cap \dots \cap \up k_n \subseteq V$. Set $k = k_1 \vee \dots \vee k_n$. Then $k \in K(X)$ with $k \le x$ and $\up k = \up k_1 \cap \dots \cap \up k_n \subseteq V$. Let $U = \up k$, so $U \in \CLF(X)$ by \cref{lem: KOF = LX}. Then $\mathcal{W}_U := \{ F \in \OF(X) : U \subseteq F \}$ is Scott-open by \cref{lem: OF is coherent}, and hence Lawson-open.  Therefore, $\mathcal{W}_U$ is an open neighborhood of $V$. Furthermore, if $F \in \mathcal{W}_U$, then $k \in F$, so $x \in F$, and hence $x \nr{R} F$. Thus, $\mathcal{W}_U \subseteq -R[x]$, and so $R[x]$ is closed.

Next, let $C$ be clopen in $\OF(X)$. By \cref{Lawson is patch}, we may assume that $C = \mathcal{W}_U - \mathcal{W}_V$ with $U, V \in \CLF(X)$. If $C = \varnothing$, then $R^{-1}[C] = \varnothing$ is clopen. Otherwise, $V \not\subseteq U$. We show that $R^{-1}[C] = -U$. If $x \in R^{-1}[C]$, then there is $W \in C$ with $x \rel{R} W$. This means that $U \subseteq W$, $V \not\subseteq W$, and $x \notin W$. Therefore, $x \notin U$, and so $x \in -U$. Conversely, if $x \in -U$, then $U \in C$ and $x\rel{R}U$, so $x \in R^{-1}[C]$. Thus, $R^{-1}[C] = -U$, and hence is clopen. 
This implies that $R$ is a continuous relation, and so $R^{-1}[F]$ is closed for each $F \in \OF(X)$ (see, e.g., \cite[Thm.~3.1.2]{Esa19}). 

Finally, let $U$ be clopen in $X$. To see that $R[U]$ is clopen in $\OF(X)$, by \cref{Lawson is patch} it suffices to assume that $U = \up k - \up l$ for some $k, l \in K(X)$. If $U = \varnothing$ then $R[U] = \varnothing$ is clopen, so we assume that $U$ is nonempty, which means that $l \not\le k$. We claim that $R[U] = -\mathcal{W}_{\up k}$. To see this, if $V \in \OF(X)$ with $V \in R[U]$, then there is $x \notin V$ with $k \le x$ and $l \not\le x$. Therefore, $k \notin V$, so $\up k \not\subseteq V$, and hence $V \in -\mathcal{W}_{\up k}$. For the reverse inclusion, if $V \in -\mathcal{W}_{\up k}$, then $k \notin V$. Therefore, $k \in U$ and $k \notin V$, so $V \in R[U]$. Thus, $R[U]$ is clopen, and hence $R$ is an interior relation.

Consequently, $R$ is a DH-relation, yielding that $(X, R, \OF(X)) \in \HD$.
\end{proof}

\begin{lemma} \label{lem: OF is coherent 3}
If $f \colon X \to Y$ is a $\CohLatL$-morphism, then $f^{-1} \colon \OF(Y) \to \OF(X)$ has a right adjoint $r \colon \OF(X) \to \OF(Y)$ which is a $\CohLatL$-morphism.
\end{lemma}

\begin{proof}
Since $f$ is a $\CohLatL$-morphism, it is a $\JML$-morphism by \cref{thm: Nachbin 2} and Lemma \ref{lem: morphisms of JML}, and hence $f^{-1} \colon \OF(Y) \to \OF(X)$ preserves arbitrary joins by \cite[Lem.~5.3]{MJ14a}. It therefore has a right adjoint $r \colon \OF(X) \to \OF(Y)$ which preserves all meets. Recall that if $U \in \OF(X)$, then $r(U) = \bigvee \{ W \in \OF(Y) : f^{-1}(W) \subseteq U\}$. Because this family is directed and each $W \in \OF(Y)$ is a union from $\CLF(Y)$, we may write
\begin{equation}
r(U) = \bigcup \{ V \in \CLF(Y) : f^{-1}(V) \subseteq U\}. \label{eqn}
\end{equation}

To show that $r$ is a $\CohLatL$-morphism, it suffices to show that $r$ preserves directed joins (see Theorems~\ref{thm: Scott facts}(\ref{Scott continuous}) and \ref{thm: Lawson facts}(\ref{Lawson continuous})). Let $\mathcal{S} \subseteq \OF(X)$ be directed and set $U = \bigcup \mathcal{S}$. Since $f^{-1}$ preserves joins, $\{ V \in \CLF(Y) : f^{-1}(V) \subseteq U\}$ is directed. Therefore,
\begin{align*}
r(U) &= \bigcup \left\{ V \in \CLF(Y) : f^{-1}(V) \subseteq U\right\} \\
&= \bigcup \left\{ V \in \CLF(Y) : f^{-1}(V) \subseteq W \textrm{ for some } W \in \mathcal{S} \right\} \\
&= \bigcup \left\{ r(W) : W \in \mathcal{S} \right\} = \bigvee \left\{ r(W) : W \in \mathcal{S} \right\},
\end{align*}
where the second equality holds since $f^{-1}(V)$ is compact for each $V \in \CLF(Y)$ and $\mathcal{S}$ is directed, and the final equality holds since the set $\{ r(W) :W \in \mathcal{S} \}$ is directed.
\end{proof}

\begin{proposition} \label{prop: functor E} 
There is a functor $\EE \colon \CohLatL \to \HD$ defined by $\EE(X) = (X, R_X, \OF(X))$ for each $X \in \CohLatL$ and $\EE(f) = (f, r)$ for each $\CohLatL$-morphism $f \colon X_1 \to X_2$, where $R_X$ is defined in Definition~\emph{\ref{def: R_X}} and $r$ is the right adjoint of $f^{-1} \colon \OF(X_2) \to \OF(X_1)$.
\end{proposition}

\begin{proof}
By \cref{lem: triple in HD}, $\EE(X) \in \HD$. Let $f \colon X_1 \to X_2$ be a $\CohLatL$-morphism. We show that $(f, r)$ is a $\HD$-morphism. The map $r$ is a $\CohLatL$-morphism by \cref{lem: OF is coherent 3}. Recall from \cref{eqn} that $r(U) = \bigcup \{ V \in \CLF(X_2) : f^{-1}(V) \subseteq U\}$ for each $U \in \OF(X_1)$. If $x \in X_1$ and $U \in \OF(X_1)$ with $x \rel{ R_{X_1}} U$, then $x \notin U$. If $V \in \OF(X_2)$ with $f^{-1}(V) \subseteq U$, then $x \notin f^{-1}(V)$, so $f(x) \notin V$. Therefore, $f(x) \notin r(U)$, and hence  $f(x) \rel{ R_{X_2}} r(U)$, so 
\cref{def: DH morphisms 1} holds.

Next, suppose that $x' \rel{ R_{X_2}} r(U)$ for some $x' \in X_2$ and $U \in \OF(X_1)$. Then $x' \notin V$ for each $V \in \CLF(X_2)$ with $f^{-1}(V) \subseteq U$. We have $\up x' = \bigcap \{ V \in \CLF(X_2) : x' \in V\}$, so $f^{-1}(\up x') = \bigcap \{ f^{-1}(V) : x' \in V \}$. If $f^{-1}(\up x') \subseteq U$, then compactness of $X_1$ implies that there are $V_1, \dots, V_n \in \CLF(X_2)$ with $f^{-1}(V_1) \cap \cdots \cap f^{-1}(V_n) \subseteq U$. If $V = V_1 \cap \dots \cap V_n$, then $V \in \CLF(X_2)$, $x' \in V$, and $f^{-1}(V) \subseteq U$, a contradiction to $x' \notin r(U)$. Therefore, $f^{-1}(\up x') \not\subseteq U$, so there is $x \in f^{-1}(\up x') - U$. Thus, $x \rel{R_{X_1}} U$ and $x' \le f(x)$, and hence \cref{def: DH morphisms 2} holds.

Finally, suppose that $f(x) \rel{R_{X_2}} V$ for some $x \in X_1$ and $V \in \OF(X_2)$. Then $f(x) \notin V$, so $x \notin f^{-1}(V)$. If $U = f^{-1}(V)$, then $U \in \OF(X_1)$, so $V \subseteq r(U)$ and $x \rel{R_{X_1}} U$. Consequently, \cref{def: DH morphisms 3} holds, and thus $(f, r)$ is a $\HD$-morphism.
\end{proof}

In what follows we will require the following result from modal logic (\cite[Lem.~3]{Esa74}, \cite[Lem.~2.17]{BBH15}). We recall that a relation $R \subseteq X \times Y$ between topological spaces is \emph{point-closed} if $R[x]$ is closed for each $x \in X$.

\begin{lemma}[Esakia's Lemma] \label{lem: Esakia}
Let $R \subseteq X \times Y$ be a point-closed relation between compact Hausdorff spaces. If $\{ C_i : i \in I\}$ is a \emph{(}nonempty\emph{)} down-directed family of closed subsets of $Y$, then
\[
R^{-1}\left[\bigcap C_i \right] = \bigcap R^{-1}[C_i].
\]
\end{lemma}

\begin{lemma} \plabel{lem: nu}
Let $(X, R, Y) \in \HD$.
\begin{enumerate}
\item \label[lem: nu]{lem: l(U) = up y} If $W \in \OF(X)$, then there is $y \in Y$ with $W = \psi(y)$.
\item \label[lem: nu]{lem: nu is an iso}\label[lemma]{lem: DH relations nu} The map $\nu \colon Y \to \OF(X)$, defined by $\nu(y) = \psi(y)$, is a $\CohLatL$-isomorphism.
\end{enumerate}
\end{lemma}

\begin{proof}
\eqref{lem: l(U) = up y} Let $W$ be an open filter of $X$. Because $W$ is a union from $\CLF(X)$ and $R$ commutes with unions, we have $\phi W = \bigcap \{ \phi U : U \in \CLF(X), U \subseteq W\}$ . Since $X$ is coherent, $\{ U \in \CLF(X) : U \subseteq W\}$ is up-directed. Therefore, $\{ \phi U : U \in \CLF(X), U \subseteq W\}$ is a down-directed family of closed sets, as each $\phi U$ is clopen (since $R$ is interior). Consequently, by Esakia's lemma,
\[
R^{-1}[\phi W] = \bigcap \{ R^{-1}[\phi U] : U \in \CLF(X), U \subseteq W\},
\]
and so
\begin{align*}
\psi\phi W &= -R^{-1}[\phi W] = \bigcup \{-R^{-1}[\phi U] : U \in \CLF(X), U \subseteq W\} \\
&= \bigcup \{ \psi\phi U : U \in \CLF(X), U \subseteq W\} = W
\end{align*}
since $\psi\phi U = U$ for each $U \in \CLF(X)$ by \cref{phi psi is 1}. Also, \cref{lem: dual iso} and \cref{lem: KOF = LX} show that $\phi U = \up l_U$ for some $l_U \in Y$. Therefore,
$\phi W = \bigcap \{\up l_U : U \subseteq W\} = \up y$
for $y = \bigvee \{ l_U : U \subseteq W \}$. Thus, $W = \psi\phi W = \psi(\up y) = -R^{-1}[\up y] = -R^{-1}[y] = \psi(y)$.

\eqref{lem: nu is an iso} By \cref{lem: open filter}, $\nu \colon Y \to \OF(X)$ is well defined, and is onto by \eqref{lem: l(U) = up y}. If $y \le y'$, then $R^{-1}[y'] \subseteq R^{-1}[y]$ since $R^{-1}[y]$ is a downset. Therefore, $\nu(y) \subseteq \nu(y')$. 
Conversely, if $y \not\le y'$, then $R^{-1}[y'] \not\subseteq R^{-1}[y]$ by \cref{def: DH relation 4}. Consequently, $\nu(y) \not\subseteq \nu(y')$. Thus, $\nu$ is an order-isomorphism, and hence a $\CohLatL$-isomorphism.
\end{proof}

\begin{proposition} \label{prop: eta is natural}
Let $\D = (X, R, Y)$ be a DH-space. Define 
\[
\kappa_\D  \colon (X, R, Y) \to (X, R_X, \OF(X))
\]
by $\kappa_\D = (1_X, \nu_\D)$, where $\nu_\D$ is defined in Lemma~\emph{\ref{lem: nu}}. Then both $\kappa_\D$ and its inverse $(1_X, \nu_\D^{-1})$ are $\HD$-morphisms, and $\kappa \colon 1_{\HD} \to \EE\FF$ is a natural isomorphism.
\end{proposition}

\begin{proof}
Let $(x, y) \in X \times Y$. Then
\[
x \rel{R} y \iff x \in R^{-1}[y] \iff x \notin -R^{-1}[y] = \nu_\D(y) \iff x \rel{R_X} \nu_\D(y).
\] 
Therefore, \cref{lem: nu,rem: morphisms} show that both $\kappa_\D$ and $\kappa^{-1}_\D$ are $\HD$-isomorphisms.

To show that $\kappa$ is natural, we need to show that the following diagram commutes for each pair of $\HD$-objects $\D = (X, R, Y), \D' = (X', R', Y')$ and each $\HD$-morphism $(f, g) \colon \D \to \D'$. 
\[
\begin{tikzcd}
(X, R, Y) \arrow[r, "(f{,}g)"] \arrow[d, "\kappa_\D"'] & (X', R', Y') \arrow[d, "\kappa_{\D'}"] \\
(X, R_X, \OF(X)) \arrow[r, "(f{,}r)"'] & (X', R_{X'}, \OF(X'))
\end{tikzcd}
\]
This amounts to showing that $r \circ \nu_\D = \nu_{\D'} \circ g$. By \cref{eqn} we have
\[
r(\nu_\D(y)) = \bigcup \{ V' \in \CLF(X') : f^{-1}(V') \subseteq \nu_\D(y) \}.
\]
Let $x' \in r(\nu_\D(y))$. If $x' \notin \nu_{\D'}(g(y))$, then $x' \rel{R'} g(y)$. Therefore, by \cref{def: DH morphisms}\eqref{def: DH morphisms 2}, there is $x \in X$ with $x \rel{R} y$ and $x' \le f(x)$. Because $x' \in r(\nu_\D(y))$, there is $V' \in \CLF(X')$ with $x' \in V'$ and $f^{-1}(V') \subseteq \nu_\D(y)$. Thus, $f(x) \in V'$ since $V'$ is an upset, so $x \in f^{-1}(V')$, and hence $x \in \nu_\D(y) = -R^{-1}[y]$. This means that $x \nr{R} y$, a contradiction. Consequently, $x' \in \nu_{\D'}(g(y))$.

For the reverse inclusion, let $x' \in \nu_{\D'}(g(y))$. Then there is $V' \in \CLF(X')$ such that $x' \in V' \subseteq \nu_{\D'}(g(y))$. If $x \in f^{-1}(V')$, then $f(x) \in \nu_{\D'}(g(y))$, so $f(x) \nr{R'} g(y)$. Therefore, $x \nr{R} y$ by \cref{def: DH morphisms}\eqref{def: DH morphisms 1}. This means that $x \in \nu_\D(y)$, and so $f^{-1}(V') \subseteq \nu_D(y)$. Thus, $x' \in r(\nu_\D(y))$.
\end{proof}

\cref{prop: eta is natural} together with $\FF\EE = 1_{\CohLatL}$ yields:

\begin{theorem} \label{thm: DH = CohLatL}
$\EE$ and $\FF$ establish an equivalence of categories between $\CohLatL$ and $\HD$.
\end{theorem}

\begin{corollary} \plabel{cor: HD = Lat}
\hfill
\begin{enumerate}
\item \label[cor: HD = Lat]{DH = HsL} $\HD$ is equivalent to $\JML$, $\BDGM$, $\HsL$, and $\CG$.
\item \label[cor: HD = Lat]{cor: HD = Lat 2}\cite[Thm.~2.15]{HD97} $\HD$ is dually equivalent to $\Lat$.
\end{enumerate}
\end{corollary}

\begin{proof}
The first item is an immediate consequence of \cref{thm: iso between CohLat_S and CohLat_L,thm: CohLatL = JML,thm: CohLatL = HsL,thm: BDGM = CohLatL,thm: DH = CohLatL,thm: CG = JM}, and the second item follows from Theorems~\ref{thm: CohLat = Lat}(\ref{thm: Nachbin 2}) and \ref{thm: DH = CohLatL}.  
\end{proof}

\begin{remark} \label{rem: functor from DH to Lat}
The duality in \cref{cor: HD = Lat 2} can be described directly as follows. A triple $(X, R, Y) \in \HD$ is sent to $\CLF(X)$ and a $\HD$-morphism $(f, g) \colon (X, R, Y) \to (X', R', Y')$ to $f^{-1} \colon \CLF(X') \to \CLF(X)$.
Going the other direction, $ A \in \Lat$ is sent to $(\Filt(A), R, \Idl(A)) \in \HD$, where $R$ is given by $F \rel{R} I$ if $F \cap I = \varnothing$, and a $\Lat$-morphism $\alpha \colon A \to B$ to $\left(\alpha^{-1}|_{\Filt(B)}, \alpha^{-1}|_{\Idl(B)} \right)$.
\end{remark}


\section{From Dunn-Hartonas to Gehrke-van Gool} \label{sec: DH and GvG}

In this section we connect the Dunn-Hartonas approach to the Gehrke-van Gool approach. The resulting spaces were introduced and studied in \cite{GvG14}. The notion of morphism considered by Gehrke and van Gool dually characterizes special lattice homomorphisms \cite[Def.~3.19]{GvG14}. We generalize their notion of morphism by working with relations instead of functions, thus giving rise to the category $\GvG$. We show that there is a functor from $\HD$ to $\GvG$. In \cref{sec: HD GvG Hg Urq}  we will show that this functor is an equivalence and in \cref{sec: conclusion} that $\GvG$ is dually equivalent to $\Lat$. Therefore, our generalized notion of morphism captures all bounded lattice homomorphisms.

Let $R \subseteq X \times Y$. Motivated by a standard technique in modal logic (see, e.g., \cite[p.~64]{CZ97}),
we define $\Diamond_R,\Box_R \colon \wp(Y)\to\wp(X)$ by 
\[
\Diamond_R B = R^{-1}[B] \quad \mbox{and} \quad \Box_R B = -R^{-1}[-B]
\]
for each $B\subseteq Y$. Similarly, define $\bd_R,\blacksquare_R \colon \wp(X)\to\wp(Y)$ by
\[
\bd_R A = R[A] \quad \mbox{and} \quad \blacksquare_R A = -R[-A]
\]
for each $A\subseteq X$. We will mainly work with $\bd_R$ and $\Box_R$, and when there is no danger of confusion, we simply write $\bd$ and $\Box$. 

The following well-known lemma (see, e.g., \cite[Prop.~2.6]{Pri84}, \cite[Thm.~3.2.1]{Esa19}) will be used throughout.

\begin{lemma} \label{lem: maximal points}
Let $X$ be a Priestley space and let $C$ be a closed subset of $X$. For each $x \in C$ there is $m \in \max C$ with $x \le m$.
\end{lemma}

\begin{definition} \label{def: X0 and Y0}
Let $X, Y$ be Priestley spaces and $R \subseteq X \times Y$ a relation. Define 
\[
X_0 = \bigcup \{ \max R^{-1}[y] : y \in Y\} \quad \mbox{and} \quad  Y_0 =  \bigcup \{ \max R[x] : x \in X\}.
\] 
\end{definition}

\begin{definition} \cite[Def.~4.18]{GvG14} \plabel{def: GvG-space}
We call a triple $(X, R, Y)$ a \emph{Gehrke-van Gool space}, or \emph{GvG-space} for short, if $X, Y$ are Priestley spaces and $R \subseteq X \times Y$ is a relation satisfying
\begin{enumerate}
\item \label[def: GvG-space]{GvG 1} $x \le x'$ iff $R[x'] \subseteq R[x]$.
\item \label[def: GvG-space]{GvG 2} $y \le y'$ iff $R^{-1}[y'] \subseteq R^{-1}[y]$.
\item \label[def: GvG-space]{GvG 3} If $U$ is a clopen upset of $X$, then $\bd U$ is a clopen downset of $Y$.
\item \label[def: GvG-space]{GvG 4} If $V$ is a clopen downset of $Y$, then $\Box V$ is a clopen upset of $X$.
\item \label[def: GvG-space]{GvG 5} If $x \nr{R} y$, then there is a clopen upset $U \subseteq X$ such that $U = \Box\bd U$, $x \in U$, and $y \notin \bd U$.
\item \label[def: GvG-space]{GvG 6} If $U$ is a clopen upset of $X$ with $X_0 \cap \Box\bd U \subseteq U$, then $U = \Box\bd U$.
\item \label[def: GvG-space]{GvG 7} If $V$ is a clopen downset of $Y$ with $Y_0 \cap V \subseteq \bd\Box V$, then $V = \bd\Box V$.
\end{enumerate}
\end{definition}

\begin{remark} \plabel{rem: properties of GvG spaces}
\hfill 
\begin{enumerate}
\item The order on $X$ in the definition above is dual to the order used in \cite{GvG14} and is more convenient to compare GvG-spaces to DH-spaces.
\item In \cite{GvG14} Condition (5) is phrased using clopen sets $U \subseteq X$ and $V \subseteq Y$. The latter can be recovered as $V = \bd U$.
\item \label[rem: properties of GvG spaces]{closed order} $R$ is a closed subset of $X \times Y$, and so both $R$ and $R^{-1}$ are point-closed. 
\item \label[rem: properties of GvG spaces]{Rx in RU} For each $U \subseteq X$ and $V \subseteq Y$, we have  $\Box\bd U = \{ x \in X : R[x] \subseteq R[U] \}$ and $\bd\Box V = \{ y \in Y : \exists x \textrm{ with } x \rel{R} y \textrm{ and } R[x] \subseteq V\}$.
\item  Since $\bd\Box V = -\blacksquare\Diamond-V$, \cref{GvG 7} is equivalent to $Y_0 \cap \blacksquare\Diamond-V \subseteq -V$ implies $V = \bd\Box V$, which is how it is phrased in \cite[Def.~4.18]{GvG14}.
\end{enumerate}
\end{remark}

\begin{definition} \label{def: LG}
Suppose that $\G = (X, R, Y)$ is a GvG-space. We let $\ClopUp(X)$ be the clopen upsets of $X$, and set
\[
\L\G = \{ U \in \ClopUp(X) : U = \Box\bd U \}.
\]
\end{definition}

\begin{remark} \label{rem: lattice ops for LG}
If $\G$ is a GvG-space, then $\L\G$ is a bounded lattice, where meet and join are given by 
\[
U_1 \wedge U_2 = U_1 \cap U_2 \quad \mbox{and} \quad U_1 \vee U_2 = \Box\bd(U_1 \cup U_2) 
\]
for each $U_1,U_2 \in \L\G$. Moreover, $X$ is the top and $\varnothing$ is the bottom of $\L\G$. To see the latter,  
let $x \in X_0$. Then $x \rel{R} y$ for some $y \in Y$, so $R[x] \ne \varnothing$. Therefore, $x \notin \Box\varnothing = \Box\bd\varnothing$, and so $X_0 \cap \Box\bd \varnothing = \varnothing$. By \cref{GvG 6}, $\varnothing = \Box\bd \varnothing$, and hence $\varnothing \in \L\G$.
\end{remark}

\begin{lemma} \plabel{lem: properties of GvG spaces 2}
Let $\G = (X, R, Y)$ be a GvG-space.
\begin{enumerate}
\item \label[lem: properties of GvG spaces 2]{GvG up x} $\up x = \bigcap \{ U \in \L\G : x \in U\}$ for each $x \in X$. 
\item \label[lem: properties of GvG spaces 2]{GvG LG basis} $\L\G$ is a basis for the open upset topology of $X$.
\item \label[lem: properties of GvG spaces 2]{GvG up y} $\down y = \bigcap \{ \bd U : U \in \L\G, y \in \bd U \}$ for each $y \in Y$.
\item \label[lem: properties of GvG spaces 2]{GvG LG basis for Y} $\{ \bd U : U \in \L\G \}$ is a basis for the open downset topology of $Y$.

\item \label[lem: properties of GvG spaces 2]{x notin U GvG} If $U \in \L\G$ and $x \notin U$, then there is $x' \in X_0$ with $x \le x'$ and $x' \notin U$.
\end{enumerate}
\end{lemma}

\begin{proof} 
\eqref{GvG up x} The inclusion $\subseteq$ is clear. For the inclusion $\supseteq$, suppose that $x \not\le x'$. Then $R[x'] \not\subseteq R[x]$ by \cref{GvG 1}, so there is $y \in Y$ with $x' \rel{R} y$ and $x \nr{R} y$. Therefore, by \cref{GvG 5}, there is $U \in \L\G$ with $x \in U$ and $y \notin \bd U$. Since $x' \rel{R} y$, we have $x' \notin U$. 

\eqref{GvG LG basis} Because each open upset is a union of clopen upsets, it is enough to see that each clopen upset $V$ is a union from $\L\G$. Let $x \in V$. If $x' \notin V$, then $x \not\le x'$, so \eqref{GvG up x} implies that there is $U_{x'} \in \L\G$ with $x \in U_{x'}$ and $x' \notin U_{x'}$. The set $\{ -U_{x'} : x' \notin V\}$ is an open cover of $-V$, so compactness implies that there are $x_1', \dots, x_n'$ with $-V \subseteq -U_{x_1'} \cup \dots \cup -U_{x_n'}$. Let $U = U_{x_1'} \cap \dots \cap U_{x_n'}$. Then $x \in U$, $U \in \L\G$ (since $\L\G$ is closed under finite intersections), and $U \subseteq V$.

The proofs of \eqref{GvG up y} and \eqref{GvG LG basis for Y} are similar.

\eqref{x notin U GvG} Let $x \notin U$. Since $U = \Box\bd U$, there is $y \in Y$ with $x \rel{R} y$ and $y \notin R[U]$. Clearly $\up x \cap R^{-1}[y]$ is nonempty, and it is closed by  \cref{closed order}. Therefore, by \cref{lem: maximal points}, there is $x' \in \max(\up x \cap R^{-1}[y])$. Thus, $x \le x'$ and $x' \in \max R^{-1}[y] \subseteq X_0$. Since $x' \rel{R} y$ and $y \notin R[U]$, we have $x' \notin U$, concluding the proof.
\end{proof}

Morphisms considered by Gehrke and van Gool are pairs of functions between GvG-spaces that dually characterize special bounded lattice homomorphisms (called admissible homomorphisms in \cite[Def.~3.19]{GvG14}). To characterize all bounded lattice homomorphisms in the language of GvG-spaces, we need to work with pairs of relations.

\begin{definition} \plabel{def: GvG morphism}
Let $\G = (X, R, Y)$ and $\G' = (X', R', Y')$ be GvG-spaces. A pair $(S, T)$ of relations $S \subseteq X \times X'$ and $T \subseteq Y \times Y'$ is a GvG-morphism provided:
\begin{enumerate}
\item \label[def: GvG morphism]{GvG morphism 1} If $U' \in \L\G'$, then $\Box_SU' \in \L\G$ and $\bd \Box_S U' = \Diamond_T \bd' U'$. 
\item \label[def: GvG morphism]{GvG morphism 2} If $x \nr{S} x'$, then there is $U' \in \L\G'$ with $x \in \Box_S U'$ and $x' \notin U'$.
\item  \label[def: GvG morphism]{GvG morphism 3} If $y \nr{T} y'$, then there is $U' \in \L\G'$ with $y \notin \Diamond_T \bd' U'$ and $y' \in \bd' U'$.
\item \label[def: GvG morphism]{GvG morphism 4} If $x\rel{R}y$, then there are $x' \in X'$ and $y' \in Y'$ with $x \rel{S} x'$, $y \rel{T} y'$, and $x' \rel{R'} y'$.
\end{enumerate}
\end{definition}

\begin{remark} \plabel{rem: equivalent condition for GvG S}
Let $(S, T) \colon \G \to \G'$ be a GvG-morphism. 
\begin{enumerate}
\item \label[rem: equivalent condition for GvG S]{S and T closed}  Both $S$ and $T$ are closed relations. To see that $S$ is closed, let $(x, x') \notin S$. Then $x \nr{S} x'$, so by \cref{GvG morphism 2}, there is $U' \in \L\G'$ with $x \in \Box_S U'$ and $x' \notin U'$. By \cref{GvG morphism 1}, $\Box_S U' \times -U'$ is an open neighborhood of $(x,x')$ disjoint from $S$. Thus, $S$ is closed. A similar argument using \cref{GvG morphism 3} shows that $T$ is closed.

\item \label[rem: equivalent condition for GvG S]{GvG S 1}Because the converse of \cref{GvG morphism 2} is obvious, we have $x \rel{S} x'$ iff for each $U' \in \L\G'$, if $x \in \Box_SU'$, then $x' \in U'$. The same applies to \cref{GvG morphism 3}.

\item \label[rem: equivalent condition for GvG S]{GvG S 2}If $x \in X$ and $y \in Y$, then $S[x]$ and $T[y]$ are upsets. To see this, suppose that $x \rel{S} x_1'$ and $x_1' \le x_2'$. If $x \nr{S} x_2'$, by \cref{GvG morphism 2} there is $U' \in \L\G'$ with $x \in \Box_SU'$ and $x_2' \notin U'$. The former implies that $S[x] \subseteq U'$, so $x_1' \in U'$. Because $U'$ is an upset, $x_2' \in U'$. This contradicts the latter, yielding that $x \rel{S} x_2'$. A similar argument shows that $T[y]$ is an upset. 
\item \label[rem: equivalent condition for GvG S]{GvG S 3}If $x' \in X'$ and $y' \in Y'$, then $S^{-1}[x']$ and $T^{-1}[y']$ are downsets.  For this, suppose $x_1 \le x_2$ and $x_2 \in S^{-1}[x']$. Then $x' \in S[x_2]$. If $x_1 \nr{S} x'$, by \cref{GvG morphism 2} there is $U' \in \L\G'$ with $x_1 \in \Box_S U'$ and $x' \notin U'$. Because $\Box_S U' \in \L\G$, it is an upset, so $x_2 \in \Box_S U'$. This implies that $S[x_2] \subseteq U'$, a contradiction since $x' \notin U'$ and $x' \in S[x_2]$. Thus, $x_1 \in S^{-1}[x']$. A similar argument shows that $T^{-1}[y']$ is a downset.
\end{enumerate}
\end{remark}

Composition of two GvG-morphisms is not relation composition, which is reminiscent to what happens in the category of generalized Priestley spaces \cite[p.~106]{BJ11}.

\begin{definition} \label{def: composition in GvG}
Let $(S_1, T_1) \colon \G_1 \to \G_2$ and $(S_2, T_2) \colon \G_2 \to \G_3$ be GvG-morphisms. 
\begin{enumerate}
\item Define $S_2 \star S_1$ by $x \rel{(S_2 \star S_1)} z$ provided that $x \in \Box_{S_1}\Box_{S_2}U$ implies $z \in U$ for each $U \in \L\G_3$.
\item  Define $T_2 \star T_1$ by $y \rel{(T_2 \star T_1)}w$ provided that $w \in \bd U$ implies $y \in \Diamond_{T_1}\Diamond_{T_2}\bd U$ for each $U \in \L\G_3$.
\item Set $(S_2, T_2) \star (S_1, T_1) = (S_2 \star S_1, T_2 \star T_1)$.
\end{enumerate}
\end{definition}

\begin{remark} \plabel{rem: fact about composition}
Let $(S_1, T_1) \colon \G_1 \to \G_2$ and $(S_2, T_2) \colon \G_2 \to \G_3$ be GvG-morphisms. 
\begin{enumerate}
\item \label[rem: fact about composition]{comp inside star} $S_2 \circ S_1 \subseteq S_2 \star S_1$ and $T_2 \circ T_1 \subseteq T_2 \star T_1$, however the equalities may not hold in general.
\item \label[rem: fact about composition]{composition equality 1} $(S_2\star S_1)[x] = \bigcap \{ U \in \L\G_3 : x \in \Box_{S_1}\Box_{S_2}U \} = \bigcap \{ U \in \L\G_3 : S_2S_1[x] \subseteq U \}$.
\item \label[rem: fact about composition]{composition equality 2} $(T_2 \star T_1)^{-1}[w] = \bigcap \{ (T_2 \circ T_1)^{-1} \bd U : w \in \bd U, \,  U \in \L\G_3\}$.
\end{enumerate}
\end{remark}

\begin{lemma} \label{lem: composition lemma for GvG}
Suppose that $(S_1, T_1) \colon \G_1 \to \G_2$ and $(S_2, T_2) \colon \G_2 \to \G_3$ are GvG-morphisms. If $U \in \L\G_3$, then
\[
\Box_{S_2\star S_1}U = \Box_{S_1}\Box_{S_2}U
\]
and
\[
\Diamond_{T_2\star T_1} \bd U = \Diamond_{T_1}\Diamond_{T_2} \bd U.
\]
\end{lemma}

\begin{proof}
For each $x \in X_1$, by \cref{composition equality 1},
\[
(S_2 \star S_1)[x] \subseteq U \iff \bigcap \{ W \in \L\G_3 : S_2S_1[x] \subseteq W \} \subseteq U \iff S_2S_1[x] \subseteq U. 
\]
Thus, $\Box_{S_2\star S_1}U = \Box_{S_1}\Box_{S_2}U$.

For each $y \in Y_1$ we show that $y \in \Diamond_{T_2\star T_1} \bd U$ iff $y \in \Diamond_{T_1}\Diamond_{T_2} \bd U$. It is enough to show that $(T_2 \star T_1)[y] \cap \bd U \ne \varnothing$ iff $(T_2\circ T_1)[y] \cap \bd U \ne \varnothing$.  By \cref{comp inside star}, $(T_2 \circ T_1)[y] \subseteq (T_2 \star T_1)[y]$, and hence $(T_2\circ T_1)[y] \cap \bd U \ne \varnothing$ implies that $(T_2 \star T_1)[y] \cap \bd U \ne \varnothing$. For the converse, suppose that $(T_2 \star T_1)[y] \cap \bd U \ne \varnothing$. Then there is $w \in \bd U$ with $y \in (T_2 \star T_1)^{-1}[w]$. Therefore, by \cref{composition equality 2},
\[
y \in \bigcap \{ (T_2 \circ T_1)^{-1} \bd W : w \in \bd W, \,  W \in \L\G_3\} \subseteq (T_2 \circ T_1)^{-1} \bd U, 
\]
which gives that $(T_2\circ T_1)[y] \cap \bd U \ne \varnothing$.
\end{proof}

\begin{proposition} \plabel{prop: composition in GvG}
Suppose $(S_1, T_1) \colon \G_1 \to \G_2$ and $(S_2, T_2) \colon \G_2 \to \G_3$ are GvG-morphisms. 
\begin{enumerate}
\item \label[prop: composition in GvG]{GvG morphism} $(S_2, T_2) \star (S_1, T_1)$ is a GvG-morphism.
\item \label[prop: composition in GvG]{GvG identity} If $\G = (X, R, Y) $ is a GvG-space, then $(\le_X, \le_Y) \colon \G \to \G$ is a GvG-morphism and is the identity morphism for $\G$.
\item \label[prop: composition in GvG]{GvG associativity} The operation $\star$ is associative.
\end{enumerate}
\end{proposition}

\begin{proof}
\eqref{GvG morphism} To simplify notation, we write $S = S_2 \star S_1$ and $T = T_2 \star T_1$. Let $U'' \in \L\G_3$ and set $U' = \Box_{S_2}U''$. Then $U' \in \L\G_2$ since $(S_2, T_2)$ is a GvG-morphism. By \cref{lem: composition lemma for GvG},
\begin{align*}
\Box_SU'' &= \Box_{S_1}\Box_{S_2}U'' = \Box_{S_1}U' \in \L\G_1
\end{align*}
because $(S_1, T_1)$ is a GvG-morphism. 

In addition, since $(S_1, T_1)$ and $(S_2, T_2)$ are GvG-morphisms, by \cref{lem: composition lemma for GvG}
\begin{align*}
\bd \Box_{S} U &= \bd \Box_{S_1} (\Box_{S_2}U) = \Diamond_{T_1} \bd \Box_{S_2}U = \Diamond_{T_1} \Diamond_{T_2} \bd U = \Diamond_{T}\bd U.
\end{align*}
This shows that \cref{GvG morphism 1} holds.

Suppose $x \nr{S} z$. By \cref{composition equality 1}, there is $U \in \L\G_3$ with $x \in \Box_{S_1}\Box_{S_2}U$ and $z \notin U$. By \cref{lem: composition lemma for GvG}, $\Box_{S_1}\Box_{S_2}U = \Box_S U$, so $x \in \Box_SU$ and $z \notin U$. Therefore, \cref{GvG morphism 2} holds. 

Suppose $y \nr{T} w$. By \cref{composition equality 2}, there is $U \in \L\G_3$ with $w  \in \bd U$ and $y \notin \Diamond_{T_1}\Diamond_{T_2}\bd U$. By \cref{lem: composition lemma for GvG}, $y \notin \Diamond_T\bd U$, Thus, \cref{GvG morphism 3} holds.

Suppose $x\rel{R_1}y$. Then there are $x' \in X_2$ and $y' \in Y_2$ with $x' \rel{R_2} y'$, $x\rel{S_1} x'$, and $y \rel{T_1} y'$. From this there are $x'' \in X_3$ and $y'' \in Y_3$ with $x'' \rel{R_3} y''$, $x' \rel{S_2} x''$, and $y' \rel{T_2} y''$. Since $S_2\circ S_1 \subseteq S$ and $T_2 \circ T_1 \subseteq T$ by \cref{comp inside star}, we have $x \rel{S} x''$ and $y \rel{T} y''$. This verifies \cref{GvG morphism 4}. Consequently, $(S_2 \star S_1, T_2 \star T_1)$ is a GvG-morphism. 

\eqref{GvG identity} Let $\G = (X, R, Y)$ be a GvG-space. We have $\Box_{\le_X} U = U$ for any upset $U \subseteq X$ and $\Diamond_{\le_Y} V = V$ for any downset $V \subseteq Y$.  Therefore, \cref{def: GvG morphism} is trivially satisfied. Thus, $(\le_X, \le_Y)$ is a GvG-morphism.

Let $(S, T) \colon \G \to \G'$ and $(S', T') \colon \G' \to \G$ be GvG-morphisms. Since $\Box_{\le_X} W = W$ for any upset $W \subseteq X$, we have $x \rel{(S \star {\le_X})} y$ iff $x \in \Box_S U$ implies $y \in U$ for each $U \in \L\G'$. On the other hand, $S \circ {\le_X} = S$ by \cref{GvG S 3}, which by \cref{GvG S 1} is equivalent to $x \in \Box_S U$ implies $y \in U$ for each $U \in \L\G'$. Therefore, $S \star {\le_X} = S \circ {\le_X} = S$, and similarly, $T \star {\le_Y} = T \circ {\le_Y} = T$. Thus, $(S, T) \star (\le_X, \le_Y) = (S, T)$, and an analogous argument gives  $(\le_X, \le_Y) \star (S', T') = (S', T')$. Consequently, $(\le_X, \le_Y)$ is the identity morphism for $\G$.

\eqref{GvG associativity} Let $(S_1, T_1) \colon \G_1 \to \G_2$, $(S_2, T_2) \colon \G_2 \to \G_3$, and $(S_3, T_3) \colon \G_3 \to \G_4$ be GvG-morphisms. By \cref{lem: composition lemma for GvG}, for each $U \in \L\G_4$, we have 
\[
\Box_{(S_3\star S_2) \star S_1}U = \Box_{S_1}\Box_{S_3 \star S_2}U = \Box_{S_1}\Box_{S_2}\Box_{S_3}U.
\]
Therefore, $x \rel{((S_3 \star S_2) \star S_1)} w$ iff $x \in \Box_{S_1}\Box_{S_2}\Box_{S_3}U$ implies $w \in U$ for each $U \in \L\G_4$. A similar calculation shows that $x \rel{(S_3 \star (S_2 \star S_1))} w$ iff $x \in \Box_{S_1}\Box_{S_2}\Box_{S_3}U$ implies $w \in U$. Thus, $(S_3 \star S_2) \star S_1 = S_3 \star (S_2 \star S_1)$. A similar argument shows that $(T_3 \star T_2) \star T_1 = T_3 \star (T_2 \star T_1)$. Consequently, $\star$ is associative.
\end{proof}

The previous lemma shows that we have a category, motivating the following definition.

\begin{definition} \label{def: GvG}
Let $\GvG$ denote the category of GvG-objects and GvG-morphisms.
\end{definition}

The rest of the section is dedicated to defining a functor $\GG \colon \HD \to \GvG$. To define $\GG$ on objects, we recall the notion of distributive joins and meets, which goes back to MacNeille \cite[Def.~3.10]{Mac37} (see also \cite{BL70,HK71}). Distributive joins and meets are also known as exact \cite{Bal84} (see also \cite{BPP14,BPW16}). As in \cite{GvG14}, we will be working with finite distributive joins and meets. Let $L$ be a lattice. For $M \subseteq L$ finite, we say that $\bigwedge M$ is a \emph{distributive meet} if
\[
a \vee \bigwedge M = \bigwedge \{ a \vee m : m \in M \}
\]
for each $a \in L$. Distributive joins are defined similarly. The next definition weakens the standard definition of a (meet-)prime element \cite[Def.~I.3.11]{GHKLMS03}.

\begin{definition} \label{def: d-prime}
Let $Z$ be a coherent lattice. 
\begin{enumerate}
\item An element $x \in Z$ is \emph{d-prime} if $x \ne 1$ and whenever $\bigwedge M$ is a finite distributive meet in $K(Z)$ and $\bigwedge M \le x$, then $m \le x$ for some $m \in M$. 
\item Let $Z_p$ be the set of d-prime elements of $Z$. 
\end{enumerate}
\end{definition}

\begin{lemma} \label{lem: Xa and Ya are Priestley}
Let $Z \in \CohLatL$. Then $Z_p$ is Lawson-closed in $Z$. Therefore, $Z_p$ is a Priestley space with the subspace topology and induced order.
\end{lemma}

\begin{proof}
Let $x \notin Z_p$. Then $x$ is not d-prime, so there is a distributive meet $k = k_1 \wedge \dots \wedge k_n$ in $K(Z)$ with $k \le x$ but all $k_i \not\le x$. Let $U = \up k - (\up k_1 \cup \dots \cup \up k_n)$. Then $U$ is Lawson-open and $x \in U$. Moreover, if $x' \in U$, then $k \le x'$ and all $k_i \not\le x'$, so $x' \notin Z_p$, and hence $U \cap Z_p = \varnothing$. Thus, $Z_p$ is closed in $Z$, yielding that $Z_p$ is a Priestley space.
\end{proof}

To associate with each DH-space a GvG-space, we need the following lemma.

\begin{lemma} \plabel{lem: producing elements from X0 or Y0}
Let $X$, $Y$ be Priestley spaces and $R \subseteq X \times Y$ a closed relation. 
\begin{enumerate}
\item \label[lem: producing elements from X0 or Y0]{producing x' in X0} If $x \rel{R} y$, then there is $x' \in \max R^{-1}[y]$ with $x \le x'$.
\item \label[lem: producing elements from X0 or Y0]{producing y' in Y0} If $x\rel{R}y$ then there is $y' \in \max R[x]$ with $y \le y'$. 
\end{enumerate}
If in addition $(X, R, Y)$ is a DH-space, then 
\begin{enumerate}[resume]
\item \label[lem: producing elements from X0 or Y0]{x is meet from X0} $x = \bigwedge (\up x \cap X_0)$ for each $x \in X$.
\item \label[lem: producing elements from X0 or Y0]{y is meet from Y0} $y = \bigwedge (\up y \cap Y_0)$ for each $y \in Y$.
\end{enumerate}
\end{lemma}

\begin{proof}
\eqref{producing x' in X0} Suppose that $x \rel{R} y$. Because $\up x \cap R^{-1}[y]$ is a nonempty closed subset of $X$ by \cref{closed order}, there is $x' \in \max(\up x \cap R^{-1}[y])$ by \cref{lem: maximal points}. Therefore, $x \le x'$ and it is clear that $x' \in \max R^{-1}[y]$. 

\eqref{producing y' in Y0} The proof is similar to that of \eqref{producing x' in X0}.

\eqref{x is meet from X0} The inequality $x \le \bigwedge (\up x \cap X_0)$ is clear. For the reverse inequality, suppose that $z \le \bigwedge (\up x \cap X_0)$. If $z \not\le x$, then $R[x] \not\subseteq R[z]$ by \cref{def: DH relation 3}. Therefore, there is $y \in Y$ with $x\rel{R}y$ and $z \nr{R} y$. By \eqref{producing x' in X0}, there is $x' \in \max R^{-1}[y] \subseteq X_0$ with $x \le x'$. Thus, $z \not\le x'$ since $R^{-1}[y]$ is a downset. This contradiction shows that $z \le x$, so $x = \bigwedge (\up x \cap X_0)$.

\eqref{y is meet from Y0} The proof is similar to that of \eqref{x is meet from X0}.
\end{proof}

\begin{definition} \label{def: from DH to GvG}
For each $(X, R, Y) \in \HD$ we set $R_p = R \cap (X_p \times Y_p)$.
\end{definition}

If $(X, R, Y) \in \HD$, then we can apply \cref{def: X0 and Y0} starting from $(X, R, Y)$ or $(X_p, R_p, Y_p)$. The following lemma shows that we get the same result.

\begin{lemma} \label{lem: X0 the same}
Let $(X, R, Y) \in \HD$. Then $X_0 = (X_p)_0$ and $Y_0 = (Y_p)_0$.
\end{lemma}

\begin{proof}
To start, we prove that $X_0 \subseteq X_p$ and $Y_0 \subseteq Y_p$. Let $x \in X_0$. Then $x \in \max R^{-1}[y]$ for some $y \in Y$. Suppose that $k = k_1 \wedge \dots \wedge k_n$ is a distributive meet in $K(X)$ with $k \le x$. Suppose that $k_i \not\le x$. Then $x < x \vee k_i$, so $(x \vee k_i) \nr{R} y$. If this happens for each $i$, then $[(x \vee k_1) \wedge \dots \wedge (x \vee k_n)] \nr{R} y$ by \cref{def: DH relation filter}. We show that $x \vee k = (x \vee k_1) \wedge \dots \wedge (x \vee k_n)$, which will show that $(x \vee k) \nr{R} y$, a contradiction since $k \le x$ and $x \rel{R} y$. The inequality $x \vee k \le (x \vee k_1) \wedge \dots \wedge (x \vee k_n)$ is clear. For the reverse inequality, it suffices to show that $m \le x \vee k$ for each $m \in K(X)$ with $m \le (x \vee k_1) \wedge \dots \wedge (x \vee k_n)$. Given such $m$, we have $m \le x \vee k_i$ for each $i$. Because $m \in K(X)$ and $X$ is algebraic, there is $m_i \in K(X)$ with $m_i \le x$ and $m \le m_i \vee k_i$. Set $s = m_1 \vee \dots \vee m_n$. Then $s \le x$ and $m \le s \vee k_i$ for each $i$. Therefore, since $k = k_1 \wedge \dots \wedge k_n$ is a distributive meet,
\[
s \vee k = (s \vee k_1) \wedge \dots \wedge (s \vee k_n).
\]
This implies that $m \le s \vee k \le x \vee k$. Thus, $x \in X_p$. A similar argument shows that $Y_0 \subseteq Y_p$.

We next show that $(X_p)_0 \subseteq X_0$. Let $x \in (X_p)_0$. Then $x \in \max R_p^{-1}[y]$ for some $y \in Y_p$. Suppose there is $x' \in X$ with $x \le x'$ and $x' \rel{R}y$. By \cref{producing x' in X0}, there is $x'' \in \max R^{-1}[y]$ with $x' \le x''$. Therefore, $x'' \in X_0 \subseteq X_p$, which forces $x'' = x$. Thus, $x \in \max R^{-1}[y]$, and hence $x \in X_0$. 

For the reverse inclusion, let $x \in X_0$. Then $x \in \max R^{-1}[y]$ for some $y \in Y$. Therefore, $x \in X_0 \subseteq X_p$. By \cref{producing y' in Y0}, there is $y' \in Y_0$ with $x \rel{R} y'$ and $y \le y'$ . Thus, $y' \in Y_0 \subseteq Y_p$, so $x \rel{R_p} y'$. By \cref{lem: Xa and Ya are Priestley}, $X_p$ and $Y_p$ are Priestley spaces. Moreover, $R_p \subseteq X_p \times Y_p$ is a closed relation by definition. Hence, \cref{producing x' in X0} applies, by which there is $x' \in \max R_p^{-1}[y']$ with $x \le x'$. Because $x' \rel{R} y'$ and $y \le y'$, we have $x' \rel{R} y$. From $x \in \max R^{-1}[y]$ it follows that $x' = x$. Therefore, $x \in \max R_p^{-1}[y'] \subseteq (X_p)_0$. This yields the reverse inclusion. A similar argument shows that $(Y_p)_0 = Y_0$.
\end{proof}

Given $(X, R, Y) \in \HD$, we recall from the beginning of this section that we write $\Box, \bd$ for $\Box_R, \bd_R$. We also write $\Box_p, \bd_p$ for $\Box_{R_p}, \bd_{R_p}$.

\begin{lemma} \plabel{lem: properties of Box and Diamond}
Let $(X, R, Y) \in \HD$.
\begin{enumerate}
\item \label[lem: properties of Box and Diamond]{bd 1} If $U = U' \cap X_p$ with $U'$ an upset of $X$, then $\bd_p U = \bd U' \cap Y_p$. 
\item \label[lem: properties of Box and Diamond]{bd 2} If $V = V' \cap Y_p$ with $V'$ a downset of $Y$, then $\Box_p V = \Box V' \cap X_p$. 
\item \label[lem: properties of Box and Diamond]{bd 3} Let $U' = \up k_1 \cup \dots \cup \up k_n$ with $k_i \in K(X)$ and set $U = U' \cap X_p$. Then $\bd_p U = Y_p - \up l$ for some $l \in K(Y)$. Therefore, if $U$ is a clopen upset of $X_p$, then $\bd_p U$ is a clopen downset of $Y_p$.
\item \label[lem: properties of Box and Diamond]{bd 4}Let $V' = -(\up l_1 \cup \dots \cup \up l_n)$ with $l_i \in K(Y)$ and set $V = V' \cap Y_p$. Then $\Box_p V = \up k \cap X_p$ for some $k \in K(X)$. Therefore, if $V'$ is a clopen downset of $Y_p$, then $\Box_p V$ is a clopen upset of $X_p$.
\item \label[lem: properties of Box and Diamond]{bd 5} If $U = \up k \cap X_p$ with $k \in K(X)$, then $\Box_p\bd_p U = U$.
\item \label[lem: properties of Box and Diamond]{bd 6} If $V =  Y_p - \up l$ with $l \in K(Y)$, then $\bd_p\Box_p V = V$.
\end{enumerate}
\end{lemma}

\begin{proof}
(1) Since $\bd_p U = R_p[U]$, the inclusion $\bd_p U \subseteq \bd U' \cap Y_p$ is clear. For the reverse inclusion, suppose that $y \in Y_p$ with $x \rel{R}y$ for some $x \in U'$. By \cref{producing x' in X0}, there is $x' \ge x$ with $x' \in X_0$ and $x' \rel{R} y$. Therefore, $x' \rel{R_p} y$ because $X_0 \subseteq X_p$ by \cref{lem: X0 the same}. Since $U'$ is an upset, $x' \in U'$, so $x' \in U$. Thus, $y \in \bd_p U$.

(2) We have $\Box_p V = \{ x \in X_p : R_p[x] \subseteq V\}$ and $\Box V' = \{ x \in X : R[x] \subseteq V'\}$. If $x \in \Box V' \cap X_p$, then $x \in X_p$ and $R[x] \subseteq V'$. Therefore, $R_p[x] = R[x] \cap Y_p \subseteq V' \cap Y_p = V$. Thus, $x \in \Box_p V$. For the reverse inclusion, let $x \in \Box_p V$ and $y \in Y$ with $x \rel{R} y$. By \cref{producing y' in Y0}, there is $y' \ge y$ with $y' \in Y_0$ and $x \rel{R} y'$. By \cref{lem: X0 the same}, $Y_0 \subseteq Y_p$. Therefore, $y' \in R_p[x]$, so $y' \in V \subseteq V'$. Because $V'$ is a downset, $y \in V'$. Thus, $R[x] \subseteq V'$, so $x \in \Box V' \cap X_p$.

(3) By (1) and \cref{lem: dual iso},
\begin{align*}
\bd_p U &= \bd U' \cap Y_p = (\bd \up k_1 \cup \dots \cup \bd \up k_n) \cap Y_p \\
&= [( - \up l_1) \cup \dots \cup ( - \up l_n)] \cap Y_p \\
&= Y_p - \bigcap \up l_i = Y_p - \up l
\end{align*}
for some $l_i \in K(Y)$, where $l = l_1 \vee \dots \vee l_n$, so $l \in K(Y)$.

If $U$ is a clopen upset of $X_p$, then $U = U' \cap X_p$ for some clopen upset $U'$ of $X$. Therefore, $U = (\up k_1 \cup \dots \cup \up k_n) \cap X_p$ for some $k_i \in K(X)$. The previous paragraph shows that $\bd_p U$ is a clopen downset of $Y_p$.

(4) By \cref{lem: dual iso},
\begin{align*}
\Box V' &= - R^{-1}[\up l_1 \cup \dots \cup \up l_n] = - (R^{-1}[\up l_1] \cup \dots \cup R^{-1}[\up l_n]) \\
&= - R^{-1}[\up l_1] \cap \dots \cap - R^{-1}[\up l_n] = \up k_1 \cap \dots \cap \up k_n 
\end{align*}
for some $k_i \in K(X)$. By (2), $\Box_p V  = \up k \cap X_p$, where $k = k_1 \vee \dots \vee k_n$, so $k \in K(X)$.

If $V$ is a clopen downset of $Y_p$, then $V = V' \cap Y_p$ for some clopen downset $V'$ of $Y$. Therefore, $V =  - (\up l_1 \cup \dots \cup \up l_n) \cap Y_p$ for some $l_i \in K(Y)$. The previous paragraph shows that $\Box_p V$ is a clopen upset of $X_p$.

(5) Set $U = \up k \cap X_p$ with $k \in K(X)$. By (1), (2), and \cref{phi psi is 1},
\begin{align*}
\Box_p\bd_p U &= \Box_p \bd_p(\up k \cap X_p) = \Box_p(\bd \up k \cap Y_p) = \Box\bd(\up k)\cap X_p \\
&= \psi\phi(\up k) \cap X_p = \up k \cap X_p = U.
\end{align*}

(6) Let $V = Y_p - \up l$ with $l \in K(Y)$. By (1), (2), and \cref{phi psi is 1},
\begin{align*}
\bd_p\Box_p V &= \bd_p\Box_p(-\up l \cap Y_p) = \bd_p(\Box(-\up l) \cap X_p) = \bd\Box(-\up l) \cap Y_p \\
&= -\blacksquare\Diamond \up l \cap Y_p = -\phi\psi(\up l) \cap Y_p = -\up l \cap Y_p = V.\qedhere
\end{align*}
\end{proof}

\begin{lemma} \label{lem: totally R-disconnected}
Let $(X, R, Y) \in \HD$. If $x \nr{R_p} y$, then there is a clopen upset $U \subseteq X_p$ with $x \in U$, $y \notin \bd_pU$, and $\Box_p\bd_p U = U$. 
\end{lemma}

\begin{proof}
Suppose that $x \nr{R_p} y$. Then $x \notin R^{-1}[y]$, and since $R^{-1}[y]$ is a closed downset of $X$, there is $k \in K(X)$ with $k \le x$ and $\up k \cap R^{-1}[y] = \varnothing$. Let $U = \up k \cap X_p$. Then $U$ is a clopen upset of $X_p$ and $x \in U$. By \cref{bd 5}, $U = \Box_p\bd_p U$. If $y \in \bd_p U$, then $x' \rel{R_p} y$ for some $x' \in U$. Therefore, $k \le x'$, so $k \rel{R} y$ by \cref{def: DH relation filter}. This contradicts the choice of $k$. Thus, $y \notin \bd_p U$.
\end{proof}

Let $Z \in \CohLatL$ and $U$ be a clopen upset of $Z_p$. Since $Z_p$ is a closed subset of $Z$ by \cref{lem: Xa and Ya are Priestley}, $U = U' \cap Z_p$ for some clopen upset $U'$ of $Z$ by Priestley duality. Therefore, $U' = \up k_1 \cup \dots \cup \up k_n$ for some $k_i \in K(Z)$ by \cref{basis,Lawson upset is Scott}.

\begin{lemma} \plabel{lem: R-regular implies R-closed}
Let $(X, R, Y) \in \HD$.
\begin{enumerate}
\item \label[lem: R-regular implies R-closed]{R-reg implies R-closed} Let $U$ be a clopen upset of $X_p$. Write $U = U' \cap X_p$ with  $U' = \up k_1 \cup \dots \cup \up k_n$ for some $k_i \in K(X)$. If $\Box_p\bd_p U \cap X_0 \subseteq U$ then $k := k_1 \wedge \dots \wedge k_n$ is a distributive meet and $U = \up k \cap X_p$. Therefore, $\Box_p\bd_p U = U$.
\item \label[lem: R-regular implies R-closed]{lem: R-coregular implies R-open} Let $V$ be a clopen downset of $Y_p$. Write $V = V' \cap X_p$ with  $V' =  - (\up l_1 \cup \dots \cup \up l_n)$ for some $l_i \in K(Y)$. If $Y_0 \cap V \subseteq \bd_p\Box_p V$ then $l := l_1 \wedge \dots \wedge l_n$ is a distributive meet and $V = Y_p - \up l$. Therefore, $\bd_p\Box_p V = V$.
\end{enumerate}
\end{lemma}

\begin{proof}
\eqref{R-reg implies R-closed} Suppose that $\Box_p\bd_p U \cap X_0\subseteq U$. If $k$ is not a distributive meet, then there is $l \in K(X)$ with $l \vee k < (l \vee k_1) \wedge \dots \wedge (l \vee k_n)$. By \cref{x is meet from X0}, there is $x \in X_0$ with $l \vee k \le x$ and $(l \vee k_1) \wedge \dots \wedge (l \vee k_n) \not\le x$. We have $k \le l\vee k \le x$. If $k_i \le x$ for some $i$, then $l \vee k_i \le x$, a contradiction, so $k \le x$ and each $k_i \not\le x$. This shows that $x \notin U$. If $y \in R_p[x]$, then $x \rel{R} y$, so $k \rel{R} y$  by \cref{def: DH relation filter}. If $k_i \nr{R} y$ for each $i$, then $k \nr{R} y$, again by \cref{def: DH relation filter}, which is false. Therefore, $k_i \rel{R} y$ for some $i$. This yields $y \in R_p[U]$, and so $R_p[x] \subseteq R_p[U]$. Thus, $X_0 \cap \Box_p\bd_p U \not\subseteq U$, a contradiction. Consequently, $k = k_1 \wedge \dots \wedge k_n$ is a distributive meet. We finish by showing that this implies that $U = \up k \cap X_p$. The inclusion $U \subseteq \up k \cap X_p$ is clear. For the other inclusion, if $x \in \up k \cap X_p$, then $k \le x$ and $x$ is d-prime. Therefore, some $k_i \le x$, so $x \in U$. Thus, $U = \up k \cap X_p$, and hence $\Box_p\bd_p U = U$ by \cref{bd 5}.

\eqref{lem: R-coregular implies R-open} is proved similarly.
\end{proof}

\begin{proposition} \label{prop: G on objects}
If $(X, R, Y)$ is a DH-space, then $(X_p, R_p, Y_p)$ is a GvG-space.
\end{proposition}

\begin{proof}
Since the orders on $X_p$ and $Y_p$ are the restrictions of the orders on $X$ and $Y$, respectively, the first two properties of \cref{def: GvG-space} hold. The remaining properties hold by \cref{lem: properties of Box and Diamond,lem: Xa and Ya are Priestley,lem: totally R-disconnected,lem: R-regular implies R-closed}.
\end{proof}

The above proposition associates a GvG-space with each DH-space. We extend this correspondence to a functor from $\HD$ to $\GvG$. For this
we require the following two facts. 

\begin{proposition} \label{lem: LG = CLF}
Let $\D = (X, R, Y) \in \HD$ and $\G = (X_p, R_p, Y_p)$. Define the map $\gamma_D \colon \CLF(X) \to \L\G$ by $\gamma_D(U)= U \cap X_p$ for each $U \in \CLF(X)$. Then $\gamma_D$ is a $\Lat$-isomorphism. 
\end{proposition}

\begin{proof}
That $\gamma_D$ is well defined follows from Lemmas~\ref{lem: KOF = LX} and \ref{lem: properties of Box and Diamond}(\ref{bd 5}), and it is clearly order-preserving. To see that it is order-reflecting, let $U, V \in \CLF(X)$ with $U \not\subseteq V$.  
By \cref{lem: KOF = LX}, $U = \up k$ and $V = \up l$ for some $k,l \in K(X)$.
Therefore, $l \not\le k$, so \cref{x is meet from X0} implies that there is $x \in X_0$ with $k \le x$ and $l \not\le x$. By \cref{lem: X0 the same}, $X_0 \subseteq X_p$, so $x \in X_p$, and hence $U \cap X_p \not\subseteq V \cap X_p$. Thus, $\gamma_D(U) \not\subseteq \gamma_D(V)$.
Finally, to see that $\gamma_D$ is onto, let $U \in \L\G$. Then  $\Box_p\bd_p U = U$, so $X_0 \cap \Box_p\bd_p U \subseteq U$. By \cref{R-reg implies R-closed}, $U = \up k \cap X_p$ for some $k \in K(X)$. Therefore,  $\up k \in \CLF(X)$ and $\gamma_D(\up k) = U$. Thus, $\gamma_D$ is an order-isomorphism, and hence a $\Lat$-isomorphism.
\end{proof}

\begin{proposition} \label{prop: G on morphisms}
Let $(f, g) \colon (X, R, Y) \to (X', R', Y')$ be a $\HD$-morphism. Define relations $S_f \subseteq X_p \times X_p'$ and $T_g \subseteq Y_p \times Y_p'$ by
\[
x \rel{S_f} x' \iff f(x) \le x' \quad \textrm{and} \quad y\rel{T_g} y' \iff g(y) \le y'. 
\]
Then $(S_f, T_g)$ is a $\GvG$-morphism.
\end{proposition}

\begin{proof}
Set $\G = (X_p, R_p, Y_p)$ and $\G' = (X_p', R_p', Y_p')$, where $R_p$ and $R_p'$ are the appropriate restrictions of $R$ and $R'$, respectively. We first verify \cref{GvG morphism 1}. Let $U' \in \L\G'$. By \cref{lem: KOF = LX,lem: LG = CLF}, there is $l \in K(X')$ with $U' = \up l \cap X_p'$. Since $S_f[x] = \up f(x) \cap X_p'$, 
\begin{align*}
\Box_{S_f} U' = \{ x \in X_p : S_f[x] \subseteq U' \} = \{ x \in X_p : l \le f(x) \} 
= f^{-1}(\up l) \cap X_p.
\end{align*}
By \cref{lem: KOF = LX}, $\up l \in \CLF(X')$. Therefore, since $f$ is a $\CohLatL$-morphism, $f^{-1}(\up l) \in \CLF(X)$. Thus, $f^{-1}(\up l) = \up k$ for some $k \in K(X)$, and hence $\Box_{S_f} U' = \up k \cap X_p$, which implies that $\Box_{S_f} U' \in \L\G$.

Next, let $y \in \Diamond_{T_g} \bd_p' U'$. Then there is $y' \in Y_p'$ with $y \rel{T_g} y'$ and $x' \in U'$ with $x' \rel{R_p'} y'$. We have $g(y) \le y'$, so $x' \rel{R'} g(y)$ by \cref{def: DH relation filter}. Therefore, there is $x \in X$ with $x \rel{R} y$ and $x' \le f(x)$ by \cref{def: DH morphisms 2}. Furthermore, we may assume that $x \in X_p$ by \cref{producing x' in X0,lem: X0 the same}. Because $x' \in U'$, we have $S_f[x] = \up f(x) \cap X_p' \subseteq U'$, so $y \in \bd_p \Box_{S_f} U'$.

Conversely, if $y \in \bd_p \Box_{S_f} U'$, then there is $x \in X_p$ with $x \rel{R_p} y$ and $S_f[x] \subseteq U'$. Since $x \rel{R_p} y$, we have $x \rel{R} y$, so $f(x) \rel{R'} g(y)$ by \cref{def: DH morphisms 1}. By Lemmas~\ref{lem: producing elements from X0 or Y0}(\ref{producing x' in X0},\ref{producing y' in Y0}) and \ref{lem: X0 the same}, there are $x' \in X_p'$ and $y' \in Y_p'$ such that $f(x) \le x'$, $g(y) \le y'$, and $x' \rel{R_p'} y'$. Therefore, $x \rel{S_f} x'$ and $y \rel{T_g} y'$. Since $S_f[x] \subseteq U'$, we have $x' \in U'$. This shows that $y' \in T_g[y] \cap \bd_p' U'$. Consequently, $y \in \Diamond_{T_g}\bd_p'U'$, and hence $\bd_p\Box_{S_f} U' = \Diamond_{T_g}\bd_p' U'$.

We next verify \cref{GvG morphism 2}. If $x \nr{S_f} x'$, then $f(x) \not\le x'$. Thus, there is $l \in K(X')$ with $l \le f(x)$ and $l \not\le x'$. Set $U' = \up l \cap X_p'$. By \cref{R-reg implies R-closed}, $U' = \Box_p'\bd_p' U'$,  $\up f(x) \cap X_p' \subseteq U'$, and $x' \notin U'$. Therefore, $S_f[x] = \up f(x) \cap X_p' \subseteq U'$, so $x \in \Box_{S_f} U'$.

We now verify \cref{GvG morphism 3}. If $y \nr{T_g} y'$, then $g(y) \not\le y'$, so there is $l \in K(Y')$ with $l \le g(y)$ and $l \not\le y'$. Set $U' = \Box_p' (Y_p' - \up l)$. Then $\bd_p' U' = Y_p' - \up l$ by \cref{lem: R-coregular implies R-open}, and hence $y' \in \bd_p' U'$. On the other hand, $T_g[y] = \up g(y) \cap Y_p' \subseteq \up l \cap Y_p'$, so $T_g[y] \cap \bd_p' U' = \varnothing$, and hence $y \notin \Diamond_{T_g} \bd_p' U'$. 

We finally verify  \cref{GvG morphism 4}. If $x\rel{R}y$, then $f(x) \rel{R'} g(y)$ by \cref{def: DH morphisms 1}. There are $x' \in X_0'$ and $y' \in Y_0'$ with $f(x) \le x'$, $g(y) \le y'$, and $x' \rel{R'} y'$ by \crefrange{def: DH morphisms 2}{def: DH morphisms 3}. Thus, by \cref{lem: X0 the same}, $x' \in X_p'$, $y' \in Y_p'$, $x \rel{S_f}x'$, and $y \rel{T_g} y'$.
\end{proof}

\begin{theorem} \label{thm: G is a functor}
There is a functor $\GG \colon \HD \to \GvG$ sending $(X, R, Y)$ to $(X_p, R_p, Y_p)$ and $(f, g)$ to $(S_f, T_g)$.
\end{theorem}

\begin{proof}
\cref{prop: G on objects} shows that $\GG$ is well defined on objects, and \cref{prop: G on morphisms} proves that it is well defined on morphisms. Let $(X, R, Y) \in \HD$. The identity morphism for $(X, R, Y)$ is $(1_X, 1_Y)$. Since $\GG(1_X, 1_Y) = (\le_X, \le_Y)$, we see that $\GG$ sends identity morphisms to identity morphisms by \cref{GvG identity}.

It is left to show that composition is preserved.  Let $(f, g) \colon (X, R, Y) \to (X', R', Y')$ and $(f', g') \colon (X', R', Y') \to (X'', R'', Y'')$ be $\HD$-morphisms. Set $S = S_{f'} \star S_{f}$ and $T = T_{g'} \star T_{g}$. We wish to show that $(S, T) = (S_{f'f}, T_{g'g})$. We prove that $S = S_{f'f}$; the proof for $T$ is similar. First suppose that $x \rel{S_{f'f}} x''$, so $f'f(x) \le x''$. Let $U \in \L\G''$, where $\G'' = (X''_p, R''_p, Y''_p)$. By \cref{lem: KOF = LX,lem: LG = CLF}, $U = \up k \cap X_p''$ for some $k \in K(X'')$. If $x \in \Box_SU$, then $x \in \Box_{S_f}\Box_{S_{f'}}U$ by \cref{lem: composition lemma for GvG}, so $S_{f'}S_f[x] \subseteq \up k$, and hence $k \le f'f(x)$. Therefore, $k \le x''$, so $x'' \in \up k \cap X_p'' = U$. Thus, $x \rel{S} x''$. Conversely, suppose that $x \rel{S} x''$. If $f'f(x) \not\le x''$, then there is $k \in K(X'')$ with $k \le f'f(x)$ and $k \not\le x''$. Set $U = \up k \cap X_p''$. Then $U \in \L\G''$ by \cref{lem: KOF = LX,lem: LG = CLF}. If $z \in S_{f'}S_f[x]$, then there is $x' \in X_p'$ with $f(x) \le x'$ and $f'(x') \le z$. Therefore, $f'f(x) \le z$. Thus, $S_{f'}S_f[x] \subseteq \up k$, and so $x \in \Box_SU$ by \cref{lem: composition lemma for GvG}. Since $x \rel{S} x''$, we then have $x'' \in U \subseteq \up k$, which is a contradiction. Consequently, $f'f(x) \le x''$, and so $x \rel{S_{f'f}} x''$. Hence, $S_{f'f} = S_{f'} \star S_{f}$, completing the proof that $\GG$ is a functor.
\end{proof}


\section{From Gehrke-van Gool to Hartung} \label{sec: GvG and Hg}

In this section we connect the Gehrke-van Gool approach to that of Hartung. The resulting spaces were introduced and studied in \cite{Har92}. The morphisms that Hartung considered dually correspond to onto bounded lattice homomorphisms \cite[Sec.~3.1]{Har92}. We introduce a more general notion of morphism between Hartung spaces, which dually correspond to all bounded lattice homomorphisms. This gives rise to the category $\Hg$. We show that there is a functor from $\GvG$ to $\Hg$. In \cref{sec: HD GvG Hg Urq}  we will show that this functor is an equivalence and in \cref{sec: conclusion} that $\Hg$ is dually equivalent to $\Lat$.

Let $R \subseteq X \times Y$ be a relation. We write $\phi, \psi$ for the Galois connection associated to $-R$ (see \cref{not: Galois connection}). We then have $\phi A = -\bd A$ for each $A \subseteq X$ and $\psi B = - \Diamond B$  for each $B \subseteq Y$. As in \cref{sec: DH and GvG}, we write $\Box$ for $\Box_R$ and $\bd$ for $\bd_R$.

\begin{definition} \label{def: lattice of a Hartung space}
Let $X, Y$ be topological spaces and $R \subseteq X \times Y$ a relation. Setting $\H = (X, R, Y)$, let
\[
\L\H = \{A \subseteq X :  A = \Box\bd A, \ A \textrm{ and } -\bd A \textrm{ are closed} \}.
\]
\end{definition}

\begin{definition} \plabel{def: le in Hartung}
Let $R \subseteq X \times Y$ be a relation.
\begin{enumerate}
\item \label[def: le in Hartung]{Hg leX} Define $\le_X$ on $X$ by  $x \le_X x'$ if $R[x'] \subseteq R[x]$. 
\item \label[def: le in Hartung]{Hg leY} Define $\le_Y$ on $Y$ by $y \le_Y y'$ if $R^{-1}[y'] \subseteq R^{-1}[y]$.
\item \label[def: le in Hartung]{Hg y-maximal} If $x \in X$ and $y \in Y$ with $x \rel{R} y$, then $x$ is $y$-\emph{maximal} if whenever $x \le_X x'$ and $x' \rel{R} y$, then $x' = x$.
\item \label[def: le in Hartung]{Hg x-maximal} If $x \in X$ and $y \in Y$ with $x \rel{R} y$, then $y$ is $x$-\emph{maximal} if whenever $y \le_Y y'$ and $x \rel{R} y'$, then $y' = y$.
\end{enumerate}
\end{definition}

\begin{definition} \plabel{def: Hartung space}
Let $X, Y$ be topological spaces and $R \subseteq X \times Y$ a relation. The triple $\H = (X, R, Y)$ is a \emph{Hartung space} provided
\begin{enumerate}
\item \label[def: Hartung space]{Hg partial orders} $\le_X$ and $\le_Y$ are partial orders;
\item \label[def: Hartung space]{Hg compact} $R$ is a compact subset of $X \times Y$ (with respect to the product topology);
\item \label[def: Hartung space]{Hg LH subbasis} $\L\H$ is a subbasis of closed sets for $X$, and $\{ -\bd A : A \in \L\H\}$ is a subbasis of closed sets for $Y$;
\item \label[def: Hartung space]{Hg closed} $A \subseteq X$ closed implies $\Box\bd A$ closed, and $B \subseteq Y$ closed implies $\bs\Diamond B$ closed;
\item \label[def: Hartung space]{Hg maximal} 
If $x \in X$, then there is $y \in Y$ such that $x$ is $y$-maximal, and if $y \in Y$, there is $x \in X$ such that $y$ is $x$-maximal;

\item \label[def: Hartung space]{Hg not R related} If $x \nr{R} y$, then there is $A \in \L\H$ with $x \in A$ and $y \notin \bd A$.
\end{enumerate}
\end{definition}

\begin{remark} \plabel{rem: Hartung issues}
\hfill
\begin{enumerate} 
\item The definition above is a reinterpretation of \cite[Def.~2.2.3]{Har92} to connect more directly with the notion of GvG-space.
\item If $\H$ is a Hartung space, then $\L\H$ is a bounded lattice whose top is $X$, bottom is $\varnothing$, join is given by $A_1 \vee A_2 = \Box\bd(A_1 \cup A_2)$, and meet by $A_1 \wedge A_2 = A_1 \cap A_2$.
\item Hartung associates to $\H = (X, R, Y)$ the lattice
\[
\{ (A, B) \in \wp(X) \times \wp(Y) :  B = \phi A, A = \psi B,\textrm{ and } A,B \textrm{ are closed} \}.
\]
Because $B$ is determined from $A$, our definition is isomorphic as a lattice to Hartung's, and is more convenient for us.
\item \label[rem: Hartung issues]{Hs dual relation} We work with the complement of Hartung's relation, which is more convenient to compare with the other categories we consider.
\end{enumerate}
\end{remark}

\begin{lemma} \plabel{rem: partial orders on Hg}
Let $\H = (X, R, Y)$ be a Hartung space, $x \in X$, and $y \in Y$.
\begin{enumerate}
\item \label[rem: partial orders on Hg]{Hg A is upset} If $A \in L\H$, then $A$ is an upset and $\bd A$ is a downset. 
\item \label[rem: partial orders on Hg]{Hg up x is intersection of As} $\up x = \bigcap \{ A \in \L\H : x \in A\}$. 
\item \label[rem: partial orders on Hg]{Hg up y is intersection of diamond As} $\up y = \bigcap \{ -\bd A : A \in \L\H, \, y \in -\bd A \}$.
\item \label[rem: partial orders on Hg]{Rx cap up y compact} $R[x] \cap \up y$ and $\up x \cap R^{-1}[y]$ are compact.
\item \label[rem: partial orders on Hg]{doubly founded} If $x \rel{R} y$, then there are $x' \in X$ and $y' \in Y$ such that $x \le_X x'$, $x'$ is $y$-maximal, $y \le_Y y'$, and $y'$ is $x$-maximal.
\end{enumerate}
\end{lemma}

\begin{proof}
\eqref{Hg A is upset} Since $A \in \L\H$, $A = \Box\bd A$, so 
\[
A = -R^{-1}[-\bd A] = -\bigcup \{ R^{-1}[y] : y \in -\bd A\} = \bigcap \{-R^{-1}[y] : y \in -\bd A\}. 
\]
\cref{Hg leX} implies that $R^{-1}[y]$ is a downset for each $y$. Therefore, $A$ is an upset. A similar argument shows that $\bd A$ is a downset.

\eqref{Hg up x is intersection of As} The inclusion $\subseteq$ is clear since each $A \in \L\H$ is an upset by \eqref{Hg A is upset}. For the reverse inclusion, suppose that $x \not\le_X x'$. Then $R[x'] \not\subseteq R[x]$, so there is $y$ with $x' \rel{R} y$ and $x \nr{R} y$. Therefore, there is $A \in \L\H$ with $x \in A$ and $y \notin \bd A$. Since $x' \rel{R} y$, we have $x' \notin A$, completing the proof. 

\eqref{Hg up y is intersection of diamond As} This is proved similarly to \eqref{Hg up x is intersection of As}. 

\eqref{Rx cap up y compact} By \eqref{Hg up x is intersection of As} and \eqref{Hg up y is intersection of diamond As}, $\up x \times \up y$ is closed in $X \times Y$, so $R \cap (\up x \times \up y)$ is closed in $R$, and hence compact because $R$ is compact by \cref{Hg compact}. Let $\pi \colon X \times Y \to Y$ be the projection map. Since $\pi$ is continuous and $R \cap (\up x \times \up y)$ is compact, $\pi[R \cap (\up x \times \up y)]$ is compact. But $\pi[R \cap (\up x \times \up y)] = R[x] \cap \up y$, so $R[x] \cap \up y$ is compact. A similar argument shows that ${\uparrow} x \cap R^{-1}[y]$ is compact. 

\eqref{doubly founded} Let $x \rel{R} y$. We show that there is $y' \in Y$ such that $y \leq_Y y'$ and $y'$ is $x$-maximal. Let $\mathcal T$ be the set of all totally ordered subsets of $R[x] \cap \up y$, ordered by inclusion. Then $\mathcal T$ is nonempty since $\{y\} \in \mathcal T$. Because the union of a chain of totally ordered subsets of $R[x] \cap \up y$ is again totally ordered, Zorn's lemma shows that there is a maximal element $C \in \mathcal T$. The collection $\{ R[x] \cap \up z : z \in C\}$ is a family of closed subsets of $R[x] \cap \up y$ by \eqref{Hg up y is intersection of diamond As}, and it clearly has the finite intersection property. Therefore, by \eqref{Rx cap up y compact}, $\bigcap \{ R[x] \cap \up z : z \in C\} \ne \varnothing$. Take $y' \in \bigcap \{ R[x] \cap \up z : z \in C\}$. Then $y \leq_Y y'$ and $x \rel{R} y'$. To see that $y'$ is $x$-maximal, let $y' \le_Y y''$ with $x \rel{R} y''$. Then $y'' \ge_Y y' \ge_Y z$ for each $z \in C$, so $C \cup \{y''\}$ is a totally ordered subset of $R[x] \cap \up y$ containing $C$. Maximality of $C$ shows that $y'' \in C$, forcing $y'' \le_Y y'$, and hence $y'' = y'$. Thus, $y'$ is $x$-maximal. That there is $x' \in X$ such that $x \leq_X x'$ and $x'$ is $y$-maximal is proved similarly but uses that $\up x \cap R^{-1}[y]$ is compact. 
\end{proof}

We next define the notion of a Hartung morphism which is based on that of a GvG-morphism (see \cref{def: GvG morphism}).

\begin{definition} \plabel{def: Hartung morphisms}
Let $\H = (X, R, Y)$ and $\H' = (X', R', Y')$ be Hartung spaces. We call a pair $(S, T)$ of relations with $S \subseteq X \times X'$ and $T \subseteq Y \times Y'$ a {\em Hartung morphism} provided 
\begin{enumerate}
\item \label[def: Hartung morphisms]{Hg morphism Box} If $A' \in \L\H'$, then $\Box_SA' \in \L\H$ and $\bd \Box_S A' = \Diamond_T \bd' A'$.
\item \label[def: Hartung morphisms]{Hg morphism not S} If $x \nr{S} x'$, then there is $A' \in \L\H'$ with $x \in \Box_SA'$ and $x' \notin A'$.
\item \label[def: Hartung morphisms]{Hg morphism not T} If $y \nr{T} y'$, then there is $A' \in \L\H'$ with $y \notin \Diamond_T\bd' A'$ and $y' \in \bd' A'$.
\item \label[def: Hartung morphisms]{Hg morphism R} If $x \rel{R} y$, then there are $x' \in X'$ and $y' \in Y'$ with $x \rel{S} x'$, $y \rel{T} y'$, and $x' \rel{R'} y'$.
\end{enumerate}
\end{definition}

As in \cref{rem: equivalent condition for GvG S}, we have:

\begin{remark} \plabel{rem: equivalent condition for Hartung S}
Let $(S, T) \colon \H \to \H'$ be a Hartung morphism between Hartung spaces $\H = (X, R, Y)$ and $\H' = (X', R', Y')$.
\begin{enumerate}
\item \label[rem: equivalent condition for Hartung S]{equivalent conditions for S(1)} 
\cref{Hg morphism not S} is equivalent to: 
\[
x \rel{S} x' \iff (\forall A' \in \L\H')\, (x \in \Box_S A' \Longrightarrow x' \in A'),
\]
and an analogous statement holds for \cref{Hg morphism not T}.
\item \label[rem: equivalent condition for Hartung S]{S[x] an upset} $S[x]$ is an upset of $X'$ for each $x\in X$ and $T[y]$ is an upset of $Y'$ for each $y\in Y$.
\item $S^{-1}[x']$ is a downset of $X$ for each $x' \in X'$ and $T^{-1}[y']$ is a downset of $Y$ for each $y' \in Y'$.
\end{enumerate}
\end{remark}

Composition of Hartung morphisms is defined as that of GvG-morphisms:

\begin{definition} \plabel{def: composition in Hartung}
Let $(S_1, T_1) \colon \H_1 \to \H_2$ and $(S_2, T_2) \colon \H_2 \to \H_3$ be Hartung morphisms. 
\begin{enumerate}
\item Define $S_2 \star S_1$ by $x\rel{ (S_2 \star S_1)} z$ provided that $x \in \Box_{S_1}\Box_{S_2}A$ implies $z \in A$ for each $A \in \L\H_3$.
\item \label[def: composition in Hartung]{star of Ts} Define $T_2 \star T_1$ by $y \rel{(T_2 \star T_1)}w$ provided that $w \in \bd A$ implies $y \in \Diamond_{T_1}\Diamond_{T_2}\bd A$ for each $A \in \L\H_3$.
\item Set $(S_2, T_2) \star (S_1, T_1) = (S_2 \star S_1, T_2 \star T_1)$.
\end{enumerate}
\end{definition}

An argument similar to that of \cref{lem: composition lemma for GvG} yields:

\begin{lemma} \label{lem: composition lemma for Hg}
Suppose that $(S_1, T_1) \colon \H_1 \to \H_2$ and $(S_2, T_2) \colon \H_2 \to \H_3$ are Hartung morphisms. If $A \in \L\H_3$, then $\Box_{S_2\star S_1}A = \Box_{S_1}\Box_{S_2}A$ and $\Diamond_{T_2 \star T_1} \bd A = \Diamond_{T_1} \Diamond_{T_2} \bd A$.
\end{lemma}

An analogous proof to \cref{prop: composition in GvG} yields:

\begin{proposition} \plabel{prop: composition in Hartung}
Suppose that $(S_1, T_1) \colon \H_1 \to \H_2$ and $(S_2, T_2) \colon \H_2 \to \H_3$ are Hartung morphisms. 
\begin{enumerate}
\item \label[prop: composition in Hartung]{Hg morphism} $(S_2, T_2) \star (S_1, T_1)$ is a Hartung morphism.
\item \label[prop: composition in Hartung]{Hg identity} If $\H = (X, R, Y)$ is a Hartung space, then $(\le_X, \le_Y) \colon \H \to \H$ is a Hartung morphism and is the identity morphism for $\H$.
\item \label[prop: composition in Hartung]{Hg associativity} The operation $\star$ is associative.
\end{enumerate}
\end{proposition}

The previous proposition shows that we have a category, motivating the following definition.

\begin{definition} \label{def: Hg}
Let $\Hg$ denote the category of Hartung spaces and Hartung morphisms.
\end{definition}

In the rest of the section we define a functor $\HH \colon \GvG \to \Hg$. Let $\G := (X, R, Y) \in \GvG$ and, as in \cref{def: X0 and Y0}, let
\[
X_0 = \bigcup \{ \max R^{-1}[y] : y \in Y\}
\quad
\mbox{and}
\quad
Y_0 = \bigcup \{ \max R[x] : x \in X \}.
\]
We topologize $X_0$ with the subspace topology of the open downset topology on $X$, and topologize $Y_0$ similarly. Set $R_0 = {R} \cap (X_0 \times Y_0)$. Our aim is to show that $ \H := (X_0, R_0, Y_0)$ is a Hartung space, which we do in a series of lemmas.

\begin{lemma} \label{lem: Hg partial orders and R related}
Define $\le_{X_0}$ on $X_0$ and $\le_{Y_0}$ on $Y_0$ by Definition~\emph{\ref{def: le in Hartung}(\ref{Hg leX},\ref{Hg leY})} applied to $R_0$. Then $\le_{X_0}$ and $\le_{Y_0}$ are the restrictions of the partial orders of $X$ and $Y$, respectively, and thus are partial orders.
\end{lemma}

\begin{proof}
Let $x, x' \in X_0$. By \cref{GvG 1}, if $x \le x'$, then $R[x'] \subseteq R[x]$. Because $R_0[x'] = R[x'] \cap Y_0$ and $R_0[x] = R[x] \cap Y_0$, we have $R_0[x'] \subseteq R_0[x]$, so $x \le_{X_0} x'$.

Conversely, suppose that $x \le_{X_0} x'$, so $R_0[x'] \subseteq R_0[x]$. Let $y \in R[x']$. By \cref{producing y' in Y0}, there is $y' \in Y_0$ with $y \le y'$ and $x' \rel{R} y'$. This yields $x' \rel{R_0} y'$, so $y' \in R_0[x'] \subseteq R_0[x] \subseteq R[x]$. Because $y \le y'$, \cref{GvG 2} gives that $y \in R[x]$. Thus, $R[x'] \subseteq R[x]$, and hence $x \le x'$ by \cref{GvG 1}.

A similar argument shows that $\le_{Y_0}$ is the restriction of the partial order of $Y$.
\end{proof}

\begin{lemma} \label{lem: compactness}
$R_0$ is a compact subset of $X_0 \times Y_0$ with respect to the product topology on $X_0 \times Y_0$.
\end{lemma}

\begin{proof}
By \cref{closed order}, $R$ is a closed subset of $X\times Y$, hence is compact. Therefore, $R$ is a compact subspace of the weaker topology of the product of the open downset topologies on $X$ and $Y$. We show that $R = \down R_0$, where $\down$ is calculated with respect to the product order on $X \times Y$ of the Priestley orders of $X$ and $Y$. Let $(x, y) \in R$, so $x \rel{R} y$. By \crefrange{producing x' in X0}{producing y' in Y0}, there are $x' \in X_0$ and $y' \in Y_0$ with $x \le x', y \le y'$, and $x' \rel{R} y'$, so $x' \rel{R_0} y'$. Hence, $(x', y') \in R_0$ and $(x,y) \le (x', y')$. Thus, $R \subseteq \down R_0$. For the reverse inclusion, suppose that $(x, y) \le (x', y') \in R_0$. Then $x' \rel{R_0} y'$, so $x' \rel{R} y'$. This yields $x \rel{R} y'$ by \cref{Hg leX}, and hence $x \rel{R} y$ by \cref{Hg leY}, concluding the proof that $R = {\downarrow}R_0$.
By \cref{lem: Hg partial orders and R related}, $\le_{X_0}$ is the restriction to $X_0$ of $\le$ on $X$, and similarly for $\le_{Y_0}$. 
Since $\le$ is the specialization order of the open upset topology, which is dual to the specialization order of the open downset topology, we obtain that
the specialization order of $X_0 \times Y_0$ is the dual of the restriction of the product order on $X \times Y$. Therefore, from $R = \down R_0$ it follows that $R$ is the saturation of $R_0$. Thus, since $R$ is compact, we conclude that $R_0$ is compact (see, e.g., \cite[p.~43]{GHKLMS03}). 
\end{proof}

The previous two lemmas show that $\H$ satisfies \crefrange{Hg partial orders}{Hg compact}. 

\begin{lemma} \plabel{lem: Galois closure preserves closed sets}
Suppose that $A' \subseteq X$ is an upset and $B' \subseteq Y$ is a downset. Set $A = A' \cap X_0$ and $B = B' \cap Y_0$. 
\begin{enumerate}
\item \label[lem: Galois closure preserves closed sets]{phi for H} $\bd_0 A = \bd A' \cap Y_0$.
\item \label[lem: Galois closure preserves closed sets]{psi for H} $\Box_0 B = \Box B' \cap X_0$.
\item \label[lem: Galois closure preserves closed sets]{boxdiamond} $\Box_0\bd_0 A = \Box\bd A' \cap X_0$.
\item \label[lem: Galois closure preserves closed sets]{diamondbox} $\bs_0\Diamond_0 B = \bs\Diamond B' \cap Y_0$.
\item \label[lem: Galois closure preserves closed sets]{boxdiamond is closed} If $A$ is closed, then $\Box_0\bd_0 A$ is closed.
\item \label[lem: Galois closure preserves closed sets]{bsdiamond is closed} If $B$ is closed, then $\bs_0\Diamond_0 B$ is closed.
\end{enumerate}
\end{lemma}

\begin{proof}
\eqref{bd 1} The inclusion $\bd_0 A \subseteq \bd A' \cap Y_0$ is clear. For the reverse inclusion, let $y \in \bd A' \cap Y_0$. Then $x \rel{R} y$ for some $x \in A'$. By \cref{producing x' in X0}, there is $x' \in X_0$ with $x \le x'$ and $x' \rel{R} y$, so $x' \rel{R_0} y$. Since $A'$ is an upset, $x' \in A' \cap X_0 = A$. Therefore, $y \in R_0[A]$, yielding the equality. 

\eqref{bd 2} The proof is similar to that of \eqref{bd 1}.

\eqref{bd 3} By \eqref{bd 1} and \eqref{bd 2} applied to $B' = \bd A'$, we have
\[
\Box_0\bd_0 A = \Box_0(\bd A'\cap Y_0) = \Box \bd A' \cap X_0.
\]

\eqref{bd 4} The proof is similar to that of \eqref{bd 3}.

\eqref{bd 5} If $A$ is closed, we may assume that $A'$ is closed (since the topology on $X_0$ is the subspace topology of the open downset topology on $X$). Therefore, $A' = \bigcap U_i$ for some family $\{U_i\}$ of clopen upsets of $X$ (see, e.g., \cite[Lem.~11.21]{DP02}), and we may assume the family is down-directed. Since $R^{-1}[y]$ is closed for each $y \in Y$ (see \cref{closed order}), \cref{lem: Esakia} shows that $\bd A' = \bigcap \bd U_i$, and so $\Box\bd A' = \bigcap \Box\bd U_i$ since $\Box$ commutes with arbitrary intersections. Each $\Box\bd U_i$ is a clopen upset, so $\Box\bd A'$ is a closed upset. Therefore, $\Box_0\bd_0 A = \Box\bd A' \cap X_0$ is closed.

\eqref{bd 6} The proof is similar to that of \eqref{bd 5}. 
\end{proof}

\begin{lemma} \plabel{lem: S versus BX}
\hfill
\begin{enumerate}
\item\label[lem: S versus BX]{subbasis 1} $\L\H$ is a subbasis of closed sets for $X_0$.
\item\label[lem: S versus BX]{subbasis 2} $\{ Y_0 -\bd U : U \in \L\G\}$ is a subbasis of closed sets for $Y_0$. 
\end{enumerate}
\end{lemma}

\begin{proof}
\eqref{subbasis 1} Let $C \subseteq X_0$ be a closed set. Then $C = D \cap X_0$ for some closed upset $D$ of $X$, and $D = \bigcap  U_i$ for some family $\{U_i\}$ of clopen upsets of $X$. By \cref{GvG LG basis}, each $U_i$ is a finite union from $\L\G$, and the result follows.

\eqref{subbasis 2} The proof is similar, but uses \cref{GvG LG basis for Y} instead of \ref{lem: properties of GvG spaces 2}\eqref{GvG LG basis}.
\end{proof}

\cref{lem: S versus BX} shows that $\H$ satisfies \cref{Hg LH subbasis}, and \crefrange{boxdiamond is closed}{bsdiamond is closed} that $\H$ satisfies \cref{Hg closed}.

\begin{lemma}
$\H$ satisfies Definition~\emph{\ref{def: Hartung space}(\ref{Hg maximal},\ref{Hg not R related})}.
\end{lemma}

\begin{proof}
For \cref{Hg maximal}, let $x \in X_0$. Then $x$ is $y$-maximal for some $y \in Y$. By \cref{producing y' in Y0}, there is $y' \ge y$ such that $y'$ is $x$-maximal, so $y' \in Y_0$. Suppose that $x \le x'$ with $x' \rel{R} y'$. Then $x' \rel{R} y$ (see \cref{GvG 2}). Thus, $x' = x$ since $x$ is $y$-maximal. Consequently, $x$ is $y'$-maximal for some $y \in Y_0$. The other half of \cref{Hg maximal} is proved similarly. 

For \cref{Hg not R related}, if $x \nr{R_0} y$, then $x \nr{R} y$, so by \cref{GvG 5} there is $U = \Box\bd U$ with $x \in U$ and $y \notin \bd U$. Set $A = U \cap X_0$. Then $x \in A$ and $\bd_0 A = \bd U \cap Y_0$ by \cref{phi for H}, so $y \notin \bd_0 A$.
\end{proof}

The above lemmas verify that $\H$ satisfies the conditions of \cref{def: Hartung space}. We thus have:

\begin{proposition} \label{prop: HH on objects}
If $\G = (X, R, Y)$ is a GvG-space, then $\H = (X_0, R_0, Y_0)$ is a Hartung space.
\end{proposition}

It follows from \cite[Ex.~4.1]{GvG14} that in general $X_0$ is properly contained in $X$. The example of Gehrke-van Gool is infinite. The next proposition shows that this is unavoidable as $X_0 = X$ in the finite case (see \cite[Ex.~5.4]{GvG14}):

\begin{proposition}
If $\G = (X, R, Y)$ is a finite GvG-space, then $X_0 = X$.
\end{proposition}

\begin{proof}
Let $x \in X$ and set $U = \up (\up x \cap X_0)$. Then $U$ is an upset, and hence a clopen upset of $X$ because $X$ is finite, hence discrete. Clearly $U \subseteq \up x$. For the other inclusion, observe that $\up x \in \L\G$ since the intersection in \cref{GvG up x} is finite and $\L\G$ is closed under finite intersections. Therefore, $U \subseteq \up x$ implies $\Box\bd U \subseteq \up x$. Thus, $X_0 \cap \Box\bd U \subseteq X_0 \cap \up x \subseteq U$, so $\Box\bd U = U$ by \cref{GvG 6}, and hence $U \in \L\G$. If $x \notin U$, then by \cref{x notin U GvG} there is $x' \in X_0$ with $x \le x'$ and $x' \notin U$, contradicting the definition of $U$. Consequently, $U = \up x$, and so $x \in X_0$ since $U = \up (\up x \cap X_0)$.
\end{proof}

Define a functor $\HH \colon \GvG \to \Hg$ by setting $\HH(X, R, Y) = (X_0, R_0, Y_0)$ for each $(X, R, Y) \in \GvG$ and $\HH(S, T) = (S_0, T_0)$ for each $\GvG$-morphism $(S, T) \colon (X, R, Y) \to (X', R', Y')$, where $S_0 = S \cap (X_0 \times X_0')$ and $T_0 = T \cap (Y_0 \times Y_0')$.
That $\HH$ is well defined on objects follows from \cref{prop: HH on objects}. In order to see that $\HH$ is well defined on morphisms, we require the following two facts.

\begin{proposition} \label{lem: LH = LG} 
Let $\G \in \GvG$. Then the map $\delta_\G \colon \L\G \to \L\H$, defined by $\delta_\G(U) = U \cap X_0$, is a $\Lat$-isomorphism. 
\end{proposition}

\begin{proof}
\cref{boxdiamond} shows that $\delta_\G$ is well defined, and clearly it is order-preserving. To show it is order-reflecting, let $U, V \in \L\G$ with $U \not\subseteq V$. Then there is $x \in U - V$. By \cref{x notin U GvG}, there is $x' \in X_0$ with $x \le x'$ and $x' \notin V$ . We have $x' \in U$ since $U$ is an upset. Consequently, $\delta_\G(U) \not\subseteq \delta_\G(V)$. 

To see that $\delta_\G$ is onto, let $A \in \L\H$. Since $A$ is closed in $X_0$, there is a closed upset $A'$ of $X$ with $A = A' \cap X_0$. Therefore, there is a down-directed family $\{U_i \}$ of clopen upsets of $X$ with $A' = \bigcap U_i$. Because $A \times -\bd_0 A \subseteq -R_0$, we have
\[
R_0 \cap \bigcap \{ (U_i\cap X_0) \times -\bd_0 A  \} = \varnothing. 
\]
Compactness of $R_0$ implies that a finite intersection is empty. Since $\{U_i\}$ is a down-directed family, there is a clopen upset $U \supseteq A'$ with $R_0 \cap ((U \cap X_0) \times -\bd_0 A) = \varnothing$. This forces $U \cap X_0 \subseteq \Box_0\bd_0 A = A$, and so $A = U \cap X_0$. We then have $A = \Box_0\bd_0 A = \Box\bd U \cap X_0$ by \cref{boxdiamond}. Because $\Box\bd U \in \L\G$, we conclude that $A = \delta_\G(U)$. Thus, $\delta_\G$ is an order-isomorphism, and hence a $\Lat$-isomorphism.
\end{proof}

\begin{lemma} \label{lem: notin U GvG case} 
Let $\G = (X, R, Y)$ and $\G' = (X', R', Y')$ be GvG-spaces. 
Let $(S, T) \colon \G \to \G'$ be a $\GvG$-morphism. If $U' \in \L\G'$ and $A' = U' \cap X_0'$, then
\[
\Box_{S_0}A' = \Box_SU' \cap X_0 \textrm{ and } \Box_{T_0} \bd_0' A'= \Box_T \bd' U' \cap Y_0. 
\]
\end{lemma}

\begin{proof}
Let $x \in \Box_S U' \cap X_0$. Then $S_0[x] = S[x] \cap X_0' \subseteq U' \cap X_0' = A'$, so $x \in \Box_{S_0} A'$. Conversely, suppose that $x \in \Box_{S_0}A'$. If $x \notin \Box_SU'$, then $S[x] \not\subseteq U'$, so there is $x' \in X'$ with $x \rel{S} x'$ and $x' \notin U'$. Since $U'\in L\G'$ and $x' \notin U'$, $R[x'] \not\subseteq R[U']$ (see \cref{Rx in RU}). Therefore, there is $y \in Y$ with $x' \rel{R} y$ and $y \notin R[U']$. By \cref{producing x' in X0}, there is $z' \in X_0'$ with $x' \le z'$ and $z' \rel{R} y$. Thus, $z' \notin U'$ since $y \notin R[U']$. By \cref{GvG S 2}, $S[x]$ is an upset, so $z' \in S[x] \cap X_0' = S_0[x]$. This forces $z' \in A' \subseteq U'$, a contradiction.  Consequently, $x \in \Box_S U' \cap X_0$. 

For the second statement, the inclusion $\Box_T \bd' U' \cap Y_0 \subseteq \Box_{T_0} \bd_0' A'$ is straightforward. For the reverse inclusion, let $y \in \Box_{T_0} \bd_0' A'$. If $y \notin \Box_T \bd' U'$, then $T[y] \not\subseteq \bd U'$, so there is $y'$ with $y \rel{T} y'$ and $y' \notin \bd U'$. By \cref{S and T closed,producing y' in Y0}, there is $z' \in Y_0'$ with $y' \le z'$ and $y \rel{T} z'$. Therefore, $y \rel{T_0} z'
$, so $z' \in T_0[y]$. Because $\bd U'$ is a downset and $y' \notin \bd U'$, we have $z' \notin \bd U'$.
This is a contradiction to $y \in \Box_{T_0} \bd_0' A'$. Thus, $y \in \Box_T \bd' U'$.
\end{proof}

\begin{proposition} \label{prop: HH on morphisms}
Let $\G,  \G'\in \GvG$, $\H = \HH(\G)$, and $\H' = \HH(\G')$. If $(S, T) \colon \G \to \G'$ is a $\GvG$-morphism, then $(S_0, T_0) \colon \H \to \H'$ is an $\Hg$-morphism.
\end{proposition}

\begin{proof}
We first verify \cref{Hg morphism Box}. Let $A' \in \L\H'$. By \cref{lem: LH = LG}, we may write $A' = U' \cap X_0$ with $U' \in \L\G'$. Then $\Box_{S_0}A' = \Box_SU' \cap X_0$ and $\Box_{T_0} \bd_0' A'= \Box_T \bd' U' \cap Y_0$ by \cref{lem: notin U GvG case}, and $\Box_{S_0} A' \in \L\H$ by \cref{lem: LH = LG,GvG morphism 1}. Thus, by \cref{phi for H,GvG morphism 1},
\[
\bd_0\Box_{S_0}A' =  \bd_0(\Box_SU' \cap X_0) = \bd\Box_S U' \cap Y_0 =\Box_T \bd' U' \cap Y_0 = \Diamond_{T_0} \bd_0' A'.
\] 

We next verify \cref{Hg morphism not S}. 
Suppose that $x \nr{S_0} x'$. Then $x \nr{S} x'$, and since $(S, T)$ is a $\GvG$-morphism, there is $U \in \L\G'$ with $x \in \Box_S U'$ and $x' \notin U'$. If $A' = U' \cap X_0$, then $A' \in \L\H$ by \cref{lem: LH = LG}. We have $\Box_{S_0}A' = \Box_S U' \cap X_0$ by the previous paragraph. Therefore, $x \in \Box_{S_0}A'$ and $x' \notin A'$. A similar argument verifies \cref{Hg morphism not T}. 

Finally, we verify \cref{Hg morphism R}. Suppose that $x \rel{R_0} y$. Then $x \rel{R} y$, and since $(S, T)$ is a $\GvG$-morphism, there are $x', y'$ with $x \rel{S} x'$, $y \rel{T} y'$, and $x' \rel{R'} y$. By \cref{lem: producing elements from X0 or Y0}, there are $x'' \in X_0'$ and $y'' \in Y_0'$ such that $x' \le x''$, $y' \le y''$, and $x'' \rel{R'} y''$, so $x'' \rel{R_0'} y''$. Also, $x \rel{S} x''$ and $y \rel{T} y''$, hence $x \rel{S_0} x''$ and $y \rel{T_0} y''$. This completes the proof that $(S_0, T_0)$ is an $\Hg$-morphism.  
\end{proof}

\begin{theorem} \label{prop: H is a functor}
There is a functor $\HH \colon \GvG \to \Hg$ sending $(X, R, Y)$ to $(X_0, R_0, Y_0)$ and $(S, T) \colon \G \to \G'$ to $(S_0, T_0)$.
\end{theorem}

\begin{proof}
By \cref{prop: HH on objects,prop: HH on morphisms}, $\HH$ is well defined on objects and morphisms. Let $\G = (X, R, Y) \in \GvG$. Then $1_\G = (\le_X, \le_Y)$. Therefore, by \cref{Hg identity},
\[
\HH(\le_X, \le_Y) = (\le_{X_0}, \le_{Y_0}) = 1_{\HH(\G)},
\]
so $\HH$ preserves identity morphisms. 

To show that composition is preserved, let $(S, T) \colon \G \to \G'$ and $(S', T') \colon \G' \to \G''$ be $\GvG$-morphisms. Set $S'' = S' \star S$ and $T'' = T' \star T$. We show that $S_0'' = S_0' \star S_0$; the proof that $T_0'' = T_0' \star T_0$ is similar. Let $x \in X_0$ and $x'' \in X_0''$. Then $x \rel{(S_0' \star S_0)} x''$ iff $x \in \Box_{S_0}\Box_{S_0'} A''$ implies $x'' \in A''$ for each $A'' \in \L\H''$. By \cref{lem: LH = LG}, we may write $A'' = U'' \cap X_0''$ for some $U'' \in \L\G''$. Then $x \rel{(S' \star S)} x''$ iff $x \in \Box_S \Box_{S'} U''$ implies $x'' \in U''$ for each $U'' \in \L\G'$. We show that $x \rel{(S_0' \star S_0)} x''$ iff $x \rel{(S' \star S)} x''$, which will show that $S''_0 = S_0' \star S_0$. First, let $x \rel{(S_0' \star S_0)} x''$. Let $U'' \in \L\G''$ and set $A'' = U'' \cap X_0''$. Suppose that $x \in \Box_S \Box_{S'} U''$. By \cref{lem: notin U GvG case},
\begin{equation}
\Box_{S_0} \Box_{S_0'} A'' = \Box_{S_0}(\Box_{S'}U'' \cap X_0') = \Box_S\Box_{S'}U'' \cap X_0. \label{eqn: double box}
\end{equation}
Therefore, $x \in \Box_S \Box_{S'}U'' \cap X_0 = \Box_{S_0}\Box_{S_0'}A''$, so $x'' \in A''$, and hence $x'' \in U''$. This shows that ${x \rel{(S' \star S)} x''}$. 

Conversely, let $x \rel{(S' \star S)} x''$ and $A'' \in \L\H''$. Suppose $x \in \Box_{S_0}\Box_{S_0'} A''$. By \cref{lem: LH = LG}, there is $U'' \in \L\G''$ with $A'' = U'' \cap X_0''$. By \cref{eqn: double box},  $\Box_{S_0}\Box_{S_0'} A'' = \Box_S \Box_{S'} U'' \cap X_0$. Thus, $x \in \Box_S\Box_{S'} U''$, so $x'' \in U''$, and hence $x'' \in U'' \cap X_0'' = A''$. This shows that $x \rel{(S_0' \star S_0)} x''$.  Consequently, $S_0'' = S_0' \star S_0$, as desired. Therefore, $\HH$ preserves composition, and hence is a functor.
\end{proof}


\section{From Hartung to Urquhart} \label{sec: Hg and Urq}

In this section we connect the Hartung approach to that of Urquhart. The resulting spaces were introduced and studied in \cite{Urq78}. The morphisms that Urquhart considered dually correspond to onto bounded lattice homomorphisms \cite[Sec.~5]{Urq78}. We introduce a more general notion of morphism between Urquhart spaces, which dually correspond to all bounded lattice homomorphisms. We show that there is a functor from $\Hg$ to the resulting category $\Urq$ of Urquhart spaces. In \cref{sec: HD GvG Hg Urq} we will show that this functor is an equivalence and in \cref{sec: conclusion} that $\Urq$ is dually equivalent to $\Lat$. 

Let $Z$ be a topological space with two quasi-orders $\le_1$ and $\le_2$. 
We write $\Up(Z, \le_i)$ for the upsets of $(Z,\le_i)$ for $i = 1,2$. By \cite[Lem.~1]{Urq78}, we have the Galois connection 
$\phi \colon \Up(Z, \le_1) \to \Up(Z, \le_2)$ and $\psi \colon \Up(Z, \le_2) \to \Up(Z, \le_1)$
given by 
\[
\phi A = -\down_2 A \quad \mbox{and} \quad \psi A = -\down_1 A,
\]
where $\down_i A$ is the $\le_i$-downset of $A$.

\begin{definition} \label{def: LU}
For a topological space $\U = (Z, \le_1, \le_2)$ with two quasi-orders, let
\[
\L\U = \{ C \in \Up(Z, \le_1) : C = \psi\phi C, \ C \textrm{ and }\phi C \textrm{ are closed} \}.
\] 
\end{definition}

Following \cite[p.~46]{Urq78}, we call the triple $(Z,\le_1,\le_2)$ a \emph{doubly ordered space} if 
\[
x \le_1 y \mbox{ and } x \le_2 y \mbox{ imply } x = y.
\]

\begin{definition}  \plabel{def: Urquhart space}
Let $\U = (Z, \le_1, \le_2)$ be a topological space with two quasi-orders. Then $\U$ is an \emph{Urquhart space} provided:
\begin{enumerate}
\item \label[def: Urquhart space]{Urq compact} $Z$ is compact and doubly ordered.
\item \label[def: Urquhart space]{Urq subbasis}The family $\L\U \cup \{  \phi C : C \in \L\U \}$ is a subbasis of closed sets for the topology on $Z$.
\item \label[def: Urquhart space]{Urq lattice ops} If $C_1, C_2 \in \L\U$, then $\phi(C_1 \cap C_2)$ and $\psi(\phi C_1 \cap \phi C_2)$ are closed.
\item \label[def: Urquhart space]{Urq dd 1} If $x \not\le_1 y$, then there is $C \in \L\U$ with $x \in C$ and $y \notin C$.
\item \label[def: Urquhart space]{Urq dd 2} If $x \not\le_2 y$, then there is $C \in \L\U$ with $x \in \phi C$ and $y \notin \phi C$.
\end{enumerate}
\end{definition}

\begin{remark} \label{rem: U is T1}
Urquhart calls a doubly ordered space doubly disconnected when it satisfies \crefrange{Urq dd 1}{Urq dd 2}. Observe that every Urquhart space is $T_1$. Indeed, if $x, y \in Z$ with $x \ne y$, then $x \not\le_1 y$ or $x \not\le_2 y$. If $x \not\le_1 y$, there is $C \in \L\U$ with $x \in C$ and $y \notin C$. Therefore, $-C$ is an open set containing $y$ but not $x$. If $x \not\le_2 y$, there is $C \in \L\U$ with $x \in \phi C$ and $y \notin C$. Thus, $-\phi C$ is an open set containing $y$ but not $x$. It follows that $Z$ is $T_1$.

Also observe that if $\U$ is an Urquhart space, then $\L\U$ is a bounded lattice, where top is $X$, bottom is $\varnothing$, $C_1 \wedge C_2 = C_1 \cap C_2$, and $C_1 \vee C_2 = \psi(\phi C_1 \cap \phi C_2)$ for each $C_1, C_2 \in \L\U$.
\end{remark}

We next define Urquhart morphisms. 

\begin{definition} \plabel{def: Urquhart morphisms}
Let $\U = (Z, \le_1, \le_2)$ and $\U' = (Z', \le_1', \le_2')$ be Urquhart spaces. Let also $\phi, \psi$ be the Galois connection maps for $\U$ and $\phi', \psi'$ for $\U'$. A pair $(P, Q)$ of relations $P, Q \subseteq Z \times Z'$ is called an \emph{Urquhart morphism} provided:
\begin{enumerate}
\item \label[def: Urquhart morphisms]{Urq Box}If $C' \in \L\U'$, then $\Box_PC' \in \L\U$ and $\phi\Box_P C' = \Box_Q \phi' C'$.
\item \label[def: Urquhart morphisms]{Urq not P} If $z \nr{P} z'$, then there is $C' \in \L\U'$ with $z \in \Box_PC'$ and $z' \notin C'$.
\item \label[def: Urquhart morphisms]{Urq not Q} If $z \nr{Q} z'$, then there is $C' \in \L\U'$ with $z \in \Box_Q\phi C'$ and $z' \notin \phi C'$.
\item \label[def: Urquhart morphisms]{Urq serial} If $z \in Z$, then $P[z] \cap Q[z] \ne \varnothing$.
\end{enumerate}
\end{definition}

Note that in \cref{Urq not Q},
\[
z \nr{Q} z' \iff \mbox{ there is } C' \in L\U' \mbox{ with } z \notin \Diamond_Q-\phi C' \mbox{ and } z' \in -\phi C',
\]
which is more in line with \cref{Hg morphism not T}.

Composition of Urquhart morphisms is defined similarly to that of GvG-morphisms and Hartung morphisms:

\begin{definition} \plabel{def: composition in Urq}
Let $(P_1, Q_1) \colon \U_1 \to \U_2$ and $(P_2, Q_2) \colon \U_2 \to \U_3$ be Urquhart morphisms. 
\begin{enumerate}
\item Define $P_2 \star P_1$ by $x\rel{ (P_2 \star P_1)} z$ if $x \in \Box_{P_1}\Box_{P_2}C$ implies $z \in C$ for each $C \in \L\U_3$.
\item \label[def: composition in Urq]{star of Q} Define $Q_2 \star Q_1$ by $y \rel{(Q_2 \star Q_1)}w$ if $y \in \Box_{Q_1}\Box_{Q_2}\phi C$ implies $w \in \phi C$ for each $C \in \L\U_3$. 
\item Set $(P_2, Q_2) \star (P_1, Q_1) = (P_2 \star P_1, Q_2 \star Q_1)$.
\end{enumerate}
\end{definition}

Observe that in \cref{star of Q},
\[
y \rel{(Q_2 \star Q_1)}w \iff 
(\forall C \in L\U_3)\, (w \in -\phi C \Longrightarrow y \in \Diamond_{Q_1}\Diamond_{Q_2}-\phi C),
\]
which is more in line with \cref{star of Ts}. We have the following analogue of \cref{rem: equivalent condition for GvG S} for Urquhart morphisms:

\begin{remark} \plabel{rem: equivalent condition for Q}
Let $(P,Q) \colon \U \to \U'$ be an Urquhart morphism. 
\begin{enumerate}
\item \label[rem: equivalent condition for Q]{eq condition for Q}  
$z \rel{P} z' \iff (\forall C \in \L\U')\, (z \in \Box_PC \Longrightarrow z' \in C)$.
\item $z \rel{Q} z' \iff (\forall C \in \L\U')\, (z \in \Box_Q\phi C \Longrightarrow z' \in \phi C)$.
\item \label[rem: equivalent condition for Q]{P[x] is an upset} $P[z]$ is a $\le_1$-upset and $Q[z]$ is a $\le_2$-upset of $Z'$ for each $z \in Z$.
\item \label[rem: equivalent condition for Q]{P inverse is a downset} $P^{-1}[z']$ is a $\le_1$-downset and $Q^{-1}[z']$ is a $\le_2$-downset of $Z$ for each $z' \in Z'$.
\end{enumerate}
\end{remark}

An argument similar to that of \cref{lem: composition lemma for GvG} yields: 
\begin{lemma} \label{lem: composition lemma for Urq}
Let $(P_1, Q_1) \colon \U_1 \to \U_2$ and $(P_2, Q_2) \colon \U_2 \to \U_3$ be Urquhart morphisms. If $C \in \L\U_3$, then
$
\Box_{P_2 \star P_1}C = \Box_{P_1}\Box_{P_2}C
$
and
$
\Box_{Q_2 \star Q_1}\phi C = \Box_{Q_1}\Box_{Q_2}\phi C.
$
\end{lemma}

An analogous proof to \cref{prop: composition in GvG} yields: 

\begin{proposition} \plabel{prop: composition in Urquhart}
Let $(P_1, Q_1) \colon \U_1 \to \U_2$ and $(P_2, Q_2) \colon \U_2 \to \U_3$ be Urquhart morphisms.
\begin{enumerate}
\item \label[prop: composition in Urquhart]{Urq morphism} $(P_2, Q_2) \star (P_1, Q_1)$ is an Urquhart morphism.
\item \label[prop: composition in Urquhart]{Urq identity} If $\U = (Z, \le_1, \le_2)$ is an Urquhart space, then $(\le_1, \le_2)$ is the identity morphism for~$\U$.
\item \label[prop: composition in Urquhart]{Urq associativity} The operation $\star$ is associative.
\end{enumerate}
\end{proposition}

We thus arrive at the following definition.

\begin{definition} \label{def: Urq}
Let $\Urq$ denote the category of Urquhart spaces and Urquhart morphisms.
\end{definition}

In the rest of the section we define a functor $\UU \colon \Hg \to \Urq$.

\begin{definition} \plabel{def: maximal pairs}
Let $\H = (X, R, Y)$ be a Hartung space.
\begin{enumerate}
\item Set $Z_\H = \{ (x, y) \in X \times Y : x \textrm{ is }y\textrm{-maximal and }y \textrm{ is }x\textrm{-maximal} \}$.
\item Topologize $Z_\H$ with the subspace topology of the product topology on $X \times Y$. 
\item Define $\le_1$ and $\le_2$ on $Z_\H$ by setting
\[
(x, y) \le_1 (x', y') \textrm{ if } x \le_X x' \quad \textrm{and} \quad (x, y) \le_2 (x', y') \textrm{ if } y \le_Y y'.
\]
\end{enumerate}
\end{definition}

\begin{remark} \plabel{rem: existence of maximal pairs}
Let $\H = (X, R, Y) \in \Hg$.
\begin{enumerate}
\item $Z_\H$ is a subset of $R$ and $\le_1$, $\le_2$ are quasi-orders on $Z_\H$.
\item \label[rem: existence of maximal pairs]{maximal pairs}\cref{doubly founded} shows that if $x \rel{R} y$, then there is $y'$ with $y \le_Y y'$ such that $y'$ is $x$-maximal. Applying it again yields $x'$ with $x \le_X x'$ such that $x'$ is $y'$-maximal. It then follows that $(x', y') \in Z_\H$. 

\item \label[rem: existence of maximal pairs]{maximal pairs 2} For each $x \in X$, there is $y' \in Y$ with $(x, y') \in Z_\H$. To see this, by \cref{Hg maximal}, there is $y \in Y$ such that $x$ is $y$-maximal. Since $x \rel{R} y$, \eqref{maximal pairs} yields $y' \ge_Y y$ such that $y'$ is $x$-maximal. To see that $x$ is $y'$-maximal, suppose $x \le_X x'$ with $x' \rel{R} y'$. Then $x' \rel{R} y$ by \cref{Hg leY}, so $x' = x$ as $x$ is $y$-maximal. Thus, $x$ is $y'$-maximal, and hence $(x, y') \in Z_\H$. A similar argument shows that for each $y \in Y$, there is $x' \in X$ with $(x', y) \in Z_\H$.
\end{enumerate}
\end{remark}

\begin{lemma} \plabel{lem: l and r}
For $\H = (X, R, Y) \in \Hg$ set $\U = (Z_\H, \le_1, \le_2)$. Suppose that $A \subseteq X$ and $B \subseteq Y$ are upsets.
\begin{enumerate}
\item \label[lem: l and r]{action of phi} $\phi( (A \times Y) \cap Z_\H) = (X \times -\bd A) \cap Z_\H$.
\item \label[lem: l and r]{action of psi} $\psi((X \times B)\cap Z_\H) = (-\Diamond B \times Y)\cap Z_\H$.
\item \label[lem: l and r]{action of psiphi} $\psi\phi((A \times Y)\cap Z_\H) = (\Box\bd A \times Y)\cap Z_\H$.
\end{enumerate}
\end{lemma}

\begin{proof}
(1) Set $C = (A \times Y)\cap Z_\H$ and let $(x, y) \in \phi C$. To show that $y \in -\bd A$, let $w \in A$. 
If $w \rel{R} y$, then by \cref{maximal pairs 2} there is $y_1 \ge_Y y$ with $(w, y_1) \in Z_\H$. Because $\phi C$ is a $\le_2$-upset and $(x, y) \le_2 (w, y_1)$, we have $(w, y_1) \in \phi C$. From $\phi C \cap C = \varnothing$, it follows that $(w, y_1) \notin C$, so $w \notin A$. This contradiction shows that $w \nr{R} y$. Therefore, $y \in -\bd A$, so $(x, y) \in (X \times -\bd A)\cap Z_\H$, and hence $\phi C \subseteq (X \times -\bd A)\cap Z_\H$.

For the reverse inclusion, suppose that $(x, y) \notin \phi C$. Because $-\phi C = \down_2 C$,  there is $(x', y') \in C$ with $(x, y) \le_2 (x', y')$. We then have $ y  \le_Y y'$ and $x' \in A$. Since $x' \rel{R} y'$ and $y \le_Y y'$, it follows that $x' \rel{R} y$ (see \cref{Hg leY}), so $y \in \bd A$. Consequently, $(x, y) \notin (X \times -\bd A)\cap Z_\H$.

(2) The proof is similar to (1). 

(3) By (1) and (2) applied to $B = -\bd A$, we have
\begin{align*}
\psi\phi((A \times Y)\cap Z_\H) &= \psi ((X \times -\bd A)\cap Z_\H) = (-\Diamond-\bd A \times Y)\cap Z_\H \\
&= (\Box\bd A  \times Y)\cap Z_\H.\qedhere
\end{align*}
\end{proof}

Let $(X, R, Y) \in \Hg$ and $\U = (Z_\H, \le_1, \le_2)$. To show that $\U$ is an Urquhart space, we require the following two results.

\begin{lemma} \label{lem: Urq Lem 5}
\cite[Lem.~5]{Urq78} Let $\U = (Z, \le_1, \le_2)$ be a compact doubly ordered space. Suppose that $\mathcal{S}$ is a bounded sublattice of $\L\U$ such that whenever $x \not\le_1 y$, there is $C \in \mathcal{S}$ with $x \in C$ and $y \notin C$. Then $\mathcal{S} = \L\U$.
\end{lemma}

\begin{proposition} \label{lem: LU = LH}
Let $\H = (X, R, Y) \in \Hg$ and $\U = (Z_\H, \le_1, \le_2)$. Then the map $\mu_\H \colon \L\H \to \L\U$, defined by $\mu_\H(A) = (A \times Y) \cap Z_\H$, is a $\Lat$-isomorphism. 
\end{proposition}

\begin{proof}
The map $\mu_\H$ is well defined by \cref{action of psiphi}, and it is clearly order-preserving. To see that it is order-reflecting, let $A_1, A_2 \in \L\H$ with $A_1 \not\subseteq A_2$. Then there is $x \in A_1 - A_2$. By \cref{maximal pairs 2}, there is $y \in Y$ with $(x, y) \in Z_\H$. Consequently, $(x, y) \in (A_1 \times Y) \cap Z_\H$ and $(x, y) \notin (A_2 \times Y) \cap Z_\H$. Thus, $\mu_\H(A_1) \not\subseteq \mu_\H(A_2)$.

To see that $\mu_\H$ is onto, let $\mathcal{S} = \{(A \times Y) \cap Z_\H : A \in \L\H\}$, so $\mathcal{S} \subseteq \L\U$ by \cref{action of psiphi}. By \cref{lem: Urq Lem 5}, to see that $\L\U = \mathcal{S}$, it suffices to show that $\mathcal{S}$ is a bounded sublattice of $\L\U$ such that if $(x_1, y_1), (x_2, y_2) \in Z_\H$ with $(x_1, y_1) \not\le_1 (x_2, y_2)$, then there is $C \in \mathcal{S}$ with $(x_1, y_1) \in C$ and $(x_2, y_2) \notin C$. First, since $\varnothing, X \in \L\H$, we have $\varnothing = (\varnothing \times Y)\cap Z_\H$ and $Z_\H = (X \times Y)\cap Z_\H$ are both in $\mathcal{S}$. Next, let $C_1, C_2 \in \mathcal{S}$, so $C_1 = (A_1 \times Y)\cap Z_\H$ and $C_2 = (A_2 \times Y)\cap Z_\H$ for some $A_1, A_2 \in \L\H$. Then
\[
C_1 \wedge C_2 = C_1 \cap C_2 = ((A_1 \cap A_2) \times Y)\cap Z_\H \in \mathcal{S}
\]
since $A_1 \cap A_2 \in \L\H$. Also, by \cref{rem: U is T1,lem: l and r},
\begin{align}
\begin{split}
C_1 \vee C_2 &= \psi(\phi C_1 \cap \phi C_2) = \psi((X \times -\bd A_1)\cap Z_\H \cap (X \times -\bd A_2) \cap Z_\H) \\
&= \psi((X \times (-\bd A_1 \cap -\bd A_2))\cap Z_\H) = \psi((X \times -\bd (A_1 \cup A_2))\cap Z_\H)\\
&= (-\Diamond-\bd( A_1 \cup A_2) \times Y) \cap Z_\H = (\Box\bd(A_1 \cup A_2) \times Y)\cap Z_\H\in \mathcal{S}\label{eqn: C1 vee C2}
\end{split}
\end{align}
since $\Box\bd(A_1 \cup A_2) \in \L\H$. Finally, suppose that $(x_1, y_1) \not\le_1 (x_2, y_2)$. Then $x_1 \not\le_X x_2$. Therefore, by \cref{Hg not R related}, there is $A \in \L\H$ with $x_1 \in A$ and $x_2 \notin A$. If $C = (A \times Y)\cap Z_\H$, then $C \in \mathcal{S}$ and $(x_1, y_1) \in C$ but $(x_2, y_2) \notin C$. Thus, $\L\H = \mathcal{S}$, and hence $\mu_\H$ is onto. Consequently, $\mu_\H$ is an order-isomorphism, and so a $\Lat$-isomorphism.
\end{proof}

\begin{proposition} \label{thm: U well defined on objects}
If $\H = (X, R, Y) \in \Hg$, then $\U = (Z_\H, \le_1, \le_2) \in \Urq$.
\end{proposition}

\begin{proof}
To verify \cref{Urq compact}, $R$ is a compact subset of $X \times Y$ by \cref{Hg compact}. To see that $Z_\H$ is compact, by \cite[p.~43]{GHKLMS03} it suffices to show that $R = \up Z_\H$, where $\up$ is calculated with respect to the product order on $X \times Y$. If $(x, y) \in R$, by \cref{maximal pairs} there is $(x', y') \in Z_\H$ with $x\le_X x'$ and $y \le_Y y'$. Therefore, $(x, y) \le (x', y')$, and hence $R = \up Z_\H$.

To see that $\U$ is doubly ordered, suppose that $(x, y) \le_1 (x', y')$ and $(x, y) \le_2 (x', y')$. Then $x \le_X x'$ and $y \le_Y y'$. We have $x' \rel{R} y'$, so $x' \rel{R} y$ by \cref{Hg leY}. Because $x$ is $y$-maximal, $x' = x$. Similarly, $y' = y$.

Next, to verify \cref{Urq subbasis}, the topology on $Z_\H$ is the subspace topology of the product topology of $X \times Y$. \cref{Hg LH subbasis} shows that
\[
\{ A \times Y : A \in \L\H \} \cup \{ X \times -\bd A : A \in \L\H\}
\]
is a subbasis of closed sets for $X \times Y$. We have $\L\U = \{(A \times Y)\cap Z_\H : A \in \L\H \}$ by \cref{lem: LU = LH}, and
\[
\{ \phi C : C \in \L\U \} = \{ (X \times -\bd A)\cap Z_\H : A \in \L\H\} 
\]
by Lemmas~\ref{lem: l and r}\eqref{action of phi} and \ref{lem: LU = LH}. Thus, $\L\U \cup \{ \phi C : C \in \L\U\}$ is a subbasis of closed sets for $Z_\H$.

To verify \cref{Urq lattice ops}, let $C_1, C_2 \in \L\U$, and write $C_i = (A_i \times Y)\cap Z_\H$ for $i = 1, 2$. By \cref{action of phi},
\[
\phi(C_1 \cap C_2) = \phi(((A_1 \cap A_2) \times Y)\cap Z_\H) = (X \times -\bd(A_1 \cap A_2))\cap Z_\H,
\]
so is closed since $A_1 \cap A_2 \in \L\H$, and hence $-\bd(A_1 \cap A_2)$ is closed in $Y$ by \cref{Hg closed}. Furthermore, by \cref{eqn: C1 vee C2}, $\psi(\phi C_1 \cap \phi C_2) = (\Box\bd(A_1 \cup A_2 ) \times Y)\cap Z_\H$, so is closed because $\Box\bd(A_1 \cup A_2) \in \L\H$. 

We next verify \cref{Urq dd 1}. Suppose $(x_1, y_1), (x_2, y_2) \in Z_\H$ with $(x_1, y_1) \not\le_1 (x_2, y_2)$. Then $x_1 \not\le_X x_2$. By \cref{Hg up x is intersection of As}, there is $A \in \L\H$ with $x_1 \in A$ and $x_2 \notin A$. Letting $C = (A \times Y)\cap Z_\H$, we have $C \in \L\U$ (see \cref{lem: LU = LH}), $(x_1, y_1) \in C$, and $(x_2, y_2) \notin C$. 

Finally, to verify \cref{Urq dd 2}, if $(x_1, y_1) \not\le_2 (x_2, y_2)$, then $y_1 \not\le_Y y_2$. By \cref{Hg up y is intersection of diamond As}, there is $A \in \L\H$ with $y_1 \notin \bd A$ and $y_2 \in \bd A$. Letting $C = (A \times Y)\cap Z_\H$, we have that $C \in \L\U$, $(x_1, y_1) \in \phi C$, and $(x_2, y_2) \notin \phi C$. This completes the proof that $\U \in \Urq$.
\end{proof}

\begin{theorem} \label{thm: U is a functor}
There is a functor $\UU \colon \Hg \to \Urq$ which sends $\H = (X, R, Y)$ to $\U = (Z_\H, \le_1, \le_2)$ and $(S, T) \colon \H \to \H'$ to $(P, Q) \colon \U \to \U'$, where
\[
(x, y) \rel{P} (x', y') \iff x\rel{S}x' \quad\textrm{and}\quad (x, y) \rel{Q} (x', y') \iff y\rel{T}y'.
\]
\end{theorem}

\begin{proof}
If $\H \in \Hg$, then $\U \in \Urq$ by \cref{thm: U well defined on objects}. Let $(S, T) \colon \H \to \H'$ be an $\Hg$-morphism. 
We show that $(P, Q) \colon \U \to \U'$ is an $\Urq$-morphism.
To verify \cref{Urq Box}, let $C' \in \L\U'$. By \cref{lem: LU = LH}, $C' = (A' \times Y') \cap Z_{\H'}$ for some $A'\in \L\H'$. We have 
\[
(x, y) \in \Box_{P}C' \iff P[(x, y)] \subseteq C' \iff S[x] \subseteq A' \iff x \in \Box_S A'.
\]
Therefore,
\begin{equation}
\Box_PC' = (\Box_SA' \times Y)\cap Z_\H \in \L\U \label{eqn: BoxP calculation}
\end{equation}
by \cref{Hg morphism Box,lem: LU = LH}.

For the second statement of \cref{Urq Box}, we have $\phi' C' = (X' \times -\bd' A')\cap Z_{\H'}$ by \cref{action of phi}. Moreover, 
\begin{align*}
(x, y) \in \Box_{Q}\phi' C' &\Longleftrightarrow Q[(x, y)] \subseteq \phi' C' \Longleftrightarrow T[y] \subseteq -\bd' A' \\
&\Longleftrightarrow y \in \Box_T-\bd' A'  \Longleftrightarrow y \in -\Diamond_T\bd' A'\\
&\Longleftrightarrow (x, y) \in (X \times -\Diamond_T \bd' A')\cap Z_\H.
\end{align*}
Therefore,
\begin{align}
\begin{split}
\Box_Q \phi' C' &= (X \times -\Diamond_T \bd' A')\cap Z_\H =  (X \times -\bd \Box_S A')\cap Z_\H\\
&= \phi( (\Box_S A' \times Y)\cap Z_\H) = \phi \Box_P C', 
\end{split}\label{box Q display}
\end{align}
where the second equality holds by \cref{Hg morphism Box}, the third equality by \cref{action of phi}, and the final equality by \cref{eqn: BoxP calculation}.

For \cref{Urq not P}, suppose $(x, y) \nr{P} (x', y')$. Then $x \nr{S} x'$. By \cref{Hg morphism not S}, there is $A' \in \L\H'$ with $x \in \Box_SA'$ and $x' \notin A'$. Set $C' = (A' \times Y')\cap Z_{\H'}$. Then $C' \in \L\U'$ by \cref{lem: LU = LH}, and $(x', y') \notin C'$. Moreover,
\begin{align*}
P[(x, y)] &= \{ (x_1, y_1) \in Z_{\H'} : (x, y) \rel{P} (x_1, y_1) \} = \{ (x_1, y_1) \in Z_{\H'} : x \rel{S} x_1 \} \\
&= \{ (x_1, y_1) \in Z_{\H'} : x_1 \in S[x] \} \subseteq (A' \times Y')\cap Z_{\H'} =  C',
\end{align*}
so $(x,y) \in \Box_P C'$.

To verify \cref{Urq not Q}, suppose $(x, y) \nr{Q} (x', y')$. Then $y \nr{T} y'$, so by \cref{Hg morphism not T}, there is $A' \in \L\H'$ with $y \notin \Diamond_T \bd' A'$ and $y' \in \bd' A'$. Set $C' = (A' \times Y')\cap Z_{\H'}$. Then $\phi' C' = (X' \times -\bd' A')\cap Z_{\H'}$ by \cref{action of phi}. Therefore, $(x', y') \notin \phi' C'$, and $(x, y) \in \Box_Q\phi' C'$ by \cref{box Q display}.

Finally, to verify \cref{Urq serial}, let $(x, y) \in Z_\H$. Then $x \rel{R} y$ and by \cref{Hg morphism R}, there are $x' \in X'$, $y' \in Y'$ with $x \rel{S} x'$, $y \rel{T} y'$, and $x' \rel{R'} y'$. By \cref{maximal pairs}, there is $(x'', y'') \in Z_{\H'}$ with $x' \le_X x''$ and $y' \le_Y y''$. By \cref{S[x] an upset}, $x \rel{S} x''$ and $y \rel{T} y''$. Therefore, $(x'', y'') \in P[(x,y)] \cap Q[(x,y)]$. Thus, $(P, Q)$ is an Urquhart morphism.

We next show that $\UU$ sends identity morphisms to identity morphisms. Let $\H \in \Hg$. By \cref{Hg identity}, the identity morphism for $\H$ is $(\le_X, \le_Y)$. If $(P, Q) = \UU(\le_X, \le_Y)$, then $(x, y) \rel{P} (x', y')$ iff $x \le_X x'$ and $(x, y) \rel{Q} (x', y')$ iff $y \le_Y y'$. Therefore, $\UU(\le_X, \le_Y) = (\le_1, \le_2)$, which  by \cref{Urq identity} is the identity morphism for $\UU(\H)$.

To finish the proof, we show that $\UU$ preserves composition. Consider $\Hg$-morphisms $(S_1, T_1) \colon \H_1 \to \H_2$ and $(S_2, T_2) \colon \H_2 \to \H_3$. Set $(S, T) = (S_2, T_2) \star (S_1, T_1)$, so $S = S_2 \star S_1$ and $T = T_2 \star T_1$. If $(x, y) \in Z_1$ and $(x', y') \in Z_3$, then
\[
(x, y) \rel{P} (x', y') \iff x \rel{S} x' \iff x \rel{(S_2 \star S_1)} x'. 
\]
Also, 
\[(x, y) \rel{(P_2 \star P_1)} (x', y') \iff (\forall C \in \L\U_3)\, \left( (x, y) \in \Box_{P_1}\Box_{P_2} C  \Longrightarrow (x', y') \in C\right). 
\]
Given such $C$ there is $A \in \L\H_3$ with $C = (A \times Y_3) \cap Z_3$. By \cref{eqn: BoxP calculation},
\[
\Box_{P_1}\Box_{P_2} C = \Box_{P_1} ((\Box_{S_2}A \times Y_2) \cap Z_2) = (\Box_{S_1}\Box_{S_2}A \times Y_1) \cap Z_1.
\]
Therefore, 
\begin{align*}
(x, y) \rel{(P_2 \star P_1)} (x', y') &\iff (\forall A \in \L\H_3)\, \left(x \in \Box_{S_1}\Box_{S_2} A \Longrightarrow x' \in A\right)\\
&\iff x \rel{(S_2\star S_1)} x'.
\end{align*}
A similar calculation shows that 
\[
(x, y) \rel{(Q_2 \star Q_1)} (x', y') \iff y \rel{(T_2 \star T_1)} y'.
\]
Thus, $(P, Q) = (P_2, Q_2) \star (P_1, Q_1)$, which shows that $\UU$ preserves composition.
\end{proof}


\section{\Plos~spaces: an alternative approach to Urquhart spaces} \label{sec: Urq and Plo}

In \cite{Plo95}, \Plos\ showed that the two quasi-orders of an Urquhart space can be replaced by a reflexive relation. In this section we introduce the category $\Plo$ of \Plos~spaces and \Plos\ morphisms and show that it is isomorphic to $\Urq$.
Following \cref{sec: DH and GvG}, for a relation $R$, we write $\Box, \bd$ instead of $\Box_R, \bd_R$ when the context is clear.

\begin{definition} \plabel{def: LP}
For the pair $ \Pl = (Z, R)$, where $Z$ is a topological space and $R$ is a relation on $Z$, we set
\[
\L\Pl = \{ C \subseteq Z : C = \Box\bd C \textrm{ and } C, -\bd C \textrm{ are closed} \}. 
\]
\end{definition}

\begin{definition} \plabel{def: Plo space}
Let $\Pl = (Z, R)$ be a topological space with a reflexive relation $R$. We call $\Pl$ a \emph{\Plos~space} provided:
\begin{enumerate}
\item \label[def: Plo space]{Plo compact} $Z$ is compact and $T_1$.
\item \label[def: Plo space]{Plo subbasis} The family $\L\Pl \cup \{  -\bd C : C \in \L\Pl \}$ is a subbasis of closed sets for the topology on $Z$.
\item \label[def: Plo space]{Plo lattice ops} If $C_1, C_2 \in \L\Pl$, then $-\bd(C_1 \cap C_2)$ and $\Box\bd(C_1 \cup C_2)$ are closed.
\item \label[def: Plo space]{Plo dd 1} If $R[y] \not\subseteq R[x]$, then there is $C \in \L\Pl$ with $x \in C$ and $y \notin C$.
\item \label[def: Plo space]{Plo dd 2} If $R^{-1}[y] \not\subseteq R^{-1}[x]$, then there is $C \in \L\Pl$ with $x \notin \bd C$ and $y \in \bd C$.
\item \label[def: Plo space]{Plo R from le} If $x \rel{R} y$, then there is $z \in Z$ with $R[z] \subseteq R[x]$ and $R^{-1}[z] \subseteq R^{-1}[y]$.
\end{enumerate}
\end{definition}

\begin{remark} \plabel{rem: basic Plo facts}
\label[rem: basic Plo facts]{LP a lattice} \cref{Plo lattice ops} implies that $\L\Pl$ is a bounded lattice whose top is $X$, bottom is $\varnothing$, $C_1 \wedge C_2 = C_1 \cap C_2$, and $C_1 \vee C_2 = \Box\bd(C_1 \cup C_2)$ for each $C_1,C_2\in\L\Pl$.
\end{remark}

\begin{definition} \plabel{def: Plo morphisms}
Let $\Pl = (Z, R)$ and $\Pl' = (Z', R')$ be \Plos~spaces. We call a pair $(P, Q)$ of relations $P, Q \subseteq Z \times Z'$ a \emph{\Plos~morphism} provided:
\begin{enumerate}
\item \label[def: Plo morphisms]{Plo Box} If $C' \in \L\Pl'$, then $\Box_PC' \in \L\Pl$ and $\bd\Box_P C' = \Diamond_Q \bd'C'$.
\item \label[def: Plo morphisms]{Plo not P} If $z \nr{P} z'$, then there is $C' \in \L\Pl'$ with $z \in \Box_PC'$ and $z' \notin C'$.
\item \label[def: Plo morphisms]{Plo not Q} If $z \nr{Q} z'$, then there is $C' \in \L\Pl'$ with $z \notin\Diamond_Q \bd' C'$ and $z' \in \bd' C'$.
\item \label[def: Plo morphisms]{Plo serial} If $z \in Z$, then $P[z] \cap Q[z] \ne \varnothing$.
\end{enumerate}
\end{definition}

We define composition of \Plos~morphisms exactly as for $\Urq$-morphisms.

\begin{definition} \label{def: composition in Plo}
Let $(P_1, Q_1) \colon \Pl_1 \to \Pl_2$ and $(P_2, Q_2) \colon \Pl_2 \to \Pl_3$ be \Plos~morphisms. 
\begin{enumerate}
\item Define $P_2 \star P_1$ by $x\rel{ (P_2 \star P_1)} z$ provided that $x \in \Box_{P_1}\Box_{P_2}C$ implies $z \in C$ for each $C \in \L\Pl_3$.
\item  Define $Q_2 \star Q_1$ by $y \rel{(Q_2 \star Q_1)}w$ provided that $y \in \Box_{Q_1}\Box_{Q_2}-\bd C$ implies $w \in -\bd C$ for each $C \in \L\Pl_3$. 
\item Set $(P_2, Q_2) \star (P_1, Q_1) = (P_2 \star P_1, Q_2 \star Q_1)$.
\end{enumerate}
\end{definition}

\begin{definition} \label{def: quasi orders from R}
For a \Plos~space $\Pl=(Z, R)$, define $\le_1$ and $\le_2$ on $Z$ by 
\[
x \le_1 z \iff R[z] \subseteq R[x] \quad\mbox{and}\quad y \le_2 z \iff R^{-1}[z] \subseteq R^{-1}[y].
\]
\end{definition}

\needspace{3\baselineskip}
\begin{remark} \plabel{Plo alternative axiom}
\hfill
\begin{enumerate}
\item \label[Plo alternative axiom]{Plo Axiom 6} By the above definition, we can phrase \cref{Plo R from le} as follows: 
\[
\mbox{If $x \rel{R} y$, then there is $z$ with $x \le_1 z$ and $y \le_2 z$}.
\]
\item \label[Plo alternative axiom]{C in LP is an upset} If $C \in \L\Pl$, then $C$ is a $\le_1$-upset and $\bd C$ is a $\le_2$-downset.
\end{enumerate}
\end{remark}

The same argument as in \cref{prop: composition in Urquhart} (which is analogous to that in \cref{prop: composition in GvG}) yields:

\begin{proposition} \label{prop: composition in Ploscica}
The \Plos~spaces and \Plos~morphisms form a category $\Plo$ where composition is given by $\star$ and the identity morphism for $\Pl$ is $(\le_1, \le_2)$.
\end{proposition}

We next define a functor $\PP \colon \Urq \to \Plo$. 

\begin{definition} \label{def: R from quasi-orders}
Let $\U = (Z, \le_1, \le_2)$ be a topological space with two quasi-orders. Define $R_\U \subseteq Z \times Z$ by 
\[
x \rel{R_\U}y \iff \exists z \in Z : x \le_1 z \mbox{ and } y \le _2 z.
\]
\end{definition}

\begin{remark} \label{rem: Plo R[x] is a downset}
It is immediate from the above definition that 
$R_\U = {\ge_2 \circ \leq_1}$, and hence $R_\U[C]= \down_2 \up_1 C$ and $R_\U^{-1}[C]=\down_1 \up_2 C$ for each $C \subseteq X$.
\end{remark}

As follows from \cite[Thm.~2.1]{Plo95}, we can recover $\le_1$ and $\le_2$ from $R_\U$:

\begin{lemma} \plabel{lem: relation between R and quasi orders}
Let $\U = (Z, \le_1, \le_2)$ be an Urquhart space and $x, y \in Z$.
\begin{enumerate}
\item \label[lem: relation between R and quasi orders]{RU from le1} $x \le_1 y$ iff  $R_\U[y] \subseteq R_\U[x]$. 
\item \label[lem: relation between R and quasi orders]{RU from le2}$x \le_2 y$ iff $R_\U^{-1}[y] \subseteq R_\U^{-1}[x]$. 
\end{enumerate}
\end{lemma}

\begin{proof}
\eqref{RU from le1} 
Suppose that $x \le_1 y$. Then $\up_1 y \subseteq \up_1x$, so
$
R_\U[y] = \down_2 \up_1 y \subseteq \down_2 \up_1 x = R_\U[x].
$
Conversely, suppose that $x \not\le_1 y$. By \cref{Urq dd 1}, there is $C \in \L\U$ with $x \in C$ and $y \notin C$. Because $C = \psi\phi C = -\down_1 \phi C$, there is $z \in \phi C$ with $y \le_1 z$. On the one hand, $z \in \up_1 y \subseteq \down_2 \up_1 y = R_\U[y]$. On the other hand, since $z \in \phi C = -\down_2 C$ and $\up_1 x \subseteq C$, we have $z \notin \down_2\up_1 x = R_\U[x]$. Thus, $R_\U[y]\not\subseteq R_\U[x]$.

\eqref{RU from le2} The proof is similar to that of \eqref{RU from le1}.
\end{proof}

If $\U \in \Urq$, it is clear that $R_\U$ is a reflexive relation, and that $\le_1, \le_2$ are definable in terms of $R_\U$ by \cref{lem: relation between R and quasi orders}. 

\begin{proposition} \plabel{prop: relation between Urq and Plo Galois correspondences}
Let $Z$ be a set, $R$ a relation on $Z$, and $\le_1, \le_2$ quasi-orders on $Z$ such that for all $x, y \in Z$,
\[
x \rel{R} y \iff \exists z \in Z : x \le_1 z \ and \ y \le_2 z.
\]
Let also $\phi, \psi$ be the Galois connection maps defined in Section~\emph{\ref{sec: Hg and Urq}}, $C$ a $\le_1$-upset, and $D$ a $\le_2$-upset. 
\begin{enumerate}
\item \label[prop: relation between Urq and Plo Galois correspondences]{phi} $\phi C = -\bd C$.
\item \label[prop: relation between Urq and Plo Galois correspondences]{psi} $\psi D = \Box-D$.
\item \label[prop: relation between Urq and Plo Galois correspondences]{psi phi} $\psi\phi C = \Box \bd C$.
\end{enumerate}
\end{proposition}

\begin{proof}
\eqref{phi} By \cref{rem: Plo R[x] is a downset}, 
$
\bd C = R[C] = \down_2 \up_1 C = \down_2 C.
$
Thus,
$
-\bd C = - \down_2 C = \phi C.
$

\eqref{psi} The proof is similar to that of \eqref{phi}.

\eqref{psi phi} Let $C$ be a $\le_1$-upset. Then $-\bd C$ is a $\le_2$-upset by \eqref{phi}. Consequently, by \eqref{phi} and \eqref{psi} applied to $D = -\bd C$, we have
$
\psi\phi C = \psi(-\bd C) = \Box \bd C.
$
\end{proof}

\begin{lemma} \plabel{lem: LP = LU}
\hfill
\begin{enumerate}
\item \label[lem: LP = LU]{LP} Let $\U = (Z, \le_1, \le_2) \in \Urq$ and $\Pl = (Z, R_\U)$. Then $\L\Pl = \L\U$.
\item \label[lem: LP = LU]{LU} Let $\Pl = (Z, R) \in \Plo$, and define $\le_1, \le_2$ as in Definition~\emph{\ref{def: quasi orders from R}}. If $\U = (Z, \le_1, \le_2)$, then $\L\U = \L\Pl$.
\end{enumerate}
\end{lemma}

\begin{proof}
\eqref{LP} Let $C \in \L\U$. By definition, $C, \phi C$ are closed and $\psi\phi C = C$. Therefore, $C \in \L\Pl$ by \cref{prop: relation between Urq and Plo Galois correspondences}. Conversely, if $C \in \L\Pl$, then $C, -\bd C$ are closed and $\Box\bd C = C$. \cref{rem: Plo R[x] is a downset} shows that $C$ is a $\le_1$-upset, so \cref{prop: relation between Urq and Plo Galois correspondences} yields that $\phi C$ is closed and $\psi\phi C = C$. Thus, $C \in \L\U$, and hence $\L\Pl = \L\U$.

\eqref{LU} The proof is similar to that of \eqref{LP}.
\end{proof}

\begin{proposition} \label{prop: PP on objects}
If $\U = (Z, \le_1, \le_2) \in \Urq$, then $\Pl := (Z, R_\U) \in \Plo$.
\end{proposition}

\begin{proof}
To verify \cref{Plo compact}, the space $Z$ is compact by definition, and it is $T_1$ by \cref{rem: U is T1}. \cref{LP} shows that $\L\Pl = \L\U$, and so Definition~\ref{def: Plo space}(\ref{Plo subbasis}-\ref{Plo dd 2}) follows from \cref{prop: relation between Urq and Plo Galois correspondences} because $\U \in \Urq$ (see \cref{def: Urquhart space}).  To verify \cref{Plo R from le}, suppose that $x \rel{R_\U} y$. Then there is $z \in Z$ with $x \le_1 z$ and $y \le_2 z$. By \cref{lem: relation between R and quasi orders}, $R_\U[z] \subseteq R_\U[x]$ and $R_\U^{-1}[z] \subseteq R_\U^{-1}[y]$. Thus, $\Pl \in \Plo$.
\end{proof}

\begin{proposition} \label{prop: PP on morphisms}
If $(P, Q) \colon \U \to \U'$ is an $\Urq$-morphism, then $(P,Q) \colon \Pl \to \Pl'$ is a $\Plo$-morphism, where $\Pl = (Z,R_\U)$ and $\Pl'=(Z',R_{\U'})$.
\end{proposition}

\begin{proof}
By \cref{LP}, $\L\Pl = \L\U$. Therefore, by \cref{prop: relation between Urq and Plo Galois correspondences}, \cref{def: Urquhart morphisms}(i) implies \cref{def: Plo morphisms}(i) for each $1 \le i \le 4$. Consequently, $(P, Q)$ is a $\Plo$-morphism.
\end{proof}

\begin{theorem} \label{prop: PP is a functor}
There is a functor $\PP \colon \Urq \to \Plo$ given by $\PP(\U) = (Z, R_\U)$ for each $\U \in \Urq$, and $\PP(P,Q) = (P,Q)$ for each $\Urq$-morphism $(P, Q)$. 
\end{theorem}

\begin{proof}
\cref{prop: PP on objects} shows that $\PP$ is well defined on objects, and \cref{prop: composition in Ploscica} that it is well defined on morphisms. It is clear that $\PP$ preserves composition and identity morphisms.  Thus, $\PP$ is a functor.
\end{proof}

We now show that $\PP$ is an isomorphism of categories. For this we require the following lemma.

\begin{lemma} \label{lem: R from quasi-orders and vice-versa}
Let $\Pl = (Z, R) \in \Plo$. If $x, y \in Z$, then $x \rel{R} y$ iff there is $z \in Z$ with $x \le_1 z$ and $y \le_2 z$.
\end{lemma}

\begin{proof}
One direction follows from \cref{Plo Axiom 6}. For the other direction, suppose that there is $z \in Z$ with $x \le_1 z$ and $y \le_2 z$. From $x \le_1 z$ we get $R[z] \subseteq R[x]$. Because $R$ is reflexive, $z \in R[z]$, so $z \in R[x]$. Therefore, $x \rel{R} z$, and hence $x \in R^{-1}[z]$. Because $y \le_2 z$, we have $R^{-1}[z] \subseteq R^{-1}[y]$. Consequently, $x \in R^{-1}[y]$, yielding $x \rel{R} y$.
\end{proof}

\begin{theorem}
$\PP \colon \Urq \to \Plo$ is an isomorphism of categories.
\end{theorem}

\begin{proof}
By \cite[p.~14]{Mac71}, to see that $\PP$ is an isomorphism of categories, it suffices to show that $\PP$ is a bijection on objects and on morphisms. If $\U = (Z, \le_1, \le_2)$, then $\le_1, \le_2$ are determined by $R_\U$ by \cref{def: quasi orders from R,lem: relation between R and quasi orders}. This shows that $\PP$ is 1-1 on objects. To see that $\PP$ is onto on objects, let $(Z, R) \in \Plo$, and set $\U = (Z, \le_1, \le_2)$, where $\le_1, \le_2$ are given in \cref{def: quasi orders from R}. We show that $\U \in \Urq$. To see that $\U$ is a doubly ordered space, suppose that $x, y \in Z$ with $x \le_1 y$ and $x \le_2 y$. If $x \ne y$, since $Z$ is $T_1$, there is 
a closed set $C$ containing $x$ but not $y$. We may assume that $C$ is a subbasic closed. Therefore, \cref{Urq subbasis} yields that either $C \in \L\Pl$ or $C = -\bd D$ for some $D \in \L\Pl$. If $C \in \L\Pl$, then $C$ is a $\le_1$-upset by \cref{C in LP is an upset}; and $C$ contains $x$ but not $y$, a contradiction to $x \le_1 y$. If $C = -\bd D$, then $C$ is a $\le_2$-upset, 
again yielding a contradiction since $x \le_2 y$. Therefore, $x = y$. This shows that $\U$ is a doubly ordered space. By \cref{LU}, $\L\U = \L\Pl$. Thus,
\cref{def: Plo space,prop: relation between Urq and Plo Galois correspondences} show 
that $\U \in \Urq$. If $x, y \in Z$, then by \cref{lem: R from quasi-orders and vice-versa}, $x \rel{R} y$ iff there is $z \in Z$ with $x \le_1 z$ and $y \le_2 z$, which happens iff $x \rel{R_\U} y$. Consequently, $R = R_\U$, so $\Pl = \PP(\U)$, and hence $\PP$ is onto on objects.

By \cref{prop: PP on morphisms}, if $(P,Q) \colon \U \to \U'$ is an $\Urq$-morphism, then $(P, Q) \colon \Pl \to \Pl'$ is a $\Plo$-morphism. Conversely, if 
$(P, Q) \colon \Pl \to \Pl'$ is a $\Plo$-morphism, then \cref{LU,prop: relation between Urq and Plo Galois correspondences} show that $(P, Q) \colon \U \to \U'$ is an Urquhart morphism, so $\PP$ is one-to-one and onto on morphisms, completing the proof that $\PP$ is an isomorphism of categories.
\end{proof}


\section{From Urquhart back to Dunn-Hartonas} \label{sec: Urq and DH}

In \cref{sec: DH and GvG,sec: GvG and Hg,sec: Hg and Urq} we defined the functors
\[
\begin{tikzcd}
\HD \arrow[r, "\GG"] & \GvG \arrow[r, "\HH"] & \Hg \arrow[r, "\UU"] & \Urq
\end{tikzcd}
\]
and in \cref{sec: Urq and Plo} we saw that $\PP \colon \Urq \to \Plo$ is an isomorphism.
In this section we complete the circle by defining a functor $\DD \colon \Urq \to \HD$.

\begin{definition}
Let $\U = (Z, \le_1, \le_2) \in \Urq$. Set
\[
\LC\U = \{ D \in \Up(Z, \le_1) : D \textrm{ is closed, } \psi\phi D = D\}
\]
and
\[
\RC\U = \{ E \in \Up(Z, \le_2) : E \textrm{ is closed, } \phi\psi E = E\}.
\]
We order $\LC\U$ and $\RC\U$ by reverse inclusion.
\end{definition}

\begin{remark}
In \cite[Thm.~3]{Urq78}, elements of $\LC\U$ are called left-closed stable sets and elements of $\RC\U$ right-closed stable sets, hence the notation.
\end{remark}
 
We show that $\LC\U$ and $\RC\U$ are coherent lattices. For this we require the following lemma.

\begin{lemma} \plabel{lem: LC and RC facts}
Let $\U = (Z, \le_1, \le_2) \in \Urq$.
\begin{enumerate}
\item \label[lem: LC and RC facts]{Urq D in C} If $D \in \LC\U$ and $C \in \L\U$, then $D \subseteq C$ iff $D \cap \phi C = \varnothing$.
\item \label[lem: LC and RC facts]{Urq E in phi C} If $E \in \RC\U$ and $C \in \L\U$, then $E \subseteq \phi C$ iff $C \cap E = \varnothing$.
\item \label[lem: LC and RC facts]{Urq finite}  If $\mathcal{S} \subseteq \LC\U$ and $C \in \L\U$ with $\bigcap \mathcal{S} \subseteq C$, then there are $D_1, \dots, D_n \in \mathcal{S}$ with $D_1 \cap \dots \cap D_n \subseteq C$.
\item \label[lem: LC and RC facts]{Urq finite 2}  If $\mathcal{S} \subseteq \RC\U$ and $C \in \L\U$ with $\bigcap \mathcal{S} \subseteq \phi C$, then there are $E_1, \dots, E_n \in \mathcal{S}$ with $E_1 \cap \dots \cap E_n \subseteq \phi C$. 
\item \label[lem: LC and RC facts]{D is intersection} If $D \in \LC\U$, then $D = \bigcap \{ C \in \L\U : D \subseteq C\}$.
\item \label[lem: LC and RC facts]{E is intersection} If $E \in \RC\U$, then $E = \bigcap \{ \phi C : C \in \L\U, E \subseteq \phi C\}$.
\end{enumerate}
\end{lemma}

\begin{proof}
\eqref{Urq D in C} This follows from \cite[Thm.~2b]{Urq78}. For the reader's convenience, we give a short proof. If $D \subseteq C$, then $D \cap \phi C \subseteq C \cap \phi C = \varnothing$.  Conversely, suppose that $D \cap \phi C = \varnothing$. Then $D \subseteq -\phi C = \down_2 C$. Therefore, $\phi(\down_2 C) \subseteq \phi D$. But $\phi(\down_2 C) = -\down_2 \down_2 C = -\down_2 C = \phi C$. Thus, $\phi C \subseteq \phi D$, so $\psi\phi D \subseteq \psi\phi C$, and hence $D \subseteq C$ since $C\in \L\U$ and $D \in \LC\U$.

\eqref{Urq E in phi C} The proof is similar to that of \eqref{Urq D in C}.

\eqref{Urq finite} If $\bigcap \mathcal{S} \subseteq C$, then $\bigcap \mathcal{S} \subseteq -\phi C$ because $C \cap \phi C = \varnothing$. Since $-\phi C$ is open, compactness implies that there are $D_1,\dots, D_n \in \mathcal{S}$ with $D_1 \cap \dots \cap D_n \subseteq -\phi C$. Since $\LC\U$ is closed under finite intersections, \eqref{Urq D in C} shows that $D_1 \cap \dots \cap D_n \subseteq C$. 

\eqref{Urq finite 2} The proof is similar to that of \eqref{Urq finite}.

\eqref{D is intersection} The inclusion $\subseteq $ is obvious. For the reverse inclusion, suppose that $x \notin D$. Since $D = \psi\phi D$, we have $x \in \down_1 \phi D = \down_1 (-\down_2 D)$. Therefore, $x \le_1 y$ for some $y \in -\down_2 D$. Consequently, $y \not\le_2 z$ for each $z \in D$. By \cref{Urq dd 2}, there is $C_z \in \L\U$ with ${y \in \phi C_z}$ and $z \notin \phi C_z$. Because $D \subseteq \bigcup \{ -\phi C_z : z \in D\}$, compactness implies that there are ${z_1, \dots, z_n \in D}$ with $D \subseteq -\phi C_{z_1} \cup \dots \cup -\phi C_{z_n}$. Let $C = \psi(\phi C_{z_1} \cap \dots \cap \phi C_{z_n})$. Then $C \in \L\U$ by \cref{rem: U is T1}.
Since $\phi C_{z_1} \cap \dots \cap \phi C_{z_n} = \phi(C_{z_1} \cup \dots \cup C_{z_n})$ and  $\phi\psi\phi=\phi$, we have
\begin{align*}
\phi C &= \phi\psi(\phi C_{z_1} \cap \dots \cap \phi C_{z_n}) = \phi\psi\phi(C_{z_1} \cup \dots \cup C_{z_n}) \\
&= \phi(C_{z_1} \cup \dots \cup C_{z_n}) = \phi C_{z_1} \cap \dots \cap \phi C_{z_n},
\end{align*} 
so $D \cap \phi C = \varnothing$. This implies that $D \subseteq C$ by \eqref{Urq D in C}. Moreover, $y \in \phi C$, yielding $y \notin C$. Finally, since $C$ is an upset and $x \le_1 y$, it follows that $x \notin C$.

\eqref{E is intersection} The proof is similar to that of \eqref{D is intersection}.
\end{proof}

\begin{proposition} \plabel{prop: LC and RC are coherent}
Let $\U \in \Urq$. 
\begin{enumerate}
\item \label[prop: LC and RC are coherent]{LC RC coherent} $\LC\U, \RC\U \in \CohLatL$.
\item \label[prop: LC and RC are coherent]{CLF versus LU} $K(\LC\U) = \L\U$ and $K(\RC\U) = \{ \phi C : C \in \L\U \}$.
\item \label[prop: LC and RC are coherent]{lem: CLF = LU} The map $\xi_U \colon \L\U \to \CLF(\LC\U)$, defined by $\xi_U(C) = \up C$, is a $\Lat$-isomorphism. 
\end{enumerate}
\end{proposition}

\begin{proof}
\eqref{LC RC coherent} Define $u \colon \Filt(\L\U) \to \LC\U$ and $v \colon \Idl(\L\U) \to \RC\U$ by
\[
u(\mathcal{F}) = \bigcap \mathcal{F} \quad \textrm{and} \quad v(\mathcal{I}) = \bigcap \{ \phi C : C \in \mathcal{I}\}
\] 
for each $\mathcal{F} \in \Filt(\L\U)$ and $\mathcal{I} \in \Idl(\L\U)$. It is straightforward to see that $u$ and $v$ are well defined and, with respect to reverse inclusion on $\LC\U$ and $\RC\U$, both are order-preserving. To see that $u$ is onto, let $D \in \LC\U$. If $\mathcal{F} = \{ C \in \L\U : D \subseteq C\}$, then it is straightforward to see that $\mathcal{F}$ is a filter and $D = \bigcap \mathcal{F}$ follows from \cref{D is intersection}. A similar argument using \cref{E is intersection} shows that $v$ is onto. 

To see that $u$ is order-reflecting, we show that $\mathcal{F} = \{ C \in \L\U : u(\mathcal{F}) \subseteq C\}$ for each $\mathcal{F} \in \Filt(\L\U)$. The inclusion $\subseteq$ is clear by the definition of $u$. For the reverse inclusion, suppose that $u(\mathcal{F}) \subseteq C$ for some $C \in \L\U$. Because $u(\mathcal{F}) = \bigcap \mathcal{F}$, \cref{Urq finite} implies that there are $C_1, \dots, C_n \in \mathcal{F}$ with $C_1 \cap \dots \cap C_n \subseteq C$. Therefore, $C \in \mathcal{F}$. Thus, for $\mathcal{F}, \mathcal{F}' \in \Filt(\L\U)$, 
\[
\mathcal{F} \subseteq \mathcal{F}' \iff u(\mathcal{F}') \subseteq u(\mathcal{F}) \iff u(\mathcal{F}) \le u(\mathcal{F}').
\]
A similar argument, using \cref{Urq finite 2} instead of \cref{Urq finite}, shows that $v$ is order-reflecting. Thus, $u$ and $v$ are order-isomorphisms, and hence $\LC\U,\RC\U \in \CohLatL$.

\eqref{CLF versus LU} Since $\LC\U$ is ordered by $\supseteq$, it follows from \cref{Urq finite} that $\L\U \subseteq K(\LC\U)$.  For the reverse inclusion, by \cref{D is intersection}, each $D \in \LC\U$ is a join in $\LC\U$ from $\L\U$. Therefore, if $D \in K(\LC\U)$, then it is a finite join from $\L\U$. Since a finite join in $\LC\U$ is a finite intersection and $\L\U$ is closed under finite intersections, $D \in \L\U$. Thus, $K(\LC\U) = \L\U$. The argument for $\RC\U$ is similar but uses \crefrange{Urq finite 2}{E is intersection} instead of \crefrange{Urq finite}{D is intersection}. 

\eqref{lem: CLF = LU} Apply \eqref{LC RC coherent}, \eqref{CLF versus LU}, and \cref{lem: KOF = LX}.
\end{proof}

For $\U\in\Urq$, define $R_\U \subseteq \LC\U \times \RC\U$ by
\[
D \rel{R_\U} E \iff D \cap E \ne \varnothing,
\]
and set $\D(\U) = (\LC\U, R_\U, \RC\U)$.

\begin{proposition} \label{D(U) is in DH}
If $\U \in \Urq$, then $\D(\U) \in \HD$.
\end{proposition}

\begin{proof}
By \cref{LC RC coherent}, $\LC\U, \RC\U \in \CohLatL$. It is left to show that $R_\U$ is a DH-relation. For \cref{def: DH relation interior} let $D \in \LC\U$. Then 
\begin{align*}
R_\U[D] &= \{ E \in \RC\U : D \cap E \ne \varnothing \} \\
&= \{ E \in \RC\U : C \cap E \ne \varnothing \ \forall C \in \L\U \textrm{ with } D \subseteq C \},
\end{align*}
where the second equality holds by \cref{D is intersection} and compactness. For $E \in \RC\U$ and $C \in \L\U$, \cref{Urq E in phi C} shows that $C \cap E \ne \varnothing$ iff $E \not\subseteq \phi C$ iff $\phi C \not\le E$. Consequently,
\begin{align*}
R_\U[D] &= \{ E \in \RC\U : \phi C \not\le E \ \forall C \in \L\U \textrm{ with } D \subseteq C\} \\
&= \{ E \in \RC\U : E \in -\up \phi C \ \forall C \in \L\U \textrm{ with } D \subseteq C \} \\
&= \bigcap \{ -\up \phi C : C \in \L\U \textrm{ with } D \subseteq C\}.
\end{align*}
Since each $-\up \phi C$ is clopen by \cref{lem: KOF = LX,CLF versus LU}, we conclude that $R_\U[D]$ is closed. A similar argument shows that $R_\U^{-1}[E]$ is closed for each $E \in \RC\U$. 

Next, we show that $R_\U[\mathcal{V}]$ is clopen for each clopen $\mathcal{V} \subseteq \LC\U$. If $\mathcal{V} = \varnothing$, then $R_\U[\mathcal{V}] = \varnothing$, so is clopen. Suppose that $\mathcal{V} \ne \varnothing$. By \cref{CLF versus LU} and Theorems~\ref{thm: Scott facts}\eqref{basis}, \ref{thm: Lawson facts}\eqref{Lawson is patch}, 
\[
\{ \up C \cap -\up C_1 \cap \dots \cap -\up C_n : C, C_1, \dots, C_n \in \L\U\}
\]
is a basis for the Lawson topology of $\LC\U$. Since $R_\U$ preserves unions, it suffices to assume that $\mathcal{V} = \up C \cap -\up{C_1} \cap \dots \cap -\up{C_n}$ for some $C, C_1,\dots, C_n \in \L\U$ with $C \not\subseteq C_i$ for each~$i$. This implies that $C \in \mathcal{V}$ (recall that $\LC\U$ is ordered by reverse inclusion). We show that $R_\U[\mathcal{V}] = -\up \phi C$. For one inclusion, by \cref{Urq E in phi C},
\[
E \in -\up \phi C \iff E \not\subseteq \phi C \iff C \cap E \ne \varnothing,
\] 
and so $C \rel{R_\U} E$. Therefore, $-\up \phi C \subseteq R_\U[\mathcal{V}]$ since $C \in \mathcal{V}$. For the reverse inclusion, if $E \in R_\U[\mathcal{V}]$, then $D \cap E \ne \varnothing$ for some $D \in \mathcal{V}$. Because $D \in \mathcal{V}$, we have $D \subseteq C$, and so $C \cap E \ne \varnothing$. Therefore, $E \not\subseteq \phi C$, and hence $E \in -\up \phi C$. Thus, $R_\U[\mathcal{V}] = -\up \phi C$, and so $R_\U[\mathcal{V}]$ is clopen. A similar argument shows that $R_\U^{-1}[\mathcal{W}]$ is clopen for each clopen $\mathcal{W} \subseteq \RC\U$. This completes the proof that $R_\U$ is an interior relation.

We next verify \cref{def: DH relation filter}. Let $D \in \LC\U$ and $E \in \RC\U$. Since the top of $\RC\U$ is $\varnothing$, we have $D \nr{R_\U} \varnothing$ for each $D \in \LC\U$. Similarly, $\varnothing \nr{R_\U} E$ for each $E \in \RC\U$. Thus, $\LC\U - R_\U[D]$ and $\RC\U - R_\U^{-1}[E]$ are nonempty.
Because $D \rel{R_\U} E$ iff $D \cap E \ne \varnothing$ and the orders on $\LC\U$ and $\RC\U$ are each reverse inclusion, $R_\U[D]$ is a downset of $\RC\U$ and $R_\U^{-1}[E]$ is a downset of $\LC\U$. Therefore, $\LC\U - R_\U[D]$ and $\RC\U - R_\U^{-1}[E]$ are upsets. 
Finally, to see that $\LC\U - R_\U[D]$ is closed under binary meets, let $D_1, D_2 \in \LC\U$ and $E \in \RC\U$. Suppose that $D_1 \nr{R_\U}E$ and $D_2 \nr{R_\U} E$. Then $D_1 \cap E = \varnothing = D_2 \cap E$. By \cref{D is intersection} and compactness, there are $C_i \in \L\U$ with $D_i \subseteq C_i$ and $C_i \cap E = \varnothing$ ($i=1,2$). Therefore, $E \subseteq \phi C_i$ by \cref{Urq E in phi C}, so $E \subseteq \phi C_1 \cap \phi C_2$. Then $C := \psi(\phi C_1 \cap \phi C_2) \subseteq \psi E$. By \cref{rem: U is T1}, $C \in \L\U$. From $C \subseteq \psi E$ we get $C \cap E = \varnothing$. Because $D_1, D_2 \subseteq C$, we have $C \le D_1, D_2$ in $\LC\U$, and hence $C \le D_1 \wedge D_2$, which means that $D_1 \wedge D_2 \subseteq C$. Consequently, $(D_1 \wedge D_2) \cap E = \varnothing$. Thus, $(D_1 \wedge D_2) \nr{R_\U} E$. This proves that $\LC\U - R_\U[D]$ is closed under binary meets. The proof that $\RC\U - R_\U^{-1}[E]$ is closed under binary meets for each $E \in \RC\U$ is similar.

For \cref{def: DH relation 3}, let $D, D' \in \LC\U$. Suppose that $R_\U[D] \subseteq R_\U[D']$. Let $C \in \L\U$ with $D' \subseteq C$. By \cref{Urq D in C}, $D' \nr{R_\U} \phi C$, so $D \nr{R_\U} \phi C$, and hence $D \subseteq C$. Therefore, $D \subseteq D'$ by \cref{D is intersection}, so $D' \le D$. The proof of \cref{def: DH relation 4} is similar but uses \crefrange{Urq E in phi C}{E is intersection} instead of \crefrange{Urq D in C}{D is intersection}. This completes the proof that $R_\U$ is a DH-relation, and thus $(\LC\U, R_\U, \RC\U) \in \HD$.
\end{proof}

We next associate a $\HD$-morphism to each $\Urq$-morphism.

\begin{definition} \label{def: fPQ and gPQ}
Let $(P, Q) \colon \U \to \U'$ be an $\Urq$-morphism. For each $D \in \LC\U$ and ${E \in \RC\U}$, define
\begin{align*}
f_{PQ}(D) &= \bigcap \{ C' \in \L\U' : D \subseteq \Box_PC'\},\\
g_{PQ}(E) &= \bigcap \{  \phi' C' : C' \in \L\U' , E \subseteq \Box_Q\phi' C'\}.
\end{align*}
\end{definition}

Because $f_{PQ}(D)$ is an intersection from $\L\U'$, we have $f_{PQ}(D) \in \LC\U'$, and similarly $g_{PQ}(E) \in \RC\U'$. Therefore, $f_{PQ} \colon \LC\U \to \LC\U'$ and $g_{PQ} \colon \RC\U \to \RC\U'$ are well-defined functions.
In order to prove that $(f_{PQ}, g_{PQ}) \colon \D(\U) \to \D(\U')$ is a $\HD$-morphism, we require the following lemma.

\begin{lemma} \plabel{lem: Urq to DH morphism facts}
Let $(P, Q) \colon \U \to \U'$ be an $\Urq$-morphism. 
\begin{enumerate}
\item \label[lem: Urq to DH morphism facts]{f(D) in C} If $D \in \LC\U$ and $C' \in \L\U'$, then $D \subseteq \Box_PC'$ iff $f_{PQ}(D) \subseteq C'$.  
\item \label[lem: Urq to DH morphism facts]{g(E) in phi C} If $E \in \RC\U$ and $C' \in \L\U'$, then $E \subseteq \Box_Q \phi'C'$ iff $g_{PQ}(E) \subseteq \phi'C'$. 
\end{enumerate}
\end{lemma}

\begin{proof}
\eqref{f(D) in C} If $D \subseteq \Box_PC'$, then $f_{PQ}(D) \subseteq C'$ by definition of $f_{PQ}$. Conversely, suppose that $f_{PQ}(D) \subseteq C'$. By definition of $f_{PQ}(D)$ and \cref{Urq finite}, there are $C_1', \dots, C_n' \in \L\U'$ with $D \subseteq \Box_PC_i'$ for each $i$ and $C_1' \cap \dots \cap C_n' \subseteq C'$. Therefore,
\[
D\subseteq \Box_PC_1' \cap \dots \cap \Box_PC_n' = \Box_P (C_1' \cap \dots \cap C_n') \subseteq \Box_P C'.
\]

\eqref{g(E) in phi C} The proof is similar to that of \eqref{f(D) in C}.
\end{proof}

\begin{proposition} \label{prop: (f_PQ  g_PQ) is a morphism}
If $(P, Q) \colon \U \to \U'$ is an $\Urq$-morphism, then $(f_{PQ}, g_{PQ}) \colon \D(\U) \to \D(\U')$ is a $\HD$-morphism.
\end{proposition}

\begin{proof}
We first point out that $f_{PQ}$ is order-preserving. If $D, D' \in \LC\U$ with $D \subseteq D'$, then $f_{PQ}(D) \subseteq f_{PQ}(D')$ follows from the definition of $f_{PQ}$. Because the orders on $\LC\U$ and $\LC\U'$ are reverse inclusions, this shows that $f_{PQ}$ is order-preserving.

We prove that $f_{PQ}$ is a $\CohLatL$-morphism. Let $D = \bigvee D_i$ be a directed join. Then $D = \bigcap D_i$ is a directed intersection. Since $f_{PQ}$ is order-preserving, $\bigvee f_{PQ}(D_i) \le f_{PQ}(D)$, which says that $f_{PQ}(D) \subseteq \bigcap f_{PQ}(D_i)$. For the reverse inclusion, by the definition of $f_{PQ}(D)$ it suffices to show that if $C' \in \L\U'$ with $D\subseteq \Box_PC'$, then there is $i$ with $f_{PQ}(D_i) \subseteq C'$. Because $D = \bigcap D_i$ is a directed intersection, \cref{Urq finite} shows that $D_i \subseteq \Box_PC'$ for some $i$, and so $f_{PQ}(D_i) \subseteq C'$. Thus, $f_{PQ}$ preserves directed joins. 

To see that $f_{PQ}$ preserves meets, suppose that $D = \bigwedge D_i$. Then $f_{PQ}(D) \le \bigwedge f_{PQ}(D_i)$ because $f_{PQ}$ is order-preserving, \cref{D is intersection} implies that a meet $\bigwedge \mathcal{S}$ in $\LC\U$ is the intersection of all $C \in \L\U$ containing $\bigcup\mathcal{S}$. Therefore, for the reverse inequality, it suffices to show that if $C' \in \L\U'$ such that $f_{PQ}(D_i) \subseteq C'$ for each $i$, then $f_{PQ}(D) \subseteq C'$. If $f_{PQ}(D_i) \subseteq C'$, then $D_i \subseteq \Box_PC'$ by \cref{f(D) in C}. If this happens for each $i$, then $D \subseteq \Box_PC'$ because $D = \bigwedge D_i$, and hence $f_{PQ}(D) \subseteq C'$. Thus, $f_{PQ}$ preserves arbitrary meets.

Observe that the left adjoint $\ell$ of $f_{PQ}$ is given by 
$\ell(D') = \Box_PD'$ for each $D' \in \LC\U'$. To see this, since $\Box_P$ preserves arbitrary intersections, by Lemmas~\ref{lem: LC and RC facts}\eqref{D is intersection} and \ref{lem: Urq to DH morphism facts}\eqref{f(D) in C}, for each $D \in \LC\U$ we have
\begin{align*}
D \subseteq \Box_PD' &\iff D \subseteq \Box_PC' \ \forall C' \in \L\U' \textrm{ with } D' \subseteq C' \\
& \iff f_{PQ}(D) \subseteq C' \ \forall C' \in \L\U' \textrm{ with } D' \subseteq C' \iff f_{PQ}(D) \subseteq D'.
\end{align*}
Thus, $\Box_PD' \le D$ iff $D' \le f_{PQ}(D)$, showing that $\ell(D') = \Box_PD'$.

We now show that $\ell$ preserves finite meets of compact elements. We have $K(\LC\U') = \L\U'$ by \cref{CLF versus LU}. 
If $\mathcal{S} \subseteq \L\U'$ is  finite, then $\bigwedge \mathcal{S}$ in $\LC\U'$ is equal to $\bigvee \mathcal{S}$ in $\L\U'$ since the order on $\LC\U'$ is reverse inclusion while the order on $\L\U'$ is inclusion. Thus, it is enough to prove the following:

\begin{claim} \label{claim: ell preserves finite joins}
$\ell$ preserves finite joins from $\L\U$.
\end{claim}

\begin{proof}
Let $\mathcal S$ be a finite subset of $\L\U$. If $\mathcal S = \varnothing$, then 
$
\Box_P\left(\bigvee \varnothing\right) = \Box_P \varnothing = \varnothing = \bigvee \varnothing. 
$
Thus, it suffices to prove that $\ell$ preserves binary joins from $\L\U'$. Let $C_1, C_2 \in \L\U'$. The inclusion $\Box_PC_1 \vee \Box_PC_2 \subseteq \Box_P(C_1 \vee C_2)$ is clear since $\Box_P$ is order-preserving. For the other inclusion, let $x \in \Box_P(C_1 \vee C_2)$.  Since $\Box_P(C_1 \vee C_2) = \Box_P\psi'(\phi'C_1 \cap \phi'C_2)$,
\[
P[x] \subseteq \psi'(\phi'C_1 \cap \phi'C_2) = -\down_1(\phi'C_1 \cap \phi'C_2),
\]
so $P[x] \cap \phi'C_1 \cap \phi'C_2 = \varnothing$ because $P[x]$ is a $\le_1$-upset. 

Suppose that $x \notin \Box_PC_1 \vee \Box_PC_2$. Since $\Box_PC_1 \vee \Box_PC_2 = \psi(\phi\Box_PC_1 \cap \phi\Box_PC_2)$, we have $x \notin \psi(\phi\Box_PC_1 \cap \phi\Box_PC_2)$, so $x \in \down_1(\phi\Box_PC_1 \cap \phi\Box_PC_2)$. Therefore, there is $y$ with $x \le_1 y$ and $y \in \phi\Box_PC_1 \cap \phi\Box_PC_2$. By \cref{Urq Box},
\[
\phi\Box_PC_1 \cap \phi\Box_PC_2 = \Box_Q \phi' C_1 \cap \Box_Q \phi' C_2 = \Box_Q(\phi'C_1 \cap \phi'C_2).
\]
Thus, $Q[y] \subseteq \phi'C_1 \cap \phi'C_2$. Since $x \le_1 y$, we have $P[y] \subseteq P[x]$ by \cref{P inverse is a downset}. Because
\[
P[x] \cap \phi'C_1 \cap \phi'C_2 = \varnothing \quad\textrm{and}\quad Q[y] \subseteq \phi'C_1 \cap \phi'C_2,
\]
we conclude that $P[y] \cap Q[y] \subseteq P[x] \cap Q[y] = \varnothing$, contradicting \cref{Urq serial}. Consequently, 
$
\Box_P(C_1 \vee C_2) \subseteq \Box_P C_1 \vee \Box_P C_2,
$
and hence $\Box_P$ preserves binary joins.
\end{proof}

Applying \cref{lem: ell preserves meets} completes the proof that $f_{PQ}$ is a $\CohLatL$-morphism. 
Similar arguments show that $g_{PQ}$ is a $\CohLatL$-morphism.

Finally, we show that $(f_{PQ}, g_{PQ})$ is a $\HD$-morphism. Because $f_{PQ}, g_{PQ}$ are $\CohLatL$-morphisms, it remains to verify the three axioms of \cref{DH maps}. For \cref{def: DH morphisms 1}, let $D \in \LC\U$ and $E \in \RC\U$ with $D \rel{R_\U} E$. Then $D \cap E \ne \varnothing$. If $f_{PQ}(D) \cap g_{PQ}(E) = \varnothing$, then
\[
\bigcap \{ C_1' \in \L\U' : D \subseteq \Box_PC_1' \} \cap \bigcap \{ \phi' C_2' : C_2' \in \L\U' ,\ E \subseteq \Box_Q\phi' C_2'\} = \varnothing.
\]
These intersections are directed, so compactness implies that there are $C_1', C_2' \in \L\U'$ with $D \subseteq \Box_PC_1'$, $E \subseteq \Box_Q\phi C_2'$, and $C_1' \cap \phi' C_2' = \varnothing$. Since there is $z \in D \cap E$, we have $z \in \Box_PC_1' \cap \Box_Q\phi' C_2'$. Therefore, $P[z] \cap Q[z] \subseteq C_1' \cap \phi' C_2' = \varnothing$. This is a contradiction since $P[z] \cap Q[z] \ne \varnothing$ by \cref{Urq serial}. Thus, $f_{PQ}(D) \cap f_{PQ}(E) \ne \varnothing$, so $f_{PQ}(D) \rel{R_{\U'}} f_{PQ}(E)$.

To verify \cref{def: DH morphisms 2}, suppose that $D' \rel{R_{\U'}} g_{PQ}(E)$ and set $D = \ell(D') \in \LC\U$. Then $D' \le f_{PQ}(D)$. If $D \cap E = \varnothing$, because $\ell$ preserves arbitrary joins in $\LC\U'$ (which are intersections), compactness implies that there is $C' \in \L\U'$ with $D' \subseteq C'$ and $\ell(C') \cap E = \varnothing$. Since $\ell(C') = \Box_PC'$, we have $\Box_PC' \cap E = \varnothing$. Therefore, $E \subseteq \phi\Box_PC'$ by \cref{Urq E in phi C}, and so $E \subseteq \Box_Q \phi'C'$ by \cref{Urq Box}. Thus, $g_{PQ}(E) \subseteq \phi'C'$  by \cref{g(E) in phi C}, and so $D' \cap g_{PQ}(E) \subseteq C' \cap \phi'C' = \varnothing$. This is a contradiction since $D' \rel{R_{\U'}} g_{PQ}(E)$. Consequently, $D \cap E \ne \varnothing$, and hence $D \rel{R_\U} E$.

A similar argument verifies \cref{def: DH morphisms 3}, completing the proof that $(f_{PQ}, g_{PQ})$ is a $\HD$-morphism.
\end{proof} 

\begin{theorem} \label{thm: D is a functor}
There is a functor $\DD \colon \Urq \to \HD$ given by $\DD(\U) = \D(\U)$ for each $\U \in \Urq$ and $\DD(P, Q) = (f_{PQ}, g_{PQ})$ for each $\Urq$-morphism $(P, Q)$, where $f_{PQ}, g_{PQ}$ are given in Definition~\emph{\ref{def: fPQ and gPQ}}.
\end{theorem}

\begin{proof}
By \cref{D(U) is in DH}, $\DD$ is well defined on objects, and it is well defined on morphisms by \cref{prop: (f_PQ  g_PQ) is a morphism}. We show that $\DD$ sends identity morphisms to identity morphisms. Let $\U = (Z, \le_1, \le_2) \in \Urq$.  By \cref{Urq identity}, the identity morphism of $\U$ is $(P, Q) := (\le_1, \le_2)$. If $z, z' \in Z$, then $z\rel{P}z'$ iff $z \le_1 z'$. Consequently, $\Box_PC = C$ for each $C \in \LC\U$ because $C$ is a $\le_1$-upset. Therefore, if $D \in \LC\U$, by \cref{D is intersection},
\[
f_{PQ}(D) = \bigcap \{ C \in \L\U : D \subseteq \Box_PC \} = \bigcap \{ C \in \L\U : D \subseteq C \} = D.
\]
This shows that $f_{PQ}$ is the identity on $\LC\U$. A similar argument shows that $g_{PQ}$ is the identity on $\RC\U$. Therefore, $(f_{PQ}, g_{PQ})$ is the identity on $\DD(\U)$. 

Next let $(P_1, Q_1) \colon \U_1 \to \U_2$ and $(P_2, Q_2) \colon \U_2 \to \U_3$ be $\Urq$-morphisms. Set $P = P_2 \star P_1$ and $Q = Q_2 \star Q_1$. To show that $\DD$ preserves composition, by definition of composition in $\HD$ it suffices to show that $f_{PQ} = f_{P_2Q_2}\circ f_{P_1Q_1}$ and $g_{PQ} = g_{P_2Q_2} \circ g_{P_1Q_1}$. If $D \in \LC\U_1$, then
\[
f_{PQ}(D) = \bigcap \{ C'' \in \L\U_3 : D \subseteq \Box_{P}C'' \} = \bigcap \{ C'' \in \L\U_3: D \subseteq \Box_{P_1}\Box_{P_2}C''\}
\]
by \cref{lem: composition lemma for Urq}. Thus,
\begin{align*}
(f_{P_2Q_2}\circ f_{P_1Q_1})(D) &= f_{P_2Q_2}(f_{P_1Q_1}(D)) = \bigcap \{ C'' \in \L\U_3 : f_{P_1Q_1}(D) \subseteq \Box_{P_2}C'' \} \\
&= \bigcap \{ C'' \in \L\U_3 : D \subseteq \Box_{P_1}\Box_{P_2}C'' \} = f_{PQ}(D),
\end{align*}
where the third equality holds by \cref{f(D) in C}. Consequently, $f_{PQ} = f_{P_2Q_2}\circ f_{P_1Q_1}$. A similar argument gives that $g_{PQ} = g_{P_2Q_2} \circ g_{P_1Q_1}$,  completing the proof that $\DD$ is a functor.
\end{proof}


\section{The resulting category equivalences} \label{sec: HD GvG Hg Urq}

In \cref{sec: DH and GvG,sec: GvG and Hg,sec: Hg and Urq} we defined the functors $\GG \colon \HD \to \GvG$, $\HH \colon \GvG \to \Hg$, and $\UU \colon \Hg \to \Urq$; in \cref{sec: Urq and Plo} we showed that the functor $\PP \colon \Urq \to \Plo$ is an isomorphism; and in \cref{sec: Urq and DH} we defined the functor $\DD \colon \Urq \to \HD$. In this section we show that the functors 
\[
\begin{tikzcd}
\HD \arrow[r, "\GG"] & \GvG \arrow[r, "\HH"] & \Hg \arrow[r, "\UU"] & \Urq \arrow[lll, bend left = 20, "\DD"] 
\end{tikzcd}
\]
 yield equivalences of $\HD, \GvG, \Hg$, and $\Urq$, and hence that each of these categories is also equivalent to $\Plo$. This concludes the circle of our equivalences. 
 
 We show that the above four functors are equivalences by producing 
 natural isomorphisms between the identity functor of each of the above four categories and the corresponding composition of the four functors. Since various of these natural isomorphisms have similar proofs, we skip some details towards the end of the section. 
 
Let $\U \in \Urq$. By \cref{D(U) is in DH}, $(\LC\U, R_\U, \RC\U) \in \HD$. Therefore, by \cref{prop: G on objects}, $(\LCUp, (R_\U)_p, \RCUp) \in \GvG$. Thus, by \cref{lem: X0 the same,prop: HH on objects}, $(\LCUO, (R_\U)_0, \RCUO) \in \Hg$. Throughout the section, we assume that an Urquhart space is the triple $\U = (Z, \le_1, \le_2)$.

\begin{lemma} \plabel{lem: connections between the lattices}
\hfill
\begin{enumerate}
\item \label[lem: connections between the lattices]{CLF from LU} Let $\D = (X, R, Y) \in \HD$. Set $\U = \UU\HH\GG(\D)$ and $\overline{\D} = \DD\UU\HH\GG(\D)$. Then
\begin{align*}
\L\U &= \{ (U \times Y) \cap Z : U \in \CLF(X) \}, \\
\CLF(\LC\U) &= \{ \up [(U \times Y) \cap Z] : U \in \CLF(X) \}.
\end{align*}
\item \label[lem: connections between the lattices]{LG from LU} Let $\G = (X, R, Y) \in \GvG$. Set $\U = \UU\HH(\G)$ and $\overline{\G} = \GG\DD\UU\HH(\G)$. Then 
\begin{align*}
\L\U &= \{ (U \times Y) \cap Z : U \in \L\G \}, \\
\L\overline{\G} &= \{ \up [(U \times Y) \cap Z]\cap \LCUp : U \in \L\G \}.
\end{align*}
\item \label[lem: connections between the lattices]{LH from LU} Let $\H = (X, R, Y) \in \Hg$. Set $\U = \UU(\H)$, and $\overline{\H} = \HH\GG\DD\UU(\H)$. Then
\begin{align*}
\L\U &= \{ (A \times Y) \cap Z : A \in \L\H \}, \\
\L\overline{\H} &= \{ \up [(A \times Y) \cap Z]\cap \LCUO : A \in \L\H \}.
\end{align*}
\end{enumerate}
\end{lemma}

\begin{proof}
\eqref{CLF from LU} Let $\G = \GG(\D)$ and $\H = \HH(\G)$. Then $\L\G = \{ U \cap X_p : U \in \CLF(X) \}$ by \cref{lem: LG = CLF} and $\L\H = \{ U \cap X_0 : U \in \CLF(X) \}$ by \cref{lem: X0 the same,lem: LH = LG}. Therefore, \cref{lem: LU = LH} yields that
\[
\L\U = \{ ((U \cap X_0) \times Y) \cap Z : U \in \CLF(X) \}.
\]
We show that $((U \cap X_0) \times Y) \cap Z = (U \times Y) \cap Z$. 
One inclusion is obvious. For the reverse inclusion, let $(x, y) \in Z$ with $x \in U$. Because $Z \subseteq X_0 \times Y_0$, we have $x \in U \cap X_0$. 
Therefore,
\[
\L\U = \{ (U \times Y) \cap Z : U \in \CLF(X) \}.
\] 
By \cref{CLF versus LU}, $K(\LC\U)=\L\U$. Thus, \cref{lem: KOF = LX} implies that
\[
\CLF(\LC\U) = \{ \up [(U \times Y) \cap Z] : U \in \CLF(X) \}.
\]

The proofs of \eqref{LG from LU} and \eqref{LH from LU} are similar to that of  \eqref{CLF from LU}.
\end{proof}

\begin{lemma} \plabel{lem: technical equalities}
\hfill
\begin{enumerate}
\item \label[lem: technical equalities]{tech 1} Let $\D = (X, R, Y) \in \HD$ and set $\UU\HH\GG(\D) = (Z, \le_1, \le_2)$. If $x \in X$ and $U \in \CLF(X)$, then 
\[
x \in U \iff 
(\up x \times Y) \cap Z \subseteq (U \times Y) \cap Z.
\]
\item \label[lem: technical equalities]{tech 2} Let $\G = (X, R, Y) \in \GvG$ and set $\UU\HH(\G) = (Z, \le_1, \le_2)$. If $x \in X$ and $U \in \L\G$, then 
\[
x \in U \iff 
(\up x \times Y) \cap Z \subseteq (U \times Y) \cap Z.
\]
\item \label[lem: technical equalities]{tech 3}Let $\H = (X, R, Y) \in \Hg$ and set $\UU(\H) = (Z, \le_1, \le_2)$. If $x \in X$ and $A \in \L\H$, then
\[
x \in A \iff 
(\up x \times Y) \cap Z \subseteq (A \times Y) \cap Z.
\]
\end{enumerate}
\end{lemma}

\begin{proof}
\eqref{tech 1} One direction is clear since $U$ is an upset. For the other, suppose that 
$x \notin U$. By \cref{phi psi is 1}, $U = \psi\phi U$. Therefore, there is $y \in \phi U$ with $x \rel{R} y$. By \cref{maximal pairs}, there are $x' \ge x$ and $y' \ge y$ with $(x', y') \in Z$. Thus, $(x',y') \in (\up x \times Y) \cap Z$. On the other hand, $x' \rel{R} y'$ and $y \le y'$ imply that $x' \rel{R} y$. Hence, $y \in \phi U$ 
yields that $x' \notin U$. Consequently,  $(x',y') \notin (U \times Y) \cap Z$, and so $(\up x \times Y) \cap Z \not\subseteq (U \times Y) \cap Z$.

The proofs of \eqref{tech 2} and \eqref{tech 3} are similar to that of  \eqref{tech 1}.
\end{proof}

\subsection{The natural isomorphism \texorpdfstring{$1_\HD \to \DD\UU\HH\GG$}{Lg}} 

\begin{lemma} \label{lem: epsilon_D}
For each $\D \in \HD$, there is a $\HD$-isomorphism $\epsilon_\D \colon \D \to \DD\UU\HH\GG(\D)$.
\end{lemma}

\begin{proof}
Let $\D = (X, R, Y)$. 
Set $\U = \UU\HH\GG(\D) := (Z, \le_1, \le_2)$ and $\overline{\D} = \DD\UU\HH\GG(\D)$, so $\overline{\D} = (\LC\U, R_\U, \RC\U)$. For $U \in \CLF(X)$ let $\overline{U} = (U \times Y) \cap Z$. By \cref{CLF from LU}, we have $\L\U = \{ \overline{U} : U \in \CLF(X) \}$ and $\CLF(\LC\U) = \{ \up \overline{U} : U \in \CLF(X) \}$.

Define $\alpha \colon X \to \LC\U$ and $\beta \colon Y \to \RC\U$ by
\[
\alpha(x)  = (\up x \times Y) \cap Z \quad\textrm{and}\quad \beta(y) = (X \times \up y) \cap Z.
\]
Since $\CLF(X)=\{ \up k : k \in K(X) \}$ (see \cref{lem: KOF = LX}), we have $\up x = \bigcap \{ U \in \CLF(X) : x \in U\}$. Therefore, $\alpha(x) = \bigcap \{ \overline{U} : U \in \CLF(X),x \in U \}$, an intersection from $\L\U$, and hence $\alpha(x) \in \LC\U$ since $\LC\U$ is closed under intersections. 
It is easy to see that $\alpha$ is order-preserving since the order on $\LC\U$ is reverse inclusion. To see that $\alpha$ is order-reflecting, suppose that $\alpha(x_1) \le \alpha(x_2)$ in $\LC\U$. Then $\alpha(x_2) \subseteq \alpha(x_1)$. By \cref{x is meet from X0}, for each $x \in X$, we have $x = \bigwedge (\up x \cap X_0)$; and by \cref{maximal pairs 2}, for each $z \in X_0$, there is $y \in Y_0$ with $(z, y) \in Z$.  Therefore,
\begin{align*}
x_1 &= \bigwedge \{ z \in X_0 : \exists y \in Y_0, (z, y) \in \alpha(x_1)\} \\
&\le \bigwedge \{ z \in X_0 : \exists y \in Y_0, (z, y) \in \alpha(x_2)\} = x_2.
\end{align*}
Thus, $\alpha$ is order-reflecting. To show that $\alpha$ is onto, let $D \in \LC\U$. By Lemmas~\ref{lem: LC and RC facts}\eqref{D is intersection} and \ref{lem: connections between the lattices}\eqref{CLF from LU}, there is a family $\{U_i\} \subseteq \CLF(X)$ with $D = \bigcap \overline{U_i}$. By \cref{lem: KOF = LX}, for each $i$ there is $k_i \in K(X)$ with $U_i = \up k_i$. If $x = \bigvee k_i$, then $\up x = \bigcap U_i$. Consequently, $\alpha(x) = D$. This shows that $\alpha$ is an order-isomorphism, hence a $\CohLatL$-isomorphism. 
Similar arguments show that $\beta$ is a $\CohLatL$-isomorphism. 

To see that $(\alpha, \beta)$ is a $\HD$-isomorphism, by \cref{rem: morphisms}, it is enough to show that $x\rel{R}y$ iff $\alpha(x) \rel{R_\U} \beta(y)$. If $\alpha(x) \rel{R_\U} \beta(y)$, then $\alpha(x) \cap \beta(y) \ne \varnothing$, so there is $(x', y') \in Z$ with $x \le x'$ and $y \le y'$. Since $(x', y') \in Z$, we have $x' \rel{R} y'$, so $x \rel{R}y$ by \cref{def: DH relation filter}. Conversely, if $x\rel{R}y$, by \cref{maximal pairs}, there are $x' \ge x$ and $y' \ge y$ with $(x', y') \in Z$, so $\alpha(x) \cap \beta(y) \ne \varnothing$, and hence $\alpha(x) \rel{R_{\U}} \beta(y)$. Setting $\epsilon_\D = (\alpha, \beta)$ yields the desired $\HD$-isomorphism.
\end{proof}

\begin{proposition} \label{prop: epsilon : 1 to DUHG}
There is a natural isomorphism $\epsilon \colon 1_{\HD} \to \DD\UU\HH\GG$.
\end{proposition}

\begin{proof}
Let $\D = (X, R, Y) \in \HD$. Then $\epsilon_\D$ is a $\HD$-isomorphism by \cref{lem: epsilon_D}. To see that $\epsilon$ is natural, let $(f, g) \colon \D \to \D'$ be a $\HD$-morphism. Set $\U = \UU\HH\GG(\D)$, $\U' = \UU\HH\GG(\D')$, and $\UU\HH\GG(f, g) = (P, Q)$. Then $\DD\UU\HH\GG(f, g) = (f_{PQ}, g_{PQ})$. We need to show that the following diagram commutes.
\[
\begin{tikzcd}[column sep = 5pc]
\D \arrow[r, "(f{,}g)"] \arrow[d, "\epsilon_\D"'] & \D' \arrow[d, "\epsilon_{\D'}"] \\
\overline{\D} \arrow[r, "(f_{PQ}{,}g_{PQ})"'] & \overline{\D'}
\end{tikzcd}
\]
Write $\epsilon_\D = (\alpha, \beta)$ and $\epsilon_{\D'} = (\alpha', \beta')$. We then need to show that $\alpha' \circ f = f_{PQ} \circ \alpha$ and $\beta' \circ g = g_{PQ} \circ \beta$. Let $x \in X$. Then $\alpha'(f(x)), f_{PQ}(\alpha(x)) \in \LC\U'$. To see that they are equal, by \cref{D is intersection} it suffices to show that
\[
\alpha'(f(x)) \subseteq C' \iff f_{PQ}(\alpha(x)) \subseteq C'
\]
for each $C' \in \L\U'$. 
Recalling that $\U'=(Z',\le_1',\le_2')$ and using the notation of \cref{lem: epsilon_D}, it follows from \cref{CLF from LU} that $C' = (U' \times Y') \cap Z' = \overline{U'}$ for some $U' \in \CLF(X')$. Since $\alpha'(f(x)) = (\up f(x) \times Y') \cap Z'$, we have $\alpha'(f(x)) \subseteq \overline{U'}$ iff $f(x) \in U'$ by \cref{tech 1}. By \cref{f(D) in C}, $f_{PQ}(\alpha(x)) \subseteq \overline{U'}$ iff $\alpha(x) \subseteq \Box_P \overline{U'}$. Let $(x_1, y_1) \in Z$ and $(x_1', y_1') \in Z'$. Recalling how $\UU$, $\HH$, and $\GG$ act on morphisms, we have
\[
(x_1, y_1) \rel{P} (x_1', y_1') \iff x_1 \rel{S_f} x_1'  \iff f(x_1) \le x_1',
\]
where $S_f$ is defined in \cref{prop: G on morphisms}. Therefore,
\begin{align*}
\Box_P \overline{U'} &= \{ (x_1, y_1) \in Z : (x_1', y_1') \in Z'\ \textrm{ and } f(x_1) \le x_1' \Longrightarrow x_1' \in U' \} \\
&= \{ (x_1, y_1) \in Z : (\up f(x_1) \times Y') \cap Z' \subseteq (U' \times Y') \cap Z' \} \\
&= \{ (x_1, y_1) \in Z : f(x_1) \in U' \} = (f^{-1}(U') \times Y) \cap Z,
\end{align*}
where the second-to-last equality follows from \cref{tech 1}. Thus,
\begin{align*}
f_{PQ}(\alpha(x)) \subseteq \overline{U'} &\iff \alpha(x) \subseteq \Box_P \overline{U'} \\
&\iff (\up x \times Y) \cap Z \subseteq (f^{-1}(U') \times Y) \cap Z \\
&\iff x \in f^{-1}(U') \iff f(x) \in U' \\
&\iff \alpha'(f(x)) \subseteq \overline{U'}, 
\end{align*}
again by \cref{tech 1}. Consequently, $f_{PQ}(\alpha(x)) = \alpha'(f(x))$, and hence $f_{PQ} \circ \alpha = \alpha' \circ f$. Similar calculations show that $g_{PQ} \circ \beta = \beta' \circ g$. Thus, $\epsilon$ is natural.
\end{proof}

\subsection{The natural isomorphism \texorpdfstring{$1_\GvG \to \GG\DD\UU\HH$}{Lg}} \label{subsec: GvG}

\begin{lemma} \label{lem: isomorphisms in GvG}
Let $\G = (X, R, Y)$ and $\G' = (X', R', Y')$ be GvG-spaces. 
Suppose $\alpha \colon X \to X'$ and $\beta \colon Y \to Y'$ are $\Pries$-isomorphisms satisfying $x \rel{R} y$ iff $\alpha(x) \rel{R'} \beta(y)$. 
Define $S_\alpha \subseteq X \times X'$ and $T_\beta \subseteq Y \times Y'$ by
\[
x \rel{S_\alpha} x' \iff \alpha(x) \le x' \quad\textrm{and}\quad y \rel{T_\beta} y' \iff \beta(y) \le y'. 
\]
Then $(S_\alpha, T_\beta) \colon \G \to \G'$ is a $\GvG$-isomorphism.
\end{lemma}

\begin{proof}
Let $U' \in \L\G'$. Since $U'$ is an upset, 
\begin{align*}
\Box_{S_\alpha} U' &= \{ x \in X : S_\alpha[x] \subseteq U' \} = \{ x \in X : \up \alpha(x) \subseteq U' \} \\
&= \{ x \in X : \alpha(x) \in U' \} = \alpha^{-1}(U').
\end{align*}
Therefore, because $U'$ is a clopen upset and $\alpha$ is a $\Pries$-morphism, $\Box_{S_\alpha} U'$ is a clopen upset. 
Similarly, $\Diamond_{T_\beta} \bd'U'=\beta^{-1}(\bd' U')$, and hence $\Diamond_{T_\beta} \bd'U'$ is a clopen downset.
We show that $(S_\alpha,T_\beta)$ is a $\GvG$-morphism. 
Since $\alpha,\beta$ are bijections such that $x \rel{R} y$ iff $\alpha(x) \rel{R'}\beta(y)$, 
\begin{equation*}
\alpha^{-1} \Box' V' = \Box \beta^{-1}  V' \quad\textrm{and}\quad \bd \alpha^{-1} U' = \beta^{-1} \bd' U'
\end{equation*}
for each $U' \subseteq X'$ and $V'\subseteq Y'$. Therefore, since $\Box'\bd' U' = U'$,  
\begin{align*}
\Box\bd(\Box_{S_\alpha}U') &= \Box\bd\alpha^{-1}(U') = \Box\beta^{-1}(\bd' U') = \alpha^{-1}(\Box'\bd' U') \\
&= \alpha^{-1}(U') = \Box_{S_\alpha} U'.
\end{align*}
Thus, $\Box_{S_\alpha} U' \in \L\G$. Moreover, 
\begin{align*}
\bd \Box_{S_\alpha} U' &= \bd \alpha^{-1}(U')= \beta^{-1}(\bd' U') = \Diamond_{T_\beta} \bd'U'.
\end{align*}
Consequently, \cref{GvG morphism 1} is satisfied.

For \cref{GvG morphism 2}, suppose that $x \nr{S_\alpha} x'$, so $\alpha(x) \not\le x'$. By \cref{GvG up x}, there is $U' \in \L\G'$ with $\alpha(x) \in U'$ and $x' \notin U'$. Since $U' $ is an upset, $S_\alpha[x] = \up \alpha(x) \subseteq U' $, so $x \in \Box_{S_\alpha}U'$. A similar proof holds for $T_\beta$, so \cref{GvG morphism 3} is satisfied. 
Suppose that $x \rel{R} y$. Then $\alpha(x) \rel{R'} \beta(y)$. Because $x\rel{S_\alpha}\alpha(x)$ and $y\rel{T_\beta}\beta(y)$, \cref{GvG morphism 4} holds. Thus, $(S_\alpha, T_\beta)$ is a $\GvG$-morphism.

Finally, to show that $(S_\alpha, T_\beta)$ is a $\GvG$-isomorphism, $(S_{\alpha^{-1}}, T_{\beta^{-1}}) \colon \G' \to \G$ is  a $\GvG$-morphism by the above argument. We prove that $(S_{\alpha^{-1}}, T_{\beta^{-1}}) \star (S_\alpha, T_\beta) = (\le_X, \le_Y)$, the identity morphism for $\G$ by \cref{GvG identity}. Since $S_{\alpha^{-1}}S_\alpha[x] = 
\up \alpha^{-1}\alpha(x) = \up x$, 
\begin{align*}
x \rel{(S_{\alpha^{-1}} \star S_\alpha)} z & \iff (\forall U \in \L\G)\,(S_{\alpha^{-1}}S_\alpha[x] \subseteq U) \Longrightarrow z \in U) \\ & \iff (\forall U \in \L\G)\,(\up x \subseteq U \Longrightarrow z \in U) \iff x \le z,
\end{align*}
where the final equivalence follows from \cref{GvG up x}. Thus, $S_{\alpha^{-1}} \star S_\alpha = {\le_X}$. Similarly, $T_{\beta^{-1}} \star T_\beta = {\le_Y}$, so $(S_{\alpha^{-1}}, T_{\beta^{-1}}) \star (S_\alpha, T_\beta) = (\le_X, \le_Y)$. A parallel argument yields that $(S_\alpha, T_\beta) \star (S_{\alpha^{-1}}, T_{\beta^{-1}}) = (\le_{X'}, \le_{Y'})$. Therefore, $(S_\alpha, T_\beta)$ is a $\GvG$-isomorphism.
\end{proof}

\begin{lemma} \label{lem: distributive joins in LG}
Let $\G = (X, R, Y)$ be a GvG-space. If $U = U_1 \vee \dots \vee U_n$ is a join in $\L\G$, then $U$ is a distributive join iff $U = U_1 \cup \dots \cup U_n$.
\end{lemma}

\begin{proof}
One implication is clear since if a finite union of elements of $\L\G$ is in $\L\G$, then it is the join in $\L\G$ and it is clearly distributive. For the other, 
set $V = U_1 \cup \dots \cup U_n$, so $U = \Box\bd V$ by \cref{rem: lattice ops for LG}. We show that $U=V$. If not, 
then there is $x \in U-V$. By \cref{x notin U GvG}, we may assume that $x \in X_0$. By definition of $X_0$, there is $y \in Y$ with $x \in \max R^{-1}[y]$. Therefore, $R^{-1}[y] \cap \up x = \{x\}$, so $R^{-1}[y] \cap \up x \cap V = \varnothing$. Since $V$ is clopen and $\up x = \bigcap \{ W \in \L\G : x \in W\}$ by \cref{GvG up x}, compactness of $X$ implies that there is $W \in \L\G$ with $x \in W$ and $R^{-1}[y] \cap W \cap V = \varnothing$. Thus, $y \in R[x]$ and $y \not\in R[W \cap V]$, so $x \notin \Box\bd(W \cap V)$. On the other hand, since $U$ is a distributive join,
\begin{align*}
W \cap U &= (W \cap U_1) \vee \dots \vee (W \cap U_n) = \Box\bd((W \cap U_1) \cup \dots \cup (W \cap U_n)) \\
&= \Box\bd (W \cap V).
\end{align*}
This is a contradiction since $x \in W \cap U$ and $x \notin \Box\bd(W \cap V)$. Therefore, $U = U_1 \cup \dots \cup U_n$.
\end{proof}

\begin{lemma} \plabel{lem: zeta G}
Let $\G = (X, R, Y)$ be a GvG-space and set $\U = \UU\HH(\G)$. 
\begin{enumerate} 
\item \label[lem: zeta G]{alpha and beta for zeta} There are $\Pries$-isomorphisms $\alpha \colon X \to \LCUp$ and $\beta \colon Y \to \RCUp$ satisfying $x\rel{R}y$ iff $\alpha(x) \rel{R_\U} \beta(y)$. 
\item \label[lem: zeta G]{zeta is an iso} There is a $\GvG$-isomorphism $\zeta_\G \colon \G \to \GG\DD\UU\HH(\G)$ given by $\zeta_\G = (S_\alpha, T_\beta)$ in the notation of Lemma~\emph{\ref{lem: isomorphisms in GvG}}.
\end{enumerate}
\end{lemma}

\begin{proof}
\eqref{alpha and beta for zeta} Since $\U := (Z, \le_1, \le_2)$ is an Urquhart space, $\LC\U$ and $\RC\U$ are coherent lattices by \cref{LC RC coherent}, and hence $\LCUp$ and $\RCUp$ are Priestley spaces by \cref{lem: Xa and Ya are Priestley}.
For $x \in X$ and $y \in Y$ define
\[
\alpha(x) = (\up x \times Y) \cap Z \quad\textrm{and}\quad
\beta(y) = (X \times \up y) \cap Z.
\]
By \cref{GvG up x}, $\up x = \bigcap \{ U \in \L\G, x \in U\}$, which shows that
\[
\alpha(x) = \bigcap \{ (U \times Y) \cap Z : U \in \L\G, x \in U \}, 
\] 
so is an intersection from $\L\U$ by \cref{LG from LU}. Thus, ${\alpha(x) \in \LC\U}$. 
We show that $\alpha(x) \in \LCUp$. Let $C_1, \dots, C_n \in \L\U$ such that $C := C_1 \wedge \dots \wedge C_n$ is a distributive meet in $\LC\U$ with $C \le \alpha(x)$. By \cref{LG from LU}, there are $U, U_1,\dots, U_n \in \L\G$ with $C_i = (U_i \times Y) \cap Z$ for each $i$ and $C = (U \times Y) \cap Z$. Since $C$ is a distributive meet in $\LC\U$, it is a distributive join in $\L\U$. Because the map $\L\G \to \L\U$ sending $U$  to $(U \times Y) \cap Z$ is an isomorphism by \cref{lem: LH = LG,lem: LU = LH},  $U = U_1 \vee \dots \vee U_n$ is a distributive join in $\L\G$. Therefore, $U = U_1 \cup \dots \cup U_n$ by \cref{lem: distributive joins in LG}. From $C \le \alpha(x)$ we get $\alpha(x) \subseteq C$, so $\alpha(x) \subseteq (U \times Y) \cap Z$, and hence $x \in U$ by \cref{tech 2}. Therefore, $x \in U_i$ for some $i$, so
\[
\alpha(x) = (\up x \times Y) \cap Z \subseteq (U_i \times Y) \cap Z = C_i,
\]
which yields $C_i \le \alpha(x)$. Thus, $\alpha(x) \in \LCUp$.

It is clear that $\alpha$ is order-preserving. To see that $\alpha$ is order-reflecting, suppose that $x \not\le x'$. By \cref{GvG up x}, there is $U \in \L\G$ with $x \in U$ and $x' \notin U$. We have $\alpha(x) \subseteq (U \times Y) \cap Z$ and $\alpha(x') \not\subseteq (U \times Y) \cap Z$ by \cref{tech 2}. Therefore, $\alpha(x')\not\subseteq\alpha(x)$, so $\alpha(x)\not\le\alpha(x')$, and hence $\alpha$ is order-reflecting. 

We next show that $\alpha$ is onto. Let $D \in \LCUp$. Then $D$ is an intersection from $\L\U$ by \cref{D is intersection}, and so $D = (C \times Y) \cap Z$ with $C$ an intersection from $\L\G$ by \cref{LG from LU}. Set 
\[
\mathcal{F} = \{ U \in \ClopUp(X) : C \subseteq U\}.
\]
Then $\mathcal{F}$ is a filter. We show that $\mathcal{F}$ is prime. By \cref{GvG LG basis}, each clopen upset of $X$ is a finite union from $\L\G$. Therefore, it suffices to show that if $U_1, \dots, U_n \in \L\G$ with $U_1 \cup \dots \cup U_n \in \mathcal{F}$, then some $U_i \in \mathcal{F}$. When this occurs, $C \subseteq U_1 \cup \dots \cup U_n$. Since $C$ is an intersection from $\L\G$, compactness of $X$ implies that there is $V \in \L\G$ with $C \subseteq V$ and $V \subseteq U_1 \cup \dots \cup U_n$. We have $V = (V\cap U_1) \cup \dots \cup (V \cap U_n)$, which is a distributive join in $\L\G$ by \cref{lem: distributive joins in LG}. Thus,
\[
(V \times Y) \cap Z = ([(V \cap U_1) \times Y] \cap Z) \cup \dots \cup ([(V \cap U_n) \times Y] \cap Z)
\]
is a distributive join in $\L\U$. Since $D = (C \times Y) \cap Z \subseteq (V \times Y) \cap Z$, in $\LC\U$ we have
\[
([(V \cap U_1) \times Y] \cap Z) \wedge \dots \wedge ([(V \cap U_n) \times Y] \cap Z) \le D.
\]
Because $D \in \LCUp$, $[(V \cap U_i) \times Y] \cap Z \le D$ for some $i$. Therefore,
\[
(C \times Y) \cap Z \subseteq ((V \cap U_i) \times Y) \cap Z \subseteq (U_i \times Y) \cap Z,
\]
so $C \subseteq U_i$ by \cref{tech 2}, and hence $U_i \in \mathcal{F}$. Thus, $\mathcal{F}$ is a prime filter. By Priestley duality, there is $x \in X$ with $\mathcal{F} = \{ U \in \ClopUp(X) : x \in U \}$. Therefore, $C = \bigcap \mathcal{F} = \up x$, so  
$
D = (C \times Y) \cap Z = (\up x \times Y) \cap Z = \alpha(x),
$
and hence $\alpha$ is onto.

To see that $\alpha$ is continuous, by Priestley duality  
it suffices to show that $\alpha^{-1}(V) \in \ClopUp(X)$ for each  $V \in \ClopUp(\LCUp)$. 
By \cref{GvG LG basis}, it is enough to assume that $V \in \L\overline{\G}$, where we recall that $\overline{\G} = \GG\DD\UU\HH(\G)$. By \cref{LG from LU}, we may assume that $V = \up C \cap \LCUp$, where $C = (U \times Y) \cap Z$ for some $U \in \L\G$. 
Therefore, 
\begin{align*}
\alpha^{-1}(V) &= \alpha^{-1}(\up C \cap \LCUp ) = \{ x \in X : C \le \alpha(x) \} \\
&= \{ x \in X : \alpha(x) \subseteq C \} = \{ x \in X : \up (x \times Y) \cap Z \subseteq (U \times Y) \cap Z \} \\
&= \{ x \in X : x \in U \} = U,	
\end{align*}
where the fifth equality holds by \cref{tech 2}. Thus, $\alpha$ is continuous. It is then a homeomorphism since it is a continuous bijection between compact Hausdorff spaces. Consequently, $\alpha$ is a $\Pries$-isomorphism. Similar arguments show that $\beta$ is a $\Pries$-isomorphism. 

Finally, we show that $x \rel{R} y$ iff $\alpha(x) \rel{R_\U} \beta(y)$ for each $x \in X, y \in Y$. First suppose $x \rel{R} y$. By \cref{maximal pairs}, there are $x' \ge x$ and $y' \ge y$ with $(x', y') \in Z$, so $\alpha(x) \cap \beta(y) \ne \varnothing$, and hence $\alpha(x) \rel{R_\U} \beta(y)$. Conversely, if $\alpha(x) \rel{R_\U} \beta(y)$, then $\alpha(x) \cap \beta(y) \ne \varnothing$. Therefore, by \cref{maximal pairs}, 
there are $x' \ge x$ and $y' \ge y$ with $(x', y') \in Z$. Since $(x', y') \in Z$ implies that $x' \rel{R} y'$, we then have $x \rel{R} y$ by \cref{def: GvG-space}(\ref{GvG 1},\ref{GvG 2}).  

\eqref{zeta is an iso} By \eqref{alpha and beta for zeta} and \cref{lem: isomorphisms in GvG}, $\zeta_\G$ is a $\GvG$-isomorphism.
\end{proof}
\begin{proposition} \label{prop:  zeta : 1 to GDUH}
There is a natural isomorphism $\zeta \colon 1_{\GvG} \to \GG\DD\UU\HH$.
\end{proposition}

\begin{proof}
Let $\G = (X, R, Y) \in \GvG$. Set $\U = \UU\HH(\G) := (Z, \le_1, \le_2)$. Also, set $\overline{\G} = \GG\DD\UU\HH(\G)$, so $\overline{\G} = (\LCUp, (R_\U)_p, \RCUp)$. By \cref{lem: zeta G}, there is a $\GvG$-isomorphism $\zeta_\G = (S_\alpha, T_\beta)$, where $\alpha \colon X \to \LCUp$ and $\beta \colon Y \to \RCUp$ are $\Pries$-isomorphisms given by
\[
\alpha(x) = (\up x \times Y) \cap Z \quad\textrm{and}\quad \beta(y) = (X \times \up y) \cap Z
\]
for all $x \in X$ and $y \in Y$. The relation $S_\alpha \subseteq X \times \LCUp$ is defined by $x\rel{S_\alpha} D$ if $\alpha(x) \le D$. Since $\alpha$ is a $\Pries$-isomorphism, we may write $D = \alpha(z)$ for some $z \in X$, and so 
\[
x \rel{S_\alpha}\alpha(z) \iff \alpha(x) \le \alpha(z) \iff x \le z.
\]
The relation $T_\beta$ is given similarly. 
To show that $\zeta$ is natural, let $(S, T) \colon \G \to \G'$ be a $\GvG$-morphism and set $(\overline{S}, \overline{T}) = \GG\DD\UU\HH(S,T)$. If $(P, Q) = \UU\HH(S,T) \subseteq Z \times Z'$, then
\[
(x, y)\rel{P}(x', y') \iff x\rel{S} x' \quad\textrm{and}\quad (x, y) \rel{Q} (x', y') \iff y \rel{T} y'. 
\]
Set $(f_{PQ}, g_{PQ}) = \DD(P,Q)$. Then $(\overline{S}, \overline{T}) = \GG(f_{PQ}, g_{PQ})$, so $\overline{S} = S_{f_{PQ}}$ and $\overline{T} = T_{g_{PQ}}$ (see \cref{prop: G on morphisms}), where $f_{PQ} \colon \LC\U \to \LC\U'$ and $g_{PQ} \colon \RC\U \to \RC\U'$ are given by 
\[
f_{PQ}(D) = \bigcap \{ C' \in \L\U' : D \subseteq \Box_PC'\}  \quad\textrm{and}\quad 
g_{PQ}(E) = \bigcap \{ \phi C' : C' \in \L\U', E \subseteq \Box_Q \phi' C' \}
\]
for each $D \in \LC\U$ and $E \in \RC\U$ (see \cref{def: fPQ and gPQ}). For $U \in \L\G$, set $\overline{U} = (U \times Y) \cap Z$. 
By \cref{LG from LU}, $\L\U = \{ \overline{U} : U \in \L\G\}$ 
and
$
\L\overline{\G} = \{ \up \overline{U} \cap \LCUp : U \in \L\G\}. 
$
Let $x \in X$. Then, by \cref{tech 2},
\[
\alpha(x) \in \up \overline{U} \iff \overline{U} \le (\up x \times Y) \cap Z \iff (\up x \times Y) \cap Z \subseteq \overline{U} \iff x \in U.
\]
Therefore, $\up\overline{U} \cap \LCUp = \alpha[U]$, and so
\begin{equation}
\L\overline{\G} = \{ \alpha[U] : U \in \L\G\}. \label{eqn: L overline G}
\end{equation}
We show that $\overline{S} \star S_\alpha = S_{\alpha'} \star S$ and $\overline{T} \star T_\beta = T_{\beta'} \star T$. This will show that the diagram below commutes, and hence that $\zeta$ is natural. 
\[
\begin{tikzcd}[column sep = 5pc]
\G \arrow[r, "(S{,}T)"] \arrow[d, "\zeta_{\G}"'] & \G' \arrow[d, "\zeta_{\G'}"] \\
\overline{\G} \arrow[r, "(\overline{S}{,}\overline{T})"'] & \overline{\G'}
\end{tikzcd}
\]
We verify the first equality by using the following claim. The proof of the second equality is similar.
\begin{claim} \plabel{GvG claim}
Let $x \in X$, $x' \in X'$, $U \in \L\G$, and $U' \in \L\G'$.
\begin{enumerate}
\item \label[GvG claim]{GvG claim 1} $\alpha(x) \rel{\overline{S}} \alpha'(x')$ iff $x \rel{S} x'$.
\item \label[GvG claim]{GvG claim 2} $\Box_{\overline{S}}\alpha'[U'] = \alpha[\Box_SU']$.
\item \label[GvG claim]{GvG claim 3} $\Box_{S_\alpha}\alpha[U] = U$.
\end{enumerate}
\end{claim}

\begin{proofclaim}
\eqref{GvG claim 1} We first show that
$\Box_P\overline{U'} = (\Box_SU' \times Y) \cap Z$. To see this, $(x, y) \in \Box_P \overline{U'}$ iff whenever $(x', y') \in Z'$ and $x \rel{S} x'$, then $(x', y') \in \overline{U'}$. Therefore, $(x, y) \in \Box_P \overline{U'}$ iff $x\rel{S} x'$ implies $x' \in U'$. Thus, $\Box_P \overline{U'} = (\Box_SU' \times Y) \cap Z$. 

Now let $x \in X$ and $x' \in X'$. Then
\begin{align*}
\alpha(x) \rel{\overline{S}} \alpha'(x') &\iff f_{PQ}\alpha(x) \le \alpha'(x') \iff \alpha'(x') \subseteq f_{PQ}\alpha(x). 
\end{align*}
Since $\Box_P\overline{U'} = (\Box_SU' \times Y) \cap Z$, we have 
\begin{align*}
f_{PQ}\alpha(x) &= \bigcap \{ \overline{U'} \in \L\U' : \alpha(x) \subseteq \Box_P \overline{U'} \} \\
&= \bigcap \{\overline{U'} : U' \in \L\G', (\up x \times Y) \cap Z \subseteq (\Box_SU' \times Y) \cap Z \} \\
&= \bigcap \{ \overline{U'} : U' \in \L\G', x \in \Box_S U'\},
\end{align*}
where the last equality holds by \cref{tech 2}. 
Because $\alpha'(x') \subseteq \overline{U'}$ iff $x' \in U'$, 
\begin{align*}
\alpha'(x') \subseteq f_{PQ}\alpha(x) &\iff \alpha'(x') \subseteq \overline{U'} \ \  \forall U' \in \L\G' \mbox{ with } x \in \Box_S U' 
\\
&\iff x' \in U' \ \  \forall U' \in \L\G' \mbox{ with } x \in \Box_S U' \\
&\iff x \rel{S} x' ,
\end{align*}

where the last equivalence holds by \cref{GvG S 1}. Thus, $\alpha(x) \rel{\overline{S}} \alpha'(x')$ iff $x \rel{S} x'$.

\eqref{GvG claim 2} Since $\alpha'$ is a $\Pries$-isomorphism, by \eqref{GvG claim 1} we have
\begin{align*}
\Box_{\overline{S}}\alpha'[U'] &= \{ \alpha(x) : \alpha(x) \rel{\overline{S}} \alpha'(x) \Longrightarrow \alpha'(x') \in \alpha'[U'] \} \\
&= \{ \alpha(x) : x \rel{S} x' \Longrightarrow x' \in U' \}= \alpha[\Box_SU'].
\end{align*}

\eqref{GvG claim 3} We have
\begin{align*}
\Box_{S_\alpha}\alpha[U] &= \{ x \in X : S_\alpha[x] \subseteq \alpha[U] \} = \{ x : \up \alpha(x) \subseteq \alpha[U]\} \\ 
&= \{ x : \alpha(x) \in \alpha[U]\} = U,
\end{align*}
where the third equality holds since $\alpha[U]$ is an upset.\qed
\end{proofclaim}

By definition, $x \rel{(\overline{S}\star S_\alpha)} \alpha'(x')$ iff $(\forall V \in \L\overline{\G'})\,(x \in \Box_{S_\alpha}\Box_{\overline{S}} V \Longrightarrow \alpha'(x') \in V)$. By \cref{eqn: L overline G}, $\L\overline{\G'} = \{ \alpha'[U'] : U' \in \L\G' \}$. Therefore, 
\[
x \rel{(\overline{S}\star S_\alpha)} \alpha'(x') \iff (\forall U' \in \L\G')\,(x \in \Box_{S_\alpha}\Box_{\overline{S}} \alpha'[U'] \Longrightarrow \alpha'(x') \in \alpha'[U']).
\]  
By \cref{GvG claim}(\ref{GvG claim 2},\ref{GvG claim 3}),
\[
\Box_{S_\alpha}\Box_{\overline{S}}\alpha'[U'] = \Box_{S_\alpha}\alpha[\Box_SU'] = \Box_SU'.
\]
This yields that 
\[
x \rel{(\overline{S}\star S_\alpha)} \alpha'(x') \iff (\forall U' \in \L\G')\,(x \in \Box_SU' \Longrightarrow x' \in U') \iff x \rel{S} x',
\]
where the last equivalence follows from \cref{GvG S 1}.
On the other hand,
\begin{align*}
 x \rel{(S_{\alpha'}\star S)} \alpha'(x') &\iff (\forall U' \in \L\G')\,(x \in \Box_S\Box_{S_{\alpha'}}\alpha'[U'] \Longrightarrow\alpha'(x') \in \alpha'[U']) \\
&\iff (\forall U' \in \L\G')\,(x \in \Box_S U' \Longrightarrow x' \in U') \\
&\iff x \rel{S} x'.
\end{align*}
Therefore, $\overline{S} \star S_\alpha = S_{\alpha'} \star S$. A similar argument shows that $\overline{T} \star T_\beta = T_{\beta'} \star T$. Thus, $\zeta$ is natural, and hence is a natural isomorphism.
\end{proof}

\subsection{The natural isomorphisms \texorpdfstring{$\eta \colon 1_\Hg \to \HH\GG\DD\UU$}{Lg} and \texorpdfstring{$\theta \colon 1_\Urq \to \UU\HH\GG\DD$}{Lg}}

\phantom{M}
\medskip

In this subsection we show that there are natural isomorphisms $\eta \colon 1_\Hg \to \HH\GG\DD\UU$ and $\theta \colon 1_\Urq \to \UU\HH\GG\DD$, which together with the natural isomorphisms of the previous two subsections yield our desired equivalences (see \cref{thm: categories are all equivalent}). The proofs in this subsection are largely similar to those in \cref{subsec: GvG}, so we mostly skip them and only indicate the differences. 
We start with $\eta \colon 1_\Hg \to \HH\GG\DD\UU$ for which we have the following analogues of Lemmas~\ref{lem: isomorphisms in GvG}, \ref{lem: zeta G}, and \cref{prop:  zeta : 1 to GDUH}.

\begin{lemma} \label{lem: isomorphisms in Hg}
Let $\H = (X, R, Y), \H' = (X', R', Y') \in \Hg$ and suppose that $\alpha \colon X \to X'$ and $\beta \colon Y \to Y'$ are  order-homeomorphisms which satisfy $x \rel{R} y$ iff $\alpha(x) \rel{R'} \beta(y)$. If $(S_\alpha, T_\beta)$ is defined by $x \rel{S_\alpha} x'$ if $\alpha(x) \le x'$ and $y \rel{T_\beta} y'$ if $\beta(y) \le y'$, then $(S_\alpha, T_\beta) \colon \H \to \H'$ is an $\Hg$-isomorphism.
\end{lemma}

\begin{lemma} \plabel{lem: eta H}
Let $\H = (X, R, Y) \in \Hg$ and set $\U = \UU(\H)$. \begin{enumerate}
\item \label[lem: eta H]{alpha and beta for eta} There are order-homeomorphisms $\alpha \colon X \to \LCUO$ and $\beta \colon Y \to \RCUO$ satisfying $x\rel{R}y$ iff $\alpha(x) \rel{R_\U} \beta(y)$. 
\item \label[lem: eta H]{eta is an iso} There is an $\Hg$-isomorphism $\eta_\H \colon \H \to \HH\GG\DD\UU(\H)$ given by $\eta_\H = (S_\alpha, T_\beta)$ in the notation of Lemma~\emph{\ref{lem: isomorphisms in Hg}}. 
\end{enumerate}
\end{lemma}

\begin{proposition} \label{prop: eta : 1 to HGDU}
There is a natural isomorphism $\eta \colon 1_{\Hg} \to \HH\GG\DD\UU$.
\end{proposition}

While the proofs of these three results are largely the same as the corresponding proofs in \cref{subsec: GvG}, there's additional argument needed in \cref{alpha and beta for eta} to guarantee that $\alpha$ and $\beta$ are well defined and onto. Recalling that $\U=(Z,\le_1,\le_2)$, for $x \in X$ and $y \in Y$ define
\[
\alpha(x) = (\up x \times Y) \cap Z \quad \mbox{and} \quad
\beta(y) = (X \times \up y) \cap Z.
\]
We only show that $\alpha$ is well defined and onto since the argument for $\beta$ is similar. By \cref{Hg up x is intersection of As}, $\up x = \bigcap \{A \in \L\H : x \in A\}$, which shows that
\[
\alpha(x) = \bigcap \{ (A \times Y) \cap Z : x \in A \in \L\H\},
\]
so is an intersection from $\L\U$ by \cref{lem: LU = LH}. Thus,  $\alpha(x) \in \LC\U$. It is left to show that $\alpha(x) \in \LCUO$.
By \cref{maximal pairs 2}, there is $y \in Y$ with $(x, y) \in Z$. Therefore, ${\alpha(x) \rel{R_\U} \beta(y)}$. Suppose there is $D \in \LC\U$ with $\alpha(x) \le D$ and $D \rel{R_\U} \beta(y)$, so $D \cap \beta(y) \ne \varnothing$. Because the order on $\LC\U$ is reverse inclusion, $D \subseteq \alpha(x)$. If $(x', y') \in D \cap \beta(y)$, then $y \le y'$ and $x \le x'$. Thus, $(x, y) \le_1 (x', y')$ and $(x, y) \le_2 (x', y')$, so $(x, y) = (x', y')$ because $\U$ is doubly ordered. Consequently, $(x, y) \in D$. By \cref{lem: LU = LH,D is intersection}, $D = (B \times Y) \cap Z$ with $B$ an intersection from $\L\H$. Therefore, $x \in B$, so $\alpha(x) \subseteq D$. Thus, $D = \alpha(x)$, and hence $\alpha(x)$ is $\beta(y)$-maximal. Consequently, $\alpha(x) \in \LCUO$.

To see that $\alpha$ is onto, let $D \in \LCUO$. Then there is $E \in \RCUO$ such that $D$ is $E$-maximal. This implies that $D \rel{R_\U} E$, so $D \cap E \ne \varnothing$. Let $(x, y) \in D \cap E$. As we saw in the previous paragraph, there is $B \subseteq X$ with $D = (B \times Y) \cap Z$ and $B$ is an intersection from $\L\H$. Since $x \in B$ and $B$ is an upset, $\up x \subseteq B$, so $\alpha(x) \subseteq D$. If this is a proper inclusion, then $D < \alpha(x)$, so $\alpha(x) \cap E = \varnothing$ because $D$ is $E$-maximal. This is false since $(x, y) \in \alpha(x) \cap E$. Therefore, $D = \alpha(x)$. Consequently, $\alpha \colon X \to \LCUO$ is onto.

We next turn our attention to $\theta \colon 1_\Urq \to \UU\HH\GG\DD$. We have the following analogues of Lemmas~\ref{lem: isomorphisms in GvG}, \ref{lem: zeta G}, and \cref{prop:  zeta : 1 to GDUH}. 

\begin{lemma} \label{lem: isomorphisms in Urq}
Let $\U = (Z, \le_1, \le_2), \U' = (Z', \le_1', \le_2') \in \Urq$ and $h \colon Z \to Z'$ be a homeomorphism and order-isomorphism with respect to both $\le_1$ and $\le_2$. Define $P_h, Q_h \subseteq Z \times Z'$ by $z \rel{P_h} z'$ if $h(z) \le_1 z'$ and $z \rel{Q_h} z'$ if $h(z) \le_2 z'$. Then $(P_h, Q_h)$ is an $\Urq$-isomorphism.
\end{lemma}

\begin{lemma} \plabel{lem: theta U}
Let $\U := (Z, \le_1, \le_2) \in \Urq$ and $\UU\HH\GG\DD(\U) := 
(\overline{Z}, \le_1, \le_2)$. 
\begin{enumerate}
\item \label[lem: theta U]{h for theta} There is $h \colon Z \to \overline{Z}$ that is a homeomorphism and order-isomorphism with respect to both $\le_1$ and $\le_2$.
\item \label[lem: theta U]{theta is an iso} There is an $\Urq$-isomorphism $\theta_\U \colon \U \to \UU\HH\GG\DD(\U)$ given by $\theta_\U = (P_h, Q_h)$ in the notation of Lemma~\emph{\ref{lem: isomorphisms in Urq}}.
\end{enumerate}
\end{lemma}

\begin{proposition} \label{prop: theta : 1 to UHGD}
There is a natural isomorphism $\theta \colon 1_\Urq \to \UU\HH\GG\DD$.
\end{proposition}

Since in the case of Urquhart spaces we work with  two quasi-orders instead of a binary relation, the corresponding proofs simplify. For the proof of \cref{lem: isomorphisms in Urq}, since the Galois connections $\phi, \psi$ are described by the two quasi-orders, a quick argument gives that for each $C' \in \L\U'$ we have $\phi \Box_{P_h} C' = \Box_{Q_h} \phi' C'$ and $\psi\phi \Box_{P_h} C' = \Box_{P_h} C'$. Thus, $\Box_{P_h} C' \in \L\U$, verifying that $(P_h,Q_h)$ satisfies \cref{Urq Box}. The proof that $(P_h,Q_h)$ satisfies the remaining conditions of \cref{def: Urquhart morphisms} is similar to that of $(S_\alpha, T_\beta)$ being a $\GvG$-morphism given in \cref{lem: isomorphisms in GvG}.

For the proof of \cref{h for theta}, define $h \colon Z\to\overline{Z}$ by $h(z) = (\up_1 z, \up_2 z)$.
The argument that $h$ is well defined and onto is similar to that of \cref{alpha and beta for eta}, and the rest to that of \cref{alpha and beta for zeta}. Lastly, the proof of \cref{prop: theta : 1 to UHGD} is similar to that of \cref{prop:  zeta : 1 to GDUH}.

Finally, putting \cref{prop: epsilon : 1 to DUHG,prop:  zeta : 1 to GDUH,prop: eta : 1 to HGDU,prop: theta : 1 to UHGD} together yields the main result of this section:

\begin{theorem} \label{thm: categories are all equivalent}
The categories $\HD$, $\GvG$, $\Hg$, and $\Urq$ are equivalent.
\end{theorem}


\section{Overview of the resulting dualities for lattices} \label{sec: conclusion}

We have provided explicit equivalences between 
various categories that were proposed in the literature as  
extensions of Priestley duality to arbitrary (not necessarily distributive) bounded lattices. 
This included developing the duals of bounded lattice homomorphisms where they were lacking. As a result, we arrive at the following diagram in \cref{Sec 12 figure}, 
where the horizontal arrows represent equivalence of categories and the upward arrows dual equivalence. 
In this section we show that each triangle in the diagram commutes up to natural isomorphism.

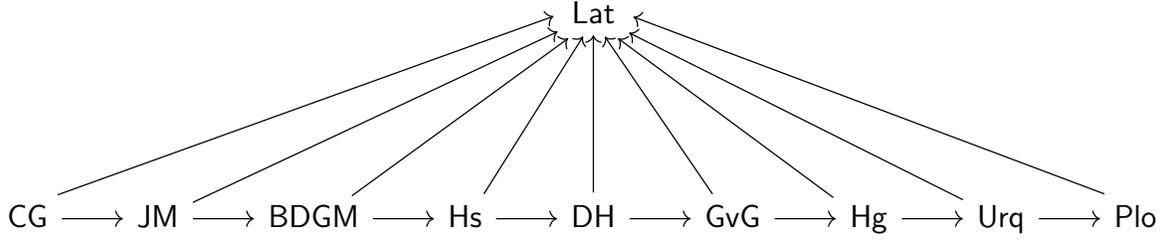
\begin{figure}[ht]
\[
\begin{tikzcd}[column sep = 2em, row sep = 5pc]
&&&& \Lat &&&& \\
\CG \arrow[urrrr, shift left = .9ex] \arrow[r] & \JML \arrow[urrr] \arrow[r] & \BDGM \arrow[urr, shift right = .3ex]\arrow[r] & \HsL \arrow[ur, shift right = .2ex] \arrow[r] & \HD \arrow[u] \arrow[r] & \GvG \arrow[ul, , shift left = .2ex] \arrow[r] & \Hg \arrow[ull, shift left = .3ex] \arrow[r] & \Urq \arrow[ulll] \arrow[r] & \Plo \arrow[ullll, shift right = .9ex]
\end{tikzcd}
\]
\caption{Various dualities for $\Lat$} \label{Sec 12 figure}
\end{figure}

By \cref{thm: CG = JM}, there is an equivalence of categories between $\CG$ and $\JML$. 
\cref{rem: functor JM to Lat} shows that there is a functor $\LL_\JML \colon \JML \to \Lat$ which sends $X \in \JML$ to $\KOF(X)$ and a $\JML$-morphism $f \colon X \to Y$ to $f^{-1} \colon \KOF(Y) \to \KOF(X)$. By \cref{rem: CG = Lat}, there is a functor $\LL_\CG \colon \CG \to \Lat$ which sends $\C \in \CG$ to $\L\C$ and a $\CG$-morphism $S \colon \C_1 \to \C_2$ to $\Box_S \colon L\C_2 \to L\C_1$. 

\begin{theorem} \label{prop: CG and JML}
The following diagram commutes up to natural isomorphism. 
\[
\begin{tikzcd}[column sep = 3em]
& \Lat & \\
\CG \arrow[ur, "\LL_\CG"] \arrow[rr, "\mathbb{H}"']&& \JML \arrow[ul, "\LL_\JML"']
\end{tikzcd}
\]
\end{theorem}

\begin{proof}
Let $\C = (X,\K) \in \CG$. Then $$\LL_{\CG}(\C) = \L\C = \{ -U : U \in \K \}$$ and $$\LL_{\JML}\H(\C) = \{ \Psi_U : U \in \K \} = \{ \{ H \in \H(\C) : U \subseteq H \} : U \in \K \}.$$ Define $\alpha_{\C} \colon \L\C \to \KOF(\H(\C))$ by $\alpha_{\C}(-U) = \Psi_U$ for each $U\in\K$. It is straightforward to see that $$-U \subseteq -V \iff \alpha_{\C}(-U) \subseteq \alpha_{\C}(-V)$$ for all $U,V\in\K$. Moreover, $\alpha_{\C}$ is onto by \cref{5.7 of CG}. Thus, $\alpha_{\C}$ is an order-isomorphism, hence a $\Lat$-isomorphism. To see that $\alpha \colon \LL_{\CG} \to \LL_{\JML} \circ \HH$ is natural, we show that the following diagram commutes.
\[
\begin{tikzcd}[column sep = 5pc]
\L\C' \arrow[r, "\Box_S"] \arrow[d, "\alpha_{\C'}"'] & \arrow[d, "\alpha_\C"] \L\C\\
\KOF(\H(\C')) \arrow[r, "f_S^{-1}"'] & \KOF(\H(\C)) 
\end{tikzcd}
\]
Let $U\in\K'$. Recalling from \cref{sec: CG} that $f_S(H) = \bigcup \{ V \in \K' : S^{-1}[V] \subseteq H\}$ for each $H \in \H(\C)$, we have
\begin{align*}
\alpha_{\C}\Box_S(-U) &= \alpha_{\C}(-S^{-1}[U]) = \Psi_{S^{-1}[U]} 
= \{ H \in \H(\C) : S^{-1}[U]  \subseteq H \} \\
&= \{ H \in \H(\C) : U \subseteq f_S(H) \} = \{ H \in \H(\C) : f_S(H) \in \Psi_U \} \\
&= f_S^{-1}[\Psi_U].
\end{align*}
Thus, $\alpha$ is a natural isomorphism, and hence the triangle above commutes up to natural isomorphism.
\end{proof}

By \cref{JML = BDGM}, there is an isomorphism of categories between $\JML$ and $\BDGM$, which we denote by $\mathbb{B}$. 
The functor $\mathbb{B}$ sends $X \in \JML$, which is an algebraic lattice with the Scott topology, to $(X, \lambda)$, where $\lambda$ is the Lawson topology. 
By \cref{rem: BDGMS duality with SLat}, there is a functor $\LL_\BDGM \colon \BDGM \to \Lat$ which sends $X \in \BDGM$ to $\CLF(X)$ and a $\BDGM$-morphism $f \colon X \to Y$ to $f^{-1} \colon \CLF(Y) \to \CLF(X)$. 
Since $\KOF(X) = \CLF(X)$ by \cref{rem: CLF = KOF}, we have $\LL_{\BDGM} \circ \mathbb{B} = \LL_{\JML}$, yielding the next result.

\begin{theorem} \label{prop: JML and BDGM}
The following diagram commutes. 
\[
\begin{tikzcd}[column sep = 3em]
& \Lat & \\
\JML \arrow[ur, "\LL_\JML"] \arrow[rr, "\mathbb{B}"']&& \BDGM \arrow[ul, "\LL_\BDGM"']
\end{tikzcd}
\]
\end{theorem}

By \cref{thm: CohLatL = HsL}, there is an isomorphism of categories between $\BDGM$ and $\HsL$, which we denote by $\HHs$. 
By \cref{Hs = Lat}, there is a functor $\LL_\HsL \colon \HsL \to \Lat$ which sends $X \in \HsL$ to $\CLF(X)$ and an $\HsL$-morphism $f \colon X \to Y$ to $f^{-1} \colon \CLF(Y) \to \CLF(X)$. The next result is then immediate from these facts.

\begin{theorem} \label{prop: BDGM and HsL}
The following diagram commutes. 
\[
\begin{tikzcd}[column sep = 3em]
& \Lat & \\
\BDGM \arrow[ur, "\LL_\BDGM"] \arrow[rr, "\HHs"']&& \HsL \arrow[ul, "\LL_\HsL"']
\end{tikzcd}
\]
\end{theorem}

By Theorems~\ref{thm: CohLatL = HsL}, \ref{thm: BDGM = CohLatL}(\ref{BDGM = CohLatL}), and \ref{thm: DH = CohLatL}, there are functors $\HsL \to \BDGM \to \CohLatL \to \HD$. We denote the composition $\HsL \to \HD$ by $\mathbb{DH}$. This functor sends $X \in \HsL$ to $(X, R_X, \OF(X))$ (see \cref{def: R_X}) and an $\HsL$-morphism $f \colon X \to Y$ to
\[
(f, r) \colon (X, R_X, \OF(X)) \to (Y, R_Y, \OF(Y)),
\]
where $r$ is the right adjoint of $f^{-1} \colon \OF(Y) \to \OF(X)$ (see \cref{lem: OF is coherent 3}). By \cref{DH = HsL}, this functor is an equivalence. \cref{rem: functor from DH to Lat} shows that there is a functor $\LL_\HD \colon \HD \to \Lat$ which sends $(X, R, Y) \in \HD$ to $\CLF(X)$ and a $\HD$-morphism $(f, g) \colon (X, R, Y) \to (X', R', Y')$ to $f^{-1} \colon \CLF(X') \to \CLF(X)$. The next result is then immediate.

\begin{theorem} \label{prop: HsL and DH}
The following diagram commutes. 
\[
\begin{tikzcd}[column sep = 3em]
& \Lat & \\
\HsL \arrow[ur, "\LL_\HsL"] \arrow[rr, "\mathbb{DH}"']&& \HD \arrow[ul, "\LL_\HD"']
\end{tikzcd}
\]
\end{theorem}

In order to finish proving the commutativity of \cref{Sec 12 figure}, we need to define functors from each of $\GvG$, $\Hg$, $\Urq$, and $\Plo$ to $\Lat$. We start with the functor from $\GvG$ to $\Lat$. 

\begin{proposition} \label{prop: The functor L for GvG}
There is a functor $\LL_\GvG \colon \GvG \to \Lat$ which sends $\G \in \GvG$ to $\L\G$ and a $\GvG$-morphism $(S, T)$ to $\Box_S$. 
\end{proposition}

\begin{proof}
By \cref{rem: lattice ops for LG}, $\LL_\GvG$ is well defined on objects. If $(S, T) \colon \G \to \G'$ is a $\GvG$-morphism, then $\Box_S \colon \L\G' \to \L\G$ is a well-defined function by \cref{GvG morphism 1}. We show that $\Box_S$ is a $\Lat$-morphism. It is clear that $\Box_S X' = X$. Since binary meet in $\L\G$ is intersection, we also have that $\Box_S(U \wedge V) = \Box_S U \wedge \Box_S V$ for each $U,V\in\L\G$. 
To see that $\Box_S\varnothing = \varnothing$, if $ x \in \Box_S \varnothing$, then $S[x] = \varnothing$. Consequently, $R[x] = \varnothing$ by \cref{GvG morphism 4}. Therefore, $x \in \Box\bd \varnothing$, which is false as $\Box\bd \varnothing = \varnothing$ by \cref{rem: lattice ops for LG}. Thus, $\Box_S \varnothing = \varnothing$. It remains to show that $\Box_S$ preserves binary joins. The inclusion $\Box_SU \vee \Box_SV \subseteq \Box_S(U\vee V)$ is clear since $\Box_S$ is order-preserving. For the other inclusion, let $x \notin \Box_SU \vee \Box_SV$. Since $\Box_SU \vee \Box_SV = \Box\bd(\Box_SU \cup \Box_SV)$, there is $y \in Y$ with $x \rel{R} y$ and $y \notin \bd(\Box_S U \cup \Box_S V)$. By \cref{GvG morphism 1},
\begin{align*}
\bd(\Box_SU \cup \Box_SV) &= \bd\Box_SU \cup \bd\Box_SV = \Diamond_T \bd' U \cup \Diamond_T \bd' V \\
&= \Diamond_T(\bd' U \cup \bd' V) = \Diamond_T(\bd'(U \cup V)).
\end{align*}
Therefore, $y \notin \Diamond_T(\bd'(U \cup V))$. Because $x \rel{R} y$, by \cref{GvG morphism 4} there are $x' \in X'$ and $y' \in Y'$ with $x \rel{S} x'$, $y \rel{T} y'$, and $x' \rel{R'} y'$. If $x \in \Box_S(U \vee V) = \Box_S(\Box'\bd'(U \cup V))$, then $x' \in \Box'\bd'(U \cup V)$, so $y' \in \bd'(U \cup V)$. This implies that $y \in \Diamond_T(\bd'(U \cup V))$, which is false. Thus, $x \notin \Box_S(U \vee V)$. Consequently, $\Box_S$ preserves binary joins, and hence is a $\Lat$-morphism. This shows that $\LL_\GvG$ is well-defined on morphisms. 

By \cref{lem: composition lemma for GvG}, $\LL_\GvG$ preserves composition. If $\G = (X, R, Y) \in \GvG$, then $(\le_X, \le_Y)$ is the identity morphism for $\G$ by \cref{GvG identity}. It is also immediate that $\Box_{\le_X} U = U$ for each $U \in \L\G$.  
Therefore, $\Box_{\le_X}$ is the identity morphism of $\L\G$, completing the proof that $\LL_\GvG$ is a functor. 
\end{proof}

\begin{theorem} \label{prop: DH and GvG}
The following diagram commutes up to natural isomorphism.
\[
\begin{tikzcd}[column sep = 3em]
& \Lat & \\
\HD \arrow[ur, "\LL_\HD"] \arrow[rr, "\GG"']&& \GvG \arrow[ul, "\LL_\GvG"']
\end{tikzcd}
\]
\end{theorem}

\begin{proof}
Let $\D = (X, R, Y) \in \HD$ and $\G = \GG(\D)$, so $\G = (X_p, R_p, Y_p)$. By \cref{lem: LG = CLF}, there is a $\Lat$-isomorphism $\gamma_D \colon \CLF(X) \to \L\G$ given by $\gamma_D(U) = U \cap X_p$ for each ${U \in \CLF(X)}$. Let $(f, g) \colon \D \to \D'$ be a $\HD$-morphism and  $(S_f, T_g) = \GG(f, g)$. By \cref{prop: G on morphisms},
\[
x \rel{S_f} x' \iff f(x) \le x' \quad\textrm{and}\quad y \rel{T_g} y' \iff g(y) \le y'.
\]
To see that $\gamma \colon \LL_{\HD} \to \LL_{\GvG}\circ\GG$ is natural, we show that the following diagram  commutes.
\[
\begin{tikzcd}[column sep = 3em]
\CLF(X') \arrow[r, "f^{-1}"] \arrow[d, "\gamma_{\D'}"'] & \CLF(X) \arrow[d, "\gamma_\D"] \\
\L\G'  \arrow[r, "\Box_{S_f}"']& \L\G
\end{tikzcd}
\]
Let $U \in \CLF(X')$. Then $\Box_{S_f} \gamma_{\D'}(U) = \Box_{S_f} (U \cap X_p')$ while $\gamma_\D(f^{-1}(U)) = f^{-1}(U) \cap X_p$. We show that these are equal. If $x \in f^{-1}(U) \cap X_p$, then $f(x) \in U$, so $S_f[x] = \up f(x) \cap X_p' \subseteq U \cap X_p'$, and hence $x \in \Box_{S_f}(U \cap X_p')$. For the reverse inclusion, suppose that $x \in \Box_{S_f} (U \cap X_p')$, so $S_f[x] \subseteq U \cap X_p'$. By \cref{lem: KOF = LX}, $U$ is a principal upset. Therefore, to show that $f(x) \in U$, by \cref{x is meet from X0} it suffices to show that $\up f(x) \cap X_0' \subseteq U$. Since $X_0' \subseteq X_p'$ (see \cref{lem: X0 the same}), 
\[
\up f(x) \cap X_0' \subseteq \up f(x) \cap X_p' = S_f[x] \subseteq U \cap X_p' \subseteq U.
\]
Thus, $f(x) \in U$, so $\gamma \colon \LL_\HD \to \LL_\GvG \circ \GG$ is a natural isomorphism, and hence the triangle above commutes up to natural isomorphism.
\end{proof}

We next turn our attention to the functor from $\Hg$ to $\Lat$. An argument similar to that of \cref{prop: The functor L for GvG} yields:

\begin{proposition} \label{prop: The functor L for Hg}
There is a functor $\LL_\Hg \colon \Hg \to \Lat$ which sends $\H \in \Hg$ to $\L\H$ and an $\Hg$-morphism $(S, T)$ to $\Box_S$. 
\end{proposition}

\begin{theorem} \label{prop: GvG and Hg}
The following diagram commutes up to natural isomorphism.
\[
\begin{tikzcd}[column sep = 3em]
& \Lat & \\
\GvG \arrow[ur, "\LL_\GvG"] \arrow[rr, "\HH"']&& \Hg \arrow[ul, "\LL_\Hg"']
\end{tikzcd}
\]
\end{theorem}

\begin{proof}
Let $\G = (X, R, Y) \in \GvG$ and $\H = \HH(\G) = (X_0, R_0, Y_0)$. By \cref{lem: LH = LG}, there is a $\Lat$-isomorphism $\delta_\G \colon \L\G \to \L\H$ given by $\delta_\G(U) = U \cap X_0$ for each $U \in \L\G$.
To show that $\delta \colon \LL_\GvG \to \LL_\Hg \circ \HH$ is natural, let $(S, T) \colon \G \to \G'$ be a $\GvG$-morphism and $(S_0, T_0) = \HH(S, T)$. We show that the following diagram  commutes.
\[
\begin{tikzcd}[column sep = 5em]
\L\G' \arrow[r, "\Box_S"] \arrow[d, "\delta_{\G'}"'] & \L\G \arrow[d, "\delta_\G"] \\
\L\H' \arrow[r, "\Box_{S_0}"'] & \L\H
\end{tikzcd}
\]
Let $U' \in \L\G'$. Then $\Box_{S_0}\delta_{\G'}(U') = \Box_{S_0} (U' \cap X_0')$ and $\delta_\G(\Box_S U') = \Box_S U' \cap X_0$. To see that the two are equal, if $x \in \Box_S U' \cap X_0$, then $S[x] \subseteq U'$, so $S_0[x] = S[x] \cap X_0' \subseteq U' \cap X_0'$, and hence $x \in \Box_{S_0}(U' \cap X_0')$. For the reverse inclusion, suppose that $x \in \Box_{S_0} (U' \cap X_0')$. Then $S_0[x] \subseteq U'$. Let $x' \in S[x]$. If $x' \notin U'$, by \cref{x notin U GvG} there is $x_0 \in X_0'$ with $x' \le x_0$ and $x_0 \notin U'$. On the other hand, by \cref{GvG S 2}, $S[x]$ is an upset. Therefore, $x_0 \in S[x] \cap X_0' = S_0[x]$, which implies that $x_0 \in U'$. The obtained contradiction shows that $x' \in U'$. Consequently, $\delta \colon \LL_\GvG \to \LL_\Hg \circ \HH$ is a natural isomorphism, and hence the triangle above commutes up to natural isomorphism.
\end{proof}

We now concentrate on the functor from $\Urq$ to $\Lat$. 

\begin{proposition} \label{prop: The functor L for Urq}
There is a functor $\LL_\Urq \colon \Urq \to \Lat$ which sends $\U \in \Hg$ to $\L\U$ and an $\Urq$-morphism $(P, Q)$ to $\Box_P$. 
\end{proposition}

\begin{proof}
The proof is essentially the same as that of \cref{prop: The functor L for GvG}, but uses \cref{claim: ell preserves finite joins} to show that if $(P, Q)$ is an $\Urq$-morphism, then $\Box_P$ preserves binary joins.
\end{proof}

\begin{theorem} \label{prop: Hg and Urq}
The following diagram commutes up to natural isomorphism.
\[
\begin{tikzcd}[column sep = 3em]
& \Lat & \\
\Hg \arrow[ur, "\LL_\Hg"] \arrow[rr, "\UU"']&& \Urq \arrow[ul, "\LL_\Urq"']
\end{tikzcd}
\]
\end{theorem}

\begin{proof}
Let $\H = (X, R, Y) \in \Hg$ and $\U = \UU(\H) = (Z, \le_1, \le_2)$. By \cref{lem: LU = LH}, there is a $\Lat$-isomorphism $\mu_\H \colon \L\H \to \L\U$, given by $\mu_\H(A) = (A \times Y) \cap Z$. Let $(S, T) \colon \H \to \H'$ be an $\Hg$-morphism and $(P, Q) = \UU(S, T)$. By \cref{thm: U is a functor},
\[
(x, y) \rel{P} (x', y') \iff x \rel{S} x' \quad\textrm{and}\quad (x, y) \rel{Q} (x', y') \iff y \rel{T} y'.
\]
We show that the following diagram  commutes.
\[
\begin{tikzcd}[column sep = 5em]
\L\H' \arrow[r, "\Box_S"] \arrow[d, "\mu_{\H'}"'] & \L\H \arrow[d, "\mu_\H"] \\
\L\U' \arrow[r, "\Box_{P}"'] & \L\U
\end{tikzcd}
\]
Let $A' \in \L\H'$. Then $\Box_P \mu_{\H'}(A') = \Box_P((A' \times Y') \cap Z')$ and $\mu_\H(\Box_S A') = (\Box_S A' \times Y) \cap Z$, which are equal by \cref{eqn: BoxP calculation}. 
Thus, $\mu \colon \LL_\Hg \to \LL_\Urq \circ \UU$ is natural, and hence the triangle above commutes up to natural isomorphism.
\end{proof}

By composing $\PP^{-1} \colon \Plo \to \Urq$ and $\LL_{\Urq} \colon \Urq \to \Lat$, we get a functor $\LL_\Plo \colon \Plo \to \Lat$ which sends $\Pl$ to $\L\Pl$ (see \cref{LU}) and a $\Plo$-morphism $(P, Q)$ to $\Box_P$. The next result is then immediate.

\begin{theorem} \label{theorem: Urq and Plo}
The following diagram commutes.
\[
\begin{tikzcd}[column sep = 3em]
& \Lat & \\
\Urq \arrow[ur, "\LL_\Urq"] \arrow[rr, "\PP"'] && \Plo \arrow[ul, "\LL_\Plo"']
\end{tikzcd}
\]
\end{theorem}

We complete the circle of equivalences with our final result of this section. 

\begin{theorem} \label{prop: Urq and DH}
The following diagram commutes up to natural isomorphism.
\[
\begin{tikzcd}[column sep = 3em]
& \Lat & \\
\Urq \arrow[ur, "\LL_\Urq"] \arrow[rr, "\DD"']&& \HD \arrow[ul, "\LL_\HD"']
\end{tikzcd}
\]
\end{theorem}

\begin{proof}
By \cref{lem: CLF = LU} there is a $\Lat$-isomorphism $\xi_U \colon \L\U \to \CLF(\LC\U)$, given by $\xi_U(C) = \up C$. For naturality, let $(P, Q) \colon \U \to \U'$ be an $\Urq$-morphism and let $(f, g) = \DD(P, Q)$. By \cref{def: fPQ and gPQ},
\[
f(D) = \bigcap \{ C' \in \LC\U' : D \subseteq \Box_P C'\} \quad\textrm{and}\quad g(E) = \bigcap \{ \phi'C' : C' \in \L\U', E \subseteq \Box_Q \phi'C' \}.
\]
We show that the following diagram  commutes.
\[
\begin{tikzcd}[column sep = 3em]
\L\U' \arrow[r, "\Box_P"] \arrow[d, "\xi_{\U'}"'] & \L\U \arrow[d, "\xi_\U"] \\
\CLF(\LC\U') \arrow[r, "f^{-1}"'] & \CLF(\LC\U)
\end{tikzcd}
\]
Let $C' \in \L\U'$. Then 
\begin{align*}
f^{-1}(\xi_{\U'}(C')) &= f^{-1}(\up C') = \{ D \in \LC\U : C' \le f(D) \} = \{ D \in \LC\U : \Box_P C' \le D \} \\
&= \up (\Box_P C') = \xi_\U(\Box_PC'),
\end{align*}
where the third equality follows from \cref{f(D) in C}.
Thus, $\xi \colon \LL_\Urq \to \LL_\HD \circ \DD$ is natural, and hence the triangle above commutes up to natural isomorphism.
\end{proof}


\section{Illustrating similarities and differences} \label{sec: similarities and differences}

In this final section, we discuss how each of the dualities for bounded lattices restricts to the distributive case. As we will see, in this case, Gehrke-van Gool, Urquhart, and \Plos\ dualities collapse to 
Priestley duality, while Celani-Gonz\'{a}lez and Hartung dualities to Stone duality for distributive lattices.
We then consider 
the action of each of the dual equivalences on a non-distributive lattice, thus
showcasing the similarities and differences of the various dualities for bounded lattices considered in the literature.

\subsection{The distributive case}

Let $A$ be a bounded distributive lattice and $X$ its Priestley space. We recall (see, e.g., \cite{Pri84} or \cite{BBGK10}) that there is an isomorphism between $\Filt(A)$ and the frame $\ClosedUp(X)$ of closed upsets of $X$ (ordered by $\supseteq$). By identifying $A$ with $\ClopUp(X)$, this isomorphism is given by
\[
\mathcal F \mapsto \bigcap \{ U \in \ClopUp(X) : U \in \mathcal F\} \quad\textrm{and}\quad C \mapsto \{ U \in \ClopUp(X) : C \subseteq U\}. 
\]
Similarly, there is an isomorphism between $\Idl(A)$ and the frame $\OpenUp(X)$ of open upsets of $X$, given by
\[
\mathcal I \mapsto \bigcup \{ U \in \ClopUp(X) : U \in \mathcal I \} \quad\textrm{and}\quad O \mapsto \{ U \in \ClopUp(X) : U \subseteq O\}.
\]

\subsection*{Jipsen-Moshier approach:} The lattice $A$ is sent to $(\Filt(A),\sigma)$, where $\sigma$ is the Scott topology on $\Filt(A)$. Thus, the Jipsen-Moshier dual of $A$ can be interpreted as $\ClosedUp(X)$ with the Scott topology, whose basis is given by $\{ \up U : U \in \ClopUp(X) \}$, where $\up$ is calculated with respect to $\supseteq$. 

\subsection*{BDGM-approach:}
The lattice $A$ is sent to $(\Filt(A), \lambda)$, where $\lambda$ is the Lawson topology on $\Filt(A)$. Thus, the BDGM-dual of $A$ can be interpreted as $\ClosedUp(X)$ with the Lawson topology, which is generated by \[
\{ \up U : U \in \ClopUp(X) \} \cup \{  -\up U : U \in \ClopUp(X) \},
\]
where again $\up$ is calculated with respect to $\supseteq$.

\subsection*{Hartonas approach:}
The Hartonas dual is the same as the BDGM-dual.

\subsection*{Celani-Gonz\'{a}lez approach:} The lattice $A$ is sent to $\C(A)=(X_A,\K_A)$ (see \cref{rem: CG = Lat}). Since $A$ is distributive, $X_A$ is exactly the set of prime filters of $A$. Therefore, $X_A$ is the Stone dual of $A$ and $\K_A$ is the collection of complements of compact opens of the spectral topology on $X_A$. Thus, the topology on $X_A$ is the co-compact topology of the spectral topology on $X_A$ (which, when viewing $X_A$ as a Priestley space, corresponds to the topology of open downsets).

\subsection*{Dunn-Hartonas approach:} The lattice $A$ is sent to the triple $(\Filt(A), R, \Idl(A))$, where recalling from \cref{sec: DH} that we are working with the complement of their relation, $R$ is given by $F \rel{R} I$ iff $ F \cap  I = \varnothing$. Interpreting $F$ as $\{ U \in \ClopUp(X) : C \subseteq U \}$ for some $C \in \ClosedUp(X)$ and $I$ as $\{ V \in \ClopUp(X) : V \subseteq O \}$ for some $O \in \OpenUp(X)$,  
\begin{align*}
C \subseteq O &\iff \bigcap \{ U \in \ClopUp(X) : C \subseteq U \} \subseteq \bigcup \{ V \in \ClopUp(X) : V \subseteq O \} \\
&\iff F \cap I \ne \varnothing, 
\end{align*}
where the second equivalence follows by compactness.
Thus, the Dunn-Hartonas dual of $A$ can be interpreted as $(\ClosedUp(X), R, \OpenUp(X))$, where $R$ is given by $C \rel{R} O$ iff $C \not\subseteq O$.

\subsection*{Gehrke-van Gool approach:} The lattice $A$ is sent to $(\Filt(A)_p, R_p, \Idl(A)_p)$, where $ F \rel{R_p}  I$ iff $F \cap I = \varnothing$. Since $A$ is distributive, a filter is d-prime iff it is prime. Thus, $\Filt(A)_p = X$. Moreover, $\Idl(A)_p := Y$ is the set of prime ideals of $A$. Therefore, the Gehrke-van Gool dual of $A$ is $(X, R, Y)$, where  
$x \rel{R} y$ iff $x \cap y = \varnothing$
for each $x\in X$ and $y \in Y$. Since there is an order-reversing bijection $Y \to X$ sending $y$ to $-y$, we may identify $(X,R,Y)$ with $X$, where $R$ is identified with the order on $X$ (see \cite[p.~452]{GvG14}), and hence interpret the Gehrke-van Gool dual as the Priestley dual of $A$.

\subsection*{Hartung approach:} The lattice $A$ is sent to $(\Filt(A)_0, R_0, \Idl(A)_0)$, where  recalling from \cref{Hs dual relation} that we are working with the complement of the Hartung relation, $\Filt(A)_0$ is the set of $I$-maximal filters for some ideal $I$, $\Idl(A)_0$ is the set of $F$-maximal ideals for some filter $F$, and $F \rel{R_0} I$ iff $F \cap I = \varnothing$ (see \cite[Def.~2.1.6]{Har92}). Because $A$ is distributive, 
these are precisely the sets of prime filters and prime ideals of $A$ (see \cite[p.~278]{Har92}). Therefore, $\Filt(A)_0 = \Filt(A)_p$,  $\Idl(A)_0 = \Idl(A)_p$, and $R_0 = R_p$. Thus, as in the Gehrke-van Gool approach, we may identify the Hartung dual of $A$ with $\Filt(A)_0$, but the topology of this space is the topology of open downsets. Consequently, 
the topology of the Hartung dual of $A$ may be interpreted as the Stone dual of $A$ with the co-compact topology of the spectral topology.

\subsection*{Urquhart approach:} The lattice $A$ is sent to $(Z, \le_1, \le_2)$, where $Z$ is the set of maximal filter-ideal pairs of $A$, $(F, I) \le_1 (F', I')$ iff $F \subseteq F'$, and $(F, I) \le_2 (F', I')$ iff $I \subseteq I'$ (see \cite[p.~47]{Urq78}). Since $A$ is distributive, 
$Z = \{ (x, -x) : x \in X \}$. Thus, $Z$ may be identified with $X$, the quasi-order $\le_1$ on $Z$ with the order $\le$ on $X$, and $\le_2$ with $\ge$ (see \cite[Thm.~5]{Urq78}). Consequently, we may interpret the Urquhart dual as the Priestley dual of $A$.

\subsection*{\Plos\ approach:} The \Plos\ dual of $A$ utilizes the same space as Urquhart (see \cite[p.~76]{Plo95}), but replaces the two quasi-orders with a relation $R$. Since $A$ is distributive, it follows from the Urquhart approach that each  maximal filter-ideal pair is of the form $(x,-x)$ for some $x \in X$. Thus, by \cref{def: R from quasi-orders}, $(x,-x) \rel{R} (x',-x')$ iff there is $x'' \in X$ such that $x \subseteq x''$ and $-x' \subseteq -x''$, which happens iff $x\le x'$. Consequently, in view of the Urquhart approach, we may interpret the \Plos\ dual as the Priestley dual of $A$. 

\bigskip

We now turn our attention to morphisms. 
Let $\alpha \colon A' \to A$ be a morphism of bounded distributive lattices. In  each of Jipsen-Moshier, BDGM, and Hartonas dualities, $\alpha$ is sent to 
$\alpha^{-1} \colon \Filt(A) \to \Filt(A')$. 
The Dunn-Hartonas dual of $\alpha$ is 
the $\HD$-morphism
\[
(f,g) \colon (\Filt(A), R, \Idl(A)) \to (\Filt(A'), R', \Idl(A')), 
\]
where $f = \alpha^{-1} \colon \Filt(A) \to \Filt(A')$ and $g = \alpha^{-1} \colon \Idl(A) \to \Idl(A')$. Let  $\varphi \colon X \to X'$ be the Priestley dual of $\alpha \colon A' \to A$. Interpreting $F$ as $\{ U \in \ClopUp(X) : C \subseteq U \}$ for some $C \in \ClosedUp(X)$, $\alpha^{-1}[F]$ is interpreted as the least closed upset of $X'$ containing $\varphi[C]$. Also, interpreting $I$ as $\{ V \in \ClopUp(X) : V \subseteq O \}$ for some $O \in \OpenUp(X)$, $\alpha^{-1}[I]$ is interpreted as the greatest open upset of $X'$ contained in $\varphi[O]$. Thus, 
we may view the Dunn-Hartonas dual of $\alpha$ as the $\HD$-morphism
\[
(f,g) \colon (\ClosedUp(X),R,\OpenUp(X)) \to (\ClosedUp(X'),R',\OpenUp(X')), 
\]
where $f$ and $g$ are given by $f(C)=\up \varphi[C]$ for each $C\in\ClosedUp(X)$ and $g(O) = - \down \varphi[-O]$ for each $O\in\OpenUp(X)$.

The $\GvG$-morphism corresponding to $\alpha$ is
\[
(S, T) \colon (\Filt(A)_p, R_p, \Idl(A)_p) \to (\Filt(A')_p, R_p', \Idl(A')_p). 
\]
With the identifications above, we may view these as the relation $S \subseteq X \times X'$ given by 
\[
x \rel{S} x' \iff \varphi(x) \subseteq x'.
\]
This relation, in turn, may be identified with $\varphi$. Thus, the Gehrke-van Gool dual of $\alpha$ may be thought of as its Priestley dual $\varphi$. The same applies to the Hartung, Urquhart, and \Plos\ duals of $\alpha$. 


\subsection{Illustrating the functors}

We conclude by illustrating the action of the functors described in this paper on the various duals of the lattice $A$ consisting of a countable antichain $A_0$ together with $0, 1$ shown in the following picture:  
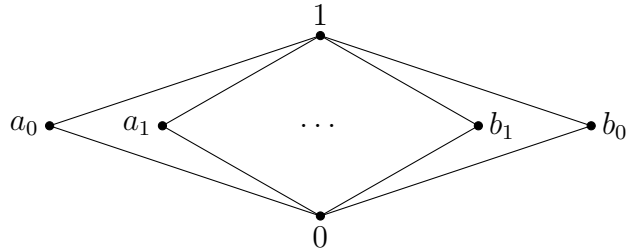
\begin{center}
\newcommand\sca{2}
\begin{tikzpicture}[scale = \scale]
\draw [fill] (0*\sca,2) circle[radius = \rad];
\draw [fill] (1.25*\sca,2) circle[radius = \rad];
\draw [fill] (4.75*\sca,2) circle[radius = \rad];
\draw [fill] (6*\sca,2) circle[radius = \rad];
\draw [fill] (3*\sca,0) circle[radius = \rad];
\draw [fill] (3*\sca,4) circle[radius = \rad];
\node at (3*\sca,2) {$\dots$};
\draw (3*\sca,0) -- (0*\sca,2) -- (3*\sca,4);
\draw (3*\sca,0) -- (1.25*\sca,2) -- (3*\sca,4);
\draw (3*\sca,0) -- (4.75*\sca,2) -- (3*\sca,4);
\draw (3*\sca,0) -- (6*\sca,2) -- (3*\sca,4);
\node [below] at (3*\sca,0) {$0$};
\node [above] at (3*\sca,4) {$1$};
\node [left] at (0*\sca,2) {$a_0$};
\node [left] at (1.25*\sca,2) {$a_1$};
\node [right] at (6*\sca,2) {$b_0$};
\node [right] at (4.75*\sca,2) {$b_1$};
\end{tikzpicture}
\end{center}

This lattice is sometimes denoted by $M_\infty$ (where $M$ stands for modular and $\infty$ for the infinite antichain in the middle), and was  considered by Gehrke and van Gool 
\cite[Ex.~4.1]{GvG14} in contrasting their duality with Hartung duality.
Letting $X = \Filt(A)$, we see that $X = \{ \up c : c \in A\}$ is dually isomorphic to $A$, and $Y = \Idl(A)$ is isomorphic to $A$.

\subsection*{Jipsen-Moshier, BDGM, and Hartonas duals}

The Jipsen-Moshier dual of $A$ is $X$ (see \cref{Sec 15 figure}) with the Scott topology. Because $K(X) = X$, the Scott topology is simply the topology of upsets of $X$. The BDGM-dual of $A$ is $X$ with the Lawson topology. Since the lower topology has subbasis $\{ -\up x : x \in X\}$, it follows that the lower topology consists of all cofinite subsets of $X$ containing $\up 1$. Consequently, each point in $\{ \up c : c \in A_0\}$ is isolated, as is $\up 0$. Therefore, the Lawson topology is the one-point compactification of the discrete space $X - \{ \up 1\}$. The Hartonas dual of $A$ is the same as the BDGM-dual.

\subsection*{Celani-Gonz\'{a}lez dual} 

The set of meet-irreducible elements of $X$ is $X_m = \{ \up c : c \in A_0\}$. The topology on $X_m$ is the restriction of the lower 
topology on $X$, and so is the cofinite topology on $X_m$. Because each element of $X$ is compact, the subbasis $\K$ of the topology on $X_m$ is $\{ X_m - \up x : x \in X\}$.

\subsection*{Dunn-Hartonas dual} 

The Dunn-Hartonas dual of $A$ is $(X, R, Y)$. By identifying $Y$ with $A$ via $d \mapsto \down d$, from the previous case we see that the Lawson topology on $Y$ is the one-point compactification of the discrete space $Y - \{0\}$ and $R$ is given by 
\[
\up c \rel{R} d \iff \up c \cap \down d = \varnothing \iff c \not\le d.
\]

\subsection*{Gehrke-van Gool dual} 

It is straightforward to see that there are no finite distributive meets of incomparable elements in $X$. Consequently, $X_p = X - \{ \up 0 \}$. Similarly, $Y_p = Y - \{1\}$. Because the topologies on $X_p$ and $Y_p$ are the subspace topologies of the Lawson topologies, the topology on $X_p$ is the one-point compactification of $X_p - \{ \up 1 \}$, and similarly for $Y_p$. Moreover, $R_p$ is the restriction of $R$ to $X_p \times Y_p$.

\subsection*{Hartung dual} 

It is easy to see that $\up 1$ is not $d$-maximal for any $d \in A$. Moreover, $\up a_n$ is $b_0$-maximal and $\up b_n$ is $a_0$-maximal for each $n$. Therefore, $X_0 = X_p - \{ \up 1 \} = X - \{ \up 0, \up 1 \}$. Similarly, $Y_0 = Y - \{0, 1\}$. The topologies on $X_0$ and $Y_0$ are the restrictions of the open downset topologies, which are exactly the cofinite topologies, and the relation $R_0$ is the restriction of $R_p$.  Thus, $R_0$ is given by $\up c \rel{R_0} d$ iff $c \ne d$. The following picture describes $X$, $X_p$, and $X_0$:

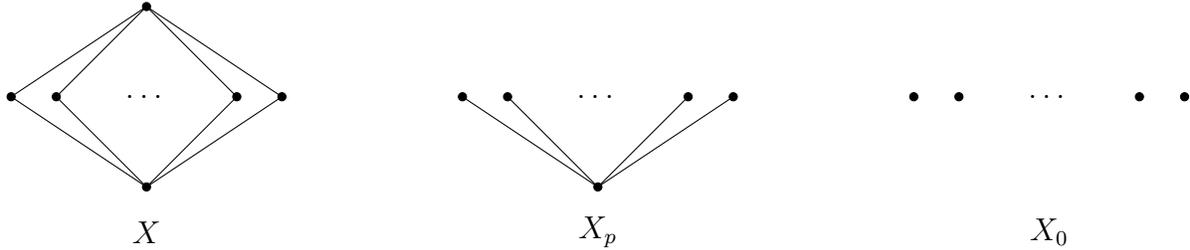
\begin{figure}[ht]
\[
\begin{tikzpicture}[scale = \scale]
\draw [fill] (0,2) circle[radius = \rad];
\node [left] at (0,2) {$\up a_0$};
\draw [fill] (1,2) circle[radius = \rad];
\node [right] at (1.1,2) {$\up a_1$};
\draw [fill] (5,2) circle[radius = \rad];
\node [left] at (4.9,2) {$\up b_1$};
\draw [fill] (6,2) circle[radius = \rad];
\node [right] at (6,2) {$\up b_0$};
\draw [fill] (3,0) circle[radius = \rad];
\node [below] at (3,0) {$\up 1$};
\draw [fill] (3,4) circle[radius = \rad];
\node [above] at (3,4) {$\up 0$};
\node at (3,2) {$\dots$};
\draw (3,0) -- (0,2) -- (3,4);
\draw (3,0) -- (1,2) -- (3,4);
\draw (3,0) -- (5,2) -- (3,4);
\draw (3,0) -- (6,2) -- (3,4);
\node at (3,-2) {$X$};
\draw [fill] (0+\spac,2) circle[radius = \rad];
\draw [fill] (1+\spac,2) circle[radius = \rad];
\draw [fill] (5+\spac,2) circle[radius = \rad];
\draw [fill] (6+\spac,2) circle[radius = \rad];
\draw [fill] (3+\spac,0) circle[radius = \rad];
\node at (3+\spac,2) {$\dots$};
\draw (3+\spac,0) -- (0+\spac,2);
\draw (3+\spac,0) -- (1+\spac,2);
\draw (3+\spac,0) -- (5+\spac,2);
\draw (3+\spac,0) -- (6+\spac,2);
\node at (3+\spac,-2) {$X_p$};
\draw [fill] (0+2*\spac,2) circle[radius = \rad];
\draw [fill] (1+2*\spac,2) circle[radius = \rad];
\draw [fill] (5+2*\spac,2) circle[radius = \rad];
\draw [fill] (6+2*\spac,2) circle[radius = \rad];
\node at (3+2*\spac,2) {$\dots$};
\node at (3+2*\spac,-2) {$X_0$};
\end{tikzpicture}
\]
\caption{$X$, $X_p$, and $X_0$} \label{Sec 15 figure}
\end{figure}

Observe that $X_0$ is exactly the Celani-Gonz\'{a}lez space $X_m$. More generally, if $X$ is a coherent lattice, then $X_0 \subseteq X_m \subseteq X_p$. To see the second inclusion, let $x \in X_m$ and let $k_1 \wedge \dots \wedge k_n$ be a distributive meet in $K(X)$ with $k_1 \wedge \dots \wedge k_n \le x$. Then 
\[
x = x \vee (k_1 \wedge \dots \wedge k_n) = (x \vee k_1) \wedge \dots \wedge (x \vee k_n),
\]
 where the second equality can be proved along the same lines as in the proof of \cref{lem: X0 the same}.
Since $x \in X_m$, we have $x = x \vee k_i$ for some $i$, and so $k_i \le x$. Thus, $x \in X_p$.

To see the first inclusion, let $Y = \OF(X)$ and define 
\[
x \rel{R} U \Longleftrightarrow x \notin U.
\]
Then $(X,R,Y)\in\HD$ by \cref{lem: triple in HD}.
Now, let $x \in X_0$. Then $x \in \max R^{-1}[U]$ for some $U \in Y$. If $x \notin X_m$, then $x = x_1 \wedge x_2$ for some $x_1,x_2 \in X$ with $x < x_1,x_2$. 
Since $x \in \max R^{-1}[U]$, we see that $x_i \nr{R} U$, so $x_1, x_2 \in U$, and hence $x = x_1 \wedge x_2 \in U$. Thus, $x 
\nr{R} U$.
The obtained contradiction 
proves that $x \in X_m$.

Our example shows that the inclusion of $X_m$ into $X_p$ may in general be proper. We currently don't have an example showing that the inclusion $X_0$ into $X_m$ may also be proper.

\subsection*{Urquhart dual} 

If $c, d \in A_0$ with $c \ne d$, then $\up c$ is $d$-maximal and vice versa. From this it follows that
$Z = \{ (\up c, d) : c, d \in A_0, c \ne d\}.$
The topology on $Z$ is the subspace topology induced from the inclusion $Z \subseteq X_0 \times Y_0$, where $X_0 \times Y_0$ is given the product topology (of two cofinite topologies). 
The quasi-order $\le_1$ is given by $(\up c, d) \le_1 (\up c', d')$ iff $\up c \subseteq \up c'$, which happens iff $c = c'$ since $c, c' \in A_0$. Similarly, $(\up c, d) \le_2 (\up c', d')$ iff $d = d'$.

\subsection*{\Plos\ dual}

The \Plos\ relation $R$ on $Z$ is given by $(\up c, d) \rel{R} (\up c', d')$ iff there is $(\up c'', d'') \in Z$ with $(\up c, d) \le_1 (\up c'', d'')$ and $(\up c', d') \le_2 (\up c'', d'')$. From the description of $\le_1, \le_2$ above, we see that $R$ is the equality relation on $Z$.

\bigskip

Finally, we illustrate how the various functors act on morphisms by considering the inclusion $\alpha \colon A' \to A$, where $A'$ is the bounded sublattice $\{0, a_0, b_0, 1\}$ of $A$.

\begin{center}
\newcommand\sca{1.2}
\newcommand\addd{5}
\newcommand\subt{-15}
\begin{tikzpicture}[scale = \scale]
\draw [fill] (0*\sca+\subt,2) circle[radius = \rad];
\draw [fill] (6*\sca+\subt,2) circle[radius = \rad];
\draw [fill] (3*\sca+\subt,0) circle[radius = \rad];
\draw [fill] (3*\sca+\subt,4) circle[radius = \rad];
\draw (3*\sca+\subt,0) -- (0*\sca+\subt,2) -- (3*\sca+\subt,4);
\draw (3*\sca+\subt,0) -- (6*\sca+\subt,2) -- (3*\sca+\subt,4);
\draw[->] (3*\sca+\subt+\addd,0) -- (3*\sca-\addd,0);
\draw[->] (3*\sca+\subt+\addd,4) -- (3*\sca-\addd,4);
\node [below] at (3*\sca+\subt,0) {$0$};
\node [above] at (3*\sca+\subt,4) {$1$};
\node [left] at (0*\sca+\subt,2) {$a_0$};
\node [right] at (6*\sca+\subt,2) {$b_0$};
\draw [fill] (0*\sca,2) circle[radius = \rad];
\draw [fill] (1.25*\sca,2) circle[radius = \rad];
\draw [fill] (4.75*\sca,2) circle[radius = \rad];
\draw [fill] (6*\sca,2) circle[radius = \rad];
\draw [fill] (3*\sca,0) circle[radius = \rad];
\draw [fill] (3*\sca,4) circle[radius = \rad];
\node at (3*\sca,2) {$\dots$};
\draw (3*\sca,0) -- (0*\sca,2) -- (3*\sca,4);
\draw (3*\sca,0) -- (1.25*\sca,2) -- (3*\sca,4);
\draw (3*\sca,0) -- (4.75*\sca,2) -- (3*\sca,4);
\draw (3*\sca,0) -- (6*\sca,2) -- (3*\sca,4);
\node [below] at (3*\sca,0) {$0$};
\node [above] at (3*\sca,4) {$1$};
\node [left] at (0*\sca,2) {$a_0$};
\node [left] at (1.25*\sca,2) {$a_1$};
\node [right] at (6*\sca,2) {$b_0$};
\node [right] at (4.75*\sca,2) {$b_1$};
\end{tikzpicture}
\end{center}

\subsection*{Jipsen-Moshier, BDGM, and Hartonas duals} Recalling that $X = \Filt(A)$ and $X' = \Filt(A')$, in all three of these cases $\alpha \colon A' \to A$ induces the morphism $f \colon X \to X'$ given by $f(F) = \alpha^{-1}(F)$ for each $F \in X$.

\subsection*{Dunn-Hartonas dual} Let $\D = (X, R, Y)$ and $\D' = (X', R', Y')$ be the Dunn-Hartonas duals of $A$ and $A'$, where using the above identifications, $Y = A$ and $Y' = A'$. Then the Dunn-Hartonas dual of $\alpha$ is the $\HD$-morphism $(f, g) \colon \D \to \D'$, where $f(F) = \alpha^{-1}(F)$ for each $F \in X$ and $g$ is given by 
\[
g(d) = \begin{cases} d & \textrm{if }d \in A', \\ 0 & \textrm{otherwise}. \end{cases}
\] 

\subsection*{Gehrke-van Gool dual} Let $\G$ and $\G'$ be the Gehrke-van Gool duals of $A$ and $A'$. The description of $\G$ is given above, while $\G' = (X'_p, R'_p, Y'_p)$, where $X'_p = \{ \up a_0, \up b_0\}$ and $Y'_p = \{ a_0, b_0\}$.  
The $\GvG$-morphism $(S, T)$ corresponding to $\alpha$ is given by
\[
\up c \rel{S} \up c' \textrm{ if }  c \in \{a_0,b_0\} \Longrightarrow c' \le c \quad \textrm{and} \quad d \rel{T} d' \textrm{ if } d\in\{a_0,b_0\} \Longrightarrow d \le d'.
\]

\subsection*{Celani-Gonz\'{a}lez dual}

The Celani-Gonz\'{a}lez duals of $A$ and $A'$ are $X_m$ and $X_m'$, where $X_m$ is described above and $X_m' = X_p'$. The $\CG$-morphism corresponding to $\alpha$ is then the  restriction of $S$ to $X_m \times X_m'$.

\subsection*{Hartung dual} Let $\H$ and $\H'$ be the Hartung duals of $A$ and $A'$. The description of $\H$ is given above and $\H' = \G'$. 
The Hartung morphism $(S_0, T_0)$ corresponding to $\alpha$ is given by $S_0 = S \cap (X_0 \times X_0')$ and $T_0 = T \cap (Y_0 \times Y_0')$.

\subsection*{Urquhart and \Plos\ duals} Let $\U$ and $\U'$ be the Urquhart duals of $A$ and $A'$. The description of $\U$ is given above, while $\U' = (Z', \le_1', \le_2')$, where $Z' = \{ (\up a_0, b_0), (\up b_0, a_0)\}$. The two quasi-orders $\le_1', \le_2'$ are both the equality relation. The Urquhart morphism $(P, Q)$ corresponding to $\alpha$ is given by 
\begin{align*}
(\up c, d) \rel{P} (\up c', d') &\iff c \in \{ a_0,b_0 \} \mbox{ implies } c' \le c; \\
(\up c, d) \rel{Q} (\up c', d') &\iff d \in \{ a_0,b_0 \} \mbox{ implies } d \le d'.
\end{align*}
Finally, let $\Pl$ and $\Pl'$ be the \Plos\ duals of $A$ and $A'$. Since the \Plos\ dual of $\alpha$ is $(P, Q)$, the $\Plo$-morphism is the same as the Urquhart morphism described above.

\section*{Acknowledgement}

We would like to thank the referee for useful comments which have improved the readability of the paper.

\newcommand{\etalchar}[1]{$^{#1}$}
\def\cprime{$'$}


\begin{thebibliography}{BDdGM24}

\bibitem[Bal84]{Bal84}
R.~N. Ball.
\newblock Distributive {C}auchy lattices.
\newblock {\em Algebra Universalis}, 18(2):134--174, 1984.

\bibitem[BBGK10]{BBGK10}
G.~Bezhanishvili, N.~Bezhanishvili, D.~Gabelaia, and A.~Kurz.
\newblock Bitopological duality for distributive lattices and {H}eyting algebras.
\newblock {\em Math. Structures Comput. Sci.}, 20(3):359--393, 2010.

\bibitem[BBH15]{BBH15}
G.~Bezhanishvili, N.~Bezhanishvili, and J.~Harding.
\newblock Modal compact {H}ausdorff spaces.
\newblock {\em J. Logic Comput.}, 25(1):1--35, 2015.

\bibitem[BDdGM24]{BDGM24}
N.~Bezhanishvili, A.~Dmitrieva, J.~de~Groot, and T.~Moraschini.
\newblock Positive modal logic beyond distributivity.
\newblock {\em Ann. Pure Appl. Logic}, 175(2):Paper No. 103374, 36, 2024.

\bibitem[BF48]{BF48}
G.~Birkhoff and O.~Frink, Jr.
\newblock Representations of lattices by sets.
\newblock {\em Trans. Amer. Math. Soc.}, 64:299--316, 1948.

\bibitem[BGHJ19]{BGHJ19}
G.~Bezhanishvili, D.~Gabelaia, J.~Harding, and M.~Jibladze.
\newblock Compact {H}ausdorff spaces with relations and {G}leason spaces.
\newblock {\em Appl. Categ. Structures}, 27(6):663--686, 2019.

\bibitem[Bir79]{Bir79}
G.~Birkhoff.
\newblock {\em Lattice theory}, volume~25 of {\em American Mathematical Society Colloquium Publications}.
\newblock Providence, R.I., third edition, 1979.

\bibitem[BJ11]{BJ11}
G.~Bezhanishvili and R.~Jansana.
\newblock Priestley style duality for distributive meet-semilattices.
\newblock {\em Studia Logica}, 98(1-2):83--122, 2011.

\bibitem[BJ13]{BJ13}
G.~Bezhanishvili and R.~Jansana.
\newblock Esakia style duality for implicative semilattices.
\newblock {\em Appl. Categ. Structures}, 21(2):181--208, 2013.

\bibitem[BL70]{BL70}
G.~Bruns and H.~Lakser.
\newblock Injective hulls of semilattices.
\newblock {\em Canad. Math. Bull.}, 13:115--118, 1970.

\bibitem[BPP14]{BPP14}
R.~N. Ball, J.~Picado, and A.~Pultr.
\newblock Notes on exact meets and joins.
\newblock {\em Appl. Categ. Structures}, 22(5-6):699--714, 2014.

\bibitem[BPW16]{BPW16}
R.~N. Ball, A.~Pultr, and J.~{Walters Wayland}.
\newblock The {D}edekind {M}ac{N}eille site completion of a meet semilattice.
\newblock {\em Algebra Universalis}, 76(2):183--197, 2016.

\bibitem[Cel03a]{Cel03b}
S.~A. Celani.
\newblock Representation of {H}ilbert algebras and implicative semilattices.
\newblock {\em Cent. Eur. J. Math.}, 1(4):561--572, 2003.

\bibitem[Cel03b]{Cel03}
S.~A. Celani.
\newblock Topological representation of distributive semilattices.
\newblock {\em Sci. Math. Jpn.}, 58(1):55--65, 2003.

\bibitem[CG20]{CG20}
S.~A. Celani and L.~J. Gonz\'{a}lez.
\newblock A categorical duality for semilattices and lattices.
\newblock {\em Appl. Categ. Structures}, 28(5):853--875, 2020.

\bibitem[CJ99]{CJ99}
S.~A. Celani and R.~Jansana.
\newblock Priestley duality, a {S}ahlqvist theorem and a {G}oldblatt-{T}homason theorem for positive modal logic.
\newblock {\em Log. J. IGPL}, 7(6):683--715, 1999.

\bibitem[Cor75]{Cor75}
W.~H. Cornish.
\newblock On {H}. {P}riestley's dual of the category of bounded distributive lattices.
\newblock {\em Mat. Vesnik}, 12(27)(4):329--332, 1975.

\bibitem[CZ97]{CZ97}
A.~Chagrov and M.~Zakharyaschev.
\newblock {\em Modal logic}, volume~35 of {\em Oxford Logic Guides}.
\newblock The Clarendon Press, Oxford University Press, New York, 1997.

\bibitem[DP02]{DP02}
B.~A. Davey and H.~A. Priestley.
\newblock {\em Introduction to lattices and order}.
\newblock Cambridge University Press, New York, second edition, 2002.

\bibitem[DST19]{DST19}
M.~Dickmann, N.~Schwartz, and M.~Tressl.
\newblock {\em Spectral spaces}, volume~35 of {\em New Mathematical Monographs}.
\newblock Cambridge University Press, Cambridge, 2019.

\bibitem[Ern91]{Ern91}
M.~Ern\'{e}.
\newblock The {ABC} of order and topology.
\newblock In {\em Category theory at work} ({B}remen, 1990), volume~18 of {\em
  Res. Exp. Math.}, pages 57--83. Heldermann, Berlin, 1991.

\bibitem[Ern93]{Ern93}
M.~Ern\'{e}, \emph{Algebraic ordered sets and their generalizations}, Algebras and orders ({M}ontreal, {PQ}, 1991), NATO Adv. Sci. Inst. Ser. C: Math. Phys.
Sci., vol. 389, Kluwer Acad. Publ., Dordrecht, 1993, pp.~113--192.

\bibitem[Ern04]{Ern04}
M.~Ern\'{e}, \emph{General {S}tone duality}, vol. 137, 2004, IV Iberoamerican
Conference on Topology and its Applications, pp.~125--158.

\bibitem[Esa74]{Esa74}
L.~L. Esakia.
\newblock Topological {K}ripke models.
\newblock {\em Soviet Math. Dokl.}, 15:147--151, 1974.

\bibitem[Esa78]{Esa78}
L.~L. Esakia.
\newblock Semantical analysis of bimodal (tense) systems.
\newblock In {\em Logic, Semantics and Methodology}, pages 87--99 (Russian). Metsniereba Press, Tbilisi, 1978.

\bibitem[Esa19]{Esa19}
L.~L. Esakia.
\newblock {\em Heyting algebras. {D}uality theory}, volume~50 of Trends in Logic. {\em Translated from the Russian by A. Evseev. Edited by G. Bezhanishvili and W. Holliday}.
\newblock Springer, 2019.

\bibitem[GHK{\etalchar{+}}03]{GHKLMS03}
G.~Gierz, K.~H. Hofmann, K.~Keimel, J.~D. Lawson, M.~Mislove, and D.~S. Scott.
\newblock {\em Continuous lattices and domains}.
\newblock Cambridge University Press, Cambridge, 2003.

\bibitem[Gol89]{Gol89}
R.~Goldblatt.
\newblock Varieties of complex algebras.
\newblock {\em Ann. Pure Appl. Logic}, 44(3):173--242, 1989.

\bibitem[Gol20]{Gol20}
R.~Goldblatt.
\newblock Morphisms and duality for polarities and lattices with operators.
\newblock {\em J. Appl. Logics}, 7(6):1019--1072, 2020.

\bibitem[GvG14]{GvG14}
M.~Gehrke and S.~J. van Gool.
\newblock Distributive envelopes and topological duality for lattices via canonical extensions.
\newblock {\em Order}, 31(3):435--461, 2014.

\bibitem[Hal62]{Hal62}
P.~R. Halmos.
\newblock {\em Algebraic logic}.
\newblock Chelsea Publishing Co., New York, 1962.

\bibitem[Har92]{Har92}
G.~Hartung.
\newblock A topological representation of lattices.
\newblock {\em Algebra Universalis}, 29(2):273--299, 1992.

\bibitem[Har97]{Har97}
C.~Hartonas.
\newblock Duality for lattice-ordered algebras and for normal algebraizable logics.
\newblock {\em Studia Logica}, 58(3):403--450, 1997.

\bibitem[Har18]{Har18}
C.~Hartonas.
\newblock Stone duality for lattice expansions.
\newblock {\em Log. J. IGPL}, 26(5):475--504, 2018.

\bibitem[Har24]{Har24}
C.~Hartonas.
\newblock Choice-free dualities for lattice expansions: Application to logics with a negation operator.
\newblock {\em Studia Logica}, pages 1--46, 2024.

\bibitem[HD97]{HD97}
C.~Hartonas and J.~M. Dunn.
\newblock Stone duality for lattices.
\newblock {\em Algebra Universalis}, 37(3):391--401, 1997.

\bibitem[HK71]{HK71}
A.~Horn and N.~Kimura.
\newblock The category of semilattices.
\newblock {\em Algebra Universalis}, 1(1):26--38, 1971.

\bibitem[HMS74]{HMS74}
K.~H. Hofmann, M.~Mislove, and A.~Stralka.
\newblock {\em The {P}ontryagin duality of compact {${\rm O}$}-dimensional semilattices and its applications}.
\newblock Lecture Notes in Mathematics, Vol. 396. Springer-Verlag, Berlin-New York, 1974.

\bibitem[Hoc69]{Hoc69}
M.~Hochster.
\newblock Prime ideal structure in commutative rings.
\newblock {\em Trans. Amer. Math. Soc.}, 142:43--60, 1969.

\bibitem[Joh82]{Joh82}
P.~T. Johnstone.
\newblock {\em Stone spaces}, volume~3 of {\em Cambridge Studies in Advanced Mathematics}.
\newblock Cambridge University Press, Cambridge, 1982.

\bibitem[Mac37]{Mac37}
H.~M. MacNeille.
\newblock Partially ordered sets.
\newblock {\em Trans. Amer. Math. Soc.}, 42(3):416--460, 1937.

\bibitem[MB23]{MB23}
J.~McDonald and K~Bimbo.
\newblock Topological duality for orthomodular lattices.
\newblock Available at https://arxiv.org/abs/2208.07430, 2023.

\bibitem[MJ14]{MJ14a}
M.~A. Moshier and P.~Jipsen.
\newblock Topological duality and lattice expansions, {I}: {A} topological construction of canonical extensions.
\newblock {\em Algebra Universalis}, 71(2):109--126, 2014.

\bibitem[ML71]{Mac71}
S.~Mac~Lane.
\newblock {\em Categories for the working mathematician}.
\newblock Graduate Texts in Mathematics, Vol. 5. Springer-Verlag, New York, 1971.

\bibitem[Nac49]{Nac49}
L.~Nachbin.
\newblock On a characterization of the lattice of all ideals of a {B}oolean ring.
\newblock {\em Fund. Math.}, 36:137--142, 1949.

\bibitem[Plo95]{Plo95}
M.~Plo\v{s}\v{c}ica.
\newblock A natural representation of bounded lattices.
\newblock {\em Tatra Mt. Math. Publ.}, 5:75--88, 1995.

\bibitem[Pri70]{Pri70}
H.~A. Priestley.
\newblock Representation of distributive lattices by means of ordered {S}tone spaces.
\newblock {\em Bull. London Math. Soc.}, 2:186--190, 1970.

\bibitem[Pri72]{Pri72}
H.~A. Priestley.
\newblock Ordered topological spaces and the representation of distributive lattices.
\newblock {\em Proc. London Math. Soc. (3)}, 24:507--530, 1972.

\bibitem[Pri84]{Pri84}
H.~A. Priestley.
\newblock Ordered sets and duality for distributive lattices.
\newblock In {\em Orders: description and roles ({L}'{A}rbresle, 1982)}, volume~99 of {\em North-Holland Math. Stud.}, pages 39--60. North-Holland, Amsterdam, 1984.

\bibitem[Sto36]{Sto36}
M.~H. Stone.
\newblock The theory of representations for {B}oolean algebras.
\newblock {\em Trans. Amer. Math. Soc.}, 40(1):37--111, 1936.

\bibitem[Sto37]{Sto37}
M.~H. Stone.
\newblock Algebraic characterizations of special {B}oolean rings.
\newblock {\em Fund. Math.}, 29:223--302, 1937.

\bibitem[Urq78]{Urq78}
A.~Urquhart.
\newblock A topological representation theory for lattices.
\newblock {\em Algebra Universalis}, 8(1):45--58, 1978.

\end{thebibliography}
\end{document}